\documentclass{amsart}

 \usepackage[all]{xy}
\usepackage{lmodern}
\usepackage[dvipsnames,svgnames,x11names,hyperref]{xcolor}
\usepackage{mathtools,url,graphicx,verbatim,amssymb,enumerate,stmaryrd,nicefrac}
\usepackage[pagebackref,colorlinks,citecolor=Mahogany,linkcolor=Mahogany,urlcolor=Mahogany,filecolor=Mahogany]{hyperref}
\usepackage{microtype}
\usepackage[margin=1.47in]{geometry}
\usepackage{setspace}
\usepackage{tikz, tikz-cd}
\usetikzlibrary{matrix,calc,arrows,patterns,decorations.pathreplacing}
\usepackage{caption}

\newtheorem*{corollary*}{Corollary}

\theoremstyle{definition}

\DeclareMathOperator*{\hocolim}{hocolim}   

\title{The homotopy type of the PL cobordism category. I}

\author{Mauricio Gomez Lopez}
\email{gomezlom@lafayette.edu}
\address{Department of Mathematics \\
	       Lafayette College, Pardee Hall\\
	       Easton, PA, 18042 \\USA}

\date{\today}

\begin{document}

\maketitle 

\begin{abstract}      
In this paper, we introduce a bordism category $\mathsf{Cob}_d^{\mathrm{PL}}$ whose objects are bundles of closed $(d-1)$-dimensional piecewise linear manifolds and whose morphisms are bundles of $d$-dimensional piecewise linear cobordisms. In the main theorem of this article, we show that the classifying space $B\mathsf{Cob}_d^{\mathrm{PL}}$  is weak homotopy equivalent to an infinite loop space. We regard $\mathsf{Cob}_d^{\mathrm{PL}}$  as the piecewise linear analogue of the category of smooth cobordisms which has been studied extensively in connection with the Madsen-Weiss Theorem, and the main result of this paper is a first step towards obtaining Madsen-Weiss type results in the context of PL topology. 
\end{abstract}  

\tableofcontents

\section{Introduction} \label{introduction}

\subsection{Background and statement of the main result} \label{intro1}

The goal of the present paper and the recent article \cite{Gomez2} is to prove a piecewise linear analogue of the following theorem by Galatius, Madsen, Tillmann, and Weiss \cite{GMTW}. 

\theoremstyle{plain}  \newtheorem*{mainfour}{Theorem (Galatius, Madsen, Tillmann, Weiss)}

\begin{mainfour}
There is a weak homotopy equivalence 
\[
B\mathsf{Cob}_d \simeq \Omega^{\infty - 1} \mathbf{MT}O(d).
\]
\end{mainfour}

In this statement, $\mathbf{MT}O(d)$ is the Madsen-Tillmann spectrum, whose $N$-th space is the Thom space
$\mathrm{Th}\big(\gamma_{d,N}^{\perp}\big)$ of the vector bundle orthogonal to the universal vector bundle
$\gamma_{d,N}$ defined over the Grassmannian $Gr_d(\mathbb{R}^N)$. The space on the left-hand side is the classifying space of the smooth cobordism category 
$\mathsf{Cob}_d$, which has played a key role in the study of stable phenomena of diffeomorphism groups.

Let us give an outline of the definition of the category $\mathsf{Cob}_d$ (see also \cite{GMTW}).
The set of objects 
of $\mathsf{Cob}_d$ is defined as follows:
let $B_N$ 
denote the set of all $(d-1)$-dimensional closed submanifolds  
$M$ in $\mathbb{R}^N$.
The natural inclusion $\mathbb{R}^N \hookrightarrow \mathbb{R}^{N+1}$
induces a map $B_N \hookrightarrow B_{N+1}$ and
the set of objects
$\mathrm{Ob}(\mathsf{Cob}_d)$ is then defined to be
the set of all tuples $(M,a)$ where
$a\in \mathbb{R}$ and $M$ is an element of
the colimit $B_{\infty}$ of the following sequence of maps
\[
\cdots \hookrightarrow B_N \hookrightarrow B_{N+1} \hookrightarrow \ldots. 
\]
A non-identity morphism between two objects $(M_0, a_0)$ and $(M_1, a_1)$ 
is a triple $(W,a_0,a_1)$, where $W$ is a $d$-dimensional compact submanifold
of a product $[a_0,a_1]\times \mathbb{R}^N$ for which there is a value $\epsilon > 0$ such that the following holds: 
\begin{itemize}
\item[(i)] $W\cap([a_0, a_0 + \epsilon) \times \mathbb{R}^N) = [a_0, a_0 + \epsilon) \times M_0$.
\item[(ii)] $W \cap ((a_1 - \epsilon, a_1]\times \mathbb{R}^N) =(a_1 - \epsilon, a_1]\times M_1$.
\item[(iii)] $\partial W = W \cap (\{a_0,a_1\} \times \mathbb{R}^N)$.
\end{itemize} 

Two morphisms $(W_1,a_0,a_1)$ and $(W_2,a_1,a_2)$ in $\mathsf{Cob}_d$ are composable 
if the outgoing boundary of $W_1$ is equal to 
the incoming boundary of $W_2$. In this case, their
composition is equal to the triple
\[
(W_1 \cup W_2, a_0, a_2).
\]
Galatius and Randal-Williams provided in \cite{GRW} a more elementary proof of the main theorem from \cite{GMTW} using scanning methods and  the spaces of manifolds $\Psi_d(\mathbb{R}^N)$ defined by Galatius in \cite{Ga}. For a fixed positive integer $N$, the underlying set of 
$\Psi_d(\mathbb{R}^N)$ is the set of all $d$-dimensional submanifolds of $\mathbb{R}^N$ which are closed as subspaces. As indicated in \cite{Ga}, the spaces $\Psi_d(\mathbb{R}^N)$  form a spectrum $\Psi_d$ by letting $N$ vary, and the main result from \cite{GMTW} is obtained in \cite{GRW} by showing that  there are two weak equivalences      
\begin{equation} \label{equiv1}
B\mathsf{Cob}_d \simeq \Omega^{\infty-1}\Psi_d
\end{equation}
\begin{equation} \label{equiv2}
\Omega^{\infty-1}\mathbf{MT}O(d) \stackrel{\simeq}{\longrightarrow} \Omega^{\infty-1}\Psi_d. 
\end{equation}
In this article, we introduce the $d$-\textit{dimensional PL cobordism category}, denoted by $\mathsf{Cob}_d^{\mathrm{PL}}$, which we consider to be the appropriate piecewise linear analogue of the category of smooth cobordisms studied in \cite{GMTW} and \cite{GRW}. Our main result gives a piecewise linear version of the equivalence (\ref{equiv1}) given above. More precisely, we prove the following theorem.

\theoremstyle{plain} \newtheorem*{mytheorem}{Theorem}                
 
\begin{mytheorem}  \label{mytheorem}

There is a weak homotopy equivalence
\[
B\mathsf{Cob}_d^{\mathrm{PL}} \stackrel{\simeq}{\longrightarrow} \Omega^{\infty-1}\Psi^{\mathrm{PL}}_d.
\]
\end{mytheorem}

The spectrum $\Psi^{\mathrm{PL}}_d$ that appears in this statement is a \textit{spectrum of PL manifolds,} analogous to the one defined by Galatius in \cite{Ga}. The $N$-th space of this spectrum is the geometric realization of a simplicial set $\Psi_d(\mathbb{R}^N)_{\bullet}$, whose role in this article is similar to the one played in \cite{GRW} by the space of smooth manifolds $\Psi_d(\mathbb{R}^N)$. As it  is the case with the equivalences indicated in (\ref{equiv1}) and (\ref{equiv2}), the superscript $\infty-1$ in  
$\Omega^{\infty-1}\Psi^{\mathrm{PL}}_d$ indicates we are taking the infinite loop space of the suspension
of $\Psi^{\mathrm{PL}}_d$.

\subsection{Outline of proof}  \label{intro2}

For the proof of our main theorem, we will follow a strategy similar to the one implemented by Galatius and Randal-Williams in \cite{GRW} to prove the weak equivalence (\ref{equiv1}). More precisely, we will introduce a filtration
\begin{equation} \label{introfil}
\psi_d(N,1)_{\bullet} \hookrightarrow \psi_d(N,2)_{\bullet} \hookrightarrow \ldots \hookrightarrow \psi_d(N,N)_{\bullet} = \Psi_d(\mathbb{R}^N)_{\bullet} 
\end{equation} 
similar to the one used in \cite{GRW} and we will show that
there are two weak equivalences
\begin{equation} \label{plequiv}
B\mathsf{Cob}_d^{\mathrm{PL}}(\mathbb{R}^N) \simeq |\psi_d(N,1)_{\bullet}|  
\end{equation}
\begin{equation} \label{plequiv2}
|\psi_d(N,1)_{\bullet}| \stackrel{\simeq}{\longrightarrow} \Omega^{N-1}|\Psi_d(\mathbb{R}^N)_{\bullet}|
\end{equation}
provided that $N-d\geq 3$. 
By allowing $N \rightarrow \infty$, we obtain the weak equivalence stated in the main theorem.
 In (\ref{plequiv}), 
$\mathsf{Cob}_d^{\mathrm{PL}}(\mathbb{R}^N)$ denotes a subcategory of $\mathsf{Cob}_d^{\mathrm{PL}}$ where morphisms (i.e., cobordisms between $d-1$ dimensional closed manifolds) are contained in $\mathbb{R}^N$. 

Even though our proof's strategy is similar to the one from \cite{GRW}, our methods differ significantly from those used by Galatius and Randal-Williams. This is especially the case for the proof of the weak equivalence (\ref{plequiv}). In the smooth case, the equivalence (\ref{plequiv}) follows quickly from the following fact:

\theoremstyle{plain} \newtheorem*{fact}{Fact}

\begin{fact}

Fix an $m$-dimensional smooth manifold $M$.
Let $W$ be a smooth $(d+m)$-dimensional submanifold
of $M\times \mathbb{R} \times (-1,1)^{N-1}$ which is closed as a subspace and
such that the projection $\pi: W \rightarrow M$
is a smooth submersion of codimension $d$. Moreover, fix a point
$\lambda_{0} \in M$. If $a\in \mathbb{R}$
is a regular value for the projection 
$x_1:\pi^{-1}(\lambda_{0}) \rightarrow \mathbb{R}$
onto the first component of $\mathbb{R}\times (-1,1)^{N-1}$, then
$a$ is also a regular value for 
$x_1:\pi^{-1}(\lambda) \rightarrow \mathbb{R}$ for all $\lambda$ sufficiently close
to $\lambda_{0}$. 

\end{fact} 

Let us discuss briefly why this result is true. Consider the fiber $W_{\lambda_0}$ of the projection $\pi$ over the point $\lambda_0$. Saying that $a$ is a regular value for the projection $x_1: W_{\lambda_0} \rightarrow \mathbb{R}$ is equivalent to saying that $W_{\lambda_0}$ is transversal to the hyperplane in $\mathbb{R}^N$ given by $x_1 = a$. Since transversality is an open condition in the $C^{\infty}$ topology, fibers of $\pi$ that are near to $W_{\lambda_0}$ will also be transversal to this hyperplane. In other words, $a$ is also a regular value for $x_1: W_{\lambda} \rightarrow \mathbb{R}$ for any $\lambda$ sufficiently close to $\lambda_0$. 

Unfortunately, the above statement is not valid for PL manifolds in general. A counterexample can be obtained as follows. First, 
we express the product $[0,1]\times \mathbb{R}$ as a union $[0,1]\times \mathbb{R} = A \cup B \cup C$, where $A = [0,1]\times (-\infty,0]$, $B$ is the convex hull of the points $(0,0)$, $(1,0)$, $(1, \frac{1}{2})$, and $C$ is the closure of the complement of $A \cup B$. We can also express the subsets $B$ and $C$ in terms of linear combinations as follows:
\begin{equation*}
B = \Big\{ \alpha\hspace{0.01cm}(1,0) + \beta(1, \frac{1}{2}) \hspace{0.03cm} : \hspace{0.03cm} 0 \leq \alpha \leq 1, \hspace{0.03cm} 0 \leq \beta \leq 1  \Big\}
\end{equation*}
\begin{equation*}
C = \Big\{ \lambda(0,1) + \beta(1, \frac{1}{2}) \hspace{0.03cm} : \hspace{0.03cm} 0 \leq \lambda, \hspace{0.03cm} 0 \leq \beta \leq 1  \Big\}
\end{equation*}
Next, we define a PL embedding $f:[0,1]\times \mathbb{R} \hookrightarrow [0,1]\times \mathbb{R} \times (-1,1)$ by gluing the following three linear maps: 

\begin{itemize}

\item[$\cdot$] $f_A: A \hookrightarrow   [0,1]\times \mathbb{R} \times (-1,1)$, defined by $f_A(t,x) = (t,x,0)$.

\item[$\cdot$] $f_B: B \hookrightarrow   [0,1]\times \mathbb{R} \times (-1,1)$, which maps $(1,0)$ and  $(1, \frac{1}{2})$ to $(1,0,0)$ and $(1, 0, \frac{1}{2})$ respectively. We then extend $f_B$ linearly to the rest of $B$. 

\item[$\cdot$] $f_C: C \hookrightarrow   [0,1]\times \mathbb{R} \times (-1,1)$, which maps $(0,1)$ and  $(1, \frac{1}{2})$ to $(0,1,0)$ and $(1, 0, \frac{1}{2})$ respectively. Again, we extend $f_C$ linearly to the rest of $C$.

\end{itemize}

The image of the embedding $f:[0,1]\times \mathbb{R} \hookrightarrow [0,1]\times \mathbb{R} \times (-1,1)$, which we will denote by $W$, is a PL submanifold of $ [0,1]\times \mathbb{R} \times (-1,1)$. Evidently, $W$ is a closed subspace (in the topological sense) of  $ [0,1]\times \mathbb{R} \times (-1,1)$. Since the map $f$ commutes with the projection onto the interval $[0,1]$,  the projection $\pi: W \rightarrow [0,1]$ is a trivial PL bundle with fiber $\mathbb{R}$. In particular, $\pi$ is a PL submersion of codimension 1 (see \S \ref{section2.2} for the definition of PL submersion). Moreover, the restriction of $f$ on $\{0\}\times \mathbb{R}$ commutes with the projection onto $\mathbb{R}$. Thus, if $W_0$ is the fiber over $0 \in [0,1]$ of the map $\pi: W \rightarrow [0,1]$, then $x =0$ is a regular value for the projection $W_0 \rightarrow \mathbb{R}$. However,  for any point $\lambda \in [0,1]$ with $\lambda \neq 0$, the pre-image of $x=0$ for the projection $W_{\lambda} \rightarrow \mathbb{R}$ is a codimension 0 submanifold. Therefore, $x=0$ cannot be a regular value for $W_{\lambda}\rightarrow \mathbb{R}$ if $\lambda \neq 0$ (see \S \ref{section4.1} for the definition of regular value for PL maps).

Since the fact stated above fails for PL manifolds, we cannot follow the argument from \cite{GRW} too closely. To fix this issue, we define a subsimplicial set $\psi_d^R(N,1)_{\bullet}$ of  $\psi_d(N,1)_{\bullet}$  for which the \textit{fiberwise regularity} described in the fact given above does hold. The definition of $\psi_d^R(N,1)_{\bullet}$ will be provided in \S \ref{section4.1}. We will then obtain (\ref{plequiv}) by showing that there are two equivalences of the form
\begin{equation} \label{introchain}
\xymatrix{
B\mathsf{Cob}_d^{\mathrm{PL}}(\mathbb{R}^N) \simeq |\psi_d^R(N,1)_{\bullet}| \ar@{^{(}->}[r]^{\quad\qquad\simeq} & |\psi_d(N,1)_{\bullet}| }
\end{equation} 
when $N-d\geq 3$. Proving that the inclusion $\psi_d^R(N,1)_{\bullet} \hookrightarrow \psi_d(N,1)_{\bullet}$ is a weak equivalence is the core of the present paper. It is also the part of the argument where we introduce most of our novel methods. Similar fiberwise regularity issues arise in the proof of the weak equivalence 
(\ref{plequiv2}), which we can also fix by applying the techniques we use to establish (\ref{plequiv}).

\subsection{Recent developments and related work}  \label{intro3}

In the recent article \cite{Gomez2}, we introduced \textit{the PL Madsen-Tillmann spectrum}
$\mathbf{MT}PL(d)$ and showed that there is a weak homotopy equivalence
\begin{equation} \label{equiv3}
\Omega^{\infty-1}\mathbf{MT}PL(d) \stackrel{\simeq}{\longrightarrow}\Omega^{\infty-1}\Psi^{\mathrm{PL}}_d.
\end{equation}
Together with this article's main theorem, the weak equivalence given above completes the proof of the piecewise linear analogue of the theorem proven by Galatius, Madsen, Tillmann, and Weiss in \cite{GMTW}. It is worth pointing out that, 
in joint work with A. Kupers, the author has proven in  \cite{GomezKupers} 
the topological version of the main result from 
 \cite{GMTW}. More concretely, in \cite{GomezKupers}, we introduce the \textit{topological cobordism category}
 $\mathsf{Cob}_d^{\mathrm{Top}}$ and the \textit{topological Madsen-Tillmann spectrum} 
 $\mathbf{MT}TOP(d)$, and we show that there is a weak equivalence of the form  
$B\mathsf{Cob}_d^{\mathrm{Top}} \simeq \Omega^{\infty -1} \mathbf{MT}TOP(d)$, provided that 
$d \neq 4$. The reason we need the restriction $d \neq 4$ is that our proof for the topological version of the weak equivalence (\ref{equiv3}) relies on smoothing theory for topological manifolds, which requires the condition 
$d \neq 4$. However, an advantage of the methods  developed in \cite{Gomez2} is that they can be easily adapted to the topological case, regardless of the dimension. Thus, by carrying out the arguments from \cite{Gomez2} in the topological setting, we obtain a new proof of the weak equivalence  $B\mathsf{Cob}_d^{\mathrm{Top}} \simeq \Omega^{\infty -1} \mathbf{MT}TOP(d)$ which works for all dimensions $d$, including $d = 4$. See \cite{Gomez2} for more details. 

We close this section by pointing out the main differences between the methods implemented in \cite{GomezKupers} and those used by the author in \cite{Gomez2} and the present paper. The proof of the weak equivalence $B\mathsf{Cob}_d^{\mathrm{Top}} \simeq \Omega^{\infty -1} \mathbf{MT}TOP(d)$ follows a strategy similar to the one used by the author 
to prove the same result in the PL category. Namely, the proof in \cite{GomezKupers} is broken down into the following steps: 
\begin{equation} \label{plequivtop}
B\mathsf{Cob}_d^{\mathrm{Top}} \hspace{0.1cm} \stackrel{(1)}{\simeq}  \hspace{0.1cm}\Omega^{\infty-1}\Psi^{\mathrm{Top}}_d  \hspace{0.1cm}
\stackrel{(2)}{\simeq} \hspace{0.1cm} \Omega^{\infty -1} \mathbf{MT}TOP(d).
\end{equation}
Here, $\Psi^{\mathrm{Top}}_d$ is 
the topological version of the spectrum $\Psi^{\mathrm{PL}}_d$. 
As mentioned before, the equivalence (2)
in (\ref{plequivtop}) is established in \cite{GomezKupers} using smoothing theory, whereas  the PL version of this equivalence 
is proven in \cite{Gomez2} using scanning techniques akin to those used by Galatius and Randal-Williams in \cite{GRW}. On the other hand, the first
equivalence of (\ref{plequivtop}) is proven in \cite{GomezKupers} using Gromov's $h$-principle. A key result that makes it possible to 
apply Gromov's $h$-principle in the topological setting is a `respectful' version of the Isotopy Extension Theorem proven
by Siebenmann (see Theorem 6.5 and Complement 6.6 of \cite{Si}). While it is claimed in \cite{Si} that this `respectful' version of the Isotopy Extension Theorem 
holds for both the $\mathrm{TOP}$ and $\mathrm{PL}$ categories, the proof of the PL version of this result does not exist in the literature
to the author's knowledge, which is the main reason why Gromov's $h$-principle cannot be used to prove the
first weak equivalence of (\ref{plequivtop}) in the PL setting. Instead, in the present paper,
we bypass Gromov's $h$-principle machinery by performing a delooping argument similar to the
one carried out by Galatius and Randal-Williams in \cite{GRW}. 

\theoremstyle{definition}  \newtheorem*{ack}{Acknowledgments}

\begin{ack}

First and foremost, I would like to express my sincere gratitude to Oscar Randal-Williams for being such a supportive advisor and mentor during my Ph.D. at the University of Copenhagen. This paper is based on the results I presented in my doctoral dissertation, and I would like to thank Oscar for his guidance, patience, and time during the development of this project.
I am also grateful to S\o ren Galatius for helpful discussions and the anonymous referee for their valuable comments and suggestions. 

\end{ack}

\section{Spaces of PL manifolds}  \label{section2}

\subsection{Standard notions from piecewise linear topology} \label{section2.1} 

We start this section by recollecting a few standard definitions from piecewise linear topology that we will use throughout this paper. A \textit{piecewise linear chart} (or PL chart for short) for a topological space $X$ is a continuous embedding $h: |K| \rightarrow X$, where $|K|$ is the geometric realization of a finite simplicial complex $K$. Two piecewise linear charts 
$h_1: |K| \rightarrow X$ and $h_2: |K'| \rightarrow X$  are said to be \textit{compatible} if either $h_1(|K|)$ and $h_2(|K'|)$ are disjoint,  or there exists subdivisions of $K$ and $K'$ which triangulate $h_1^{-1}(h_1(|K|)\cap h_2(|K'|))$ and $h_2^{-1}(h_1(|K|)\cap h_2(|K'|))$ respectively and such that the composition
\[
\xymatrix{ 
h_1^{-1}(h_1(|K|)\cap h_2(|K'|)) \ar[r]^{\hspace{0.35cm}h_1} & h_1(|K|)\cap h_2(|K'|) \ar[r]^{\hspace{-0.35cm}h^{-1}_2} & h_2^{-1}(h_1(|K|)\cap h_2(|K'|))}
\]
is linear on each simplex of the subdivision of $K$ contained in $h_1^{-1}(h_1(|K|)\cap h_2(|K'|))$. A \textit{piecewise linear space} (or PL space for short) is a second countable Hausdorff space $X$ together with a collection $\Lambda$ of PL charts satisfying the following properties: 

\begin{itemize}

\item[(i)] Any two charts in $\Lambda$ are compatible. 

\item[(ii)] For any point $x \in X$, there is a PL chart $h: |K| \rightarrow X$ in $\Lambda$ such that $h(|K|)$ is a neighborhood of $x$. 

\item[(iii)] The collection $\Lambda$ is maximal. That is, if $h: |K| \rightarrow X$ is a PL chart which is compatible with every chart in $\Lambda$, then $h: |K| \rightarrow X$ must also belong to $\Lambda$. 

\end{itemize}

The collection of charts $\Lambda$ is typically called a \textit{PL structure on X}. Now, consider a PL space $X$ with PL structure 
$\Lambda$, and let $X_0 \subseteq X$ be a subspace of $X$. We say that $X_0$ is a \textit{PL subspace of $X$} if, for any point 
$x \in X_0$, we can find a chart $h: |K| \rightarrow X$ in the PL structure $\Lambda$
whose image is contained in $X_0$ and has the property that
$\mathrm{Im}\hspace{0.05cm}h$ is a neighborhood of $x$ in $X_0$. 
Note that $X_0$ is itself a PL space. 
The PL structure on $X_0$ is the subcollection $\Lambda_0$ of $\Lambda$ consisting of all charts $h: |K| \rightarrow X$ such that 
$\mathrm{Im}\hspace{0.05cm}h \subseteq X_0$.

Now, consider two PL spaces $X$ and $Y$. A continuous map $f: X \rightarrow Y$ is said to be a \textit{PL map} if, for each $x \in X$, we can find PL charts $h_0: |K| \rightarrow X$ and $h_1:|K'| \rightarrow Y$ for $X$ and $Y$ satisfying the following:

\begin{itemize}

\item[(i)] $x$ is in the interior of $h_0(|K|)$ and $f(x)$ is in the interior of $h_1(|K'|)$. 

\item[(ii)] $f\big( h_0(|K|) \big) \subseteq h_1(|K'|)$. 

\item[(iii)] The composition $h_1^{-1} \circ f \circ h_0$ maps each simplex of $K$ linearly  to a simplex of $K'$. 

\end{itemize}

We say that $f: X \rightarrow Y$ is a \textit{PL homeomorphism} if its both a PL map and a homeomorphism. 
Moreover, we say that a PL map $g: X \rightarrow Y$ is a \textit{PL embedding} if 
$\mathrm{Im}\hspace{0.05cm}g$ is a PL subspace of $Y$ and the map 
$g: X \rightarrow \mathrm{Im}\hspace{0.05cm}g$ is a PL homeomorphism.

For any open set $U$ in a Euclidean space $\mathbb{R}^n$, we can define a canonical piecewise linear structure on $U$ by taking all PL charts compatible with inclusions  of the form $|K| \hookrightarrow U$, where $K$ is any finite simplicial complex consisting of linear simplices contained in $U$. 
We say that $M$ is a \textit{PL manifold of dimension n} if $M$ is a PL space and, for each $x \in M$, we can find a piecewise linear embedding $h: U \rightarrow M$ which is defined on some open set $U \subseteq \mathbb{R}^n$ and whose image is a neighborhood of the point $x$. Similarly, for any open set $U \subseteq \mathbb{R}^n$, we can also define a canonical piecewise linear structure on the intersection $U \cap \mathbb{R}^n_+$, where $\mathbb{R}^n_+$ denotes the subspace of $\mathbb{R}^n$ defined by the inequality $x_n \geq 0$.  Then, we can define an \textit{n-dimensional PL manifold with boundary} to be a PL space $M$ such that, for each $x \in M$, there is a piecewise linear embedding 
$h: U \cap \mathbb{R}^n_+ \rightarrow M$ whose image is a neighborhood of $x$. 

\subsection{Spaces of PL manifolds} \label{section2.2}

Our goal in this subsection is to define the simplicial sets $\Psi_d(U)_{\bullet}$ (where $U$ is an arbitrary open set inside some $\mathbb{R}^N$) that will act as the PL analogues of the spaces of smooth manifolds defined by Galatius in \cite{Ga}.  However, instead of just defining a simplicial set $\Psi_d(U)_{\bullet}$, 
we will construct a much more general object. Namely, for each open set $U$, we will define a \textit{PL set}, i.e., a contravariant functor 
$\Psi_d(U): \mathbf{PL}^{op} \rightarrow \mathbf{Sets}$ defined on the category $\mathbf{PL}$ of PL spaces and PL maps. Roughly speaking, for a given PL space $P$, the elements of the set $\Psi_d(U)(P)$ will be families of $d$-dimensional PL submanifolds of $U$ parameterized by $P$. As we shall see later, there is a canonical embedding $\iota:\Delta \hookrightarrow \mathbf{PL}$, where $\Delta$ denotes the category of finite sets $[p]$ and 
order-preserving functions. The simplicial set $\Psi_d(U)_{\bullet}$ is then obtained by precomposing $\Psi_d(U): \mathbf{PL}^{op} \rightarrow \mathbf{Sets}$ 
with the induced embedding $\iota^{op}: \Delta^{op} \hookrightarrow \mathbf{PL}^{op}$.

To define the sets $\Psi_d(U)(P)$, we will need the following notion from PL topology which perhaps is not as standard as the material that we reviewed in the previous subsection. 

\theoremstyle{definition} \newtheorem{sbmn}{Definition}[section]

\begin{sbmn} \label{sbmn}

A PL map
$\pi: E \rightarrow P$ is said to be a
\textit{PL submersion of codimension d} if, for
each $x \in E$, there is an open neighborhood $V$ of $\pi(x)$
in $P$ and an open piecewise linear embedding
$h: V\times \mathbb{R}^d \rightarrow E$ such that 
$x$ is in the image of $h$ and 
$\pi\circ h$ is equal to the standard projection 
$\mathrm{pr}_1: V\times   \mathbb{R}^d \rightarrow V$. 

\end{sbmn}

Note that each fiber of a PL submersion $\pi$ of codimension $d$ is a $d$-dimensional PL manifold. An open piecewise linear  embedding $h: V\times \mathbb{R}^d \rightarrow E$ satisfying $\pi \circ h = \mathrm{pr}_1$ and $x \in \mathrm{Im}\hspace{0.1cm}h$ is typically called a \textit{submersion chart around $x$}. We will also need the following construction for the definition of the functor
$\Psi_d(U): \mathbf{PL}^{op} \rightarrow \mathbf{Sets}$.
 
\theoremstyle{definition} \newtheorem{pullback}[sbmn]{Definition}

\begin{pullback} \label{pullback}

Fix a PL space $P$ and an open set $U \subseteq \mathbb{R}^N$. Let $W$ be a closed PL subspace of the product $P \times U$ such that the projection 
$\pi: W \rightarrow P$ is a PL submersion of codimension $d$. Given a PL map $f: Q \rightarrow P$, the subspace 
\[
\big\{ (q, x) \in Q \times U \hspace{0.2cm} | \hspace{0.2cm} f(q) = \pi(x)   \big\}
\]
of $Q \times U$ will be called \textit{the pull-back of W along f,} and we will denote it by $f^*W.$\end{pullback}

Since pull-backs of submersions are again submersions, we have that the standard projection 
$\tilde{\pi}: f^*W \rightarrow Q$ is also a PL submersion of codimension $d$. Moreover, note that $f^*W$ is a closed PL subspace of $Q \times U$. With this construction at hand, we can now define the PL sets
 $\Psi_d(U): \mathbf{PL}^{op} \rightarrow \mathbf{Sets}$. 

\theoremstyle{definition} \newtheorem{spacemangen}[sbmn]{Definition}

\begin{spacemangen} \label{spacemangen} 

Given an open subset $U \subseteq \mathbb{R}^N$ and a non-negative integer $d$, we define the functor 
$\Psi_d(U): \mathbf{PL}^{op} \rightarrow \mathbf{Sets}$ as follows:

\begin{itemize}

\item[$\cdot$] For each PL space $P$, $\Psi_d(U)(P)$ is the set of all closed PL subspaces $W$ 
of $P \times U$ with the property that the standard projection 
$\pi: W \rightarrow P$ is a PL submersion of codimension $d$. 

\item[$\cdot$] For a PL map $f: Q \rightarrow P$, the function $\Psi_d(U)(f): \Psi_d(U)(P) \rightarrow \Psi_d(U)(Q)$ will send an element $W$ of $\Psi_d(U)(P)$ to the pull-back $f^*W$ of $W$ along $f$. From now on, we will denote the function 
$\Psi_d(U)(f)$ simply by $f^*$.  
\end{itemize}

\end{spacemangen}

Consider a PL space $P$. For an element 
$W$ of $\Psi_d(U)(P)$, the fiber
over a point $\lambda$ in $P$ of the 
standard projection $\pi: W \rightarrow P$ 
is a $d$-dimensional PL submanifold
of $\{ \lambda \} \times U$ which is also closed
as a subspace.
Thus, we can view an element $W$ of 
$\Psi_d(U)(P)$ as a family of 
$d$-dimensional PL submanifolds of $U$,
closed as subspaces,
which are parameterized \textit{piecewise linearly} by the base-space
$P$. 

In this paper, we will mostly be working with the functor $\Psi_d(\mathbb{R}^N): \mathbf{PL}^{op} \rightarrow \mathbf{Sets}$, 
i.e., the PL set corresponding to $U = \mathbb{R}^N$. Besides $\Psi_d(\mathbb{R}^N)$, the following PL sets will
also play a pivotal role in the proof of our main theorem.

\theoremstyle{definition} \newtheorem{spacemanfil}[sbmn]{Definition}

\begin{spacemanfil} \label{spacemanfil}
Fix two integers $N > 0$ and $d\geq 0$. For any integer $0 \leq k \leq N$, let $\psi_d(N,k): \mathbf{PL}^{op} \rightarrow \mathbf{Sets}$ 
be the PL set defined as follows: 

\begin{itemize}

\item[$\cdot$] For each PL space $P$, we define $\psi_d(N,k)(P)$ to be the subset of $\Psi_d(\mathbb{R}^N)(P)$ consisting of all elements $W$ such that $W \subseteq P  \times \mathbb{R}^{k}\times (-1,1)^{N-k}$. That is, each fiber of the projection 
$\pi: W \rightarrow P$ is a PL submanifold of $\mathbb{R}^N$ contained in $\mathbb{R}^{k}\times (-1,1)^{N-k}$. 

\item[$\cdot$] For a PL map $f: Q \rightarrow P$, the function $\psi_d(N,k)(f)$ is equal to 
the restriction of the pull-back function $f^{*}: \Psi_d(\mathbb{R}^N)(P) \rightarrow \Psi_d(\mathbb{R}^N)(Q)$ on
$\psi_d(N,k)(P)$. By abuse of notation, we will also denote $\psi_d(N,k)(f)$ simply by $f^*$.

\end{itemize}

\end{spacemanfil} 

Note that  $\Psi_d(\mathbb{R}^N) = \psi_d(N,N)$. Also, 
for any PL space $P$, we evidently have that
$\psi_d(N,0)(P) \subseteq \psi_d(N,1)(P) \subseteq \ldots \subseteq \psi_d(N,N-1)(P) \subseteq \Psi_d(\mathbb{R}^N)(P)$. 
Thus, the PL sets $\psi_d(N,k)$ form a filtration for $\Psi_d(\mathbb{R}^N): \mathbf{PL}^{op} \rightarrow \mathbf{Sets}$. 

\theoremstyle{definition} \newtheorem{sbmn.conv}[sbmn]{Convention}

\begin{sbmn.conv} \label{sbmn.conv}

Before giving any more definitions, we need to fix the following conventions, which we will use throughout the rest of this article: 

\begin{itemize}

\item[(1)] Unless noted otherwise, we will denote the elements of the standard basis of $\mathbb{R}^{p+1}$ by 
$e_0, \ldots, e_{p}$. In particular, the indexing of the standard basis vectors will typically start at $j=0$.

\item[(2)] For any non-negative integer $p$, we define \textit{the standard p-simplex} $\Delta^p$ to be the convex hull of the vectors $e_0, \ldots, e_{p}$ of the standard basis in $\mathbb{R}^{p+1}$. 
 
 \item[(3)] For any morphism $\eta:[p] \rightarrow [q]$ in $\Delta$, we will denote by  
$\widetilde{\eta}$ the linear map $\Delta^p \rightarrow \Delta^q$ defined by $e_j \mapsto e_{\eta(j)}$.

\end{itemize}

\end{sbmn.conv}

The embedding  $\iota: \Delta \rightarrow \mathbf{PL}$ that we referred to at the 
beginning of this subsection is defined by $[p]\mapsto \Delta^p$ on objects and by
$\eta \mapsto \widetilde{\eta}$ on morphisms. Then, 
for any PL set $\mathcal{F}: \mathbf{PL}^{op} \rightarrow \mathbf{Sets}$ 
(not just those of the form $\Psi_d(U)$), we obtain a simplicial set $\mathcal{F}_{\bullet}$ by 
precomposing $\mathcal{F}$ with the induced embedding $\iota^{op}$. We will typically call
$\mathcal{F}_{\bullet}$ \textit{the underlying simplicial set of $\mathcal{F}$.} Note that
the set of $p$-simplices of $\mathcal{F}_{\bullet}$ is equal to $\mathcal{F}(\Delta^p)$.
Clearly, any natural transformation 
$\gamma: \mathcal{F} \Rightarrow \mathcal{G}$ between PL sets will induce a simplicial set map
$\gamma_{\bullet}:  \mathcal{F}_{\bullet} \rightarrow \mathcal{G}_{\bullet}$. Moreover, it is evident that the operation 
$\gamma \mapsto \gamma_{\bullet}$ respects compositions and identity maps. Thus, the correspondences 
$\mathcal{F} \mapsto \mathcal{F}_{\bullet}$, $\gamma \mapsto \gamma_{\bullet}$ define a functor 
$\Gamma: \mathbf{PLsets} \rightarrow \mathbf{Ssets}$, where 
$\mathbf{PLsets}$ is the category of 
PL sets. Morphisms in this category are given by natural transformations. 

 In the following remarks, we will make several useful observations about the 
PL sets $\Psi_d(U): \mathbf{PL}^{op} \rightarrow \mathbf{Sets}$ and their 
underlying simplicial sets $\Psi_d(U)_{\bullet}$.

\theoremstyle{definition} \newtheorem{spacerem0}[sbmn]{Remark}

\begin{spacerem0} \label{spacerem0} 
Note that, for any integer $0\leq k \leq N$, the simplicial set $\psi_d(N,k)_{\bullet}$ underlying the functor 
$\psi_d(N,k):  \mathbf{PL}^{op} \rightarrow \mathbf{Sets}$ is a subsimplicial set of $\Psi_d(\mathbb{R}^N)_{\bullet}$.  
Evidently, the filtration $\psi_d(N,0) \subseteq \psi_d(N,1) \subseteq \ldots \subseteq \psi_d(N,N-1) \subseteq \Psi_d(\mathbb{R}^N)$  of PL sets
induces an analogous filtration for the simplicial set $\Psi_d(\mathbb{R}^N)_{\bullet}$. 

\end{spacerem0} 

\theoremstyle{definition} \newtheorem{spacerem}[sbmn]{Remark}

\begin{spacerem} \label{spacerem}
\textbf{(The base-point of} $\Psi_d(U)_{\bullet}$\textbf{)}
We point out that, for any open set $U \subseteq \mathbb{R}^N$ and any $p\geq 0$, \textit{the empty manifold}
$\varnothing_p \subseteq \Delta^{p}\times U$
is also an element of $\Psi_d(U)_{p}$.  The empty simplices 
$\varnothing_p$ form a subsimplicial set of $\Psi_d(U)_{\bullet}$, which we will denote by $\varnothing_{\bullet}$. We will usually regard the geometric realization $|\varnothing_{\bullet}|$ as the canonical base-point for the space
$|\Psi_d(U)_{\bullet}|$. 
\end{spacerem}

\theoremstyle{definition} \newtheorem{spacerem2}[sbmn]{Remark}

\begin{spacerem2} \label{spacerem2}
\textbf{(The sheaf} $\Psi_d$\textbf{)} 
Given any two open sets $U, V \subseteq \mathbb{R}^N$ with $V \subseteq U$, we can define a 
\textit{restriction morphism} $r_{U,V}: \Psi_d(U) \Rightarrow \Psi_d(V)$ between the PL sets $\Psi_d(U)$ and 
$ \Psi_d(V)$. Namely, let $r_{U,V}$ be the natural transformation such that, for any PL space $P$,  
the component $ \Psi_d(U)(P) \stackrel{r_{U,V}}{\longrightarrow} \Psi_d(V)(P)$ corresponding to $P$ is the function which sends an element $W \in \Psi_d(U)(P)$ to the intersection $W\cap(P\times V)$.
If $\mathcal{O}(\mathbb{R}^N)$ is the poset of open sets in $\mathbb{R}^N$, then the correspondences
\begin{equation} \label{pre-sheaf}
U \mapsto \Psi_d(U) \qquad V \subseteq U \mapsto r_{U,V}
\end{equation}
define a contravariant functor $\Psi_d: \mathcal{O}(\mathbb{R}^N)^{op} \rightarrow \mathbf{PLsets}$. 
In fact, it is not hard
to prove that $\Psi_d: \mathcal{O}(\mathbb{R}^N)^{op} \rightarrow \mathbf{PLsets}$ is actually a sheaf of PL sets. This sheaf descends to a sheaf of simplicial sets $\Psi_d: \mathcal{O}(\mathbb{R}^N)^{op} \rightarrow \mathbf{Ssets}$ via the functor 
$\Gamma: \mathbf{PLsets} \rightarrow \mathbf{Ssets}$ that we defined immediately after 
Convention \ref{sbmn.conv}. 
\end{spacerem2}

\theoremstyle{definition} \newtheorem{spacerem3}[sbmn]{Remark}

\begin{spacerem3} \label{spacerem3}

Given an element $W$ of a set $\Psi_d(U)(P)$
and a point $\lambda \in P$, we shall typically denote the fiber 
of the projection $\pi:W \rightarrow P$ over $\lambda$ by $W_{\lambda}$.

\end{spacerem3}

We end this subsection with the following proposition, which provides a convenient tool for producing elements in a set of the form $\Psi_d(U)(P)$.

\theoremstyle{plain} \newtheorem{pullemb}[sbmn]{Proposition}

\begin{pullemb} \label{pullemb} 

Fix a PL space $P$ and two open sets $U,V \subseteq \mathbb{R}^N$. 
Then, for any element $W$ in $\Psi_d(V)(P)$ and any
open piecewise linear embedding 
$H: P\times U \rightarrow P \times V$
which commutes with the projection onto $P$,
the pre-image $H^{-1}(W)$ is an element of the set
$\Psi_d(U)(P)$. 
\end{pullemb}

\begin{proof}
First, it is clear that the intersection of $W$ and $\mathrm{Im}\hspace{0.05cm}H$
is a closed PL subspace of $\mathrm{Im}\hspace{0.05cm}H$.
Since $H$ is a PL homeomorphism
from $P\times U$ to $\mathrm{Im}\hspace{0.05cm}H$,
it follows that $H^{-1}(W)$ is a closed 
PL subspace  of $P\times U$. Next,
note that the standard projection $\pi: H^{-1}(W) \rightarrow P$ is equal to the 
composition of $H|_{H^{-1}(W)}$,
which is a PL homeomorphism
from $H^{-1}(W)$ to $W$,
and the standard projection $\widetilde{\pi}: W \rightarrow P$. 
Since $\widetilde{\pi}$ is a PL submersion of codimension $d$, then so is 
the projection  $\pi: H^{-1}(W) \rightarrow P$.
\end{proof}

\subsection{Properties of $\Psi_d(U)_{\bullet}$}  \label{section2.3}    

In this subsection, 
we will prove that any simplicial set of the form $\Psi_d(U)_{\bullet}$
acts as a \textit{moduli space of PL manifolds} 
in the sense that, for any PL space $P$, $\Psi_d(U)_{\bullet}$  classifies all elements 
of the set $\Psi_d(U)(P )$. 
This property of $\Psi_d(U)_{\bullet}$ will be a consequence of Theorem \ref{classsub}, 
which states a more general classification result for a special type of PL set, 
which we will introduce in Definition \ref{quasiPL} below.
For this definition, we shall use the following terminology: Suppose that 
$\mathcal{F}: \mathbf{PL}^{op} \rightarrow \mathbf{Sets}$ is a PL set. If $P$ is a PL space
  and $Q$ is a PL subspace of $P$, then the map 
  $\mathcal{F}(P) \rightarrow \mathcal{F}(Q)$ induced by the 
  inclusion $Q \hookrightarrow P$ shall be called a \textit{restriction map}. 
  
\theoremstyle{definition}  \newtheorem{quasiPL}[sbmn]{Definition} 

\begin{quasiPL} \label{quasiPL}
A \textit{quasi-PL space}  is a PL set $\mathcal{F}: \mathbf{PL}^{op} \rightarrow \mathbf{Sets}$ which satisfies 
the following gluing condition: For any PL space $P$ and any locally finite collection of closed PL subspaces $\{ Q_i \}_i$ which covers $P$, the diagram of restriction maps
\begin{equation} \label{gluing.condition}
\xymatrix{ \mathcal{F}(P) \ar[r] & \prod_i  \mathcal{F}(Q_i)  \ar@<0.7ex>[r] \ar@<-0.7ex>[r]  & 
\prod_{i,j}  \mathcal{F}(Q_i \cap Q_j)
}
\end{equation}
is an equalizer diagram of sets.  
\end{quasiPL}

\theoremstyle{definition}  \newtheorem{quasiPLrem}[sbmn]{Remark} 

\begin{quasiPLrem} \label{quasiPLrem}
It is worth pointing out that any quasi-PL space must also necessarily be a sheaf on PL spaces.
In other words, if $\mathcal{F}$ is a PL set satisfying the gluing condition described in 
Definition \ref{quasiPL}, then it must also satisfy the analogous gluing condition for arbitrary open covers. 
One can prove this fact using Lemma 2.1.6 from \cite{Jo} and standard subdivision arguments from PL topology. 
However, we shall not use this property of quasi-PL spaces in the remainder of this paper. 
\end{quasiPLrem}

Before we can state Theorem \ref{classsub}, we need to review yet another notion from PL topology: \textit{Triangulations}. 
First, suppose that $K$ is a locally finite simplicial complex in some $\mathbb{R}^m$, possibly with $m = \infty$. Note that
the geometric realization $|K|$ (i.e., the union of all the simplices in $K$) is itself a PL space. The PL structure on $|K|$ is the collection of all PL charts compatible with all inclusions of the form $|L|\hookrightarrow |K|$, where $L$ is any finite subcomplex of $K$.  
A \textit{triangulation} for a PL space $P$ is a tuple $(K,h)$ where $K$ is a locally finite simplicial complex in $\mathbb{R}^{m}$ (possibly with $m = \infty$) and
$h$ is a PL homeomorphism $|K| \stackrel{\cong}{\longrightarrow} P$.  
As proven in \cite{plhud}, any PL space $P$ admits a triangulation. If $P$ is compact, we may assume that $K$ is a finite simplicial complex in some finite-dimensional $\mathbb{R}^m$.  

Consider a locally finite simplicial complex $K$ in $\mathbb{R}^{m}$ (again, we allow $m=\infty$) and fix an ordering $\leq$ on the set of vertices of $K$. In order to state Theorem \ref{classsub}, we need to explain how the tuple  $(K, \leq)$ induces a simplicial set  $K_{\bullet}$. A $p$-simplex in $K_{\bullet}$ is a non-decreasing sequence $v_{0}\leq \ldots \leq v_{p}$ of vertices, possibly with repetitions, with the property that the set $\{ v_0,\ldots, v_p \}$  spans a simplex of $K$. For a morphism 
$\eta: [q] \rightarrow [p]$ in the category $\Delta$, the induced structure map 
$\eta^*: K_p \rightarrow K_q$ sends a $p$-simplex $v_{0}\leq \ldots \leq v_{p}$ to 
$v_{\eta(0)}\leq \ldots \leq v_{\eta(q)}$. Finally, we will also use the following notation  
in the statement of Theorem \ref{classsub}:

\begin{itemize}

\item[$\cdot$] Consider a PL set $\mathcal{F}$. As we do for PL
sets of the form $\Psi_d(U)$, we shall denote by $f^*$ the map 
$\mathcal{F}(Q) \rightarrow \mathcal{F}(P)$ induced by a PL map $f: P \rightarrow Q$. Moreover,
as we also do for $\Psi_d(U)$, we shall denote a generic element of a set $\mathcal{F}(P)$ by $W$.

\item[$\cdot$] If $v_0, \ldots, v_p$ is a collection of $p+1$ points (possibly with repetitions)  in $\mathbb{R}^m$
and if $e_0,\ldots, e_p$ are the elements of the standard basis
of $\mathbb{R}^{p+1}$, 
then we will denote by
$s_{v_0\ldots v_p} : \Delta^p \rightarrow \mathbb{R}^m$
the linear map defined
by $e_j \mapsto v_j$.

\end{itemize}

\theoremstyle{plain}  \newtheorem{classsub}[sbmn]{Theorem} 

\begin{classsub} \label{classsub}

Consider a PL set $\mathcal{F}: \mathbf{PL}^{op} \rightarrow \mathbf{Sets}$.
If $\mathcal{F}$ is a quasi-PL space, then $\mathcal{F}$ also has the following properties: 

\begin{itemize}

\item[(i)] Fix a PL space $P$. For any choice of triangulation $(K,h)$ of $P$  and ordering $\leq$ on the vertices of $K$,
the function of sets 
\[
\mathcal{S}_{K,h} : \mathcal{F}(P) \rightarrow \mathbf{Ssets}(K_{\bullet}, \mathcal{F}_{\bullet})
\]
which sends an element $W$
of $\mathcal{F}(P)$ to the morphism of simplicial sets 
$F_W: K_{\bullet} \rightarrow \mathcal{F}_{\bullet}$
defined by
$F_W(v_0\leq \ldots \leq v_p) = s_{v_0\ldots v_p}^*h^*W $
is a bijection. 

\item[(ii)] The underlying simplicial set $\mathcal{F}_{\bullet}$ is Kan. 

\end{itemize}

\end{classsub}

\theoremstyle{definition} \newtheorem{classsub.remark}[sbmn]{Remark}

\begin{classsub.remark} \label{classsub.remark}
Before proving Theorem \ref{classsub}, let us make a few comments.  
At a first glance, the definition of the morphism $F_W$ appearing in (i) might be hard to digest, so let us spell it out more carefully. 
Fix a triangulation $(K,h)$ of $P$ and an ordering $\leq$ on $\mathrm{Vert}(K)$. For any $p$-simplex 
$v_0\leq \ldots \leq v_p$ of $K_{\bullet}$, the image $F_W(v_0\leq \ldots \leq v_p)$ is obtained by first taking the pull-back of $W$ along the PL homeomorphism $h:|K| \rightarrow P$, and then pulling back 
the element $h^*W \in \mathcal{F}(|K|)$ along the 
linear map $s_{v_0\ldots v_p}: \Delta^p \rightarrow \mathbb{R}^m$ defined before the statement of Theorem \ref{classsub}. 
Note that, since the set $\{v_0, \ldots, v_p\}$ spans a simplex of $K$, the image of 
$s_{v_0\ldots v_p}$ lands in $|K|$. For this reason, it is possible to pull back $h^*W$ along
$s_{v_0\ldots v_p}$.  We close this remark by making three observations 
about the function $\mathcal{S}_{K,h}$  defined in part (i) of Theorem \ref{classsub}:

\begin{itemize}

\item[$\cdot$] First, note that the function $\mathcal{S}_{K,h}$ 
does not only depend on the triangulation $(K,h)$ 
but also on the order relation we fix on $\mathrm{Vert}(K)$.

\item[$\cdot$] Second, note that it is possible to define the function 
$\mathcal{S}_{K,h}: \mathcal{F}(P) \rightarrow \mathbf{Ssets}(K_{\bullet}, \mathcal{F}_{\bullet})$ 
even if $\mathcal{F}$ is not a quasi-PL space. 
The point of the theorem is that, if 
$\mathcal{F}$ happens to be a quasi-PL space, then the function 
$\mathcal{S}_{K,h}$ is a bijection. 

\item[$\cdot$] Finally, it is worth pointing out that in the case when $P = \Delta^p$ and $(K,h)$ is the canonical triangulation of $\Delta^p$ (i.e., take $K$ to be the standard simplicial complex triangulating $\Delta^p$ and $h$ to be the identity map $\Delta^p \rightarrow \Delta^p$), then the function $\mathcal{S}_{K,h}$ agrees with the natural bijection 
$\mathcal{F}_{p} \rightarrow \mathbf{Ssets}(\Delta^p_{\bullet}, \mathcal{F}_{\bullet})$. 

\end{itemize}

\end{classsub.remark}

\begin{proof}[\textbf{\textit{Proof of Theorem \ref{classsub}:}}] 
To prove part (i), we shall need the following notation:

\begin{itemize}

\item[$\cdot$] Given $p+1$ points $v_0, \ldots, v_p$ in $\mathbb{R}^m$, we defined $s_{v_0\ldots v_p}: \Delta^p \rightarrow \mathbb{R}^m$ to be the linear map which sends the element $e_j$ of the standard basis of $\mathbb{R}^{p+1}$ to $v_j$. If the points $v_0, \ldots, v_p$ are in general position (i.e., their convex hull is a $p$-simplex), then $s_{v_0\ldots v_p}$ is a PL homeomorphism from $\Delta^p$ to the convex hull of $v_0, \ldots, v_p$. In this case, we will denote the inverse of $s_{v_0\ldots v_p}$ by $t_{v_0\ldots v_p}$. 

\item[$\cdot$] We will denote the inverse of the PL homeomorphism $h: |K| \rightarrow P$ by $\tilde{h}$. 

\item[$\cdot$] Finally, given a collection of points $v_0, \ldots, v_p$ in $\mathbb{R}^m$, we will denote its convex hull by $\langle v_0, \ldots, v_p  \rangle$. 

\end{itemize}

Now, fix a PL space $P$, a triangulation $(K,h)$ of $P$, and an ordering $\leq$ of the vertices of $K$. Also, 
let $K_{\bullet}$ be the simplicial set induced by the tuple $(K, \leq)$. 
If $\Lambda$ is the set of all non-degenerate simplices $v_0 < \ldots < v_p$ of $K_{\bullet}$, then
 \[
 \Big\{ h(\langle v_0, \ldots, v_p  \rangle)  \Big\}_{v_0 < \ldots < v_p \in \Lambda} 
 \]
is a collection of closed PL subspaces of $P$  which is locally finite and covers $P$. 
For simplicity, we shall denote each PL subspace $h(\langle v_0, \ldots, v_p  \rangle)$ by 
$Q_{v_0\ldots v_p}$.  
We will prove that the function 
$\mathcal{S}_{K,h} : \mathcal{F}(P) \rightarrow \mathbf{Ssets}(K_{\bullet}, \mathcal{F}_{\bullet})$
is a bijection by constructing an inverse,
which we shall denote by $\mathcal{T}_{K,h}$.
For any simplicial set map $g: K_{\bullet} \rightarrow \mathcal{F}_{\bullet}$,
we will construct the image $\mathcal{T}_{K,h}(g)$ under the inverse $\mathcal{T}_{K,h}: \mathbf{Ssets}(K_{\bullet}, \mathcal{F}_{\bullet}) \rightarrow  \mathcal{F}(P)$
as follows: 

\begin{itemize}

\item[$\cdot$] First, for any non-degenerate simplex $v_0 < \ldots < v_p$ of $K_{\bullet}$,
we will denote the image $g(v_0 < \ldots < v_p) \in \mathcal{F}_p$ by $W^{v_0\ldots v_p}$. 

\item[$\cdot$]  Next, for each non-degenerate simplex $v_0 < \ldots < v_p$, 
take the element $\tilde{h}^*t^*_{v_0\ldots v_p}W^{v_0 \ldots v_p}$ of $\mathcal{F}(Q_{v_0\ldots v_p})$
obtained by first pulling back $W^{v_0\ldots v_p}$ along $t_{v_0\ldots v_p}$, and then pulling back
$t_{v_0\ldots v_p}^*W^{v_0\ldots v_p}$ along the inverse of $h: |K| \rightarrow P$. 

\item[$\cdot$] Suppose that $v_0 < \ldots < v_p$ and $v'_0 < \ldots < v'_{q}$ are two non-degenerate
simplices of $K_{\bullet}$ such that the convex hulls 
$\langle v_0, \ldots, v_p  \rangle$ and $\langle v'_0, \ldots, v'_{q}  \rangle$ 
have a non-empty intersection, and let $Q$ be the intersection of the closed subspaces
$Q_{v_0\ldots v_p}$ and $Q_{v'_0\ldots v'_q}$. Then, since 
$g: K_{\bullet} \rightarrow \mathcal{F}_{\bullet}$ is a morphism of 
simplicial sets, it is not hard to show that the restrictions 
of  $\tilde{h}^*t^*_{v_0\ldots v_p}W^{v_0 \ldots v_p}$ and
$\tilde{h}^*t^*_{v'_0\ldots v'_q}W^{v'_0 \ldots v'_q}$ 
over $Q$ agree with each other. Therefore, given that
$\mathcal{F}$ is a quasi-PL space, there must exist 
a unique element $W^{g}$ of $\mathcal{F}(P)$ such that, for each
non-degenerate simplex $v_0 < \ldots < v_p$, the restriction of $W^g$ over 
$Q_{v_0\ldots v_p}$ is equal to $\tilde{h}^*t^*_{v_0\ldots v_p}W^{v_0 \ldots v_p}$. 
We define this element $W^g$ to be the image $\mathcal{T}_{K,h}(g)$. 
\end{itemize}

It is straightforward to verify that the function 
$\mathcal{T}_{K,h}: \mathbf{Ssets}(K_{\bullet}, \mathcal{F}_{\bullet}) \rightarrow  \mathcal{F}(P)$,
as we defined it above, is indeed an inverse for $\mathcal{S}_{K,h}$. We can therefore conclude
that $\mathcal{S}_{K,h}$ is a bijection between $\mathcal{F}(P)$ and $\mathbf{Ssets}(K_{\bullet}, \mathcal{F}_{\bullet})$. 

Now, let us turn to part (ii) of this theorem. Fix then a
simplicial set map of the form $g: \Lambda^p_{k\bullet} \rightarrow \mathcal{F}_{\bullet}$.
To prove that the  simplicial set $\mathcal{F}_{\bullet}$ is Kan, we
must show that $g$ admits an extension of the form $\Delta^p_{\bullet} \rightarrow \mathcal{F}_{\bullet}$. 
First, let $K$ be the standard simplicial complex triangulating the horn 
$\Lambda^p_k$ of the standard $p$-simplex $\Delta^p$.
If $\leq$ is the canonical order relation among the vertices of $\Lambda^p_k$, then
clearly the simplicial set $K_{\bullet}$ induced by $(K, \leq)$ is isomorphic to 
$\Lambda^p_{k\bullet}$. Thus, part (i) of this theorem guarantees that there is 
a bijection of sets 
$\mathcal{F}(\Lambda^p_k) \stackrel{\cong}{\longrightarrow} \mathbf{Ssets}(\Lambda^p_{k\bullet}, \mathcal{F}_{\bullet})$,
which we shall denote simply by $\mathcal{S}$.  
If $W$ is the element of $\mathcal{F}(\Lambda^p_k)$ obtained by pulling back the map $g: \Lambda^p_{k\bullet} \rightarrow \mathcal{F}_{\bullet}$
along the bijection $\mathcal{S}$, then $W$ has the property that 
$g(v_0 \leq \ldots \leq v_q) = s^*_{v_0\ldots v_q}W$ for any simplex 
$v_0 \leq \ldots \leq v_q$  of $\Lambda^p_{k\bullet}$ 
(again, $s_{v_0\ldots v_q}$ is the linear map 
$\Delta^q \rightarrow \Lambda^p_k$ which sends $e_j$ to $v_j$). 
If we now pull back the element $W$ along a PL retraction 
$r:\Delta^p \rightarrow \Lambda^p_k$, we obtain an element $W' := r^*W$ of
$\mathcal{F}(\Delta^p)$ which (via part (i) of this theorem) induces a
morphism of simplicial sets $f: \Delta^p_{\bullet} \rightarrow \mathcal{F}_{\bullet}$
which maps a simplex $v_0 \leq \ldots \leq v_q$ of $\Delta^p_{\bullet}$
to $s^*_{v_0\ldots v_q}W'$. Since $r:\Delta^p \rightarrow \Lambda^p_k$ is a retraction,
we have that $W'$ agrees with $W$ over $\Lambda^p_k$, which implies 
that $s^*_{v_0\ldots v_q}W' = s^*_{v_0\ldots v_q}W$ for any simplex 
$v_0 \leq \ldots \leq v_q$  of $\Lambda^p_{k\bullet}$. It follows that 
$f: \Delta^p_{\bullet} \rightarrow \mathcal{F}_{\bullet}$ extends 
$g: \Lambda^p_{k\bullet} \rightarrow \mathcal{F}_{\bullet}$, and we have thus proven that 
the underlying simplicial set $\mathcal{F}_{\bullet}$ is Kan.
\end{proof}

The next theorem will allow us to interpret $\Psi_d(U)_{\bullet}$ as a moduli space of PL manifolds. 

 \theoremstyle{plain} \newtheorem{kan}[sbmn]{Theorem}

\begin{kan}  \label{kan}
Any PL set of the form $\Psi_d(U)$ is a quasi-PL space and, therefore,
satisfies the properties listed in Theorem \ref{classsub}. 
In particular, the underlying simplicial set $\Psi_d(U)_{\bullet}$ is Kan. 
\end{kan}
 
 To discuss the proof of this theorem, it is convenient to introduce some notation.
 
 \theoremstyle{definition} \newtheorem{gluesub.not}[sbmn]{Note}

\begin{gluesub.not} \label{gluesub.not}
Fix an open set $U \subseteq \mathbb{R}^N$ and a PL space $P$. Throughout the rest of this article, we shall adopt the following notation: 
Let $W$ be a PL subspace of $P\times U$, and let $\pi:W \rightarrow P$ be the standard projection onto $P$ (we are not assuming that $\pi$ is necessarily a submersion). If $S$ is a PL subspace 
of $P$, then we will denote the pre-image $\pi^{-1}(S)$ by $W_S$, and we will denote the restriction 
$\pi|_{W_S}$ simply by $\pi_S$. Note that if $W$ is an element of $\Psi_d(U)(P)$, then $W_{S}$ is an element of  $\Psi_d(U)(S)$. 
\end{gluesub.not}
 
The key step in the proof of Theorem \ref{kan} 
is to show that compatible collections of PL submersions 
can be glued together to produce a global PL submersion. 
Let us try to articulate this idea more clearly.
Fix an open subset $U \subseteq \mathbb{R}^N$, a PL space $P$, and a closed PL subspace 
$W$ of $P\times U$ (we are not assuming that $W \in \Psi_d(U)(P)$ for the time being).  
As we have done before,
we shall denote the standard projection $W \rightarrow P$ by $\pi$. 
Now, suppose that 
$\{Q_i\}_{i \in \Lambda}$ is a locally finite collection of closed PL subspaces that covers $P$ 
with the property that, for each $i$ in $\Lambda$, the restriction $W_{Q_i}$ is an element of $\Psi_d(U)(Q_i)$. 
In particular, the restriction $\pi_{Q_i}$ is a PL submersion of codimension $d$. 
Then, 
as we shall see in the proof of Theorem \ref{kan}, the fact that each individual 
$\pi_{Q_i}$ is a PL submersion of codimension $d$ will imply that the map
$\pi: W \rightarrow P$ is also a PL submersion of codimension $d$.
Said differently, if $\pi$ is \textit{locally} a PL submersion (with respect to a locally finite cover of closed PL subspaces), then
it must also be so \textit{globally}. 
The following lemma is the key tool we will use to perform this 
gluing procedure for submersions.

\theoremstyle{plain} \newtheorem{pregluesub}[sbmn]{Lemma}

\begin{pregluesub} \label{pregluesub}

Let $\sigma$ be a $p$-dimensional simplex (with $p > 0$) inside some 
$\mathbb{R}^m$, $W$ an element of the set 
$\Psi_d(U)(\sigma)$, $\lambda_0$ a point in $\partial \sigma$, and $y_0$ a point in the fiber over $\lambda_0$ of the projection $\pi: W \rightarrow \sigma$. Suppose that $g: V \times \mathbb{R}^d \rightarrow W_{\partial \sigma}$ and $f: V' \times \mathbb{R}^d \rightarrow W$ are submersion charts around $y_0$ for $\pi_{\partial \sigma}: W_{\partial\sigma} \rightarrow \partial \sigma$ and $\pi: W \rightarrow \sigma$ respectively with the property that $\mathrm{Im}\hspace{0.05cm}g \subseteq \mathrm{Im}\hspace{0.05cm} f$. Then, after possibly shrinking $V$, we can find an open neighborhood $\widetilde{V}$ of $\lambda_0$ in $\sigma$ such that $\widetilde{V} \cap \partial \sigma = V$ and a submersion chart $\widetilde{g}: \widetilde{V} \times \mathbb{R}^d \rightarrow W$ around $y_0$ for the submersion $\pi: W \rightarrow \sigma$ such that 
$\widetilde{g}|_{V\times \mathbb{R}^d} = g$. 

\end{pregluesub}

\begin{proof}
First, let $h: V\times \mathbb{R}^d \rightarrow V \times \mathbb{R}^d$ be the PL embedding obtained by taking the composition
\[
\xymatrix{
V\times \mathbb{R}^d \ar[r]^{\hspace{0.1cm}g} & \mathrm{Im}\hspace{0.05cm}g \ar[r]^{f^{-1}} & V\times \mathbb{R}^d,} 
\] 
where the second map is just the restriction of $f^{-1}$ on $\mathrm{Im}\hspace{0.05cm}g$. Note that 
$h$ is actually an open PL embedding. Next, let 
$c: \partial \sigma \times [0,1] \rightarrow \sigma$ be a collar for the boundary $\partial \sigma$, i.e., $c$ is a PL embedding with the property that $c(\lambda,0) = \lambda$ for any point $\lambda \in \partial \sigma$. 
After shrinking $V$ if necessary, we can find a value 
$0 < \delta < 1$ such that $c\big(V\times [0,\delta)\big) \subseteq V'$. 
Throughout the rest of this proof, we will denote the image 
$c\big(V\times [0,\delta)\big)$ by $\widetilde{V}$.
Now, let $h_2: V\times \mathbb{R}^d \rightarrow \mathbb{R}^d$ be the second component of the map
$h: V\times \mathbb{R}^d \rightarrow V \times \mathbb{R}^d$, and define a new PL embedding 
$\widetilde{h}: \widetilde{V} \times \mathbb{R}^d \rightarrow \widetilde{V} \times \mathbb{R}^d$ by setting 
$\widetilde{h}(c(\lambda,t), x) = (c(\lambda,t), h_2(\lambda,x))$ for any point $(c(\lambda,t), x)$ in $\widetilde{V}\times \mathbb{R}^d$. 
Our desired submersion chart $\widetilde{g}: \widetilde{V} \times \mathbb{R}^d \rightarrow W$ will be defined as 
$\widetilde{g} := f \circ \widetilde{h}.$  It is straightforward to 
verify that $\widetilde{g}(\lambda_0,0) = y_0$ and that $\widetilde{g}$ commutes with the projection onto $\widetilde{V}$. Moreover, since the image of $h$ is open, it follows that the image of $\widetilde{h}$ is also open, which then implies that the image  of $\widetilde{g}$ is open in $W$. Thus, we can conclude that $\widetilde{g}$ is a submersion chart around $y_0$ for the map $\pi: W \rightarrow \sigma$. Finally, we also have that $\widetilde{g}$ extends $g$ since for any point $(\lambda, x) \in V \times \mathbb{R}^d$ we have the following: 
\[
\widetilde{g}(\lambda, x) = f \circ \widetilde{h}(c(\lambda,0), x) = f(c(\lambda,0), h_2(\lambda,x)) =  f(\lambda, h_2(\lambda, x)) =  f (h(\lambda, x)) =  g(\lambda,x) 
\]
\end{proof}

We will now give the proof of Theorem \ref{kan}. To do so, we need to review some basic terminology from PL topology. Let $K$ be some locally finite simplicial complex 
and fix a vertex $a_0$ of $K$. \textit{The star of $a_0$ in K,} denoted by $\mathrm{st}(a_0, K)$, is the subcomplex of $K$ consisting of all simplices $\sigma \in K$ 
satisfying one of the following two conditions:

\begin{itemize}

\item[(i)] $\sigma$ contains $a_0$, or

\item[(ii)] $\sigma$ is a face of a simplex $\sigma' \in K$ that contains $a_0$.  

\end{itemize} 

\textit{The link of $a_0$ in $K$,} denoted by $\mathrm{lk}(a_0, K)$, is the subcomplex of $\mathrm{st}(a_0, K)$ consisting of those simplices that do not contain $a_0$. It is not hard to show that the subspace $|\mathrm{st}(a_0, K)| - |\mathrm{lk}(a_0, K)|$ is an open neighborhood of the vertex $a_0$ in $|K|$. We will also need the following definition: For a non-negative integer $j \geq 0$, \textit{the j-th skeleton of $K$} (denoted by $K^j$) is the subcomplex of $K$ consisting of all simplices of dimension at most $j$.

\begin{proof}[\textbf{\textit{Proof of Theorem \ref{kan}:}}] 
Consider an open set $U \subseteq \mathbb{R}^N$. 
Our goal in this proof is to show that the PL set $\Psi_d(U)$ satisfies the gluing condition
described in Definition \ref{quasiPL}. 
Fix then  a PL space $P$ and let $\{Q_i\}_{i\in \Lambda}$ be a locally finite collection of closed PL subspaces 
of $P$ which covers $P$. By Theorem 3.6 of \cite{plhud}, it is possible to find a triangulation 
$h: |K| \rightarrow P$ of $P$ which also triangulates each PL subspace of the collection $\{Q_i\}_{i\in \Lambda}$. 
Thus, it is enough to prove that $\Psi_d(U)$ satisfies the desired gluing condition in the
case when the locally finite cover of closed PL subspaces is of the form 
$\{h(\sigma)\}_{\sigma \in K}$, where $(K, h)$ is a triangulation for $P$
and the collection $\{h(\sigma)\}_{\sigma \in K}$ is indexed by the simplices of $K$. 
In fact, by taking pull-backs, it suffices to prove that $\Psi_d(U)$ 
satisfies this gluing condition in the special case
when $P = |K|$ for some locally finite simplicial complex $K$ in 
$\mathbb{R}^m$ (possibly with $m = \infty$) and 
$\{Q_i\}_{i\in \Lambda}$ is the collection $\{ \sigma \}_{\sigma \in K}$
of simplices of $K$. We shall from now on only focus on this case. 
At this point,  the reader is encouraged to review the notation introduced
in Note \ref{gluesub.not}. 

Take now an element $(W^{\sigma})_{\sigma \in K}$ of the product $\prod_{\sigma \in K}\Psi_d(U)(\sigma)$ with the
property that, if $\beta$ is a face of a simplex $\sigma \in K$, then the restriction $W^{\sigma}_{\beta}$ of $W^{\sigma}$ over 
$\beta$ is equal to $W^{\beta}$.  
We need to show that there is a unique element $W$ of 
$\Psi_d(U)(|K|)$ such that, for any simplex $\sigma \in K$, the restriction 
$W_{\sigma}$ of $W$ over $\sigma$ agrees with $W^{\sigma}$.
Our candidate for such a
$W$ is defined by taking the following union in the space $|K|\times U$: 
\begin{equation} \label{candidate.union}
W := \bigcup_{\sigma \in K} W^{\sigma}. 
\end{equation}
We must prove two things: (1) $W$ is a closed PL subspace of $|K|\times U$, and (2) the standard projection 
$\pi: W \rightarrow |K|$ is a PL submersion of codimension $d$. 
To prove the first claim, note that the right-hand side of (\ref{candidate.union}) is a locally finite union
of closed PL subspaces. 
Then, since one can always  find common triangulations for a
finite collection of compatible PL charts (see Lemma 3.2 in \cite{plhud}), any locally finite union of PL subspaces 
is also a PL subspace. We thus have that $W$ is a PL subspace of $|K|\times U$. 
Moreover, via an elementary general topology argument, one can prove that the union of a locally finite collection of closed subspaces in a topological space is itself closed, which implies that
$W$ is indeed a closed PL subspace of $|K|\times U$. This concludes the proof 
of the first claim. 
It is evident that the element 
$W$ we just constructed satisfies $W_{\sigma} = W^{\sigma}$ for any
simplex $\sigma \in K$. Also, since each $W^{\sigma}$ was assumed to be an element
of $\Psi_d(U)(\sigma)$, the restriction of the projection $\pi: W \rightarrow |K|$ on 
any $W_{\sigma}$ is a PL submersion of codimension $d$. Following the conventions  
we introduced in Note \ref{gluesub.not}, for each simplex $\sigma \in K$
we shall denote the restriction of $\pi$ on $W_{\sigma}$
by $\pi_{\sigma}$. 

Next, we will prove that the projection $\pi: W \rightarrow |K|$ is a PL submersion of codimension $d$. 
Take any point $y_0$ in the PL space $W$, and denote the image $\pi(y_0)$ by $a_0$. 
After subdividing $K$ if necessary, we may assume that $a_0$ is a vertex of $K$. Now, let $\mathrm{st}(a_0, K)$ be the star of $a_0$ in $K$ and let $p$
 be the maximal value of the set  
 $\{\mathrm{dim}(\sigma) \hspace{0.05cm} : \hspace{0.05cm} \sigma \in \mathrm{st}(a_0, K) \}$. 
 Our goal is to construct a submersion chart $h:V\times \mathbb{R}^d \rightarrow W$ around $y_0$ for the map 
 $\pi: W \rightarrow |K|$ over an open neighborhood $V$ of $a_0$ contained in 
 $|\mathrm{st}(a_0, K)| - |\mathrm{lk}(a_0, K)|$.
We will construct this chart via an induction argument on the skeleta $K^j$ of $K$. 
In other words, starting with $j=0$, we will inductively produce submersion charts around $y_0$ for each projection 
$\pi_{|K^j|}: W_{|K^j|} \rightarrow |K^j|$. Since $p$ is the maximal value in 
 $\{\mathrm{dim}(\sigma) \hspace{0.05cm} : \hspace{0.05cm} \sigma \in \mathrm{st}(a_0, K) \}$, this induction will terminate when we reach $j=p$. 
First, note that the $0$-th skeleton $K^0$ is just a discrete set of points in $\mathbb{R}^m$. In particular, each point in $K^0$ is its own neighborhood. Thus, producing a submersion chart around $y_0$ in the case $j=0$ simply amounts to choosing a chart $h: \mathbb{R}^d \rightarrow W_{a_0}$ for the PL manifold $W_{a_0}$ with the property that $y_ 0 \in \mathrm{Im}\hspace{0.05cm}h$. This takes care of the base case $j=0$. 
Now, take any $j$ in $\{0, \ldots, p -1\}$ and suppose we have a  
 submersion chart $g: V \times \mathbb{R}^d \rightarrow W_{|K^j|}$ around $y_0$ for the projection $\pi_{|K^j|}: W_{|K^j|} \rightarrow |K^j|$. By taking a smaller open neighborhood of $a_0$ if necessary, we may assume that $V$ is contained in 
 $|\mathrm{st}(a_0, K^j)| - |\mathrm{lk}(a_0, K^j)|$.
Using Lemma \ref{pregluesub}, we will extend $g$ slightly (after possibly shrinking its image) over the interior of each $(j+1)$-simplex of $K$ that contains $a_0$. 
Suppose then that $\sigma_1, \ldots, \sigma_q$ are the 
$(j+1)$-simplices of $K$ that have $a_0$ as a vertex. For each 
 $i \in \{1, \ldots, q\}$, $V_i$ will denote the intersection $V \cap \partial \sigma_i$, and $g_i$ will denote the restriction of 
 the chart $g:V \times \mathbb{R}^d \rightarrow W_{|K^j|}$ on $V_i \times \mathbb{R}^d$. Additionally,  $V_0$ will denote the intersection of $V$ with the union of all simplices of $K$ of dimension at most $j$ that contain $a_0$ but which are not contained in any $(j+1)$-simplex, and $g_0$ will denote the restriction of $g:V \times \mathbb{R}^d \rightarrow W_{|K^j|}$ on $V_0 \times \mathbb{R}^d$. 
 Since $\pi_{\sigma_i}: W_{\sigma_i} \rightarrow \sigma_i$ is a PL submersion of codimension $d$ for each  $i \in \{1, \ldots, q\}$, we can find a submersion chart $f_i: V'_i \times \mathbb{R}^d \rightarrow W_{\sigma_i}$ around $a_0$ for each $\pi_{\sigma_i}$. 
 For each chart $f_i$, the set $V'_i$ is an open neighborhood of $a_0$ in $\sigma_i$. 
 By shrinking $V$ and reparameterizing $g$ if necessary, we may assume that 
 $\mathrm{Im}\hspace{0.05cm}g_i \subseteq \mathrm{Im}\hspace{0.05cm}f_i$ for all $i$ in $\{1, \ldots, q\}$. 
 Therefore, after possibly shrinking $V$ further, Lemma \ref{pregluesub} ensures that  
 we can find for each $i \in \{1, \ldots, q\}$ an open neighborhood
$\widetilde{V}_i$ of $a_0$ in $\sigma_i$ and a submersion chart
$\widetilde{g}_i: \widetilde{V}_i \times \mathbb{R}^d \rightarrow W_{\sigma_i}$ around $y_0$ for the submersion 
$\pi_{\sigma_i}: W_{\sigma_i} \rightarrow \sigma_i$ such that $\widetilde{V}_i \cap \partial \sigma_i = V_i$ and 
$\widetilde{g}_i|_{V_i \times \mathbb{R}^d} = g_i$. 
Note that any two maps in the collection $g_0, \widetilde{g}_1, \ldots, \widetilde{g}_q$ are equal when restricted on the intersection of their domains. Thus, if $\widetilde{V}$ denotes the union of the sets $V_0, \widetilde{V}_1, \ldots, \widetilde{V}_q$, we can produce a map 
$\widetilde{g}: \widetilde{V}\times \mathbb{R}^d \rightarrow W_{|K^{j+1}|}$ by gluing all the charts  $g_0, \widetilde{g}_1, \ldots, \widetilde{g}_q$. 
In other words, $\widetilde{g}: \widetilde{V}\times \mathbb{R}^d \rightarrow W_{|K^{j+1}|}$ is the unique map 
satisfying $\widetilde{g}|_{V_0\times \mathbb{R}^d} = g_0$ and $\widetilde{g}|_{\widetilde{V}_i\times \mathbb{R}^d} = \widetilde{g}_i$ for each  $i \in \{1, \ldots, q\}$.  
By pre-composing with a linear translation if necessary, we can also guarantee 
that $\widetilde{g}(a_0, 0) = y_0$, where $0$ denotes the origin in $\mathbb{R}^d$.  

Recall that we assumed $V \subseteq |\mathrm{st}(a_0, K^j)| - |\mathrm{lk}(a_0, K^j)|$. Because of this assumption, we can choose the sets $\widetilde{V}_i$ to be small enough so that the set 
$\widetilde{V}$ is contained in $|\mathrm{st}(a_0, K^{j+1})| - |\mathrm{lk}(a_0, K^{j+1})|$. If this is the case, then $\widetilde{V}$
is a subspace of $|\mathrm{st}(a_0, K^{j+1})| - |\mathrm{lk}(a_0, K^{j+1})|$ with the additional property that,
for any simplex $\delta \in K^{j+1}$ containing $a_0$, the intersection $\widetilde{V}\cap \delta$ is open in 
$\delta$. It follows that 
$\widetilde{V}$ is an open neighborhood of $a_0$ in $|K^{j+1}|$. 
Moreover, it is evident that $\widetilde{g}: \widetilde{V}\times \mathbb{R}^d \rightarrow W_{|K^{j+1}|}$ 
commutes with the projection onto $\widetilde{V}$. 
Thus, to conclude that $\widetilde{g}$ is a submersion chart for the map $\pi_{|K^{j+1}|}: W_{|K^{j+1}|} \rightarrow |K^{j+1}|$, we would just need to prove that $\widetilde{g}$ is a PL homeomorphism onto its image. 
However, instead of proving this, it is much easier to prove that $\widetilde{g}$ is a PL homeomorphism when restricted to an appropriate subspace of $\widetilde{V}\times \mathbb{R}^d$ that
contains $(a_0, 0)$. Indeed, pick an open neighborhood of  $a_0$ contained in $\widetilde{V}$. By abuse of notation, we will denote this new open neighborhood of $a_0$ by $V$. 
Also, let us choose $V$ so that $\overline{V} \subseteq \widetilde{V}$. 
Since $|\mathrm{st}(a_0, K^{j+1})|$ is compact and $\widetilde{V} \subseteq |\mathrm{st}(a_0, K^{j+1})|$, the condition 
$\overline{V} \subseteq \widetilde{V}$ guarantees that $\overline{V}$ is also compact. 
By construction, the map $\widetilde{g}: \widetilde{V}\times \mathbb{R}^d \rightarrow W_{|K^{j+1}|}$ is piecewise linear and a bijection onto its image. Then, since $\overline{V} \times [-1,1]^d$ is compact and $W_{|K^{j+1}|}$ is Hausdorff, the restriction 
$\widetilde{g}|_{\overline{V} \times [-1,1]^d}$ is a PL homeomorphism from $\overline{V} \times [-1,1]^d$ to 
$\widetilde{g}(\overline{V} \times [-1,1]^d)$. Consequently, the restriction of $\widetilde{g}$ on $V \times (-1,1)^d$ is also a PL homeomorphism onto its image. Finally, by reparameterizing the domain of $\widetilde{g}|_{V \times (-1,1)^d}$ with a PL homeomorphism $\mathbb{R}^d \stackrel{\cong}{\longrightarrow} (-1,1)^d$, we obtain a submersion chart
$V\times \mathbb{R}^d \rightarrow W_{|K^{j+1}|}$ around $y_0$ for the projection 
$\pi_{|K^{j+1}|}: W_{|K^{j+1}|} \rightarrow |K^{j+1}|$. 

By iterating the above procedure $p$ times (where, recall, $p$ is the maximum of the set  
$\{\mathrm{dim}(\sigma) \hspace{0.05cm} : \hspace{0.05cm} \sigma \in \mathrm{st}(a_0, K) \}$),
we can produce
 a submersion chart $h:V\times \mathbb{R}^d \rightarrow W_{|K^p|}$ around $y_0$ for the projection 
$\pi_{|K^p|}: W_{|K^p|} \rightarrow |K^p|$ over some open neighborhood $V$ of $a_0$ contained in 
$|\mathrm{st}(a_0, K^p)| - |\mathrm{lk}(a_0, K^p)|$. 
However, since $p$ is the maximal possible dimension of any simplex in $\mathrm{st}(a_0, K)$, we have that $V\cap |K| = V \cap |K^p|$, which implies that
the projections $\pi: W \rightarrow |K|$ and $\pi_{|K^p|}: W_{|K^p|} \rightarrow |K^p|$ agree over $V$. 
It follows that $h:V\times \mathbb{R}^d \rightarrow W_{|K^p|}$ is actually a  submersion chart around $y_0$ for 
the projection $\pi: W \rightarrow |K|$. 
Since the point $y_0 \in W$ that we fixed at the beginning was arbitrary, we can conclude that $\pi: W \rightarrow |K|$ is a PL submersion of codimension $d$,
and it follows that $W$ is an element of $\Psi_d(U)(|K|)$.
\end{proof}

\theoremstyle{definition} \newtheorem{moreclasssub}[sbmn]{Remark}

\begin{moreclasssub}  \label{moreclasssub}
The proof of  
Theorem \ref{kan} works verbatim to prove that the PL sets $\psi_d(N,k)$ introduced 
in Definition \ref{spacemanfil} are also quasi-PL spaces.
In particular, each simplicial set of the form $\psi_d(N,k)_{\bullet}$
is Kan. 
\end{moreclasssub}

We shall adopt the following terminology in the remainder of this paper. 

\theoremstyle{definition} \newtheorem{defclass}[sbmn]{Definition}

\begin{defclass}  \label{defclass}

Fix a PL set $\mathcal{F}: \mathbf{PL}^{op} \rightarrow \mathbf{Sets}$ (not necessarily a quasi-PL space),
a triangulation $(K,h)$ for a PL space $P$,  
and an order relation on the set of vertices $\mathrm{Vert}(K)$.  
Also, let $\mathcal{S}_{K,h}: \mathcal{F}(P) \rightarrow \mathbf{Ssets}(K_{\bullet}, \mathcal{F}_{\bullet})$ 
be the function defined in Theorem \ref{classsub}. 
For an element $W$
in $\mathcal{F}(P)$, we will say that the map
of simplicial sets $F_W := \mathcal{S}_{K,h}(W)$ 
\textit{classifies $W$ relative to the triangulation $(K,h)$}.

\end{defclass}

If $\mathcal{F}$ happens to be a quasi-PL space, then Theorem \ref{classsub} asserts that any morphism 
$g: K_{\bullet} \rightarrow \mathcal{F}_{\bullet}$ can be realized as the classifying map 
of some element in $\mathcal{F}(P)$. 

As we shall see in Proposition \ref{concord.homotopy} below,
we can use classifying maps (in the sense of Definition \ref{defclass}) 
to produce homotopies of maps into spaces of the form $|\mathcal{F}_{\bullet}|$,
where $\mathcal{F}$ is an arbitrary PL set. 
For the statement of Proposition \ref{concord.homotopy},
we shall need the following definition, which will play a fundamental
role in many of the arguments that will appear later in this paper. 
Even though we will mainly apply this definition to PL sets of the form 
$\Psi_d(U)$ and $\psi_d(N,k)$, we choose to formulate it for arbitrary PL sets. 
 
\theoremstyle{definition}  \newtheorem{concord}[sbmn]{Definition}

\begin{concord} \label{concord}
Fix a PL set $\mathcal{F}: \mathbf{PL}^{op} \rightarrow \mathbf{Sets}$, 
a PL space $P$, and two elements
 $W_0$ and $W_1$ of the set $\mathcal{F}(P)$.
 Also, let $i_0$, $i_1: P \rightarrow [0,1] \times P$ be the embeddings defined 
by $i_0(\lambda) = (0,\lambda)$ and   $i_1(\lambda) = (1,\lambda)$ respectively. 
We say that $W_0$ and $W_1$ are \textit{concordant} if there is 
an element $\widetilde{W}$ in $\mathcal{F}([0,1]\times P)$
such that $i_0^{*}\widetilde{W} = W_0$ and $i_1^*\widetilde{W} = W_1$. 
In this case, we say that $\widetilde{W}$
is a \textit{concordance from $W_0$ to $W_1$}.
\end{concord}

Now we turn to Proposition \ref{concord.homotopy}, which tells us that concordances are a good source of homotopies. 

\theoremstyle{plain} \newtheorem{concord.homotopy}[sbmn]{Proposition}

\begin{concord.homotopy} \label{concord.homotopy}
Let $\mathcal{F}$ be a PL set, $P$ a compact $PL$ space, $(K,h)$ a triangulation of $P$, and $\leq$ an order relation on  
$\mathrm{Vert}(K)$. By the compactness of $P$, we may assume that $K$ is a finite simplicial complex in some Euclidean space 
$\mathbb{R}^m$. Moreover, consider two elements $W$ and $W'$ of the set 
$\mathcal{F}(P)$, and let 
\[
F_{W}: K_{\bullet} \rightarrow \mathcal{F}_{\bullet} \qquad F_{W'}: K_{\bullet} \rightarrow \mathcal{F}_{\bullet}
\]
be the classifying maps of $W$ and $W'$ relative to the triangulation $(K,h)$ (see Definition \ref{defclass}). 
If $\widetilde{W} \in \mathcal{F}([0,1]\times P)$  
is a concordance from $W$ to $W'$, then $\widetilde{W}$ induces a
homotopy $H: [0,1] \times |K_{\bullet}| \rightarrow |\mathcal{F}_{\bullet}|$ with 
$H_0 = |F_W|$ and $H_1 = |F_{W'}|$. 
\end{concord.homotopy}

\begin{proof}
First, identify the product $[0,1] \times |K|$ with its image under the embedding 
$[0,1]\times |K| \hookrightarrow \mathbb{R}^{m+1}$ defined by 
$(t,x) \mapsto (t,x)$.  Next, 
for each simplex $\sigma$ of $K$, subdivide the product $[0,1]\times \sigma$
with the simplicial complex $\widetilde{K}_{\sigma}$ obtained by pulling back 
the standard triangulation of the prism $[0,1]\times \Delta^{\mathrm{dim}(\sigma)}$ 
along the obvious linear isomorphism $[0,1]\times\sigma \stackrel{\cong}{\longrightarrow} [0,1]\times \Delta^{\mathrm{dim}(\sigma)}$. 
By assembling all the triangulations $\widetilde{K}_{\sigma}$, we obtain 
a simplicial complex $\widetilde{K}$ which  
subdivides $[0,1]\times |K|$ and agrees with $K$ on 
$\{0\} \times |K|$ and $\{1\}\times |K|$. Finally, fix an ordering $\widetilde{\leq}$ on the vertices of $\widetilde{K}$ which extends the order relation $\leq$ on $\mathrm{Vert}(K)$. Evidently, the tuple 
$(\widetilde{K}, \mathrm{Id}_{[0,1]}\times h)$ is a triangulation of the product $[0,1]\times P$. 
Now, let $F_{\widetilde{W}}: \widetilde{K}_{\bullet} \rightarrow \mathcal{F}_{\bullet}$
be the classifying map of the concordance $\widetilde{W}$
 relative to $(\widetilde{K}, \mathrm{Id}_{[0,1]}\times h)$. Also, for $j \in \{ 0,1 \}$, let
 $I_j: K_{\bullet} \hookrightarrow \widetilde{K}_{\bullet}$ be the simplicial set map induced by the inclusion
 $i_j: |K| \hookrightarrow |\widetilde{K}|$
 defined by $i_j(\lambda) = (\hspace{0.08cm}j,\lambda)$. 
 Since $(\widetilde{K}, \widetilde{\leq})$ agrees with $(K,\leq)$ when restricted to the bottom and top face
 of $[0,1]\times |K|$, we have that $F_W = F_{\widetilde{W}}\circ I_0$ and $F_{W'} = F_{\widetilde{W}}\circ I_1$. 
 Thus, after identifying $|\widetilde{K}_{\bullet}|$ with $[0,1]\times |K_{\bullet}|$, 
 the geometric realization of $F_{\widetilde{W}}$ 
 gives a homotopy $H$ from $|F_{W}|$ to $|F_{W'}|$. 
\end{proof}

\theoremstyle{definition} \newtheorem{semi.class.version}[sbmn]{Remark}

\begin{semi.class.version} \label{semi.class.version} 

Fix a quasi-PL space $\mathcal{F}$ and let $\widetilde{\mathcal{F}}_{\bullet}$ be the semi-simplicial set obtained by forgetting the degeneracies of the underlying simplicial set $\mathcal{F}_{\bullet}$. By doing an argument nearly identical to the one we did in the proof of Theorem \ref{classsub}, one can obtain a version of this theorem for the semi-simplicial set $\widetilde{\mathcal{F}}_{\bullet}$. More precisely, let $P$ be a PL space, and fix a
triangulation $(K,h)$ of $P$ and an order relation $<$ on $\mathrm{Vert}(K)$. Also, let 
$\Delta_{\mathrm{inj}}$ be the category of finite sets $[p]$ and \textit{injective} order-preserving functions. 
The tuple $(K, <)$ induces a semi-simplicial set $K_{\bullet}$ as follows: A $p$-simplex in $K_{\bullet}$ is a strictly increasing 
chain of vertices $v_0 < v_1 < \ldots < v_p$ that span a simplex of $K$. Additionally, we define structure maps in exactly the same way we did in the case of simplicial sets. 
Then, if $\mathbf{Semi}\text{-}\mathbf{Ssets}$ denotes the category of semi-simplicial sets, one can show that the function 
\[
\mathcal{F}(P) \rightarrow \mathbf{Semi}\text{-}\mathbf{Ssets}(K_{\bullet}, \widetilde{\mathcal{F}}_{\bullet})
\]
which sends an element $W\in \mathcal{F}(P)$
 to the morphism of semi-simplicial sets 
$F_W: K_{\bullet} \rightarrow \widetilde{\mathcal{F}}_{\bullet}$
defined by
$F_W(v_0 < \ldots < v_p) = s_{v_0\ldots v_p}^*h^*W $
is a bijection. 
Also, the proof we did in part (ii) of Theorem \ref{classsub} works verbatim for semi-simplicial sets. Consequently,
$\widetilde{\mathcal{F}}_{\bullet}$ is Kan if $\mathcal{F}$ is a quasi-PL space.
Moreover, Definition \ref{defclass} and Proposition \ref{concord.homotopy} admit versions
that apply to semi-simplicial sets, and 
the proof of Proposition \ref{concord.homotopy}  can be replicated without any issues in the
semi-simplicial setting.
All of the above observations apply to the semi-simplicial sets
$\widetilde{\Psi}_d(U)_{\bullet}$ and $\widetilde{\psi}_d(N,k)_{\bullet}$ corresponding to the PL sets
$\Psi_d(U)$ and $\psi_d(N,k)$ respectively.  
\end{semi.class.version}

\subsection{The subdivision map} \label{secsubmap}  
In this subsection, we will perform a construction that works for PL sets in general, not 
just quasi-PL spaces. In particular, all the constructions we shall describe in this subsection
work for PL sets of the form $\Psi_d(U)$ and $\psi_d(N,k)$.
Consider then an arbitrary PL set $\mathcal{F}: \mathbf{PL}^{op} \rightarrow \mathbf{Sets}$. As we did
in Remark \ref{semi.class.version}, we shall denote by $\widetilde{\mathcal{F}}_{\bullet}$ the semi-simplicial set obtained
by forgetting the degeneracies of the underlying simplicial set $\mathcal{F}_{\bullet}$. 
Our goal is to show that it is possible to 
construct a map
 $\rho: |\widetilde{\mathcal{F}}_{\bullet}| \rightarrow  |\widetilde{\mathcal{F}}_{\bullet}|$ 
 from the geometric realization $|\widetilde{\mathcal{F}}_{\bullet}|$ to itself  
 satisfying the following properties:  
 
 \begin{enumerate}
 
 \item $\rho$ is homotopic to the identity map on $ |\widetilde{\mathcal{F}}_{\bullet}|$. 
 
 \item For any morphism of semi-simplicial sets $f:X_{\bullet} \rightarrow \widetilde{\mathcal{F}}_{\bullet}$
and any non-negative integer $r$,
there is a unique morphism 
$g: \mathrm{sd}^r X_{\bullet} \rightarrow \widetilde{\mathcal{F}}_{\bullet}$
making the following diagram commute
\[
\xymatrix{
\left|X_{\bullet}\right| \ar[r]^{\left|f\right|}  & \left| \widetilde{\mathcal{F}}_{\bullet}\right| \ar[d]^{\rho^r} \\
\left|\mathrm{sd}^r X_{\bullet}\right| \ar[u]^{\cong}\ar[r]^{\left|g\right|} & \left| \widetilde{\mathcal{F}}_{\bullet}\right|. }
\]
 
 \end{enumerate}
 
 In the second property stated above, $ \mathrm{sd}^r X_{\bullet} $ denotes the \textit{$r$-th barycentric subdivision} of $X_{\bullet}$, and the left vertical map in the previous diagram is the canonical 
homeomorphism from  $\left|\mathrm{sd}^r X_{\bullet}\right|$ to $\left| X_{\bullet} \right|$. 
In view of this second property, the automorphism 
 $\rho$ will be called \textit{the subdivision map of} $\widetilde{\mathcal{F}}_{\bullet}$. 

To explain the construction of the map  
$\rho: |\widetilde{\mathcal{F}}_{\bullet}| \rightarrow  |\widetilde{\mathcal{F}}_{\bullet}|$, we need to introduce some notation: 

\begin{itemize}

\item[$\cdot$] By abuse of notation, we shall denote the canonical simplicial complex triangulating the standard $p$-simplex $\Delta^p$ simply by $\Delta^p$. 

\item[$\cdot$] For any face $F$ of the standard $p$-simplex $\Delta^p$, we will denote its barycentric point by $bF$. 

\item[$\cdot$] $\mathrm{sd}\hspace{0.05cm}\Delta^p$ will denote \textit{the first barycentric subdivision} of the standard simplicial complex triangulating $\Delta^p$. That is, $\mathrm{sd}\hspace{0.05cm}\Delta^p$ is the simplicial complex whose vertices are all barycentric points $bF$ and whose 
higher-dimensional simplices are all possible convex hulls $\langle bF_0, bF_1, \ldots bF_p \rangle$ corresponding to strictly increasing flags of faces $F_0 \subset F_1 \subset \ldots \subset F_p$. 

\item[$\cdot$] We can define an order relation $<_{\mathrm{s}p}$ on 
$\mathrm{Vert}(\mathrm{sd}\hspace{0.05cm}\Delta^p)$ by setting $bF_i <_{\mathrm{s}p} bF_j$ if and only if 
$F_i \subset F_j$. Also, we will denote the obvious order relation on $\mathrm{Vert}(\Delta^p)$ by $<_p$. 

\item[$\cdot$] Fix a semi-simplicial set $Y_{\bullet}$. As typically done in the literature, we will denote by 
$\mathrm{Ex}\hspace{0.05cm} Y_{\bullet}$ the semi-simplicial set whose $p$-simplices are all semi-simplicial morphisms of the form $\mathrm{sd} \hspace{0.05cm}\Delta^p_{\bullet} \rightarrow Y_{\bullet}$. 

\item[$\cdot$] For any PL set $\mathcal{F}$, we will denote the semi-simplicial set $\widetilde{\mathcal{F}}_{\bullet}$
simply by $\mathcal{F}_{\bullet}$. 
Also, if $W$ is a $p$-simplex of $\mathcal{F}_{\bullet}$, we will denote its 
characteristic map $\Delta^p_{\bullet} \rightarrow \mathcal{F}_{\bullet}$ by $f_W$.

\end{itemize}

\theoremstyle{definition} \newtheorem{construction.rho}[sbmn]{Note}

\begin{construction.rho} \label{construction.rho} 

\textbf{(Construction of the subdivision map)} Fix a PL set $\mathcal{F}$, and
let $\Delta^p_{\bullet}$ and 
$\mathrm{sd}\hspace{0.05cm}\Delta^p_{\bullet}$ be the 
semi-simplicial sets induced respectively by the ordered simplicial complexes 
$(\Delta^p, <_p)$ and $(\mathrm{sd}\hspace{0.05cm}\Delta^p, <_{\mathrm{s}p} )$ (see Remark \ref{semi.class.version}). 
The key observation for constructing the 
subdivision map $\rho$ is the following: Given any $p$-simplex $W$ of the semi-simplicial set $\mathcal{F}_{\bullet}$ 
(i.e., an element of $\mathcal{F}(\Delta^p)$), we can triangulate the base-space $\Delta^p$ with 
$\mathrm{sd}\hspace{0.05cm}\Delta^p$. Thus, by the semi-simplicial version of Definition \ref{defclass}, the triangulation 
$(\mathrm{sd}\hspace{0.05cm}\Delta^p, \mathrm{Id}_{\Delta^p})$ and the order relation 
$<_{\mathrm{s}p}$ will induce a classifying map 
$\mathrm{sd}\hspace{0.05cm}\Delta^p_{\bullet} \rightarrow \mathcal{F}_{\bullet}$, which we will denote by 
$f_{\mathrm{s},W}$ throughout this subsection.  It is not hard to verify that the assignment $W \mapsto f_{\mathrm{s},W}$ is natural with respect to order-preserving inclusions $\theta: [q] \rightarrow [p]$; i.e., for any $p$-simplex $W$ of 
$\mathcal{F}_{\bullet}$ and any morphism  $\theta: [q] \rightarrow [p]$ in 
$\Delta_{\mathrm{inj}}$, the following diagram commutes: 
\begin{equation} \label{construction.rho.diagram}
\xymatrix{ 
\mathrm{sd}\hspace{0.05cm}\Delta^q_{\bullet} \ar[rr]^{\theta_{\mathrm{s}}} \ar[rd]^{f_{\mathrm{s},\theta^*W}} & & \mathrm{sd}\hspace{0.05cm}\Delta^p_{\bullet}  \ar[ld]_{f_{\mathrm{s},W}} \\
 & \mathcal{F}_{\bullet} & }
\end{equation}
The top arrow is the natural morphism between subdivisions induced by $\theta$
(see Remark \ref{construction.rho.remark} below). It follows that the assignments $W \mapsto f_{\mathrm{s},W}$ define a morphism of semi-simplicial sets
$\widetilde{\gamma}: \mathcal{F}_{\bullet} \rightarrow \mathrm{Ex}\hspace{0.05cm} \mathcal{F}_{\bullet}$.  Then, if 
$\gamma: \mathrm{sd} \hspace{0.05cm} \mathcal{F}_{\bullet} \rightarrow \mathcal{F}_{\bullet}$ is the adjoint of 
$\widetilde{\gamma}$, we define \textit{the subdivision map} 
$\rho: |\mathcal{F}_{\bullet}| \rightarrow |\mathcal{F}_{\bullet}|$ to be the composition 
\[
\xymatrix{
|\mathcal{F}_{\bullet}| \ar[r]^{h^{-1}} & |\mathrm{sd} \hspace{0.05cm} \mathcal{F}_{\bullet}| 
\ar[r]^{|\gamma|} & |\mathcal{F}_{\bullet}|,}
\]
where $h^{-1}$ is the inverse of the canonical homeomorphism 
$h: |\mathrm{sd} \hspace{0.05cm} \mathcal{F}_{\bullet}|  \rightarrow |\mathcal{F}_{\bullet}|$. 

\end{construction.rho}

\theoremstyle{definition} \newtheorem{construction.rho.remark}[sbmn]{Remark}

\begin{construction.rho.remark} \label{construction.rho.remark} 
Let $\theta: [q] \rightarrow [p]$ be a morphism in $\Delta_{\mathrm{inj}}$. The top map $\theta_{\mathrm{s}}$ in 
(\ref{construction.rho.diagram}) is defined by 
$bF_{i_0} <_{\mathrm{s}q} \ldots <_{\mathrm{s}q} bF_{i_m} \mapsto 
\widetilde{\theta}(bF_{i_0}) <_{\mathrm{s}p} \ldots <_{\mathrm{s}p} \widetilde{\theta}(bF_{i_m})$, where 
$\widetilde{\theta}$ is the linear map $\Delta^q \rightarrow \Delta^p$ that sends a basis vector 
$e_j \in \mathbb{R}^{q+1}$ to  $e_{\theta(j)} \in \mathbb{R}^{p+1}$. We can guarantee that each 
image $\widetilde{\theta}(bF_{i_k})$ is a barycentric point in $\Delta^p$ because the map $\widetilde{\theta}$
is injective, and it therefore maps barycentric points to barycentric points, which is a property that does not hold for simplicial maps in general. This is the reason why we work with semi-simplicial sets for the construction of $\rho$ instead of simplicial sets. 
\end{construction.rho.remark}

As mentioned earlier, the subdivision map satisfies the following fundamental property. 

\theoremstyle{plain}  \newtheorem{subhopid}[sbmn]{Proposition}

\begin{subhopid} \label{subhopid}
The subdivision map $\rho: |\mathcal{F}_{\bullet}| \rightarrow |\mathcal{F}_{\bullet}|$
is homotopic to the identity map on $|\mathcal{F}_{\bullet}|$. 
\end{subhopid}

\begin{proof}
The construction of the homotopy between the subdivision map $\rho$ and 
$\mathrm{Id}_{|\mathcal{F}_{\bullet}|}$
is similar to the construction of $\rho$ itself, although it might seem substantially longer due to all the 
preliminaries we need to discuss. 
First, consider the product $[0,1]\times \Delta^p$, which we will view as a subspace of 
$\mathbb{R}^{p+2}$. We will triangulate $[0,1]\times \Delta^p$ with a simplicial complex 
$\mathcal{R}\hspace{0.05cm}\Delta^p$ defined as follows. On $\{0\} \times \Delta^p$, $\mathcal{R}\Delta^p$ will agree with the standard simplicial complex triangulating $\Delta^p$. On $\{1\} \times \Delta^p$, $\mathcal{R}\Delta^p$ agrees with the first barycentric subdivision $\mathrm{sd}\hspace{0.05cm}\Delta^p$. Finally, a subset of the form
\[
\big\{ (0, e_{i_0}), \ldots, (0, e_{i_k}) \big\} \cup \big\{ (1, bF_{j_0}), \ldots, (1, bF_{j_{k'}}) \big\} 
\] 
will span a $(k+ k')$-simplex of $\mathcal{R}\hspace{0.05cm}\Delta^p$ if and only if 
we have the following:

\begin{itemize}

\item[i)] The convex hull $\langle bF_{j_0}, \ldots, bF_{j_{k'}} \rangle$ is a simplex of 
$\mathrm{sd}\hspace{0.05cm}\Delta^p$, and 

\item[ii)] $\langle e_{i_0}, \ldots, e_{i_k} \rangle$ is a face of the smallest face in the collection 
$\{F_{j_0}, \ldots, F_{j_{k'}} \}$.

\end{itemize}

The reader may have perhaps noticed that $\mathcal{R}\hspace{0.05cm}\Delta^p$ is the 
triangulation typically used to prove excision for singular homology. Moreover, we 
can turn 
$\mathcal{R}\Delta^p$ into an ordered simplicial complex by defining the following order $<_{\mathrm{r}p}$ on
$\mathrm{Vert}(\mathcal{R}\hspace{0.05cm}\Delta^p)$: 

\begin{itemize}

\item[-] The restriction of $<_{\mathrm{r}p}$ on $\{0\} \times \Delta^p$ and $\{1\} \times \Delta^p$ agrees with $<_p$ and $<_{\mathrm{s}p}$ respectively. 

\item[-] $(0, e_i) <_{\mathrm{r}p} (1, bF_{j})$ for any pair of vertices  $(0, e_i)$ and $(1, bF_{j})$. 

\end{itemize}

Now, let $\mathcal{R}\hspace{0.05cm}\Delta^p_{\bullet}$ be the semi-simplicial set induced by 
the pair $(\mathcal{R}\hspace{0.05cm}\Delta^p, <_{\mathrm{r}p})$ and,  
for any morphism $\theta: [q] \rightarrow [p]$ in $\Delta_{\mathrm{inj}}$, let 
$\theta_{\mathrm{r}}: \mathcal{R}\hspace{0.05cm}\Delta^q_{\bullet} \rightarrow \mathcal{R}\hspace{0.05cm}\Delta^p_{\bullet}$ be the morphism which maps a simplex of the form 
$(0, e_{i_0}) <_{\mathrm{r}q} \ldots <_{\mathrm{r}q} (0, e_{i_k}) <_{\mathrm{r}q} (1, bF_{j_0}) <_{\mathrm{r}q} \ldots <_{\mathrm{r}q} (1, bF_{j_{k'}})$ to 
\[
(0, \widetilde{\theta}(e_{i_0})) <_{\mathrm{r}p} \ldots <_{\mathrm{r}p} (0, \widetilde{\theta}( e_{i_k})) <_{\mathrm{r}p} (1, \widetilde{\theta}( bF_{j_0})) <_{\mathrm{r}p} \ldots <_{\mathrm{r}p} (1, \widetilde{\theta}( bF_{j_{k'}})).
\]
In the previous definition, $\widetilde{\theta}$ denotes once again the linear map $\Delta^q \rightarrow \Delta^p$ which sends $e_j$ to $e_{\theta(j)}$. To proceed with our construction, we need to introduce a new kind of
endofunctor on $\mathbf{Semi}\text{-}\mathbf{Ssets} $ which will play in this proof a role 
similar to the one played by the functor $\mathrm{Ex}(\cdot)$ in the construction of the
subdivision map $\rho$. Namely, for any semi-simplicial set $Y_{\bullet}$, we define 
$\mathrm{Ex}^{\mathrm{r}}\hspace{0.05cm}Y_{\bullet}$ to be the semi-simplicial set
whose $p$-simplices are all semi-simplicial maps of the form 
$\mathcal{R}\hspace{0.05cm}\Delta^p_{\bullet} \rightarrow Y_{\bullet}$. The 
structure map of $\mathrm{Ex}^{\mathrm{r}}\hspace{0.05cm}Y_{\bullet}$ 
induced by a morphism $\theta \in \Delta_{\mathrm{inj}}([q], [p])$ is given by pre-composition 
with $\theta_{\mathrm{r}}$. Next, we will define a morphism 
$\mathcal{F}_{\bullet} \rightarrow \mathrm{Ex}^{\mathrm{r}}\hspace{0.05cm} \mathcal{F}_{\bullet}$ as follows.
For any $p$-simplex $W$ of $\mathcal{F}_{\bullet}$, the product $[0,1] \times W$ is an element 
of $\mathcal{F}([0,1]\times \Delta^p)$. Then, by the semi-simplicial version of
Definition \ref{defclass}, the triangulation 
$(\mathcal{R}\hspace{0.05cm}\Delta^p, \mathrm{Id}_{[0,1]\times \Delta^p})$
and the order relation $<_{\mathrm{r}p}$
induce a classifying map $\mathcal{R}\hspace{0.05cm}\Delta^p_{\bullet} \rightarrow \mathcal{F}_{\bullet}$
for $[0,1]\times W$, which we will denote by $f_{\mathrm{r}, W}$. 
As it was the case for the maps  $f_{\mathrm{s}, W}$ defined in 
Note \ref{construction.rho}, it is not hard to verify that
$f_{\mathrm{r},W}\circ \theta_{\mathrm{r}} = f_{\mathrm{r}, \theta^*W}$ for all $W \in \mathcal{F}_p$ and all morphisms
$\theta: [q] \rightarrow [p]$ in $\Delta_{\mathrm{inj}}$. Therefore, all the assignments 
$W \mapsto f_{\mathrm{r}, W}$  assemble into a map of semi-simplicial sets 
$\mathcal{F}_{\bullet} \rightarrow \mathrm{Ex}^{\mathrm{r}}\hspace{0.05cm} \mathcal{F}_{\bullet}$, 
which we will denote by $\widetilde{\Gamma}$ throughout the rest of this proof. 

In the same way that the functor $Y_{\bullet} \mapsto \mathrm{Ex}\hspace{0.05cm}Y_{\bullet}$
is right adjoint to $X_{\bullet} \mapsto \mathrm{sd}\hspace{0.05cm}X_{\bullet}$, the functor 
$Y_{\bullet} \mapsto \mathrm{Ex}^{\mathrm{r}}\hspace{0.05cm}Y_{\bullet}$ is right adjoint to a functor
$\mathcal{R}: \mathbf{Semi}\text{-}\mathbf{Ssets} \rightarrow  \mathbf{Semi}\text{-}\mathbf{Ssets}$ whose construction is similar to that of $\mathrm{sd}: \mathbf{Semi}\text{-}\mathbf{Ssets} \rightarrow  \mathbf{Semi}\text{-}\mathbf{Ssets}$; i.e.,
we define $\mathcal{R}$ via a colimit construction involving the simplex category.
Recall that, for any semi-simplicial set 
$X_{\bullet}$,  \textit{the simplex category} $\Delta\downarrow X$ is the category 
whose objects are semi-simplicial set maps $f:\Delta^p_{\bullet} \rightarrow X_{\bullet}$, and morphisms are commutative diagrams of the form
\begin{equation} \label{simplex.morph}
\xymatrix{
\Delta^p_{\bullet} \ar[r]^{f} \ar[d]_{\theta} & X_{\bullet}  \ar@{=}[d]\\
\Delta^q_{\bullet} \ar[r]^{g}  & X_{\bullet}. }
\end{equation}
The left vertical map in this diagram is induced by a morphism $\theta: [p] \rightarrow [q]$ in $\Delta_{\mathrm{inj}}$. Then, $\mathcal{R}: \mathbf{Semi}\text{-}\mathbf{Ssets} \rightarrow  \mathbf{Semi}\text{-}\mathbf{Ssets}$ is the functor defined by $X_{\bullet} \mapsto \mathcal{R}\hspace{0.05cm}X_{\bullet}$, where 
$\mathcal{R}\hspace{0.05cm}X_{\bullet}$ is the colimit of the functor 
$\Delta\downarrow X \rightarrow \mathbf{Semi}\text{-}\mathbf{Ssets}$ which sends an object 
 $\Delta^p_{\bullet} \rightarrow X_{\bullet}$ to $\mathcal{R}\hspace{0.05cm}\Delta^p_{\bullet}$ and
 sends a morphism of the form (\ref{simplex.morph}) to $\theta_{\mathrm{r}}$ (to obtain
 $\mathrm{sd}\hspace{0.05cm}X_{\bullet}$, we would use instead the functor that sends 
  $\Delta^p_{\bullet} \rightarrow X_{\bullet}$ to $\mathrm{sd}\hspace{0.05cm}\Delta^p_{\bullet}$, and a
  morphism (\ref{simplex.morph}) to $\theta_{\mathrm{s}}$). For any semi-simplicial set map
  $g: X_{\bullet} \rightarrow Y_{\bullet}$, the induced map $\mathcal{R}\hspace{0.05cm}g$ is the one obtained via the universal property of colimits. It is straightforward to verify that
  $X_{\bullet} \mapsto \mathcal{R}\hspace{0.05cm}X_{\bullet}$ is left adjoint to the functor $Y_{\bullet} \mapsto \mathrm{Ex}^{\mathrm{r}}\hspace{0.05cm}Y_{\bullet}$. 
 
We are now ready to define the homotopy between $\mathrm{Id}_{|\mathcal{F}_{\bullet}|}$ 
and the subdivision map. Let then 
$\Gamma: \mathcal{R}\hspace{0.05cm} \mathcal{F}_{\bullet} \rightarrow \mathcal{F}_{\bullet}$
be the adjoint of the map $\widetilde{\Gamma}: \mathcal{F}_{\bullet} \rightarrow 
\mathrm{Ex}^{\mathrm{r}}\hspace{0.05cm} \mathcal{F}_{\bullet}$ given by $W \mapsto f_{\mathrm{r}, W}$. 
Note that for each $p \geq 0$, we have two obvious inclusions 
$j_{0\hspace{-0.03cm},\hspace{0.05cm} p}: \Delta^p_{\bullet} \hookrightarrow \mathcal{R}\hspace{0.05cm}\Delta^p_{\bullet}$ and
$j_{1\hspace{-0.03cm},\hspace{0.05cm} p}: \mathrm{sd}\hspace{0.05cm}\Delta^p_{\bullet} \hookrightarrow \mathcal{R}\hspace{0.05cm}\Delta^p_{\bullet}$. By the way we defined the morphisms 
$f_{\mathrm{r}, W}: \mathcal{R}\hspace{0.05cm}\Delta^p_{\bullet} \rightarrow \mathcal{F}_{\bullet}$, 
we have for each $p$-simplex $W$ of $\mathcal{F}_{\bullet}$ that  
$f_{W} = f_{\mathrm{r}, W} \circ j_{0\hspace{-0.03cm},\hspace{0.05cm} p}$
and $f_{\mathrm{s}, W} = f_{\mathrm{r}, W} \circ j_{1\hspace{-0.03cm},\hspace{0.05cm} p}$.
Recall that $f_W$ is the characteristic map $\Delta^p_{\bullet} \rightarrow \mathcal{F}_{\bullet}$ of $W$,
and $f_{\mathrm{s}, W}$ is the classifying map of $W$ with respect to the triangulation
$(\mathrm{sd}\hspace{0.05cm}\Delta^p, \mathrm{Id}_{\Delta^p})$. Then, 
by a standard colimit argument, the collection of maps 
$\{ \hspace{0.07cm} j_{0\hspace{-0.03cm},\hspace{0.05cm} p}\}_{p\geq 0}$
will induce  an inclusion
$\mathcal{J}_0: \mathcal{F}_{\bullet} \hookrightarrow  \mathcal{R}\hspace{0.05cm} \mathcal{F}_{\bullet}$
such that $\Gamma \circ \mathcal{J}_0 = \mathrm{Id}_{\mathcal{F}_{\bullet}}$. 
Similarly, the collection  $\{ \hspace{0.07cm} j_{1\hspace{-0.03cm},\hspace{0.05cm} p}\}_{p\geq 0}$ induces an inclusion 
$\mathcal{J}_1:\mathrm{sd}\hspace{0.05cm} \mathcal{F}_{\bullet} \hookrightarrow  \mathcal{R}\hspace{0.05cm}\mathcal{F}_{\bullet}$ with the property that $\Gamma \circ \mathcal{J}_1 = \gamma$, where $\gamma$
is the morphism $\mathrm{sd} \hspace{0.05cm} \mathcal{F}_{\bullet} \rightarrow \mathcal{F}_{\bullet}$
that we obtained in Note \ref{construction.rho}.
The reader should think of $\mathcal{J}_0$ and
$\mathcal{J}_1$ as the inclusions into the `bottom' and `top' face of 
$\mathcal{R}\hspace{0.05cm} \mathcal{F}_{\bullet}$. Finally, for $j \in \{0,1\}$, 
let  $i_j$ 
be the inclusion $|\mathcal{F}_{\bullet}| \hookrightarrow [0,1] \times |\mathcal{F}_{\bullet}|$
defined by  $x \mapsto (\hspace{0.07cm}j, x)$. Using the fact that
the geometric realization functor $|\cdot|$ commutes with colimits, 
we can show that there exists a homeomorphism 
$H: |\mathcal{R}\hspace{0.05cm}\mathcal{F}_{\bullet}| \rightarrow [0,1] \times |\mathcal{F}_{\bullet}|$
satisfying the following: 
\[
H\circ |\mathcal{J}_0| = i_0 \qquad 
H\circ |\mathcal{J}_1| = i_1\circ h.
\]
In the right-hand equality, $h$ is the canonical homeomorphism
$|\mathrm{sd} \hspace{0.05cm} \mathcal{F}_{\bullet}|  \stackrel{\cong}{\longrightarrow} |\mathcal{F}_{\bullet}|$. 
Putting together all the constructions discussed in this proof, we obtain the following commutative diagram:
\begin{equation} \label{subdivision.comm}
\xymatrix{
|\mathcal{F}_{\bullet}| \ar[d]_{i_1} & & |\mathrm{sd} \hspace{0.05cm} \mathcal{F}_{\bullet}| \ar[d]_{|\mathcal{J}_1|}
\ar[ll]_{\hspace{0.35cm}h}  \ar@/^/[drr]^{|\gamma|} & & \\
[0,1]\times |\mathcal{F}_{\bullet}| & & |\mathcal{R}\hspace{0.05cm}\mathcal{F}_{\bullet}| \ar[ll]_{\hspace{0.4cm}H} 
\ar[rr]^{|\Gamma|} & & 
|\mathcal{F}_{\bullet}| \\
|\mathcal{F}_{\bullet}| \ar[u]^{i_0} & & |\mathcal{F}_{\bullet}| \ar[ll]_{\quad\mathrm{Id}_{|\mathcal{F}_{\bullet}|}} \ar@/_/[urr]_{\quad\mathrm{Id}_{|\mathcal{F}_{\bullet}|}} \ar[u]^{|\mathcal{J}_0|} & & } 
\end{equation}
Recall that the subdivision map $\rho$ is equal to the composition $|\gamma|\circ h^{-1}$. 
Then,  if  
$\mathcal{H}: [0,1]\times |\mathcal{F}_{\bullet}| \rightarrow |\mathcal{F}_{\bullet}|$
is the map
defined by $\mathcal{H} := |\Gamma| \circ H^{-1}$,
the commutativity of (\ref{subdivision.comm}) implies that
$\mathcal{H}\circ i_0 = \mathrm{Id}_{|\mathcal{F}_{\bullet}|}$
and  $\mathcal{H}\circ i_1 = \rho$. In other words, we have shown that $\mathcal{H}$
is a homotopy between $\mathrm{Id}_{|\mathcal{F}_{\bullet}|}$ and the subdivision map
$\rho$, which concludes the proof. 
\end{proof}

In the next proposition, we will prove the second property of the subdivision map that we discussed at the 
beginning of this subsection. 

\theoremstyle{plain} \newtheorem{unisub}[sbmn]{Proposition}

\begin{unisub} \label{unisub}
For any map of semi-simplicial sets $f: X_{\bullet} \rightarrow \mathcal{F}_{\bullet}$
and any  integer $r>0$, there is a unique morphism
$g: \mathrm{sd}^r \hspace{0.05cm} X_{\bullet} \rightarrow \mathcal{F}_{\bullet}$
which makes the following diagram commute
\[
\xymatrix{
\left|X_{\bullet}\right| \ar[r]^{ \hspace{-0.45cm} \left|f\right|}  & \left| \mathcal{F}_{\bullet} \right| \ar[d]^{\rho^r} \\
\left|\mathrm{sd}^r \hspace{0.05cm} X_{\bullet}\right| \ar[u]^{\cong}\ar[r]^{ \left|g\right|} & \left|  \mathcal{F}_{\bullet} \right|. }
\]
The left vertical map in this diagram is the canonical homeomorphism
$\left|\mathrm{sd}^r \hspace{0.05cm}  X_{\bullet}\right| \stackrel{\cong}{\longrightarrow} |X_{\bullet}|$.
\end{unisub} 
 
\begin{proof}
 
In the case $r=1$, we simply take $g = \gamma \circ \mathrm{sd}\hspace{0.05cm}f$, where 
$\mathrm{sd}\hspace{0.05cm}f$ is the map between subdivisions induced by 
$f: X_{\bullet} \rightarrow \mathcal{F}_{\bullet}$ and
$\gamma: \mathrm{sd} \hspace{0.05cm} \mathcal{F}_{\bullet} \rightarrow \mathcal{F}_{\bullet}$
is the map we defined in Note \ref{construction.rho}. For the case $r>1$, 
we just iterate the previous argument.    
\end{proof}

In the next result, we highlight a very useful special case of Proposition \ref{unisub}. 

\theoremstyle{plain} \newtheorem{unisub.cor}[sbmn]{Proposition}

\begin{unisub.cor} \label{unisub.cor}

Let $(K,h)$ be a triangulation of a PL space $P$ and $\leq$ an order relation on the set of vertices of $K$. Also,  fix an element $W$ of the set $\mathcal{F}(P)$. Then, if we apply Proposition \ref{unisub} 
to the classifying map 
$f: K_{\bullet} \rightarrow \mathcal{F}_{\bullet}$ 
of $W$ relative to the triangulation $(K,h)$, the map  
$g: \mathrm{sd}^r \hspace{0.05cm} K_{\bullet} \rightarrow \mathcal{F}_{\bullet}$
appearing at the bottom of the diagram
\[
\xymatrix{
\left|K_{\bullet}\right| \ar[r]^{ \left|f\right|}  & \left| \mathcal{F}_{\bullet} \right| \ar[d]^{\rho^r} \\
\left|\mathrm{sd}^r \hspace{0.05cm} K_{\bullet}\right| \ar[u]^{\cong}\ar[r]^{  \left|g\right|} & \left|  \mathcal{F}_{\bullet} \right| }
\] 
is the classifying map of $W$ relative to the triangulation $(\mathrm{sd}^r \hspace{0.05cm} K, h)$ of $P$. 
\end{unisub.cor}

\begin{proof}
It is enough to prove this result assuming that $r =1$. First, we will consider the
case when $W$ is an element of $\mathcal{F}(\Delta^p)$ (i.e., $W$ is a $p$-simplex of $\mathcal{F}_{\bullet}$)
and $f_W: \Delta^p_{\bullet} \rightarrow \mathcal{F}_{\bullet}$ is the classifying map of $W$ relative to the standard triangulation $(\Delta^p, \mathrm{Id}_{\Delta^p})$ of $\Delta^p$
(i.e., $f_W$ is the characteristic map of the simplex $W$). 
In this case, we need to check that the diagram  
\begin{equation} \label{diag.unisub1}
\xymatrix{
\left|\Delta^p_{\bullet}\right| \ar[rr]^{  \left|f_W\right|}  & & \left| \mathcal{F}_{\bullet} \right| 
\ar[d]^{\rho} \\
\left|\mathrm{sd} \hspace{0.05cm} \Delta^p_{\bullet}\right| \ar[u]^{\cong}\ar[rr]^{  
\left|f_{\mathrm{s}, W}\right|} & & \left|  \mathcal{F}_{\bullet} \right| }
\end{equation}
commutes, where the bottom map $f_{\mathrm{s}, W}$ is the classifying map of $W$ relative to the triangulation
$(\mathrm{sd} \hspace{0.05cm} \Delta^p, \mathrm{Id}_{\Delta^p})$ of $\Delta^p$ (we are using the same notation that we introduced in Note \ref{construction.rho}).  To show this, we first note that 
(\ref{diag.unisub1}) agrees with the outer rectangle of the diagram   
\begin{equation} \label{diag.unisub2}
\xymatrix{
|\Delta^p_{\bullet}| \ar[rr]^{  \left|f_W\right|} & & |\mathcal{F}_{\bullet}| \ar[d]^{\cong} \\
|\mathrm{sd} \hspace{0.05cm} \Delta^p_{\bullet}| \ar[u]^{\cong} \ar@{=}[d] 
\ar[rr]^{|\mathrm{sd} \hspace{0.05cm} f_W|}  
& &
 |\mathrm{sd} \hspace{0.05cm} \mathcal{F}_{\bullet}| \ar[d]^{|\gamma|} \\
|\mathrm{sd} \hspace{0.05cm} \Delta^p_{\bullet}| 
\ar[rr]^{ \left|f_{\mathrm{s}, W}\right|} & & |\mathcal{F}_{\bullet}|, \\
}
\end{equation}
where the two vertical maps in the top rectangle are the obvious homeomorphisms, and the right-vertical map
in the bottom rectangle is the geometric realization of the morphism
$\gamma: \mathrm{sd} \hspace{0.05cm} \mathcal{F}_{\bullet} \rightarrow \mathcal{F}_{\bullet}$ that we defined in
Note \ref{construction.rho}. The top rectangle in (\ref{diag.unisub2}) is clearly 
commutative. 
To see that the bottom rectangle also commutes,  let us consider first the diagram
\begin{equation} \label{diag.unisub3}
\xymatrix{
\Delta^p_{\bullet} \ar[rr]^{ f_W }  \ar@{=}[d] & &  \mathcal{F}_{\bullet}
\ar[d]^{\widetilde{\gamma}} \\
\Delta^p_{\bullet}   \ar[rr]^{ \widetilde{f}_W }  
& & \mathrm{Ex}\hspace{0.05cm}\mathcal{F}_{\bullet}, }
\end{equation}
where the right-vertical map $\widetilde{\gamma}$ is the adjoint 
of $\gamma: \mathrm{sd} \hspace{0.05cm} \mathcal{F}_{\bullet} \rightarrow \mathcal{F}_{\bullet}$ 
(see Note \ref{construction.rho}) 
and the bottom map is the characteristic map of 
$f_{\mathrm{s},W}: \mathrm{sd} \hspace{0.05cm} \Delta^p_{\bullet} \rightarrow \mathcal{F}_{\bullet}$
(when viewed as a $p$-simplex of $\mathrm{Ex}\hspace{0.05cm}\mathcal{F}_{\bullet}$). 
 Since the morphism $\widetilde{\gamma}$ maps $W$ to 
$f_{\mathrm{s},W}$, the diagram given in (\ref{diag.unisub3}) is commutative. Thus, since 
the functors $Y_{\bullet} \mapsto \mathrm{Ex}\hspace{0.05cm}Y_{\bullet}$
and $X_{\bullet} \mapsto \mathrm{sd}\hspace{0.05cm}X_{\bullet}$
are adjoint to each other, 
it follows that the bottom rectangle 
of (\ref{diag.unisub2}) also commutes. Consequently, we can conclude that (\ref{diag.unisub1})
is commutative. 

Now, consider a PL space $P$, and fix a triangulation 
$(K,h)$ of $P$ and an order relation $\leq$ on $\mathrm{Vert}(K)$. Moreover,
pick an element $W \in \mathcal{F}(P)$, and let
$f: K_{\bullet} \rightarrow \mathcal{F}_{\bullet}$ and 
$g: \mathrm{sd}\hspace{0.05cm}K_{\bullet} \rightarrow \mathcal{F}_{\bullet}$ be the classifying maps
of $W$ with respect to the triangulations 
$(K, h)$ and $(\mathrm{sd}\hspace{0.05cm}K, h)$ respectively. 
To finish this proof, we need to show 
that the diagram 
\begin{equation} \label{diag.unisub4}
\xymatrix{
\left|K_{\bullet}\right| \ar[r]^{  \left|f\right|}  & \left| \mathcal{F}_{\bullet} \right| \ar[d]^{\rho} \\
\left|\mathrm{sd} \hspace{0.05cm} K_{\bullet}\right| \ar[u]^{\cong}\ar[r]^{  \left|g\right|} & \left|  \mathcal{F}_{\bullet} \right| }
\end{equation}
commutes, where the left-vertical map is the canonical homeomorphism from 
$\left|\mathrm{sd} \hspace{0.05cm} K_{\bullet}\right|$ to $|K_{\bullet}|$. 
To show this, it is enough to prove that this diagram is locally commutative. In other words,
for any simplex $\sigma$ of $K$, we must show that the diagram
\begin{equation} \label{diag.unisub5}
\xymatrix{
\left|\sigma_{\bullet}\right| \ar[r]^{  \left|f\right|}  & \left| \mathcal{F}_{\bullet} \right| \ar[d]^{\rho} \\
\left|\mathrm{sd} \hspace{0.05cm} \sigma_{\bullet}\right| \ar[u]^{\cong}\ar[r]^{   \left|g\right|} & \left|  \mathcal{F}_{\bullet} \right| }
\end{equation}
 is commutative. 
 In this last diagram, $\sigma_{\bullet}$ is the semi-simplicial set induced by the order 
 $\leq$ and the subcomplex of $K$ which triangulates $\sigma.$ 
 However, after identifying
 $\sigma_{\bullet}$ and $\mathrm{sd} \hspace{0.05cm} \sigma_{\bullet}$ with 
 $\Delta^{p}_{\bullet}$ and $\mathrm{sd}\hspace{0.05cm}\Delta^p_{\bullet}$ respectively
 (where $p$ is the dimension of $\sigma$), (\ref{diag.unisub5})
 turns into a diagram of the form (\ref{diag.unisub1}), which we already proved is commutative. 
Therefore, (\ref{diag.unisub5}) must also be commutative. 
 Since we can do this argument for any simplex $\sigma$
 of $K$, we can conclude that (\ref{diag.unisub4}) commutes. 
 \end{proof}

The next proposition will be used in several of the arguments presented 
in \S \ref{section5} and \S \ref{section6}.  

\theoremstyle{plain} \newtheorem{subsimphop}[sbmn]{Proposition}

\begin{subsimphop} \label{subsimphop}
Fix a PL set $\mathcal{F}$ and let $\mathcal{F}_{\bullet}$ be the semi-simplicial set
obtained by forgetting the degeneracies of the underlying simplicial set of $\mathcal{F}$. Also,
let $X_{\bullet}$ be a sub-semi-simplicial set of $\mathcal{F}_{\bullet}$ with the property that $|X_{\bullet}|$ is invariant both under the subdivision map $\rho$ of 
$\mathcal{F}_{\bullet}$ and the homotopy $\mathcal{H}$ between 
$\mathrm{Id}_{|\mathcal{F}_{\bullet}|}$ and $\rho$ (see the proof of Proposition \ref{subhopid}). Then, any map of pairs 
\[
(\Delta^p, \partial\Delta^p) \stackrel{f}{\longrightarrow} (|\mathcal{F}_{\bullet}|, |X_{\bullet}|)
\]
is homotopic, as a map of pairs, to a composite of the form
\[
(\Delta^p, \partial\Delta^p) \stackrel{f'}{\longrightarrow} (|K_{\bullet}|, |K'_{\bullet}|)  \stackrel{|g|}{\longrightarrow} 
(|\mathcal{F}_{\bullet}|, |X_{\bullet}|), 
\]
where $(K_{\bullet}, K'_{\bullet})$ is a pair of semi-simplicial sets induced by a pair $(K,K')$ of finite ordered simplicial complexes $(K,K')$.  
\end{subsimphop}

\begin{proof}

Since $\Delta^p$ is compact, the image of  
$f$ only intersects finitely many simplices
of $|\mathcal{F}_{\bullet}|$. Let
$L_{\bullet}$ be the sub-semi-simplicial set
of $\mathcal{F}_{\bullet}$ generated by 
these simplices, and let $L'_{\bullet}$
be the sub-semi-simplicial set of $L_{\bullet}$
generated by those simplices which intersect
the image $f(\partial\Delta^p)$. If
$i: L_{\bullet} \hookrightarrow \mathcal{F}_{\bullet}$ is the natural inclusion from $L_{\bullet}$ to 
$\mathcal{F}_{\bullet}$, then Proposition \ref{unisub} guarantees that there exists 
a map of semi-simplicial sets
$g: \mathrm{sd}^2 \hspace{0.05cm} L_{\bullet} \rightarrow \mathcal{F}_{\bullet}$ 
whose geometric realization makes the following diagram commute
\[
\xymatrix{
\left|L_{\bullet}\right| \ar[r]^{  \left|i\right|} 
& \left| \mathcal{F}_{\bullet} \right| \ar[d]^{ \rho^2} \\
\left|\mathrm{sd}^2 \hspace{0.05cm} L_{\bullet}\right|  \ar[r]^{   \left|g\right|}  
\ar[u]^{  \vspace{-0.7cm} h}  
& \left|  \mathcal{F}_{\bullet} \right|. }
\]
The left-vertical map $h$ is the canonical homeomorphism 
$h:\left|\mathrm{sd}^2L_{\bullet}\right| \rightarrow\left|L_{\bullet}\right|$. 
Now, let us express the map $f$ as the composition 
\[
(\Delta^p,\partial\Delta^p) \stackrel{\hat{f}}{\longrightarrow}  (\left|L_{\bullet}\right|, \left|L'_{\bullet}\right|) \stackrel{|i|}{\longrightarrow} 
(|\mathcal{F}_{\bullet} |, |X_{\bullet}|),
\] 
where $\hat{f}$ is the map obtained by restricting the target of $f$ to $\left|L_{\bullet}\right|$.
Since the subdivision map $\rho$ is homotopic to 
the identity map $\mathrm{Id}_{|\mathcal{F}_{\bullet}|}$, it follows that
$f$ is homotopic to the composite
\begin{equation} \label{eq1}
(\Delta^p,\partial\Delta^p)  \stackrel{f'}{\longrightarrow}  (\left|\mathrm{sd}^2L_{\bullet}\right|,\left|\mathrm{sd}^2L'_{\bullet}\right|) \stackrel{\left|g\right|}{\longrightarrow} (\left| \mathcal{F}_{\bullet} \right|, \left| X_{\bullet}\right|),
\end{equation}
where $f' = h^{-1} \circ \hat{f}$. 
In (\ref{eq1}), we can guarantee that $g(\mathrm{sd}^2 \hspace{0.05cm} L'_{\bullet}) \subseteq X_{\bullet}$ because we are assuming that 
$|X_{\bullet}|$ is invariant under the subdivision map $\rho$.  Furthermore, since
$|X_{\bullet}|$ is also invariant under the homotopy $\mathcal{H}$ between $\mathrm{Id}_{|\mathcal{F}_{\bullet}|}$ 
and $\rho$, the map given in (\ref{eq1}) 
and $f$ are homotopic as maps of pairs. 

Finally, as explained in \cite{RSD}, the second barycentric
subdivision $\mathrm{sd}^2\hspace{0.05cm}Y_{\bullet}$ of a (locally)
finite semi-simplicial set $Y_{\bullet}$
is isomorphic to a semi-simplicial set induced by  
a (locally) finite ordered simplicial complex. Therefore, in (\ref{eq1})
we can replace 
$(\mathrm{sd}^2 \hspace{0.05cm} L_{\bullet},\mathrm{sd}^2 \hspace{0.05cm} L'_{\bullet})$ 
with a pair $(K_{\bullet},K_{\bullet}')$ induced by 
finite ordered simplicial complexes $(K,K')$.  
\end{proof}

\theoremstyle{definition} \newtheorem{simp.approx}[sbmn]{Remark}

\begin{simp.approx} \label{simp.approx}
The reader might be wondering why we do not simply take $(K,K')$ to be a subdivision of
the pair $(\Delta^p, \partial\Delta^p)$ in the previous proposition and apply a semi-simplicial version of the 
Simplicial Approximation Theorem to proof this result. To the author's knowledge, there
are two simplicial approximation theorems that hold in the semi-simplicial setting. Both of these can
be found in \cite{RSD} (in that paper, semi-simplcial sets are called $\Delta$-\textit{sets}). 
However, neither of these two simplicial approximation theorems can be used to prove Proposition \ref{subsimphop}.
For one of these, one needs $X_{\bullet}$ to be Kan, a condition that fails in all the cases in which we will 
apply Proposition \ref{subsimphop} later in this paper. For the other version, one needs
the map $\partial \Delta^p \rightarrow |X_{\bullet}|$ to be \textit{simplicial} 
(the definition of simplicial map between geometric realizations 
of semi-simplicial sets can be found on page 327 of \cite{RSD}).  
However, not all maps between geometric realizations of semi-simplicial sets 
are simplicial, so we cannot use this other version to prove Proposition \ref{subsimphop} either. 
\end{simp.approx}

\subsection{The spectrum of PL manifolds}    \label{section2.5}

As explained in the introduction, the spaces 
$|\Psi_d(\mathbb{R}^N)_{\bullet}|$ will be the stages of a spectrum, and it is the purpose
of this subsection to define the structure maps
\[
\mathcal{E}_N: S^1\wedge |\Psi_d(\mathbb{R}^N)_{\bullet}| \rightarrow |\Psi_d(\mathbb{R}^{N+1})_{\bullet}|
\]
of this spectrum.  

To define $\mathcal{E}_N$, we will first define a \textit{translation map} $T^+:  \Delta^1_{\bullet} \times\Psi_d(\mathbb{R}^N)_{\bullet} \rightarrow \Psi_d(\mathbb{R}^{N+1})_{\bullet}$. Geometrically, this map will push any element $W$ to $+\infty$ along the extra coordinate direction in $\mathbb{R}^{N+1}$. Let us fix once and for all a PL homeomorphism
\begin{equation} \label{plhomeo.fixed}
f:[0,1) \rightarrow [0,\infty).
\end{equation}
The definition of the translation map $T^+$ requires several steps. 

\textit{Step 1.} For any non-negative integer $p$, let $F_p:  [0,1)\times\Delta^p\times \mathbb{R}^N \rightarrow [0,1]\times\Delta^p\times \mathbb{R}^{N+1}$ be the PL embedding defined by
\[
(t,\lambda, x_1, \ldots,x_N) \mapsto (t,\lambda,x_1, \ldots, x_N, f(t)).
\]
For each $p$-simplex $W \in \Psi_d(\mathbb{R}^N)_p$, we will denote the image of $[0,1)\times W$ under the PL embedding $F_p$ by $\widetilde{W}^{+}$. It is straightforward to verify that 
 $\widetilde{W}^{+}$ is an element of $\Psi_d(\mathbb{R}^{N+1})([0,1]\times \Delta^p)$. Moreover, $\widetilde{W}^{+}$ is a concordance from $W$ to $\varnothing$ 
 when viewed as elements of $\Psi_d(\mathbb{R}^{N+1})_p$. 

 \textit{Step 2.} Consider the product $[0,1] \times \Delta^p$, which we will view as a subspace of 
$\mathbb{R}^{p+2}$. Also, let $K_p$ be the canonical simplicial complex triangulating $[0,1] \times \Delta^p$
 (i.e., $K_p$ is the simplicial complex typically used to prove homotopy invariance for singular homology). 
We can define an order relation $\leq$ on  
 $\mathrm{Vert}(K^p) = \mathrm{Vert}(\{0\}\times \Delta^p) \cup \mathrm{Vert}(\{1\}\times \Delta^p)$ by  
 ordering all vertices in  $\{0\}\times \Delta^p$ and $\{1\}\times \Delta^p$ in the obvious way, and
declaring all top vertices to be greater than the bottom ones. It is not hard to see that, 
after identifying $[0,1]$ with $\Delta^1$, the simplicial set $K_{\bullet}^p$ induced by 
$(K^p,\leq)$ is equal to $\Delta^1_{\bullet}\times \Delta^p_{\bullet}$. Then, for any $p$-simplex $W$ of 
$\Psi_d(\mathbb{R}^N)_{\bullet}$, the concordance
$\widetilde{W}^{+}$ that we defined in Step 1 will induce a map of the form
$F_{\widetilde{W}^{+}}: \Delta^1_{\bullet}\times\Delta^p_{\bullet} \rightarrow  \Psi_d(\mathbb{R}^{N+1})_{\bullet}$, i.e., 
$F_{\widetilde{W}^{+}}$ is the map that classifies $\widetilde{W}^{+}$ relative to the triangulation 
$(K^p, \mathrm{Id}_{[0,1]\times \Delta^p})$.   From now on, for any 
$W \in \Psi_d(\mathbb{R}^N)_p$, 
we will denote the map $F_{\widetilde{W}^{+}}$ simply by $T^{W}$. 

 \textit{Step 3.} Finally, for any $p$-simplex $W$ of $\Psi_d(\mathbb{R}^N)_{\bullet}$ and any morphism $\eta:[q]\rightarrow [p]$ in the category $\Delta$, the diagram 
\begin{equation} \label{pret2}
\xymatrix{
\Delta^1_{\bullet}\times\Delta_{\bullet}^{q}  \ar[r]^{\mathrm{Id}_{\Delta^1_{\bullet}}\times\eta} \ar[d]_{T^{\eta^*W}} & \Delta^1_{\bullet} \times\Delta_{\bullet}^{p}\ar[d]^{T^{W}} \\
\Psi_d(\mathbb{R}^{N+1})_{\bullet} \ar@{=}[r] & \Psi_d(\mathbb{R}^{N+1})_{\bullet}}
\end{equation}   
is commutative. Then, by a standard colimit argument, there exists a unique map 
$T^+:  \Delta^1_{\bullet} \times \Psi_d(\mathbb{R}^{N})_{\bullet}\rightarrow \Psi_d(\mathbb{R}^{N+1})_{\bullet}$ 
such that  $T^{+}\circ \big( \mathrm{Id}_{\Delta^1_{\bullet}}\times f_W  \big) = T ^W$ for any $p$-simplex 
$W$ in $\Psi_d(\mathbb{R}^{N})_{\bullet}$. Again, recall that $f_W: \Delta^p_{\bullet} \rightarrow \Psi_d(\mathbb{R}^{N})_{\bullet}$ is the characteristic map of the simplex $W$. This morphism 
$T^+$ is our desired translation map.

We can follow the procedure outlined above to define a \textit{negative} translation map $T^-:  \Delta^1_{\bullet} \times \Psi_d(\mathbb{R}^{N})_{\bullet}\rightarrow \Psi_d(\mathbb{R}^{N+1})_{\bullet}$ which pushes elements $W$ of $\Psi_d(\mathbb{R}^N)_{\bullet}$ to $-\infty$ along the extra coordinate direction in $\mathbb{R}^{N+1}$. For the definition of $T^-$, we need to use the PL homeomorphism $-f: [0,1) \rightarrow (-\infty, 0]$
instead of  the map $f:[0,1)\rightarrow [0,\infty)$ that we fixed in (\ref{plhomeo.fixed}).  

Next, via canonical homeomorphisms, we can identify the domains of $|T^+|$ and $|T^-|$ with the products 
$[0,1]\times |\Psi_d(\mathbb{R}^{N})_{\bullet}|$ and
$[-1,0]\times |\Psi_d(\mathbb{R}^{N})_{\bullet}|$ respectively. Once we do these identifications, we can glue $|T^+|$ and $|T^-|$ 
 along the subspace $\{0\} \times |\Psi_d(\mathbb{R}^{N})_{\bullet}| $ to produce a map 
\begin{equation} \label{pre.structure.map}
\mathcal{T}_N: [-1,1] \times |\Psi_d(\mathbb{R}^{N})_{\bullet}| \rightarrow |\Psi_d(\mathbb{R}^{N+1})_{\bullet}|.
\end{equation}
Finally, let us denote the canonical base-points of $|\Psi_d(\mathbb{R}^{N})_{\bullet}|$ and $|\Psi_d(\mathbb{R}^{N+1})_{\bullet}|$ by $\bullet_N$ and $\bullet_{N+1}$ respectively (see Remark \ref{spacerem}).  By construction, 
the map $\mathcal{T}_N$ in (\ref{pre.structure.map}) maps both $[-1,1]\times \{\bullet_N\}$ and $\{-1,1\}\times |\Psi_d(\mathbb{R}^{N})_{\bullet}|$ to  the base-point $\bullet_{N+1}$. Therefore, $\mathcal{T}_N$ descends to a map of the form
\begin{equation} \label{spectrummap}
\mathcal{E}_N: S^1\wedge |\Psi_d(\mathbb{R}^N)_{\bullet}| \rightarrow |\Psi_d(\mathbb{R}^{N+1})_{\bullet}|.
\end{equation}
We are now ready to give the following definition.

\theoremstyle{definition} \newtheorem{specpsi}[sbmn]{Definition}
 
\begin{specpsi} \label{specpsi}

The \textit{spectrum of PL manifolds}, denoted by $\Psi^{\mathrm{PL}}_d$, is the  
spectrum whose $N$-th space is equal to $|\Psi_d(\mathbb{R}^N)_{\bullet}|$
and whose structure maps $\mathcal{E}_N$ are those defined in (\ref{spectrummap}).

\end{specpsi}

\theoremstyle{definition} \newtheorem{contractible.choice}[sbmn]{Remark}
 
\begin{contractible.choice} \label{contractible.choice}

We close this section by observing that the weak homotopy type of the spectrum $\Psi^{\mathrm{PL}}_d$ does not depend on the map
$f: [0, 1) \rightarrow [0,\infty)$ that we fixed in (\ref{plhomeo.fixed}). 
Indeed, we can view any PL homeomorphism 
$[0,1) \stackrel{\cong}{\longrightarrow} [0,\infty)$ as a 0-simplex of the 
simplicial set $\mathcal{H}_{\bullet}$ whose $p$-simplices are PL homeomorphisms 
$\Delta^p \times [0,1)  \stackrel{\cong}{\longrightarrow} \Delta^p \times [0,\infty)$ which commute with the projection onto $\Delta^p$. It is not hard to prove that this simplicial set $\mathcal{H}_{\bullet}$ is contractible.   Now, instead of just using one single PL homeomorphism 
$f: [0, 1) \rightarrow [0,\infty)$, we can define a spectrum $\Phi$ where we use all possible PL homeomorphisms 
$[0,1)  \stackrel{\cong}{\longrightarrow}  [0,\infty)$ at once. More precisely, let $\varnothing_{\bullet}$ denote once again the subsimplicial set of $\Psi_d(\mathbb{R}^N)_{\bullet}$ consisting of all the empty manifolds $\varnothing$
(see Remark \ref{spacerem}). 
The $N$-th stage of $\Phi$ will be the space
$\mathcal{X}_N$ obtained by taking the geometric realization
$|\Psi_d(\mathbb{R}^N)_{\bullet} \times \mathcal{H}_{\bullet}|$ and then collapsing the subspace 
$|\varnothing_{\bullet} \times \mathcal{H}_{\bullet}|$ to a point, which we will denote by $\tilde{\bullet}_N$. From now on, we take 
$\tilde{\bullet}_N$ to be the base-point of $\mathcal{X}_N$. The structure maps 
$S^1\wedge \mathcal{X}_N \rightarrow \mathcal{X}_{N+1}$ are obtained as follows: 

\begin{itemize}

\item[$\cdot$] First, for any $p$-simplex $(W,G)$ of $\Psi_d(\mathbb{R}^N)_{\bullet} \times \mathcal{H}_{\bullet}$, we let 
$\widetilde{W}^{G,+}$ denote the image of $[0,1)\times W$ under the PL embedding 
$ [0,1)\times\Delta^p\times \mathbb{R}^N \rightarrow [0,1]\times\Delta^p\times \mathbb{R}^{N+1}$
defined by $(t,\lambda, x_1, \ldots,x_N) \mapsto (t,\lambda,x_1, \ldots, x_N, G(\lambda,t)).$ Again, it is not hard to
verify that $\widetilde{W}^{G,+}$ is an element of 
$\Psi_d(\mathbb{R}^{N+1})([0,1]\times \Delta^p)$. 

\item[$\cdot$] For each simplex $(W,G)$ of $\Psi_d(\mathbb{R}^N)_{\bullet} \times \mathcal{H}_{\bullet}$, let 
$F_{\widetilde{W}^{G,+}}: \Delta^1_{\bullet}\times\Delta^p_{\bullet} \rightarrow  \Psi_d(\mathbb{R}^{N+1})_{\bullet}$  be the classifying map of $\widetilde{W}^{G,+}$ relative to the triangulation
$(K^p, \mathrm{Id}_{[0,1]\times \Delta^p})$ of $[0,1]\times \Delta^p$.
By essentially repeating the same argument given in Step 3 above, we can glue all the morphisms 
$F_{\widetilde{W}^{G,+}}$ together to produce a translation map 
$T^+:  \Delta^1_{\bullet} \times \Psi_d(\mathbb{R}^{N})_{\bullet}\times \mathcal{H}_{\bullet}\rightarrow \Psi_d(\mathbb{R}^{N+1})_{\bullet}$. Then, we define a new morphism 
$\mathcal{T}^+: \Delta^1_{\bullet} \times \Psi_d(\mathbb{R}^{N})_{\bullet}\times \mathcal{H}_{\bullet} \rightarrow \Psi_d(\mathbb{R}^{N+1})_{\bullet} \times \mathcal{H}_{\bullet}$ by setting 
$\mathcal{T}^+ := T^+ \times \mathrm{Id}_{\mathcal{H}_{\bullet}}$. 

\item[$\cdot$] In a similar fashion, we can also construct a map of simplicial sets of the form  
$\mathcal{T}^-:  \Delta^1_{\bullet} \times \Psi_d(\mathbb{R}^{N})_{\bullet}\times \mathcal{H}_{\bullet} \rightarrow \Psi_d(\mathbb{R}^{N+1})_{\bullet} \times \mathcal{H}_{\bullet}$. Then, after identifying the domains of 
$|\mathcal{T}^+|$ and $|\mathcal{T}^-|$ with the products
$[0,1]\times |\Psi_d(\mathbb{R}^{N})_{\bullet}\times \mathcal{H}_{\bullet}|$
and $[-1,0]\times |\Psi_d(\mathbb{R}^{N})_{\bullet}\times \mathcal{H}_{\bullet}|$ respectively, we can glue 
$|\mathcal{T}^+|$ and $|\mathcal{T}^-|$ together to produce a new map 
$\mathcal{T}_N: [-1,1] \times |\Psi_d(\mathbb{R}^{N})_{\bullet} \times \mathcal{H}_{\bullet}| 
\rightarrow |\Psi_d(\mathbb{R}^{N+1})_{\bullet} \times \mathcal{H}_{\bullet}|.$ The composition of 
$\mathcal{T}_N$ and the quotient map 
$|\Psi_d(\mathbb{R}^{N+1})_{\bullet} \times \mathcal{H}_{\bullet}| \rightarrow \mathcal{X}_{N+1}$ maps 
both $\{-1,1\}\times|\Psi_d(\mathbb{R}^{N})_{\bullet} \times \mathcal{H}_{\bullet}| $ and 
$[-1,1]\times |\varnothing_{\bullet} \times \mathcal{H}_{\bullet}|$ to the base-point $\tilde{\bullet}_{N+1}$. Thus, 
$\mathcal{T}_N$ descends to a map $\widetilde{\mathcal{E}}_N: S^1\wedge \mathcal{X}_{N} \rightarrow \mathcal{X}_{N+1}$.
This is our desired structure map. 
\end{itemize}

Now consider again the PL homeomorphism $f: [0,1) \rightarrow [0,\infty)$ that we fixed in 
(\ref{plhomeo.fixed}). For each $N$, we can define a morphism
$\delta_N: \Psi_d(\mathbb{R}^{N})_{\bullet} \rightarrow \Psi_d(\mathbb{R}^{N})_{\bullet} \times \mathcal{H}_{\bullet}$ 
by setting $W \mapsto (W, \mathrm{Id}_{\Delta^p} \times f)$ for all $p$-simplices $W$. These morphisms 
$\delta_N$ induce a map of spectra  $\delta: \Psi_d^{\mathrm{PL}} \rightarrow \Phi$ which, by the contractibility of $\mathcal{H}_{\bullet}$, is a weak homotopy equivalence. Thus, regardless of the PL homeomorphism $f$ we use, the spectrum $\Psi_d^{\mathrm{PL}}$ will always have the same weak homotopy type as $\Phi$.

\end{contractible.choice}

\section{The piecewise linear cobordism category}   \label{section4}

In this section, we will introduce the  
\textit{piecewise linear cobordism category} $\mathsf{Cob}_d^{\mathrm{PL}}$, which is the central
object of study in this paper. Throughout this article, we have used simplicial
sets to define piecewise linear versions of the
spaces studied in \cite{GRW}, and we will
apply this strategy again for the definition of $\mathsf{Cob}_d^{\mathrm{PL}}$. 
In particular, $\mathsf{Cob}_d^{\mathrm{PL}}$ will be defined as  
a \textit{non-unital simplicial category}, i.e.,
a non-unital category whose sets of objects and morphisms
are actually simplicial sets, and whose
structure maps are morphisms of simplicial sets. 

Besides defining the piecewise linear cobordism category, in this section we will
begin the proof of the following result, which is the first step to
prove the main theorem of this article.

\theoremstyle{plain} \newtheorem{DesCob}{Theorem}[section]

\begin{DesCob} \label{DesCob}
There is a weak homotopy equivalence
\[
B\mathsf{Cob}_d^{\mathrm{PL}} \simeq |\psi_d(\infty,1)_{\bullet}|.
\]
\end{DesCob}  

In this statement, $\psi_d(\infty,1)_{\bullet}$ is the colimit of the sequence
\[
\xymatrix{
\cdots \ar@{^{(}->}[r] &  \psi_d(N-1,1)_{\bullet}  \ar@{^{(}->}[r]^{\hspace{0.4cm}i_{N-1}} &  \psi_d(N,1)_{\bullet}  \ar@{^{(}->}[r]^{\hspace{-0.3cm}i_{N}} &  \psi_d(N+1,1)_{\bullet}  \ar@{^{(}->}[r] & \cdots}
\]
where  $i_N$ is the map of simplicial sets induced by the inclusion $\mathbb{R}^N \hookrightarrow \mathbb{R}^{N+1}$ defined by 
$(x_1, \ldots, x_N) \mapsto (x_1, \ldots, x_N, 0)$. 

Before we define the non-unital category $\mathsf{Cob}_d^{\mathrm{PL}}$, we need to introduce the notion of \textit{fiberwise regular value.} 

\subsection{Fiberwise regular values}  \label{section4.1}  

First, we shall discuss the notion of \textit{regular value} for PL maps. 

\theoremstyle{definition} \newtheorem{regular}[DesCob]{Definition}

\begin{regular} \label{regular} 
Let $P$ and $Q$ be PL spaces, and let $f:P\rightarrow Q$ be a proper PL map.
A point $q \in Q$ is said to be a \textit{regular value of $f$}
if there is an open neighborhood $U$ of $q$
in $Q$ and a PL homeomorphism $h: f^{-1}(q)\times U \rightarrow f^{-1}(U)$
such that $f\circ h$ agrees with the standard projection
$f^{-1}(q)\times U \rightarrow U$.
\end{regular}

We remark that the notion of regular value for PL maps was already introduced by Williamson in \cite{cobordism}. However, in \cite{cobordism}, this definition was only formulated for maps between compact PL spaces.
In this paper, we decided to define regular values for proper PL maps since the most prominent consequences of regularity involve these kinds of maps. The same observation applies to the use of regular values in differential topology, as exemplified by results such as Ehresmann's Fibration Theorem. 

The following proposition, which we will use frequently in the remaining sections, can be viewed as a PL version of Sard's Theorem. For a proof of this result, see Theorem 1.3.1 in \cite{cobordism}, or Lemma 1.7 in Essay III of \cite{KS}. 
 
 \theoremstyle{plain} \newtheorem{williamson}[DesCob]{Proposition}

\begin{williamson} \label{williamson}
Let $P$ be a compact PL space and 
$f:P \rightarrow \Delta^p$  a 
PL map. Suppose that there is
a triangulation
$h: |K| \rightarrow P$  of $P$ 
such that the composition $f\circ h$
is a simplicial map when we take $\Delta^p$ with its standard simplicial complex structure.
Then, any point $\lambda_{0}$ in 
$\mathrm{Int}\hspace{0.03cm}\Delta^p$ is a regular value of
$f$.  
\end{williamson}

\theoremstyle{definition} \newtheorem{remark.williamson}[DesCob]{Remark}

\begin{remark.williamson} \label{remark.williamson}

The following statement is a standard fact from piecewise linear topology: \textit{Let $P$ be a PL space. If $P\times \mathbb{R}$ is a PL manifold of dimension $m+1$, then $P$ is a PL manifold of dimension $m$.} This result also holds for PL manifolds with boundary. By induction, we can generalize the previous statement as follows: If $P\times \mathbb{R}^p$ is a PL manifold (with boundary) of dimension $m+p$, then $P$ is a PL manifold (with boundary) of dimension $m$. Therefore, if $M$ and $N$ are PL manifolds and  $\lambda_0 \in N - \partial\hspace{0.03cm}N$ is a regular value for a proper PL map of the form
$f:M \rightarrow N$, it follows that $f^{-1}(\lambda_0)$ is a PL submanifold of $M$ of dimension 
$\mathrm{dim}(M) - \mathrm{dim}(N)$. Additionally, in the case when $M$ is a PL manifold with boundary, we have that $f^{-1}(\lambda_0)$ is a \textit{proper} PL submanifold, i.e., the intersection $\partial M \cap f^{-1}(\lambda_0)$ is equal to the boundary of $f^{-1}(\lambda_0)$.  
\end{remark.williamson}

\theoremstyle{definition} \newtheorem{remark.williamson2}[DesCob]{Remark}

\begin{remark.williamson2} \label{remark.williamson2}
The following discussion is a good example to illustrate how Proposition \ref{williamson} plays in the PL category a role similar to the one played by Sard's Theorem in differential topology. 
Let then $P$ be a PL space and $W$ an element of the set $\psi_d(N,k)(P)$. Also, fix a point $\lambda \in P$, 
and let $W_{\lambda}$ be the fiber over $\lambda$ of the standard projection $\pi: W \rightarrow P$. 
As explained in \S \ref{section2.2}, we can view the fiber $W_{\lambda}$ as a PL submanifold of $\mathbb{R}^N$ which is closed as a subspace. In fact, by the definition of $\psi_d(N,k)$, the fiber $W_{\lambda}$ is strictly contained 
in $\mathbb{R}^k \times (-1,1)^{N-k}$. Now, let $x_k: W_{\lambda} \rightarrow \mathbb{R}^k$ be the projection of $W_{\lambda}$ onto the first
component of $\mathbb{R}^k \times (-1,1)^{N-k}$. Using Proposition \ref{williamson}, we can prove that 
the map $x_k: W_{\lambda} \rightarrow \mathbb{R}^k$  has an abundance of regular values. Indeed,
since $W$ is a closed PL subspace of $\mathbb{R}^k \times (-1,1)^{N-k}$, the projection 
$x_k: W_{\lambda} \rightarrow \mathbb{R}^k$ is a proper PL map. Therefore,
by Theorem 3.6 in \cite{plhud}, we can find locally finite simplicial complexes
$K$ and $L$ such that $|K| = W_{\lambda}$ and $|L| = \mathbb{R}^k$, and which make $x_k$ simplicial. 
Proposition \ref{williamson} then implies that each point
$x \in \mathbb{R}^k$ which lies 
in the interior of a $k$-simplex of $L$ is a regular value of  $x_k: W_{\lambda} \rightarrow \mathbb{R}^k$.
Consequently, the set of regular values of $x_k$ is dense in $\mathbb{R}^k$. 
\end{remark.williamson2}

We are now ready to state the definition of fiberwise regular value. In this definition, given a point $a \in \mathbb{R}^k$, 
$B(a,\delta)$ will denote the open ball
of radius $\delta>0$ centered at $a$ with respect to the norm $\left\|\cdot\right\|$
in $\mathbb{R}^k$ defined by 
\begin{equation} \label{norm.pl}
\left\|(x_1, \ldots,x_k)\right\| = \mathrm{max}\{|x_1|, \ldots, |x_k| \}. 
\end{equation}

\theoremstyle{definition} \newtheorem{fibregular}[DesCob]{Definition}

\begin{fibregular} \label{fibregular}

Fix a PL space $P$ and let
$W \subseteq P \times \mathbb{R}^{k}\times (-1,1)^{N-k}$ be an element 
of the set $\psi_d(N,k)(P)$. Also, let 
$\pi: W \rightarrow P$ be the standard 
projection onto $P$, and let $x_k: W \rightarrow \mathbb{R}^k$ be
the projection onto the second factor of $P \times \mathbb{R}^{k}\times (-1,1)^{N-k}$.
A value $a_0\in \mathbb{R}^k$ is said to be
a \textit{fiberwise regular value} of the projection $x_k: W \rightarrow \mathbb{R}^k$
if for every point $w$ in the pre-image $x_k^{-1}(a_0)$ there is a 
$\delta >0$, an open neighborhood $V$ of $\lambda_0=\pi(w)$ in $P$,  
and a PL homeomorphism
\[ 
h:  V \times B(a_0,\delta) \times (\pi,x_k)^{-1}\big((\lambda_0,a_0)\big) \longrightarrow
(\pi,x_k)^{-1}\big(V\times B(a_0,\delta)\big)  
\]
such that $(\pi, x_k) \circ h$ 
is equal to the standard projection onto
$V \times B(a_0,\delta)$.
\end{fibregular}

\theoremstyle{definition} \newtheorem{remark.fib.reg}[DesCob]{Remark}

\begin{remark.fib.reg} \label{remark.fib.reg}
Let us make a couple of comments about the previous definition. Fix again a PL space $P$ and
 an element $W  \subseteq P \times \mathbb{R}^{k}\times (-1,1)^{N-k}$ of the set
 $\psi_d(N,k)(P)$. As we have done before in this article, for any point $\lambda$ in $P$, we will denote  the fiber of the projection $\pi: W \rightarrow P$ over $\lambda$ by $W_{\lambda}$. If  $a_0$ is a fiberwise regular value of the projection 
$x_k: W \rightarrow \mathbb{R}^k$, then for any point $\lambda \in P$ we have that $a_0$ is a regular value (in the sense of Definition \ref{regular}) of the map $W_{\lambda} \rightarrow \mathbb{R}^k$ obtained by restricting $x_k$ to the fiber $W_{\lambda}$.  In particular, for any $\lambda \in P$, the pre-image $(\pi,x_k)^{-1}\big((\lambda,a_0)\big)$ will be a PL submanifold of dimension $d - k$ of $W_{\lambda}$ (see the discussion given in Remark \ref{remark.williamson}). However, besides guaranteeing regularity on each fiber, Definition \ref{fibregular} also asserts that the restriction of the projection $\pi: W \rightarrow P$ on the pre-image 
$x_k^{-1}(a_0)$ is a PL bundle whose fibers are compact manifolds of dimension $d -k$. Thus, after identifying 
$P \times\{a_0\}\times (-1,1)^{N-k}$ with  $P\times (-1,1)^{N-k}$, we have that the pre-image $x_k^{-1}(a_0)$ is an 
element of the set
$\psi_{d-k}(N-k, 0)(P)$. Moreover, if $P$ happens to be compact, we can make a stronger claim. In this case, the compactness of $P$ guarantees that there is a $\delta > 0$ such that the restriction of the map 
$(\pi,x_k): W \rightarrow P \times \mathbb{R}^k$ on $x_k^{-1}\big(B(a_0, \delta)\big)$ is also a PL bundle whose fibers are compact PL manifolds of dimension $d -k$. Thus, by considering the restriction of the projection $(\pi, x_k)$ on 
$x_k^{-1}\big(B(a_0, \delta)\big)$, we can view 
$x_k^{-1}\big(B(a_0, \delta)\big)$ as an 
element of the set $\psi_{d-k}(N - k, 0)\big(P \times B(a_0, \delta)\big)$. 
In fact, assuming that $P$ is compact, it is easy to prove the following equivalence: 
$a_0 \in \mathbb{R}^k$ is a fiberwise regular value of $x_k: W \rightarrow \mathbb{R}^k$ if and only if there exists a 
$\delta > 0$ such that the pre-image $x_k^{-1}\big(B(a_0, \delta)\big)$ is an element of 
$\psi_{d-k}(N - k, 0)\big(P \times B(a_0, \delta)\big)$.
\end{remark.fib.reg}

\theoremstyle{definition} \newtheorem{notation.fib.reg}[DesCob]{Notation}

\begin{notation.fib.reg} \label{notation.fib.reg}
If $W$ is an element of  $\psi_d(N,k)(P)$ and $a_0 \in \mathbb{R}^k$ is a fiberwise regular value of $x_k: W \rightarrow \mathbb{R}^k$, we will typically denote the pre-image $x_k^{-1}(a_0)$ by $W_{a_0}$. 
\end{notation.fib.reg}

In this section and the next, we will mostly be concerned with projections of the form 
$x_1: W \rightarrow \mathbb{R}$ where $W$ is a simplex of $\psi_d(N,1)_{\bullet}$, i.e., $W$ belongs to a set 
$\psi_d(N,1)(\Delta^p)$ for some $p \geq 0$. The following simplicial set will be essential for our proof of the equivalence $B\mathsf{Cob}_d^{\mathrm{PL}} \simeq |\psi_d(\infty,1)_{\bullet}|$.

\theoremstyle{definition} \newtheorem{psireg}[DesCob]{Definition}

\begin{psireg} \label{psireg}
$\psi_d^R(N,1)_{\bullet}$ is the subsimplicial set
of $\psi_d(N,1)_{\bullet}$ which consists of all
simplices $W$ with the property that the map
$x_1:W \rightarrow \mathbb{R}$ has a  fiberwise regular value.
\end{psireg}

We point out that it is possible to construct elements of $\psi_d(N,1)_{\bullet}$ which do not have any fiberwise regular values. However, we will give an example of such an element $W$ in \S \ref{section6} (more specifically, Remark \ref{example.non.fib}) since this construction will rely on methods that we will introduce in that section.  

Our strategy to prove Theorem \ref{DesCob} goes as follows. First, we will prove in this section that we have an equivalence of the form $B\mathsf{Cob}_d^{\mathrm{PL}} \simeq |\psi_d^{R}(\infty,1)_{\bullet}|$. Then, in Section \S \ref{section5}, we will prove that the inclusion $\psi_d^{R}(N,1)_{\bullet} \hookrightarrow \psi_d(N,1)_{\bullet}$ is a weak homotopy equivalence, as long as we have $N - d\geq 3$. 

\theoremstyle{definition} \newtheorem{remark.psireg}[DesCob]{Remark}

\begin{remark.psireg} \label{remark.psireg}

Let us make a couple of comments regarding the simplicial set $\psi_d^R(N,1)_{\bullet}$.

(1) First, we will discuss the kind of submersions that $\psi_d^R(N,1)_{\bullet}$ classifies. 
Let then $P$ be a PL space and 
$(K,h)$ a triangulation of $P$. Also, fix an order relation $\leq$ on the set of vertices $\mathrm{Vert}(K)$, and let
$K_{\bullet}$ be the simplicial set induced by the pair 
$(K,\leq)$. Finally, let $\mathcal{S}_{K,h}:\Psi_d(\mathbb{R}^N)(P) \rightarrow \mathbf{Ssets}(K_{\bullet}, \Psi_d(\mathbb{R}^N)_{\bullet})$ be the function defined in the statement of Theorem \ref{classsub}. 
Since $\mathcal{S}_{K,h}$ is a bijection, it admits an inverse, which we will denote 
by $\mathcal{T}_{K,h}$ in this note. 
Now, consider an arbitrary simplicial set map of the form
$g: K_{\bullet} \rightarrow \psi_d^R(N,1)_{\bullet}$. Via the methods that we used in the proof of Theorem \ref{classsub}, we can produce an element $W^g$ in the set $\psi_d(N,1)(P) \subseteq \Psi_d(\mathbb{R}^N)(P)$ which satisfies 
$\mathcal{S}_{K,h}(W^g) = g$. However, since the target of $g$ is $\psi_d^R(N,1)_{\bullet}$, this
element $W^g$ will have the following property: 

\begin{itemize}

\item[(*)] For each simplex $\sigma$ of $K$, the restriction of 
$x_1: W^g \rightarrow \mathbb{R}$ on $W^g_{h(\sigma)}$ has a fiberwise regular value.

\end{itemize}

In this statement, 
$h(\sigma)$ is the image of $\sigma$ under the PL homeomorphism $h: |K| \rightarrow P$
given in the triangulation  $(K,h)$, and 
$W^g_{h(\sigma)}$ is the restriction of $W^g$ over $h(\sigma)$. On the other hand, if we start with an element $W \in \psi_d(N,1)(P)$ satisfying property (*), then it is not hard to verify that the image of the simplicial set map 
$F_W := \mathcal{S}_{K,h}(W)$ is contained in $\psi_d^R(N,1)_{\bullet}$. Therefore, if we restrict the inverse 
$\mathcal{T}_{K,h}$ to
$\mathbf{Ssets}(K_{\bullet}, \psi_d^R(N,1)_{\bullet})$, we obtain a bijection from $\mathbf{Ssets}(K_{\bullet}, \psi_d^R(N,1)_{\bullet})$ to the subset of $\psi_d(N,1)(P)$ consisting of all $W$ satisfying property (*), i.e., all $W$ which admit a fiberwise regular value when restricted over the image $h(\sigma)$ of a simplex $\sigma$ of $K$. However, such a $W$ does not need to have a global fiberwise regular value that works over all $P$.

(2) Let $\widetilde{\psi}_d(N,1)_{\bullet}$ and $\widetilde{\psi}^R_d(N,1)_{\bullet}$ be the semi-simplicial sets obtained by 
forgetting degeneracies in 
the simplicial sets $\psi_d(N,1)$ and $\psi_d^R(N,1)$ respectively. In \S \ref{section5}, to prove that $\psi_d^{R}(N,1)_{\bullet} \hookrightarrow \psi_d(N,1)_{\bullet}$ is a weak homotopy equivalence, we will show that the corresponding inclusion of semi-simplicial sets
$\widetilde{\psi}_d^{R}(N,1)_{\bullet} \hookrightarrow \widetilde{\psi}_d(N,1)_{\bullet}$ is a weak equivalence. We will do this by proving that any map
of the form $f:(\Delta^p, \partial\Delta^p) \rightarrow \big( |\widetilde{\psi}_d(N,1)_{\bullet}|, |\widetilde{\psi}^R_d(N,1)_{\bullet}| \big)$ represents the trivial class in  $\pi_p\big(|\widetilde{\psi}_d(N,1)_{\bullet}|, |\widetilde{\psi}_d^R(N,1)_{\bullet}| \big)$.
Ideally, we would like 
$f$ to be homotopic (as a map of pairs) to the realization of a morphism 
$g: (\Delta^p_{\bullet}, \partial \Delta^p_{\bullet}) \rightarrow ( \widetilde{\psi}_d(N,1)_{\bullet}, \widetilde{\psi}^R_d(N,1)_{\bullet})$ 
of pairs of semi-simplicial sets. Unfortunately, it is not necessarily true that $f$ is homotopic to the realization of such a morphism
$g$ because the semi-simplicial set $\widetilde{\psi}_d^R(N,1)_{\bullet}$ is not Kan (the same goes for the simplicial set 
$\psi_d^R(N,1)_{\bullet}$). To get around this issue, we will use Proposition \ref{subsimphop}. By this proposition, $f$ will be homotopic (as a map of pairs) to a composition of the form
\[
(\Delta^p, \partial\Delta^p) \stackrel{f'}{\longrightarrow} \big(|K_{\bullet}|, |K'_{\bullet}|\big)  \stackrel{|g|}{\longrightarrow} 
\big(|\widetilde{\psi}_d(N,1)_{\bullet}|, |\widetilde{\psi}_d^R(N,1)_{\bullet}| \big),
\]
where $(K_{\bullet}, K'_{\bullet})$ is a pair of finite semi-simplicial sets induced by a pair $(K,K')$ of finite ordered simplicial complexes. Then, we will show that the realization of the morphism 
$g: (K_{\bullet}, K'_{\bullet}) \rightarrow ( \widetilde{\psi}_d(N,1)_{\bullet}, \widetilde{\psi}^R_d(N,1)_{\bullet})$
is homotopic (as a map of pairs) to a map that sends $|K_{\bullet}|$ to
$|\widetilde{\psi}^R_d(N,1)_{\bullet})|$. During this deformation of $|g|$, we will keep the first map $f'$ fixed. 
\end{remark.psireg}

\subsection{The non-unital simplicial category $\mathsf{Cob}_d^{\mathrm{PL}}$}    \label{section4.2} 

Before we give the definition of $\mathsf{Cob}_d^{\mathrm{PL}}$, we need to clarify what we mean by non-unital simplicial category. For this definition, we will use the following notation: Given two maps $f: X_{\bullet} \rightarrow Z_{\bullet}$ and $g: Y_{\bullet} \rightarrow Z_{\bullet}$ of simplicial sets, we will denote their fiber product by $X \hspace{0.105cm}_f\hspace{-0.1cm}\times_g Y$. 

\theoremstyle{definition} \newtheorem{simpnoncat}[DesCob]{Definition}

\begin{simpnoncat} \label{simpnoncat}
A \textit{non-unital simplicial category} $\mathcal{C}$ 
consists of the following data:
\begin{enumerate}
\item A simplicial set of objects $\mathcal{O}_{\bullet}$.
\item A simplicial set of morphisms $\mathcal{M}_{\bullet}$.
\item Maps of simplicial sets 
$s,t: \mathcal{M}_{\bullet} \rightarrow \mathcal{O}_{\bullet}$, called
respectively the \textit{source} and \textit{target map} of $\mathcal{C}$.
\item A map of simplicial sets 
$\mu: \mathcal{M} \hspace{0.105cm}_t\hspace{-0.1cm}\times_s \mathcal{M} \rightarrow \mathcal{M}_{\bullet}$,
called the \textit{composition map}, which satisfies the following conditions:
\begin{itemize}
\item[$\cdot$] $s\big( \mu(f,g)\big) = s(f)$ and $t\big( \mu(f,g) \big) = t(g)$
\item[$\cdot$] $\mu\big(\mu(f,g), h\big) = \mu\big( f, \mu(g,h) \big)$ 
whenever $t(f) = s(g)$ and $t(g) = s(h)$.   
\end{itemize} 
\end{enumerate}
\end{simpnoncat}
Typically, we will denote the image $\mu(f,g)$ by $g \circ f$. Note that we can also define a non-unital simplicial category as a simplicial object $\Delta^{op}\rightarrow\mathbf{Non}$-$\mathbf{Cat}$ in
the category of non-unital categories.This point of view is more convenient for formulating the notion of \textit{nerve of a non-unital simplicial category}, which we will do in Definition \ref{nervenoncat} below. For this definition, we shall denote the category of simplicial sets by $\mathbf{Fun}\big(\Delta^{op},\mathbf{Sets}\big)$ and the category of semi-simplicial sets by $\mathbf{Fun}\big(\Delta^{op}_{\mathrm{inj}},\mathbf{Sets}\big)$. Also, we shall denote by $\mathcal{N}:\mathbf{Non}\text{-}\mathbf{Cat} \rightarrow \mathbf{Fun}\big(\Delta^{op}_{\mathrm{inj}},\mathbf{Sets}\big)$ the functor which sends a non-unital category to its nerve. 

\theoremstyle{definition} \newtheorem{nervenoncat}[DesCob]{Definition}

\begin{nervenoncat} \label{nervenoncat}
Let $\mathcal{C}: \Delta^{op}\rightarrow\mathbf{Non}$-$\mathbf{Cat}$ be a non-unital simplicial category.
\textit{The nerve of} $\mathcal{C}$ is the functor $\mathcal{N}\mathcal{C}: \Delta^{op}_{\mathrm{inj}} \rightarrow \mathbf{Fun}\big(\Delta^{op}, \mathbf{Sets}\big)$ obtained by applying the categorical exponential law to the composition $\mathcal{N}\circ \mathcal{C}: \Delta^{op} \rightarrow  \mathbf{Fun}\big(\Delta^{op}_{\mathrm{inj}}, \mathbf{Sets}\big)$. 

\end{nervenoncat}

\theoremstyle{definition}  \newtheorem{notadd4}[DesCob]{Convention}

\begin{notadd4}

To make our arguments easier to formulate, we shall make one small adjustment to the definition of 
$\psi_d(N, 1)_{\bullet}$. Namely, 
throughout the rest of this section, we shall suppose that $p$-simplices of $\psi_d(N, 1)_{\bullet}$ are contained in $\mathbb{R}\times \Delta^p\times (-1,1)^{N-1}$ instead of $\Delta^p \times \mathbb{R} \times (-1,1)^{N-1}$. Moreover, we point out that fiberwise regular values are preserved by structure maps of $\psi_d(N, 1)_{\bullet}$. In other words, if $W \in \psi_d(N,1)_p$ and $a\in \mathbb{R}$ is a fiberwise regular value of $x_1: W \rightarrow \mathbb{R}$, then $a$ will also be a fiberwise regular value of $x_1: \eta^*W \rightarrow \mathbb{R}$ for any morphism $\eta: [q] \rightarrow [p]$ in the category $\Delta$. 

\end{notadd4}

We can now start developing the definition of the
PL cobordism category $\mathsf{Cob}_d^{\mathrm{PL}}$. 
First, we will define a version of this category in which morphisms are manifolds contained in $\mathbb{R}^N$
for some fixed $N$ (see Definition 3.7 from \cite{GRW}). For this definition, it is important to keep in mind that if $W$ is a $p$-simplex of 
$\psi_d(N,1)_{\bullet}$ and $a$ is a fiberwise regular value of the projection $x_1: W \rightarrow \mathbb{R}$ onto the first factor of $\mathbb{R}\times \Delta^p \times (-1,1)^{N-1}$, then the pre-image $x^{-1}(a)$ (which we will denote by $W_a$) is a $p$-simplex of $\psi_{d-1}(N-1,0)_{\bullet}$ (see Remark \ref{remark.fib.reg}). Also,  in this definition, we will use the following notation: If $W \in \psi_d(N,1)_p$ and $a \in \mathbb{R}$, then $W + a$ will denote the image of $W$ under the map $\mathbb{R} \times \Delta^p \times (-1,1)^{N-1} \rightarrow \mathbb{R} \times \Delta^p \times (-1,1)^{N-1}$ given by $(t, \lambda, x) \mapsto (t + a, \lambda, x)$. It is evident that $W + a$ is also a $p$-simplex of 
$ \psi_{d}(N,1)_{\bullet}$. 

\theoremstyle{definition}  \newtheorem{plcob}[DesCob]{Definition}

\begin{plcob} \label{plcob}

$\mathsf{Cob}_d^{\mathrm{PL}}(\mathbb{R}^N)$ is the non-unital simplicial category whose simplical sets of objects and morphisms $\mathcal{O}_{\bullet}$, $\mathcal{M}_{\bullet}$ and structure maps $s, t, \mu$ are defined as follows: 

\begin{enumerate}

\item $\mathcal{O}_{\bullet} = \psi_{d-1}(N-1,0)_{\bullet}$.

\item The set of $p$-simplices of $\mathcal{M}_{\bullet}$ is the subset of $\psi_d(N,1)_p\times (0, \infty)$
of tuples $(W, a)$ which satisfy the following conditions:  

\begin{itemize}

\item[\textit{(i)}]  $a$ and $0$ are fiberwise regular values of the projection $x_1: W \rightarrow \mathbb{R}$.

\item[\textit{(ii)}]  There exists an $\epsilon >0 $ such that
\[
x_1^{-1}\big((-\infty, \epsilon)\big) = (-\infty, \epsilon) \times W_0
\]
\[
x_1^{-1}\big((a -\epsilon, \infty)\big) = (a - \epsilon, \infty)\times W_a.
\]
Given a morphism $\eta: [q] \rightarrow [p]$ in $\Delta$, the structure map $\eta^*:\mathcal{M}_p \rightarrow \mathcal{M}_q$ sends a $p$-simplex $(W, a)$ to $(\eta^*W, a)$. 

\end{itemize}

\item $s: \mathcal{M}_{\bullet} \rightarrow \psi_{d-1}(N-1,0)_{\bullet}$ and $t: \mathcal{M}_{\bullet} \rightarrow \psi_{d-1}(N-1,0)_{\bullet}$ map a morphism $(W, a)$ to $W_0$ and $W_a$ respectively. 

\item If $(W, a)$ and $(W', a')$ are two morphisms with the property that $ W_a = W'_0$, then the map $\mu$ sends the tuple $\big((W,a), (W', a')\big)$ to the morphism $(\widetilde{W}, a + a')$, 
where $\widetilde{W}$ is equal to the union
\[
\Big( W \cap x_1^{-1}\big((-\infty, a]\big) \Big) \cup \Big( (W' + a)\cap x_1^{-1}\big([a,\infty)\big) \Big).
\] 

\end{enumerate}

\end{plcob}

For the definition of the PL cobordism category, we need versions of $\psi_d(N,1)_{\bullet}$  and $\psi_{d-1}(N-1,0)_{\bullet}$ in which the background space is infinite dimensional. First, $\psi_d(\infty,1)_{\bullet}$ will denote the simplicial set whose $p$-simplices are subsets $W \subseteq \mathbb{R}\times\Delta^p \times (-1,1)^{\infty - 1}$ for which we can find a positive integer $N$ such that $W \subseteq \mathbb{R}\times\Delta^p \times (-1,1)^{N- 1}$ and $W \in \psi_d(N,1)_p$. Structure maps for $\psi_d(\infty,1)_{\bullet}$ are also defined by taking pull-backs. It is evident how we can also define the infinite-dimensional version of $\psi_{d-1}(N-1,0)_{\bullet}$, which we will denote by $\psi_{d-1}(\infty-1,0)_{\bullet}$. With all of these preliminaries, we can now state the central definition of this article.  

\theoremstyle{definition} \newtheorem{plcobcat}[DesCob]{Definition} 
 
\begin{plcobcat}  \label{plcobcat} 
\textit{The PL cobordism category} $\mathsf{Cob}_d^{\mathrm{PL}}$ is the non-unital
simplicial category obtained by carrying out the construction given in Definition \ref{plcob} with 
$\psi_d(N,1)_{\bullet}$ and $\psi_{d-1}(N-1, 0)_{\bullet}$ replaced by 
$\psi_d(\infty,1)_{\bullet}$ and $\psi_{d-1}(\infty-1, 0)_{\bullet}$ respectively.

\end{plcobcat}

\theoremstyle{definition} \newtheorem{remarkcob}[DesCob]{Remark} 

\begin{remarkcob} \label{remarkcob}
Let $(W,a)$ be a morphism in $\mathsf{Cob}_d^{\mathrm{PL}}$ with $W \in \psi_d(\infty,1)_{p}$.
Moreover, suppose that $W\subseteq \mathbb{R}\times \Delta^p \times (-1,1)^{N-1}$,
and let $x_1: W \rightarrow \mathbb{R}$ be the projection onto the first
component of $\mathbb{R}\times\Delta^p\times (-1,1)^{N-1}$. Note that the 
underlying PL manifold $W$ is completely determined by the
pre-image $x_1^{-1}\big([0,a] \big)$, which we will 
denote by $W_{[0,a]}$. Similarly, for any $\lambda \in \Delta^p$, 
we will denote the intersection of
the fiber $W_{\lambda}$ of $\pi: W\rightarrow \Delta^p$
with  $x_1^{-1}\big([0,a] \big)$ by $W_{\lambda,[0,a]}$.
For a fixed $\lambda_0 \in \Delta^p$, it is evident 
that $W_{\lambda_0,[0,a]}$ is a $d$-dimensional
piecewise linear cobordism whose boundary components  
are contained in $\{0\}\times (-1,1)^{N-1}$ and 
$\{a\}\times (-1,1)^{N-1}$. Moreover, since $0$
and $a$ are fiberwise regular values of 
$x_1:W \rightarrow \mathbb{R}$,  
the PL version of Ehresmann's Fibration Theorem implies that
the restriction $\pi|_{W_{[0,a]}}: W_{[0,a]} \rightarrow \Delta^p$
is a trivial PL bundle whose fibers
are all PL homeomorphic to the cobordism 
$W_{\lambda_0,[0,a]}$. This observation
justifies the name \textit{PL cobordism category}. 
\end{remarkcob}

Consider again the non-unital category $\mathsf{Cob}_d^{\mathrm{PL}}(\mathbb{R}^N)$. For $k \geq 1$, we will describe a very explicit model for the $k$-th nerve (the simplicial set of $k$ composable morphisms) of  $\mathsf{Cob}_d^{\mathrm{PL}}(\mathbb{R}^N)$, which we will denote by $N_k\mathsf{Cob}_d^{\mathrm{PL}}(\mathbb{R}^N)_{\bullet}$. The set of $p$-simplices of  $N_k\mathsf{Cob}_d^{\mathrm{PL}}(\mathbb{R}^N)_{\bullet}$ will be the set of all tuples of the form 
$(W, a_1 < \ldots < a_k)$ such that $0 < a_1$, $W \in \psi_{d}(N,1)_{p}$, $(W, a_k)$ is a morphism in 
$\mathsf{Cob}_d^{\mathrm{PL}}(\mathbb{R}^N)$,   and $0, a_1, \ldots, a_k$ are fiberwise regular values of $x_1: W \rightarrow \mathbb{R}$. Moreover, there must exist an $\epsilon > 0$ such that 
\[
x_1^{-1}\big( (a_{j} - \epsilon, a_j + \epsilon) \big) = (a_{j} - \epsilon, a_j + \epsilon) \times W_{a_j} 
\]
for all $j =1,\ldots, k-1$. The structure maps of $N_k\mathsf{Cob}_d^{\mathrm{PL}}(\mathbb{R}^N)_{\bullet}$ are again defined by taking pull-backs. In the case $k=0$, we set 
$N_0\mathsf{Cob}_d^{\mathrm{PL}}(\mathbb{R}^N)_{\bullet} = \psi_{d-1}(\infty-1, 0)_{\bullet}$. 

Next, we will assemble the simplicial sets $\{ N_k\mathsf{Cob}_d^{\mathrm{PL}}(\mathbb{R}^N)_{\bullet}\}_{k\geq 0}$ into a functor 
\[
\mathcal{N}\mathsf{Cob}_d^{\mathrm{PL}}(\mathbb{R}^N): \Delta^{op}_{\mathrm{inj}} \rightarrow \mathbf{Ssets} 
\]
as follows. At the level of objects, we define $[k] \mapsto N_k\mathsf{Cob}_d^{\mathrm{PL}}(\mathbb{R}^N)_{\bullet}$. To describe the map 
$\eta^*: N_k\mathsf{Cob}_d^{\mathrm{PL}}(\mathbb{R}^N)_{\bullet} \rightarrow N_{k'}\mathsf{Cob}_d^{\mathrm{PL}}(\mathbb{R}^N)_{\bullet}$ induced by a morphism $\eta: [k'] \rightarrow [k]$ in $\Delta_{\mathrm{inj}}$, with $k ' > 0$, we consider two cases: 

\begin{itemize}

\item[(i)] If $\eta(0) = 0$, then $\eta^*\big((W, a_1 < \ldots < a_k)\big) = (W', a_{\eta(1)} < \ldots < a_{\eta(k')} )$, where $W' = x^{-1}_1\big( (-\infty, a_{\eta(k')}] \big) \cup \big([a_{\eta(k')}, \infty) \times W_{a_{\eta(k')}}  \big)$. 

\item[(ii)] If $\eta(0) > 0$, then 
\[
\eta^*\big((W, a_1 < \ldots < a_k)\big) = (W', a_{\eta(1)}  - a_{\eta(0)} < \ldots < a_{\eta(k')} - a_{\eta(0)} ),
\]
where $W'$ is obtained by first taking the union
\[
x^{-1}_1\big( [a_{\eta(0)}, a_{\eta(k')}] \big) \cup \big( (-\infty,a_{\eta(0)}]\times W_{a_{\eta(0)}}\big) \cup \big([a_{\eta(k')}, \infty) \times W_{a_{\eta(k')}} \big)
\]
and then taking the image of this union under the PL homeomorphism $\mathbb{R} \times \Delta^p \times (-1,1)^{N-1} \rightarrow \mathbb{R} \times \Delta^p \times (-1,1)^{N-1} $ defined by 
\[
(t, \lambda, x) \mapsto (t - a_{\eta(0)}, \lambda, x). 
\]

\end{itemize}
When $k' = 0$, the map $\eta^*: N_k\mathsf{Cob}_d^{\mathrm{PL}}(\mathbb{R}^N)_{\bullet} \rightarrow N_{0}\mathsf{Cob}_d^{\mathrm{PL}}(\mathbb{R}^N)_{\bullet}$ induced by $\eta: [0] \rightarrow [k]$ is given by $(W, a_1 < \ldots < a_k) \mapsto W_{a_{\eta(0)}}$. That is, we just get the level set of the fiberwise regular value $a_{\eta(0)}$. The functor  $\mathsf{Cob}_d^{\mathrm{PL}}(\mathbb{R}^N): \Delta^{op}_{\mathrm{inj}} \rightarrow \mathbf{Ssets}$ that we have just defined is \textit{the nerve of} $\mathsf{Cob}_d^{\mathrm{PL}}(\mathbb{R}^N)$. We can define the nerve 
$\mathcal{N}\mathsf{Cob}_d^{\mathrm{PL}}$ of the PL cobordism category $\mathsf{Cob}_d^{\mathrm{PL}}$ 
in a completely analogous fashion. 

Note that, by forgetting degeneracies in $N_k\mathsf{Cob}_d^{\mathrm{PL}}(\mathbb{R}^N)_{\bullet}$ and 
$N_k\mathsf{Cob}_d^{\mathrm{PL}}$, both $\mathcal{N}\mathsf{Cob}_d^{\mathrm{PL}}(\mathbb{R}^N)$ and 
$\mathcal{N}\mathsf{Cob}_d^{\mathrm{PL}}$ turn into functors of the form $\Delta_{\mathrm{inj}}^{op} \rightarrow \mathbf{Semi}$-$\mathbf{Ssets}$, which we can view as bi-semi-simplicial sets $\mathcal{N}\mathsf{Cob}_d^{\mathrm{PL}}(\mathbb{R}^N)_{\bullet, \bullet}$ and  $\mathcal{N}\mathsf{Cob}_{d\bullet, \bullet}^{\mathrm{PL}}$ by the categorical exponential law. We will use this observation to formulate the following definition. 

\theoremstyle{definition}  \newtheorem{classspace}[DesCob]{Definition}

\begin{classspace} \label{classspace}

The classifying spaces $B\mathsf{Cob}_d^{\mathrm{PL}}(\mathbb{R}^N)$ and $B\mathsf{Cob}_d^{\mathrm{PL}}$ are defined as the geometric realizations of 
$\mathcal{N}\mathsf{Cob}_d^{\mathrm{PL}}(\mathbb{R}^N)_{\bullet, \bullet}$ and  $\mathcal{N}\mathsf{Cob}_{d\bullet, \bullet}^{\mathrm{PL}}$ respectively. 

\end{classspace}

\subsection{The equivalence $B\mathsf{Cob}_d^{\mathrm{PL}} \simeq |\psi_d^{R}(\infty,1)_{\bullet}|$}   \label{section4.3}

In what remains of this section, we will show that $B\mathsf{Cob}_d^{\mathrm{PL}}$
is weak homotopy equivalent to $\left|\psi_d^R(\infty,1)_{\bullet}\right|$. In order
to do this, we need to introduce the following poset model
for the non-unital category $\mathsf{Cob}_d^{\mathrm{PL}}(\mathbb{R}^N)$. 
For this definition,
$\widetilde{\psi}_d(N,1)_{\bullet}$ will denote once again 
the semi-simplicial set obtained from $\psi_d(N,1)_{\bullet}$ 
after forgetting degeneracies.

\theoremstyle{definition} \newtheorem{cobPOset}[DesCob]{Definition}

\begin{cobPOset} \label{cobPOset}

Let $\mathcal{D}_d(\mathbb{R}^N)_{\bullet,\bullet}$ be the bi-semi-simplicial set defined as follows: 

\begin{itemize}

\item[(i)] The set of $(p,q)$-simplices is equal to the subset of $\widetilde{\psi}_d(N,1)_p\times \mathbb{R}^{q+1}$ consisting of tuples $(W, a_0 < a_1 < \ldots < a_q)$ with the property that each $a_i$ is a fiberwise regular value of $x_1: W \rightarrow \mathbb{R}$. 

\item[(ii)] Given morphisms $\eta: [p'] \rightarrow [p]$ and $\theta: [q'] \rightarrow [q]$ in $\Delta_{\mathrm{inj}}$, the structure map 
$(\eta \times \theta)^*: \mathcal{D}_d(\mathbb{R}^N)_{p,q}\rightarrow \mathcal{D}_d(\mathbb{R}^N)_{p',q'} $ induced by $\eta \times \theta$ is given by
\[
(W, a_0 < a_1 < \ldots < a_q) \mapsto  (\eta^*W, a_{\theta(0)} < a_{\theta(1)} < \ldots < a_{\theta(q')}).
\]
In other words, maps in the $p$-direction are given by pull-backs, and in the $q$-direction they are defined by forgetting fiberwise regular values. 
\end{itemize}

\end{cobPOset}

We also need a version of $\mathcal{D}_d(\mathbb{R}^N)_{\bullet,\bullet}$ where manifolds are required to be cylindrical near the fiberwise regular values. More specifically, we define a bi-semi-simplicial set 
$\mathcal{D}_d^{\perp}(\mathbb{R}^N)_{\bullet,\bullet}$ as follows:

\begin{itemize}

\item[(i)] $\mathcal{D}_d^{\perp}(\mathbb{R}^N)_{p,q}$ is the subset of $\mathcal{D}_d(\mathbb{R}^N)_{p,q}$
consisting of elements $(W, a_0 < \ldots < a_q)$ for which there is an $\epsilon > 0$ such that
\[
x_1^{-1}\big( (a_{j} - \epsilon, a_j + \epsilon) \big) = (a_{j} - \epsilon, a_j + \epsilon) \times W_{a_j} 
\]
for all $j=0,\ldots,q$. 

\item[(ii)] The structure maps of $\mathcal{D}_d^{\perp}(\mathbb{R}^N)_{\bullet,\bullet}$  coincide with those of 
$\mathcal{D}_d(\mathbb{R}^N)_{\bullet,\bullet}$. 

\end{itemize}

We will prove that $B\mathsf{Cob}_d^{\mathrm{PL}} \simeq |\psi_d^R(\infty,1)_{\bullet}|$ by showing that there is
a zig-zag of weak equivalences 
\begin{equation} \label{chain.equiv.cob}
\left\|\mathcal{N}\mathsf{Cob}_d^{\mathrm{PL}}(\mathbb{R}^N)_{\bullet,\bullet}\right\| \stackrel{\simeq}{\leftarrow} 
\left\|\mathcal{D}_d^{\perp}(\mathbb{R}^N)_{\bullet,\bullet}\right\| \stackrel{\simeq}{\rightarrow}
\left\|\mathcal{D}_d(\mathbb{R}^N)_{\bullet,\bullet}\right\| \stackrel{\simeq}{\rightarrow}
\left|\widetilde{\psi}_d^R(N,1)_{\bullet}\right| \stackrel{\simeq}{\rightarrow} \left|\psi_d^R(N,1)_{\bullet}\right|
\end{equation}
and then letting $N$ go to infinity.
The right-most equivalence is the natural quotient
map from the geometric realization of the semi-simplicial set 
$\widetilde{\psi}^R_d(N,1)_{\bullet}$ to
the geometric realization of $\psi_d^R(N,1)_{\bullet}$.  If $X_{\bullet}$ is a simplicial set and $\widetilde{X}_{\bullet}$ is the semi-simplicial set obtained by forgetting degeneracies, then the natural quotient map $\widetilde{X}_{\bullet} \rightarrow X_{\bullet}$ is a weak homotopy equivalence. 
This fact is proven in \cite{RSD}. 

In order to show that $B\mathsf{Cob}_d^{\mathrm{PL}} \simeq |\psi_d^R(\infty,1)_{\bullet}|$, we will  start by comparing the spaces
$\left\|\mathcal{N}\mathsf{Cob}_d^{\mathrm{PL}}(\mathbb{R}^N)_{\bullet,\bullet}\right\|$ and
$\left\|\mathcal{D}_d^{\perp}(\mathbb{R}^N)_{\bullet,\bullet}\right\|$. To do so, let us explain first how we will define a function 
\[
\mathcal{D}_d^{\perp}(\mathbb{R}^N)_{p,q} \stackrel{f_{p,q}}{\longrightarrow} \mathcal{N}\mathsf{Cob}_d^{\mathrm{PL}}(\mathbb{R}^N)_{p,q}
\] 
for any pair $p, q$ with $p \geq 0$ and $q > 0$. Given an element $(W, a_0 < \ldots < a_q)$ of $\mathcal{D}_d^{\perp}(\mathbb{R}^N)_{p,q}$, we define its image in $\mathcal{N}\mathsf{Cob}_d^{\mathrm{PL}}(\mathbb{R}^N)_{p,q}$ as follows:

\begin{itemize}

\item[$\cdot$] First, we take the pre-image of $[a_0, a_q]$ under the map $x_1: W \rightarrow \mathbb{R}$. Using the same style of notation that we used in Remark \ref{remarkcob}, we will denote this pre-image by $W_{[a_0, a_q]}$. Next, we define $W(a_0,a_q)$ to be the union 
\[
W_{[a_0, a_q]} \cup \big( (-\infty, a_0]\times W_{a_0} \big) \cup  \big( [a_q, \infty) \times W_{a_q} \big) 
\]
where $W_{a_0}$ and $W_{a_q}$ are the level sets of $x_1: W \rightarrow \mathbb{R}$ corresponding to the values $a_0$ and $a_q$ respectively. 

\item[$\cdot$] Finally, define
\[
f_{p,q}\big((W, \hspace{0.05cm} a_0 <  \ldots < a_q)\big) = (W(a_0, a_q) - a_0, \hspace{0.05cm} a_1 - a_0 <  \ldots < a_{q} - a_0)
\]
where $W(a_0, a_q) - a_0$ is the image of $W(a_0,a_q)$ under the PL automorphism  of
$\mathbb{R}\times \Delta^p \times (-1,1)^{N-1}$ given by $(t, \lambda, x) \mapsto (t-a, \lambda, x)$. Note that 
$0$ and  $a_1 - a_0, \ldots , a_q - a_0$ are indeed fiberwise regular values of the projection $x_1$ from 
$W(a_0, a_q) - a_0$ to $\mathbb{R}$.  
\end{itemize} 

In the case $q=0$, we define $f_{p,0}\big(W, a_0\big) = W_{a_0}$. It is straightforward to verify that all the functions $f_{p,q}$ can be assembled into a map $f_{\bullet, \bullet}$ of bi-semi-simplicial sets. With this construction, we can now establish the first equivalence in (\ref{chain.equiv.cob}).  

\theoremstyle{plain}  \newtheorem{CtoD}[DesCob]{Proposition} 
 
\begin{CtoD} \label{CtoD}
The morphism of bi-semi-simplicial sets
$\mathcal{D}^{\perp}_d(\mathbb{R}^N)_{\bullet,\bullet} \stackrel{f_{\bullet,\bullet}}{\longrightarrow} \mathcal{N}\mathsf{Cob}_d^{\mathrm{PL}}(\mathbb{R}^N)_{\bullet,\bullet}$ defined above induces a weak homotopy equivalence 
\[
\left\|\mathcal{D}^{\perp}_d(\mathbb{R}^N)_{\bullet,\bullet}\right\| \stackrel{\simeq}{\longrightarrow} 
\left\| \mathcal{N}\mathsf{Cob}_d^{\mathrm{PL}}(\mathbb{R}^N)_{\bullet, \bullet}\right\|.
\]

\end{CtoD}

\begin{proof}
We will prove this proposition by showing that, for each $q\geq 0$, the map of semi-simplicial sets $f_{\bullet, q}$ is a weak homotopy equivalence. For this argument, we shall identify the set $\mathcal{N}\mathsf{Cob}_d^{\mathrm{PL}}(\mathbb{R}^N)_{p, 0}$ with the subset of $\psi_d(N,1)_p \times \mathbb{R}$ consisting of all tuples of the form $(\mathbb{R}\times M, 0)$, where $M \in \psi_{d-1}(N-1,0)_p$. By doing this identification, it will not be necessary to treat $q = 0$ as a separate case.

Fix an integer $q \geq 0$. If $\mathcal{N}\mathsf{Cob}_d^{\mathrm{PL}}(\mathbb{R}^N)_{\bullet, q} \stackrel{i_{q}}{\longrightarrow} \mathcal{D}^{\perp}_d(\mathbb{R}^N)_{\bullet,q}$
is the obvious inclusion of semi-simplicial sets, then the composition $f_{\bullet, q} \circ i_q$ is equal to the identity map on $\mathcal{N}\mathsf{Cob}_d^{\mathrm{PL}}(\mathbb{R}^N)_{\bullet, q}$. Therefore, it suffices to show that the inclusion $i_q$ is a weak homotopy equivalence. Moreover, since both  $\mathcal{D}^{\perp}_d(\mathbb{R}^N)_{\bullet,q}$
and $\mathcal{N}\mathsf{Cob}_d^{\mathrm{PL}}(\mathbb{R}^N)_{\bullet, q}$ are Kan, it is enough to prove that the geometric realization of a morphism of pairs of the form
\begin{equation} \label{morphism.CtoD}
h: (\Delta^{k}_{\bullet},\partial\Delta^{k}_{\bullet}) \longrightarrow 
(\mathcal{D}_d^{\perp}(\mathbb{R}^N)_{\bullet,q}, \mathcal{N}\mathsf{Cob}_d^{\mathrm{PL}}(\mathbb{R}^N)_{\bullet, q} )
\end{equation}
represents a trivial class in $\pi_k(|\mathcal{D}_d^{\perp}(\mathbb{R}^N)_{\bullet,q}|, |\mathcal{N}\mathsf{Cob}_d^{\mathrm{PL}}(\mathbb{R}^N)_{\bullet, q}|)$. Fix then such a morphism $h$ and let 
$(W, a_0 < \ldots < a_q)$ be the $k$-simplex of $\mathcal{D}_d^{\perp}(\mathbb{R}^N)_{\bullet,q}$ classified by $h$; i.e., if $\sigma$ is the unique non-degenerate $k$-simplex of $\Delta^k_{\bullet}$, then $h(\sigma) = (W, a_0 <  \ldots < a_q)$.
Note that $(W, a_0 < \ldots < a_q)$ satisfies the following properties: 

\begin{itemize}

\item[$\cdot$] $a_0 = 0$. 

\item[$\cdot$] If $x_1: W \rightarrow \mathbb{R}$ is the projection onto the first component of 
$\mathbb{R}\times \Delta^k \times (-1, 1)^{N-1}$, then 
there is a value $\epsilon > 0$ such that 
\[
x_1^{-1}\big( (a_j - \epsilon, a_j + \epsilon) \big) = (a_j - \epsilon, a_j + \epsilon)  \times W_{a_j}
\]
for all $j = 0, \ldots, q$. 

\item[$\cdot$] If $\delta_i: \Delta^{k-1} \hookrightarrow \Delta^{k}$ is the inclusion that maps $\Delta^{k-1}$ to the $i$-th face of $\Delta^k$, then $(\delta_i^{*}W, a_1 < \ldots < a_q)$ is an element of 
$\mathcal{N}\mathsf{Cob}_d^{\mathrm{PL}}(\mathbb{R}^N)_{k-1, q}$.   

\end{itemize}
The essential step in this proof is to construct a concordance between $W$ and $W(0, a_q)$ by stretching the pre-images $x_1^{-1}\big( ( - \epsilon, \epsilon) \big)$ and $x_1^{-1}\big( (a_q - \epsilon, a_q + \epsilon) \big)$ towards $-\infty$ and $\infty$ respectively. To this end, fix a value $\epsilon'$ such that $0 < \epsilon' < \epsilon$ and pick an isotopy of open PL embeddings $f: \mathbb{R} \times [0,1] \rightarrow \mathbb{R}  \times [0,1]$ satisifying the following conditions: 
\begin{itemize}

\item[$\cdot$] $f_0$ is the identity on $\mathbb{R}$.

\item[$\cdot$] $f$ fixes all points in $[- \epsilon',a_q + \epsilon'] \times [0,1]$.

\item[$\cdot$] $f_1$ maps $\mathbb{R}$ onto $(-\epsilon, a_q + \epsilon)$.

\end{itemize}
The product of maps
$f\times \mathrm{Id}_{\Delta^k\times \mathbb{R}^{N-1}}$,
which we shall denote by $e$, is an open PL
embedding from 
$\mathbb{R} \times [0,1] \times \Delta^k \times \mathbb{R}^{N-1}$
to itself which commutes with the projection onto 
$[0,1] \times \Delta^k$. Then, if $[0,1]\times W$ denotes the constant concordance from $W$ to itself, Proposition \ref{pullemb} guarantees that the pre-image $\widehat{W} : = e^{-1}([0,1]\times W)$ is an element of the set $\psi_d(N,1)([0,1] \times \Delta^k)$. By the way we defined the embedding $e$, it is clear that $\widehat{W}$ is a concordance between $W$ and $W(0, a_q)$. Moreover, note that $0, a_1, \ldots, a_q$ are fiberwise regular values of $x_1: \widehat{W} \rightarrow \mathbb{R}$. Thus, any homotopy
$H:  [0,1] \times  \Delta^k  \rightarrow |\psi_d(N,1)_{\bullet}|$ induced by $\widehat{W}$ (see Proposition \ref{concord.homotopy}) will admit a lift $\widetilde{H}: [0,1] \times  \Delta^k  \rightarrow |\mathcal{D}_d^{\perp}(\mathbb{R}^N)_{\bullet,q}|$ with the property that $\widetilde{H}_0 = |h|$ and $\widetilde{H}_1(\Delta^k) \subseteq |\mathcal{N}\mathsf{Cob}_d^{\mathrm{PL}}(\mathbb{R}^N)_{\bullet, q}|$. 
Finally, since the embedding $e$ does not change the fibers of $[0,1] \times W$ over $[0,1]\times \partial \Delta^k$, the homotopy $\widetilde{H}$ will map $[0,1]\times \partial \Delta^k$ to 
$|\mathcal{N}\mathsf{Cob}_d^{\mathrm{PL}}(\mathbb{R}^N)_{\bullet, q}|$. Therefore, the map 
$h: (\Delta^p, \partial \Delta^p) \rightarrow (|\mathcal{D}_d^{\perp}(\mathbb{R}^N)_{\bullet,q}|, |\mathcal{N}\mathsf{Cob}_d^{\mathrm{PL}}(\mathbb{R}^N)_{\bullet, q}|)$ that we fixed represents the trivial class in 
$\pi_k(|\mathcal{D}_d^{\perp}(\mathbb{R}^N)_{\bullet,q}|, |\mathcal{N}\mathsf{Cob}_d^{\mathrm{PL}}(\mathbb{R}^N)_{\bullet, q}|)$. 
\end{proof}

In Proposition \ref{DtoD} below, we will show that the inclusion $\mathcal{D}_d^{\perp}(\mathbb{R}^N)_{\bullet,\bullet} \hookrightarrow \mathcal{D}_d(\mathbb{R}^N)_{\bullet,\bullet}$ induces a weak homotopy equivalence between geometric realizations. 
Let us introduce some notation that we will use in this proof. First, fix a $p$-simplex $W \subseteq \mathbb{R}\times\Delta^{p}\times (-1,1)^{N-1}$ of $\psi_d(N,1)_{\bullet}$. Also, let $x_1: W \rightarrow \mathbb{R}$ and $\pi: W\rightarrow \Delta^p$ denote once again the projections onto the first and second components of $\mathbb{R}\times\Delta^{p}\times (-1,1)^{N-1}$ respectively. For any subsets $A \subseteq \Delta^p$ and $S \subseteq \mathbb{R}$, we shall denote the pre-image of $S\times A$ under the map 
$(x_1, \pi): W \rightarrow \mathbb{R}\times \Delta^p$ by $W_{S,A}$. Moreover, for a fiberwise regular value $a$ of 
$x_1: W \rightarrow \mathbb{R}$, we shall continue to denote the pre-image $x_1^{-1}(a)$ by $W_a$.

\theoremstyle{plain} \newtheorem{DtoD}[DesCob]{Proposition}

\begin{DtoD} \label{DtoD} 
The map $\left\|\mathcal{D}_d^{\perp}(\mathbb{R}^N)_{\bullet,\bullet}\right\| \rightarrow \left\|\mathcal{D}_d(\mathbb{R}^N)_{\bullet,\bullet}\right\|$ 
induced by the inclusion
$i_{\bullet,\bullet}: \mathcal{D}_d^{\perp}(\mathbb{R}^N)_{\bullet,\bullet} \hookrightarrow \mathcal{D}_d(\mathbb{R}^N)_{\bullet,\bullet}$
is a weak homotopy equivalence. 
\end{DtoD}

\begin{proof}

Since  $\mathcal{D}_d^{\perp}(\mathbb{R}^N)_{\bullet,q}$ and 
$\mathcal{D}_d(\mathbb{R}^N)_{\bullet,q}$ are Kan for all $q \geq 0$, it suffices to show that the geometric realization of any map of the form 
\begin{equation} \label{morphism.DtoD}
h: (\Delta^k_{\bullet}, \partial\Delta^k_{\bullet}) \rightarrow (\mathcal{D}_d(\mathbb{R}^N)_{\bullet,q},\mathcal{D}_d^{\perp}(\mathbb{R}^N)_{\bullet,q}) 
\end{equation}
represents the trivial class in 
$\pi_k(|\mathcal{D}_d(\mathbb{R}^N)_{\bullet,q}|, |\mathcal{D}_d^{\perp}(\mathbb{R}^N)_{\bullet,q}|)$. Fix then a morphism $h$ of the form (\ref{morphism.DtoD}) and let 
$(W, a_0 < \ldots < a_q)$ be the $k$-simplex classified by $h$. Moreover, for each $j =0, \ldots, q$, let us denote the restriction of $W_{a_j}$ over $\partial\Delta^k$ by $W_{a_j, \hspace{0.03cm}\partial\Delta^k}$. Since  $h\big(  \partial\Delta^k_{\bullet}\big)$ is contained in $\mathcal{D}_d^{\perp}(\mathbb{R}^N)_{\bullet,q}$, there exists an $\epsilon > 0$ with the property that 
\begin{equation} \label{property.DtoD}
W_{(a_j - \epsilon, a_j + \epsilon), \hspace{0.03cm} \partial\Delta^k} =  (a_j - \epsilon, a_j + \epsilon)\times W_{a_j, \hspace{0.03cm}\partial\Delta^k}
\end{equation}
for all $j=0,\ldots, q$. We can also assume that the value $\epsilon$ is small enough so that all the intervals 
$(a_j - \epsilon, a_j + \epsilon)$ are mutually disjoint. Our goal in this proof is to transform $W$, via a concordance, into an element of $\psi_d(N,1)_k$ for which the property given in (\ref{property.DtoD}) holds globally over $\Delta^k$, not just $\partial \Delta^k$.  
In other words, we want to make the level sets of $x_1: W \rightarrow \mathbb{R}$ constant near each $a_j$. To do this, let us fix a value $j_0$ in $\{0,\ldots, q\}$. Since $a_{j_0}$ is a fiberwise regular value of $x_1: W \rightarrow \mathbb{R}$, there exists a value $\delta>0$ such that $W_{(a_{j_0} - \delta, a_{j_0} + \delta), \Delta^k}$ is an element of the set
$\psi_{d-1}(N-1,0)\big((a_{j_0} - \delta, a_{j_0} + \delta)\times\Delta^k\big)$ (see Remark \ref{remark.fib.reg}). By shrinking either $\epsilon$ or $\delta$, we may assume $\epsilon = \delta$. To make the level sets near $a_{j_0}$ constant, we will pull back $W_{(a_{j_0} - \epsilon, a_{j_0} + \epsilon), \Delta^k}$ along a
 PL  homotopy 
\[
J: [0,1] \times (a_{j_0} - \epsilon, a_{j_0} + \epsilon)\times\Delta^k \longrightarrow (a_{j_0} - \epsilon, a_{j_0} + \epsilon)\times\Delta^k
\]  
satisfying the following conditions: 

\begin{itemize}

\item[(i)] $J_0$ is the identity map on 
$(a_{j_0} - \epsilon, a_{j_0} + \epsilon)\times\Delta^k$, and $J$ commutes with the projection onto $\Delta^k$.

\item[(ii)] $J$ is the constant homotopy outside of $[a_{j_0} - \frac{\epsilon}{2}, a_{j_0} + \frac{\epsilon}{2}]\times\Delta^k$.

\item[(iii)] If $(t, x) \in  [0,1] \times (a_{j_0} - \epsilon, a_{j_0} + \epsilon)$ is a point in the triangle spanned
by $(0,a_{j_0})$, $(1, a_{j_0} + \frac{\epsilon}{4})$ and $(1, a_{j_0} - \frac{\epsilon}{4})$, then $J$ maps 
$(t,x, \lambda)$ to $(a_{j_0}, \lambda)$ for any point $\lambda \in \Delta^k$. 
\end{itemize}

Fix then such a PL homotopy $J$ and denote the pull-back $J^*W_{(a_{j_0} - \epsilon, a_{j_0} + \epsilon), \Delta^k}$ by $\widetilde{W}$.  Since $\widetilde{W}$ is an element of the set
$\psi_{d-1}(N-1, 0)\big( [0,1] \times (a_{j_0} - \epsilon, a_{j_0} + \epsilon)\times\Delta^k \big)$,
the projection $\widetilde{W} \rightarrow [0,1] \times (a_{j_0} - \epsilon, a_{j_0} + \epsilon)\times\Delta^k$ is a PL submersion of codimension $d-1$. Thus, the projection $\pi: \widetilde{W} \rightarrow [0,1]\times \Delta^k$ obtained by forgetting the factor $(a_{j_0} - \epsilon, a_{j_0} + \epsilon)$
is a PL submersion of codimension $d$, and we can therefore view $\widetilde{W}$ as an element of the set $\Psi_d\big( (a_{j_0} - \epsilon, a_{j_0} + \epsilon) \times \mathbb{R}^{N-1}\big)\big( [0,1] \times \Delta^k \big)$.
In this last statement, $\Psi_d\big( (a_{j_0} - \epsilon, a_{j_0} + \epsilon) \times \mathbb{R}^{N-1}\big)$ is the PL set 
obtained by applying Definition \ref{spacemangen} to the open subset $U = (a_{j_0} - \epsilon, a_{j_0} + \epsilon) \times \mathbb{R}^{N-1}$ of $\mathbb{R}^N$. 

Let $[0,1]\times W$ denote again the constant concordance from $W$ to itself.
Since the functor $\Psi_d: \mathcal{O}(\mathbb{R}^N)^{op} \rightarrow \mathbf{PLsets}$ defined by 
$U \mapsto \Psi_d(U)$ is a sheaf of PL sets (see Remark \ref{spacerem2}), we can extend 
$\widetilde{W}$  to a unique element $\widehat{W}$ of 
$\Psi_d(\mathbb{R}^N)\big([0,1]\times \Delta^k \big)$ satisfying the following: 
\begin{itemize}

\item[(i$^*$)] $\widehat{W}$ agrees with $\widetilde{W}$ when we restrict the first coordinate of the background space $\mathbb{R}^N$ to $(a_{j_0} - \epsilon, a_{j_0} + \epsilon)$. 

\item[(ii$^*$)]  $\widehat{W}$ agrees with the constant concordance $[0,1]\times W$ when we restrict the first coordinate of the background space $\mathbb{R}^N$ to $\mathbb{R} - [a_{j_0} - \frac{\epsilon}{2}, a_{j_0} + \frac{\epsilon}{2}]$. 

\end{itemize}
Moreover, it is evident that all the fibers of the projection $\widehat{W} \rightarrow [0,1]\times \Delta^k$ are contained in 
$\mathbb{R} \times (-1,1)^{N-1}$, so $\widehat{W}$ is in fact an element of $\psi_d(N,1)([0,1]\times \Delta^k)$. Now, consider the projection $x_1:\widehat{W} \rightarrow \mathbb{R}$  onto the first component of the background space $\mathbb{R}\times (-1,1)^{N-1}$. Property (i$^*$) above says that the pre-image 
$x^{-1}((a_{j_0} - \epsilon, a_{j_0} + \epsilon))$ is equal to $\widetilde{W}$. Then, since 
$\widetilde{W}$ is in $\psi_{d-1}(N-1, 0)\big( [0,1] \times (a_{j_0} - \epsilon, a_{j_0} + \epsilon)\times\Delta^k \big)$,
the equivalence stated at the end of Remark  \ref{remark.fib.reg} implies that 
$a_{j_0}$ is a fiberwise regular value of the map $x_1:\widehat{W} \rightarrow \mathbb{R}$.
On the other hand, since $\widehat{W}$ coincides with $[0,1]\times W$ in 
$x_1^{-1}(\mathbb{R} - [a_{j_0} - \frac{\epsilon}{2}, a_{j_0} + \frac{\epsilon}{2}])$, the remaining values from the set $\{a_0, \ldots, a_q\}$ are also fiberwise regular values for $x_1: \widehat{W} \rightarrow \mathbb{R}$. 
Furthermore, property (i) of the map $J$ implies that the concordance $\widehat{W}$ starts at $W$, whereas property (iii) implies that the other end of the concordance $\widehat{W}$, which we shall denote by $W'$, satisfies
\begin{equation} \label{property.DtoD2}
W'_{(a_{j_0} - \frac{\epsilon}{4}, a_{j_0} +  \frac{\epsilon}{4}), \hspace{0.03cm} \Delta^k} =  
(a_{j_0} - \frac{\epsilon}{4}, a_{j_0} +  \frac{\epsilon}{4})\times W_{a_{j_0}}. 
\end{equation}
Then, any homotopy $H: [0,1] \times \Delta^k  \rightarrow |\psi_d(N,1)_{\bullet}|$ induced by $\widehat{W}$ admits a lift $\widetilde{H}:  [0,1] \times \Delta^k \rightarrow |\mathcal{D}_d(\mathbb{R}^N)_{\bullet,q}|$ such that
$\widetilde{H}_0 = |h|$ and $\widetilde{H}_1$ is equal to the geometric realization of the map 
$\Delta^k_{\bullet} \rightarrow \mathcal{D}_d(\mathbb{R}^N)_{\bullet,q}$ that classifies the tuple 
$(W', a_0 < \ldots < a_q)$. 
Moreover, note that the construction of $\widehat{W}$ preserves the condition given in (\ref{property.DtoD}) for all 
$j = 0, \ldots, q$. Thus, for all times $t \in [0,1]$, the homotopy $\widetilde{H}$ maps $\partial\Delta^k$ to   
$|\mathcal{D}_d^{\perp}(\mathbb{R}^N)_{\bullet,q}|$. By iterating this argument for each fiberwise regular value 
$a_0, \ldots, a_q$, we obtain a homotopy 
$F_t: (\Delta^k, \partial \Delta^k) \rightarrow (|\mathcal{D}_d(\mathbb{R}^N)_{\bullet,q}|, |\mathcal{D}_d^{\perp}(\mathbb{R}^N)_{\bullet,q}|)$ such that $F_0= |h|$ and $F_1(\Delta^k) \subseteq  |\mathcal{D}_d^{\perp}(\mathbb{R}^N)_{\bullet,q}|$. In other words,  the map $|h|$ represents the trivial class in 
$\pi_k(|\mathcal{D}_d(\mathbb{R}^N)_{\bullet,q}|, |\mathcal{D}_d^{\perp}(\mathbb{R}^N)_{\bullet,q}|)$.
\end{proof}

Finally, we compare the spaces $\left\|\mathcal{D}(\mathbb{R}^N)_{\bullet,\bullet}\right\|$
and $\left|\widetilde{\psi}_d^R(N,1)_{\bullet}\right|$.

\theoremstyle{plain} \newtheorem{DtoPsi}[DesCob]{Proposition}

\begin{DtoPsi} \label{DtoPsi}
There is a weak equivalence 
$\left\|\mathcal{D}_d(\mathbb{R}^N)_{\bullet,\bullet}\right\| \stackrel{\simeq}{\rightarrow} \left|\widetilde{\psi}_d^R(N,1)_{\bullet}\right|$.
\end{DtoPsi}

\begin{proof}
For each $p$-simplex $W$ of $\widetilde{\psi}_d^R(N,1)_{\bullet}$,
let $(\mathbb{R}_W,\leq)$ be the sub-poset of
$(\mathbb{R}, \leq)$ consisting of all values 
$a \in \mathbb{R}$ which are fiberwise regular values of the
projection $x_1: W \rightarrow \mathbb{R}$.  
If we apply the geometric realization
functor along the second simplicial direction of
$\mathcal{D}_d(\mathbb{R}^N)_{\bullet,\bullet}$, we obtain
a semi-simplicial space $G_{\bullet}$ whose space of $p$-simplices
is equal to
\[
G_p = \coprod_{W \in \psi_d^R(N,1)_p}\{W\}\times B(\mathbb{R}_W, \leq). 
\]
But since each $B(\mathbb{R}_W,\leq)$
is contractible, it follows that each component of the map of 
semi-simplicial spaces $g_{\bullet}: G_{\bullet} \rightarrow \widetilde{\psi}_d^R(N,1)_{\bullet}$
defined by $(W,\lambda) \mapsto W$ is a weak homotopy equivalence. Therefore,
the  induced map between geometric realizations
\[
\left|g_{\bullet}\right|: \left\|\mathcal{D}_d(\mathbb{R}^N)_{\bullet}\right\| \rightarrow \left|\widetilde{\psi}_d^R(N,1)_{\bullet}\right|
\]
is also a weak homotopy equivalence. 
\end{proof}

By combining the previous three propositions, and by letting $N \rightarrow \infty$ in (\ref{chain.equiv.cob}), we obtain 
the following. 

\theoremstyle{plain}  \newtheorem{preDesCob}[DesCob]{Proposition}

\begin{preDesCob} \label{preDesCob}

There is a weak homotopy equivalence
$B\mathsf{Cob}_d^{\mathrm{PL}} \simeq |\psi_d^R(\infty,1)_{\bullet}|$.

\end{preDesCob}

In the next section, we will prove that the inclusion
$\psi_d^R(N,1)_{\bullet} \hookrightarrow \psi_d(N,1)_{\bullet}$
is a weak equivalence when $N-d \geq 3$, thus completing the proof of Theorem \ref{DesCob}.

\section{The inclusion $\psi_d^R(N,1)_{\bullet} \hookrightarrow \psi_d(N,1)_{\bullet}$}   \label{section5}  

We will devote this entire section to proving the following result. 

\theoremstyle{plain} \newtheorem{longman}{Theorem}[section]

\begin{longman} \label{longman}
If $N-d \geq 3$, then the inclusion $\psi_d^R(N,1)_{\bullet} \hookrightarrow \psi_d(N,1)_{\bullet}$
is a weak homotopy equivalence.
\end{longman}

As we mentioned in Remark \ref{remark.psireg}, instead of proving directly that the inclusion 
$\psi_d^R(N,1)_{\bullet} \hookrightarrow \psi_d(N,1)_{\bullet}$ 
is a weak homotopy equivalence,
we will show instead that the corresponding morphism
of semi-simplicial sets 
\[
\widetilde{\psi}_d^R(N,1)_{\bullet} \hookrightarrow \widetilde{\psi}_d(N,1)_{\bullet}
\]
is a weak equivalence. This is sufficient for proving Theorem \ref{longman} because 
the natural quotient maps
$|\widetilde{\psi}_d^R(N,1)_{\bullet}| \stackrel{\simeq}{\rightarrow} |\psi_d^R(N,1)_{\bullet}|  $
and $|\widetilde{\psi}_d(N,1)_{\bullet}| \stackrel{\simeq}{\rightarrow} |\psi_d(N,1)_{\bullet}|$
make the diagram 
\[
\xymatrix{
|\widetilde{\psi}_d^R(N,1)_{\bullet}| \ar@{^{(}->}[r] \ar[d]^{\simeq} & |\widetilde{\psi}_d(N,1)_{\bullet}| \ar[d]^{\simeq} \\
|\psi_d^R(N,1)_{\bullet}| \ar@{^{(}->}[r] & |\psi_d(N,1)_{\bullet}| }
\]
commutative. As also explained in Remark \ref{remark.psireg},
the reason we will prove Theorem \ref{longman} 
for semi-simplicial sets is that 
our proof will make use of Proposition \ref{subsimphop}, which is a result that only applies to 
semi-simplicial sets. Throughout this section, 
as we did in \S \ref{secsubmap},   we will denote
the semi-simplicial sets $\widetilde{\psi}^R_d(N,1)_{\bullet}$ and
$\widetilde{\psi}_d(N,1)_{\bullet}$  simply
by $\psi_d^R(N,1)_{\bullet}$ and $\psi_d(N,1)_{\bullet}$ respectively.

Our essential tool to prove Theorem \ref{longman} is the following proposition. 

\theoremstyle{plain} \newtheorem{makereg}[longman]{Proposition}

\begin{makereg} \label{makereg}

Suppose that $N-d\geq 3$. Also, fix a compact PL space $P$ and let $i_0, i_1: P \hookrightarrow [0,1] \times P$ be the inclusions defined by 
$i_j(x) = (j,x)$ for $j=0,1$. Given any $W$ in $\psi_d(N,1)(P)$ and any real value $\beta$, we can find a concordance $\widetilde{W} \in \psi_d(N,1)([0,1] \times P)$ with the following properties: 

\begin{itemize}

\item[(i)] $i^*_0\widetilde{W} = W$.

\item[(ii)] $\widetilde{W}$ agrees with the constant concordance $[0,1] \times W$ when we restrict the background space to $(-\infty, \beta) \times (-1,1)^{N-1}$. 

\item[(iii)] For the element $W'= i_1^*\widetilde{W}$, there exists a finite open cover $\{ U_j\}_{j \in \Lambda}$ of $P$ such that, for each $j\in \Lambda$, the projection $x_1: W'_{U_j} \rightarrow \mathbb{R}$ 
onto the second component of $U_j \times \mathbb{R} \times (-1,1)^{N-1}$
has a fiberwise regular value $a_j > \beta$. 

\end{itemize}

\end{makereg}

\begin{proof}[\textbf{Proof of Theorem \ref{longman}}]

At this point, it is convenient to give the 
proof of Theorem \ref{longman} assuming the result given in 
Proposition \ref{makereg}. Consider then an arbitrary map of pairs 
\[
f:(\Delta^p, \partial\Delta^p) \rightarrow \big( \left|\psi_d(N,1)_{\bullet}\right|, \left|\psi^R_d(N,1)_{\bullet}\right| \big).
\] 
We wish to show that $f$ represents the trivial class in 
$\pi_p\big(|\psi_d(N,1)_{\bullet}|,$ $ |\psi_d^R(N,1)_{\bullet}| \big)$.
First, it is straightforward to verify that the space $|\psi_d^R(N,1)_{\bullet}|$ is invariant  both under the subdivision map 
$\rho$ of 
$\psi_d(N,1)_{\bullet}$ and the homotopy $\mathcal{H}$ between $\rho$ and the identity map 
$\mathrm{Id}_{|\psi_d(N,1)_{\bullet}|}$ that we constructed in the proof of Proposition \ref{subhopid}
(we proved Proposition \ref{subhopid} for arbitrary PL sets; in particular, it holds for the PL set $\psi_d(N,1)$).  
Then, Proposition \ref{subsimphop} implies that $f$ is homotopic, as a map of pairs $(\Delta^p, \partial \Delta^p) \rightarrow \big( \left|\psi_d(N,1)_{\bullet}\right|, \left|\psi^R_d(N,1)_{\bullet}\right|\big)$, to a composition of the form
\begin{equation} \label{firstsimp}
(\Delta^p,\partial\Delta^p) \stackrel{f'}{\longrightarrow} (|K_{\bullet}|,|K'_{\bullet}|) \stackrel{\left|g\right|}{\longrightarrow} \big(\left|\psi_d(N,1)_{\bullet}\right|, \left|\psi^R_d(N,1)_{\bullet}\right|\big),
\end{equation}
where $(K_{\bullet},K'_{\bullet})$ is a pair of   
semi-simplicial sets induced by a pair of finite ordered
simplicial complexes $(K,K')$.   Moreover, 
we may assume that $(K, K')$ is a pair of simplicial complexes 
in some Euclidean space $\mathbb{R}^m$. As we did in Section \S \ref{section2}, we will denote the union of the simplices of $K$ (resp. $K'$) by $|K|$ (resp. $|K'|$).  Now, let $W \in \psi_d(N,1)(|K|)$ be the element classified 
(relative to the triangulation $(K, \mathrm{Id}_{|K|})$ of $|K|$) by 
the morphism $g$ that appears in (\ref{firstsimp}).  If $\sigma$ is any simplex of $K'$, then the condition $g(K'_{\bullet})\subseteq \psi^R_d(N,1)_{\bullet}$ implies that the restriction of $x_1: W \rightarrow \mathbb{R}$ on $W_{\sigma}$ has a fiberwise regular value $a_{\sigma}$. Pick such a fiberwise regular value 
$a_{\sigma}$  for each simplex $\sigma \in K'$, and fix once and for all a real number $\beta$
grater than $\mathrm{max}\{ a_{\sigma}: \text{ }\sigma \in K' \}$. 

By applying Proposition \ref{makereg} to $W \in \psi_d(N,1)(|K|)$ and the value $\beta$ we fixed above, we can produce a concordance $\widetilde{W}$ from $W$ to an element $W' \in \psi_d(N,1)(|K|)$ for which there is an open cover $U_1, \ldots, U_p$ of $|K|$ and real values $a_1, \ldots, a_p$ in $(\beta,\infty)$ with the property that 
$a_j$ is a fiberwise regular value of $x_1: W'_{U_j} \rightarrow \mathbb{R}$ for $j=1, \ldots, p$.  
Then, by the version of Proposition \ref{concord.homotopy} for semi-simplicial sets 
(see Remark \ref{semi.class.version}), $\widetilde{W}$ will induce a homotopy 
$\widetilde{H}: [0,1]\times |K_{\bullet}| \rightarrow |\psi_d(N,1)_{\bullet}|$
from $|g|$ to the geometric realization of the morphism $g': K_{\bullet} \rightarrow \psi_d(N,1)_{\bullet}$ which classifies the element $W'$ relative to the triangulation $(K, \mathrm{Id}_{|K|})$.
Moreover, according to Proposition \ref{makereg}, $\widetilde{W}$ agrees with the constant concordance $[0,1]\times W$ when we restrict the background space to $(-\infty, \beta) \times \mathbb{R}^{N-1}$. Thus, the homotopy $\widetilde{H}$ maps $[0,1]\times |K'_{\bullet}|$ to  $|\psi_d^R(N,1)_{\bullet}|$.   
It follows that the composition $\widetilde{H}_t\circ f'$ is a homotopy from $|g|\circ f'$ to $|g'|\circ f'$ which maps $\partial \Delta^p$ to $|\psi_d^R(N,1)_{\bullet}|$ for all times $t \in [0,1]$.  

Now, by the Lebesgue Number Lemma, we can find a large enough positive integer $r$ so that the $r$-th barycentric subdivision $\mathrm{sd}^{r}K$ is subordinate to the open cover $\{U_1, \ldots, U_p\}$ (i.e., each simplex of 
$\mathrm{sd}^{r}K$ is contained in some open set $U_i$). 
Recall that, for each set $U_i$ in the cover $\{U_1, \ldots, U_p\}$, the projection
$x_1: W'_{U_i} \rightarrow \mathbb{R}$ has a fiberwise regular value.
Consequently, 
for each simplex $\sigma$ of $\mathrm{sd}^{r}K$,
the projection $x_1: W'_{\sigma} \rightarrow \mathbb{R}$
has a fiberwise regular value. 
Now, let us apply Proposition \ref{unisub} 
to the PL set $\psi_d(N,1)$. 
According to this proposition, for the integer $r$ that we fixed above, 
there exists a semi-simplicial set map $g'': \mathrm{sd}^{r}K_{\bullet} \rightarrow \psi_d(N,1)_{\bullet}$ which 
makes the following diagram commute:
\begin{equation} \label{diagram.proof.longman}
\xymatrix{
\left|K_{\bullet}\right| \ar[r]^{ \hspace{-0.45cm} \left|g'\right|}  & \left| \psi_d(N,1)_{\bullet} \right| \ar[d]^{\rho^r} \\
\left|\mathrm{sd}^r \hspace{0.05cm} K_{\bullet} \right| \ar[u]^{\cong}\ar[r]^{  \hspace{-0.4cm} \left|g''\right|} & 
\left| \psi_d(N,1)_{\bullet}  \right|. }
\end{equation}
In this diagram, the left-vertical  map is again the canonical homeomorphism from 
$|\mathrm{sd}^r \hspace{0.05cm} K_{\bullet} |$ to $|K_{\bullet}|$ and $\rho^r$ 
 is the map obtained by composing the subdivision map $r$ times with itself. Furthermore, 
 by Proposition \ref{unisub.cor},  
 the morphism $g'': \mathrm{sd}^{r}K_{\bullet} \rightarrow \psi_d(N,1)_{\bullet}$ is equal to the classifying map 
 of $W'$ relative to the triangulation $(\mathrm{sd}^r\hspace{0.05cm}K, \mathrm{Id}_{|K|})$ of $|K|$. 
However, as we mentioned before, $W'$ is an element of 
 $\psi_d(N,1)(|K|)$ with the property that, for any simplex
 $\sigma \in \mathrm{sd}^r\hspace{0.05cm}K$,  the projection
 $x_1: W'_{\sigma} \rightarrow \mathbb{R}$
has a fiberwise regular value. 
It follows that the image of $g''$ must lie entirely in $\psi_d^R(N,1)_{\bullet}$
(see the discussion given in part (1) of Remark \ref{remark.psireg}.
Even though we focused only on simplicial sets in that remark, the same observations hold for 
semi-simplicial sets).  
Then, by the commutativity of (\ref{diagram.proof.longman}), the image of 
$\rho^r\circ |g'|$ is contained in $|\psi_d^R(N,1)_{\bullet}|$.  Also, since 
$|\psi_d^R(N,1)_{\bullet}|$ is invariant under the homotopy $\mathcal{H}$ between 
$\mathrm{Id}_{|\psi_d(N,1)_{\bullet}|}$ and the subdivision map $\rho$, we have that 
$\rho^r\circ |g'| \circ f'$ and $ |g'| \circ f'$ are homotopic as maps of pairs of the form 
$(\Delta^p, \partial \Delta^p) \rightarrow ( \left|\psi_d(N,1)_{\bullet}\right|, \left|\psi^R_d(N,1)_{\bullet}\right|)$. 
 
Thus, we have obtained the following chain of homotopies:
\begin{equation} \label{chain.of.homotopies}
f \text{ } \sim \text{ } |g|\circ f' \text{ } \sim \text{ }  |g'|\circ f' \text{ } \sim \text{ } \rho^r\circ |g'|\circ f'.
\end{equation}
Since $\rho^r \circ |g'|\circ f'$ maps $\Delta^p$ to $|\psi_d^R(N,1)_{\bullet}|$, and any two consecutive maps in 
(\ref{chain.of.homotopies}) are homotopic as maps of pairs 
$(\Delta^p, \partial \Delta^p) \rightarrow ( \left|\psi_d(N,1)_{\bullet}\right|, \left|\psi^R_d(N,1)_{\bullet}\right|)$,
we can conclude that $f$ represents the trivial class in 
$\pi_p\big(|\psi_d(N,1)_{\bullet}|,$ $ |\psi_d^R(N,1)_{\bullet}| \big)$.
\end{proof}

\theoremstyle{definition} \newtheorem{makereg.remark}[longman]{Remark}

\begin{makereg.remark} \label{makereg.remark}

We will spend most of this section developing the proof of Proposition \ref{makereg}. Before we begin discussing the details of this argument, it is worth noting that it is enough to prove 
Proposition \ref{makereg} in the case when the base-space $P$ is a closed PL manifold. Indeed, 
suppose that we have proven this proposition for all
closed PL manifolds, and fix a compact PL space $P$ and an element $W$ of 
$\psi_d(N,1)(P)$.  Since $P$ is compact, we can assume that $P$ is embedded in some sufficiently high-dimensional Euclidean space $\mathbb{R}^m$. Next, take a regular neighborhood $M'$ of $P$ in $\mathbb{R}^m$, which exists by the compactness of $P$ (see \cite{JLB}, \cite{RS} for a thorough discussion of regular neighborhoods). This neighborhood $M'$ satisfies the following two properties:

\begin{itemize}

\item[(1)] $M'$ is an $m$-dimensional compact PL manifold with boundary such that $P \subseteq M' - \partial M'$.

\item[(2)] There exists a PL retraction $r: M' \rightarrow P$.

\end{itemize}

Thus, we can pull back $W$ along the retraction $r$ to produce an element $r^*W$ of 
$\psi_d(N,1)(M')$. Since  $r: M' \rightarrow P$ is a retraction, the restriction of $r^*W$ over $P$ is equal to $W$.
Next, by taking a sufficiently high-dimensional space $\mathbb{R}^n$
and a PL embedding $j: M' \hookrightarrow [0,\infty)\times \mathbb{R}^n$ such that 
$j^{-1}(\{0\}\times\mathbb{R}^n)= \partial M'$, 
we may assume that $M'$ is a \textit{proper} PL submanifold of $[0,\infty)\times \mathbb{R}^n$. In particular, we have $M'\cap(\{0\}\times\mathbb{R}^n) = \partial M'$. We will now construct 
a closed PL submanifold $M$ of $(-\infty, \infty)\times \mathbb{R}^n \cong \mathbb{R}^{n+1}$
by \textit{doubling} $M'$; i.e., $M$ is the PL submanifold of $\mathbb{R}^{n+1}$ obtained as
follows:
\begin{itemize}
\item[$\cdot$] First, we take the image $M'' \subseteq (-\infty,0]\times\mathbb{R}^n$ of 
$M'$ under the piecewise linear map $q: M' \rightarrow (-\infty,0]\times\mathbb{R}^n$ which multiplies the first coordinate
of any point $x \in M'$ by -1. Note that $\partial M' = \partial M''$

\item[$\cdot$] Then, $M$ is defined as the union $M'' \cup M'$.

\end{itemize}

Observe that the pull-back $q^*r^*W$ of $r^*W$ along the PL map $q$ agrees with $r^*W$ over 
$\partial M'$. Thus, since $\psi_d(N,1)$ is a quasi-PL space (see Remark \ref{moreclasssub}), 
we can glue $q^*r^*W$ and $r^*W$ to produce an element 
$W'$ of $\psi_d(N,1)(M)$. Note that, by construction, we have that $W'_{M'} = r^*M$. Consequently, since the restriction of $r^*W$ over $P$ is $W$, we also have that $W'_P = W$. 
Now, since we are assuming that Proposition \ref{makereg} holds for $M$, there exists a concordance 
$\widetilde{W} \in \psi_d(N,1)([0,1]\times M)$ which starts at $W'$ and satisfies all of the desired properties. But since the restriction of $W'$ over $P$ is equal to $W$, the element $\widetilde{W}_{[0,1]\times P}$ of 
$ \psi_d(N,1)([0,1]\times P)$ obtained by restricting $\widetilde{W}$ over $[0,1]\times P$ is a concordance which starts at $W$. It is evident that $\widetilde{W}_{[0,1]\times P}$ will also satisfy all the properties listed in Proposition \ref{makereg}. 

\end{makereg.remark}

\subsection{The Isotopy Extension Theorem}  \label{section50}

The reason we need the codimension restriction $N-d \geq 3$ in Proposition \ref{makereg} (and, therefore, also in Theorem \ref{longman}) is that our proof for this proposition will use the PL version of \textit{the Isotopy Extension Theorem}. Before we recall the statement of this foundational result (Theorem \ref{isotopy} below), we need to review the following terminology from PL topology. First, fix two PL manifolds $M^m$ and $Q^q$ (possibly with boundary) with $m < q$. Also, denote the interval $[-1,1]$ by $I$ and the point $(0,\ldots,0)$ in the cartesian product $I^n$ by $\mathbf{0}$. An 
\textit{n-isotopy of $M$ in $Q$} is a PL embedding $f: I^n \times M \rightarrow  I^n \times Q$ which commutes with the projection onto $I^n$. We can think of an $n$-isotopy as a collection of PL embeddings 
$f_{t}: M \hookrightarrow Q$ parametrized by $t \in I^n$. An $n$-isotopy  $f:  I^n \times M \rightarrow I^n \times Q$ is said to be \textit{allowable} if there is an $(m-1)$-dimensional PL submanifold $N$ of $\partial M$ such that $f_t^{-1}(\partial Q) = N$ for all $t \in I^n$. In this  definition, we allow $N = \varnothing$ and $N = \partial M$. Finally, we say that an $n$-isotopy of $M$ in $Q$ is \textit{fixed} on a subset $X \subseteq M$ if $f_t|_X = f_{\mathbf{0}}|_X$ for all  $t \in I^n$.   

\theoremstyle{plain} \newtheorem{isotopy}[longman]{Theorem}

\begin{isotopy} \label{isotopy} \textbf{(The Isotopy Extension Theorem)} 
Fix two PL manifolds $M$, $Q$ of dimensions $m$ and $q$ respectively, and let $f: I^n \times M \rightarrow I^n\times Q$
be an allowable  $n$-isotopy of $M$ in $Q$ which is fixed on $f_{\mathbf{0}}^{-1}(\partial Q)$. If $M$ is compact and $q - m \geq 3$, then there is a PL homeomorphism $H: I^n \times Q \rightarrow I^n \times Q$ satisfying the following:

\begin{itemize}

\item[(i)] $H$ commutes with the projection onto $I^n$ and $H_{\mathbf{0}} = \mathrm{Id}_Q$. 

\item[(ii)] $H_t \circ f_{\mathbf{0}} = f_t$ for all $t \in I^n$. 

\item[(iii)] $H_t$ fixes $\partial Q$ pointwise for all $t \in I^n$. 

\end{itemize}

\end{isotopy}

A PL homeomorphism  $H: I^n \times Q \rightarrow I^n \times Q$ satisfying the conditions given in (i) is called an \textit{ambient n-isotopy of $Q$}. If an ambient $n$-isotopy $H$ of $Q$ satisfies (ii) above, then we say that $H$ \textit{extends} $f$. Property (iii) says that $H$ is \textit{fixed on} $\partial Q$. The proof of 
the Isotopy Extension Theorem can be found in \cite{ext} (see also Remark \ref{hudson.remark} below).  

\theoremstyle{definition} \newtheorem{isotopy.addendum}[longman]{Remark}

\begin{isotopy.addendum} \label{isotopy.addendum}

 The Isotopy Extension Theorem can be strengthened as follows. Suppose that 
 $f: I^n \times M \rightarrow I^n\times Q$ is an $n$-isotopy satisfying all the assumptions given in 
 Theorem \ref{isotopy}. In particular, we require that 
 $\mathrm{dim}\hspace{0.05cm}Q - \mathrm{dim}\hspace{0.05cm}M \geq 3$. Moreover, suppose that there is a PL subspace $B$ of $I^n$ with the following properties: 
 \begin{itemize}
 \item[$\cdot$] $\mathbf{0} \in B$.
 \item[$\cdot$] $B$ is a PL retract of $I^n$. 
 \item[$\cdot$] There exists a PL homeomorphism $h:B \times Q \rightarrow B \times Q$ which commutes with the projection onto $B$ and extends the isotopy $f$ \textit{over} $B$. In other words, 
 $h_t \circ f_{\mathbf{0}} = f_t$ for all $t \in B$. Moreover, assume that $h_{\mathbf{0}} = \mathrm{Id}_Q$. 
 \end{itemize}
Then, with these additional assumptions, we can find an ambient $n$-isotopy
 $H$ of $Q$ which satisfies all the claims given in the Isotopy Extension Theorem, but which also
 agrees with
 $h: B \times Q \rightarrow B \times Q$ over the PL subspace $B$, i.e., $H_t = h_t$ for any $t\in B$. 
 Let us explain why such an $H$ exists. First, fix a PL retraction $r: I^n \rightarrow B$ and an
 ambient $n$-isotopy $\widetilde{H}: I^n \times Q \rightarrow I^n \times Q$ 
 satisfying all the claims listed in the Isotopy Extension Theorem. 
 Next, consider the ambient $n$-isotopy  $H:  I^n \times Q \rightarrow I^n \times Q$ of $Q$
 given by 
  $H_t = \widetilde{H}_t \circ \widetilde{H}^{-1}_{r(t)}\circ h_{r(t)}$ for any $t \in I^n$. We claim that this $H$ satisfies all the desired properties. First, since $r|_{B} = \mathrm{Id}_B$, it follows that $H_t = h_t$ for all $t\in B$. Moreover, 
  by pre-composing $H_t$ with $f_{\mathbf{0}}$, we obtain the following:
  \[
  H_t \circ f_{\mathbf{0}} = \widetilde{H}_t \circ \widetilde{H}^{-1}_{r(t)}\circ h_{r(t)} \circ f_{\mathbf{0}} 
  =  \widetilde{H}_t \circ \widetilde{H}^{-1}_{r(t)}\circ f_{r(t)} 
  = \widetilde{H}_t \circ f_{\mathbf{0}} = f_t.
  \] 
 Thus, $H$ extends $f$. Finally, it is straightforward to verify that $H$ is fixed on $\partial Q$ 
 and $H_{\mathbf{0}} = \mathrm{Id}_Q$.  

\end{isotopy.addendum}

\theoremstyle{definition} \newtheorem{isotopy.over}[longman]{Remark}

\begin{isotopy.over} \label{isotopy.over}

The Isotopy Extension Theorem only concerns isotopies parameterized by $I^n$ for some $n$. However, for several of the arguments given in this section, it will be convenient to work with isotopies parameterized by more general 
base-spaces. With this in mind, we define the following: Let $P$ be a PL space. We say that a PL embedding 
$f: P \times M \rightarrow P \times Q$ is an \textit{isotopy of $M$ in $Q$ over $P$} if $f$ commutes with the projection onto $P$. We define \textit{ambient isotopies over} $P$ in a similar fashion. 

\end{isotopy.over}

There is another version of the Isotopy Extension Theorem (also proven in \cite{ext}) which does not require the codimension condition $q - m \geq 3$. To state this version, we need to introduce some terminology. First, let us fix two PL manifolds $M$ and $Q$ (possibly with boundary)
of dimensions $m$ and $q$ respectively.
We say that an $n$-isotopy $f: M\times I^n \rightarrow Q \times I^n$ 
is \textit{locally trivial} if, for any point $(x,t)\in M \times I^n$, there are open neighborhoods $U$ of $x$ in $M$ and 
$V$ of $t$ in $I^n$, and a PL embedding 
$\alpha:  (-1,1)^{q - m} \times U \times V \rightarrow Q\times V$ 
which commutes with the projection onto $V$ and such that 
$\alpha|_{\{0\}\times U \times V} = f|_{U\times V}$ once we identify  
$\{0\}\times U \times V$ with $U\times V$. With this terminology in place, we can now state the following alternative version of the Isotopy Extension Theorem discussed in \cite{ext} (see also \cite{KL}). 

\theoremstyle{plain} \newtheorem{loc.trivial.iso}[longman]{Theorem}

\begin{loc.trivial.iso} \label{loc.trivial.iso}
Let $M$ and $Q$ be PL manifolds (possibly with boundary) and suppose that $M$ is compact. 
If $f: M\times I^n \rightarrow Q \times I^n$ is an allowable and locally trivial $n$-isotopy, then there exists
an ambient $n$-isotopy $H: Q\times I^n \rightarrow Q \times I^n$ which extends $f$ and is fixed on $\partial Q$.

\end{loc.trivial.iso}

Throughout this article, we will mostly use the version of the Isotopy Extension Theorem 
stated in Theorem \ref{isotopy}. We will only apply the version given in Theorem \ref{loc.trivial.iso} 
in the proof of Proposition \ref{propmonoid}. It is worth pointing out that Hudson proves the Isotopy Extension Theorem 
in \cite{ext} assuming only that the $n$-isotopy $f$ is \textit{locally unknotted} (see page 652 of \cite{ext}),
which is a property weaker than local triviality. As shown in \cite{ext}, Theorem \ref{isotopy} is a consequence
of this version of the Isotopy Extension Theorem that requires only the $n$-isotopy to be locally unknotted. 

In reality, for the proof of Proposition \ref{makereg}, we will mostly use the following corollary of Theorem \ref{isotopy}.  

\theoremstyle{plain} \newtheorem{isotopy2}[longman]{Corollary}

\begin{isotopy2} \label{isotopy2}

Let $Q$ be a PL manifold of dimension $q$ and $M$ a compact submanifold of $Q$ of dimension $m$. Moreover,
suppose that $f:I^n \times M \rightarrow I^n \times Q$ is an allowable $n$-isotopy such that $f_{\mathbf{0}}$ is equal to the natural inclusion $M \hookrightarrow Q$. Then, if $q - m \geq 3$, 
there exists an ambient $n$-isotopy  $H: I^n \times Q \rightarrow  I^n \times Q$ 
satisfying all the claims listed in Theorem \ref{isotopy} 
and which is homotopic (via a homotopy of ambient $n$-isotopies) to the identity ambient $n$-isotopy 
$\mathrm{Id}_{I^n\times Q}$. In other words, there is an ambient isotopy 
$\widetilde{H}:[0,1]\times I^n \times Q \rightarrow [0,1]\times I^n \times Q$ of $Q$ over 
$[0,1]\times I^n$ (see Remark \ref{isotopy.over} above) such that 
$\widetilde{H}_0 = H$ and $\widetilde{H}_1 = \mathrm{Id}_{I^n \times Q}$. Moreover, we can also guarantee
that $\widetilde{H}$ fixes $[0,1]\times I^n \times \partial Q$ pointwise and $\widetilde{H}_{s,\mathbf{0}} = \mathrm{Id}_Q$
for all $s \in [0,1]$.

\end{isotopy2}

\begin{proof}
In this proof, we will modify our notation slightly and denote the point $\mathbf{0} = (0,\ldots, 0)$ of $I^n$ 
by $\hat{\mathbf{0}}$. Also, we will denote the PL subspace
\[
\big( [0,1] \times \{ \hat{\mathbf{0}} \} \big) \cup \big(\{1\} \times I^n\big)
\]
of $[0,1] \times I^n$ by $B$. Note that $B$ is a PL retract of $[0,1] \times I^n$. 
Let us fix then a PL retraction $r: [0,1] \times I^n \rightarrow B$. Moreover, 
fix a PL homotopy 
$c: [0,1] \times I^n \rightarrow I^n$ such that $c_0(\lambda) = \lambda$ for all $\lambda \in I^n$ and
$B \subseteq c^{-1}(\hat{\mathbf{0}})$. For example, we can construct such a PL homotopy 
$c$ as follows: 

\begin{itemize}

\item[$\cdot$] First, fix a pair of simplicial complexes $(K,L)$ which triangulates $([0,1] \times I^n, B)$. Moreover,
by subdividing $K$ further if necessary, we may assume that $K$ also has a subcomplex $K'$ which triangulates
the face $\{0\} \times I^n$. 

\item[$\cdot$] Next, we  define $c: [0,1] \times I^n \rightarrow I^n$ by setting $c(s, \lambda) = \hat{\mathbf{0}}$
for  any vertex $(s, \lambda) \in L$ and $c(s, \lambda) = \lambda$ for 
any vertex $(s, \lambda) \notin L$. Then, we extend $c$ linearly to all higher-dimensional simplices of 
$K$. Note that this definition forces $c$ to map each vertex 
$(0, \lambda)$ of $K'$ to $\lambda$. Therefore, $c_0$ maps $I^n$ identically to itself. Also, by the way we
defined $c$, it is clear that  $B$ is contained in $c^{-1}(\hat{\mathbf{0}})$.

\end{itemize}

The map $\widetilde{f}: [-1,1]\times I^n \times M \rightarrow [-1,1]\times I^n \times Q$ defined by 
$(s, \lambda, x) \mapsto (s, \lambda, f_{c_{|s|}(t)}(x))$ is an isotopy of PL embeddings of $M$  in $Q$ parameterized by $[-1,1]\times I^n$. In other words, $\widetilde{f}$ is an 
$(n+1)$-isotopy of $M$ in $Q$. Moreover, since $f$ is allowable, then so is $\widetilde{f}$. Thus, since 
$q - m \geq 3$, we can apply the Isotopy Extension Theorem to produce an ambient isotopy 
$\widetilde{H}:[-1,1]\times I^n \times Q \rightarrow [-1,1]\times I^n \times Q$ of $Q$ over $[-1,1]\times I^n$
which extends $\widetilde{f}$ and fixes $[-1,1]\times I^n \times \partial Q$ pointwise. For the rest of this proof, we will only need to use the restrictions 
$\widetilde{f}|_{ [0,1]\times I^n \times M}$ and $\widetilde{H}|_{ [0,1]\times I^n \times Q}$.
By abuse of notation, we shall denote these restrictions simply by $\widetilde{f}$ and $\widetilde{H}$ respectively. 
Note that, after we perform this change of notation, $\widetilde{H}$ still extends $\widetilde{f}$.

Recall that $f_{\hat{\mathbf{0}}}$ is equal to the natural inclusion $M \hookrightarrow Q$. 
Since $B  \subseteq c^{-1}(\hat{\mathbf{0}})$, it follows that $\widetilde{f}_{(s,t)}$ is equal to the natural inclusion 
$M \hookrightarrow Q$ for any $(s,\lambda) \in B$. Thus, the identity ambient isotopy
$\mathrm{Id}_{[0,1]\times [-1,1]^{n}\times Q}$ extends $\widetilde{f}$ over $B$. 
Moreover, recall that we fixed a PL retraction $r: [0,1]\times I^n \rightarrow B$ at the beginning of this proof.
Then, by doing an argument almost identical to the one given in Remark \ref{isotopy.addendum}, we may assume that  $\widetilde{H}|_{B\times Q} = \mathrm{Id}_{B\times Q}$.

To complete this proof, let $H: I^n \times Q \rightarrow  I^n \times Q$ 
be the ambient $n$-isotopy of $Q$ defined by 
$H := \widetilde{H}_0$. Since $c_0$ maps $I^n$ identically to itself, we have that $\widetilde{f}_0 = f$. 
Then, since $\widetilde{H}$ extends $\widetilde{f}$, it follows that
$H$ extends $f$. On the other hand, since $\{1\} \times I^n \subseteq B$ and 
$\widetilde{H}|_{B\times Q} = \mathrm{Id}_{B\times Q}$, the ambient $n$-isotopy $\widetilde{H}_1$ is equal to 
$\mathrm{Id}_{I^n\times Q}$. Thus, $\widetilde{H}$ is a homotopy of ambient $n$-isotopies from $H$ to 
$\mathrm{Id}_{I^n \times Q}$. As mentioned before, $\widetilde{H}$ fixes $[0,1]\times I^n \times \partial Q$ 
pointwise. Additionally, 
since $[0,1] \times \{ \hat{\mathbf{0}}\} \subseteq B$ and $\widetilde{H}|_{B\times Q} = \mathrm{Id}_{B\times Q}$, we also have that $\widetilde{H}_{(s, \hat{\mathbf{0}})} = \mathrm{Id}_Q$ for all $s \in [0,1]$.
Therefore, we have verified that $H$ and $\widetilde{H}$ satisfy all the required properties, which completes the proof. 
\end{proof}

\theoremstyle{definition} \newtheorem{hudson.remark}[longman]{Remark}

\begin{hudson.remark} \label{hudson.remark}
It is worth pointing out that the version of the Isotopy Extension Theorem that appears in \cite{ext} is slightly different from the one we stated in Theorem \ref{isotopy}. In \cite{ext}, the isotopy is defined over the cube $[0,1]^n$, not $[-1,1]^n$. Moreover, the reference point used in \cite{ext} to perform the isotopy extension is also $\mathbf{0}\in [0,1]^n$, but now 
this base-point lies on the boundary of the base-space instead of its interior, as it is in Theorem \ref{isotopy}. Nevertheless, our version of the Isotopy Extension Theorem can be derived easily from the one given in \cite{ext}. 
\end{hudson.remark}

\subsection{Producing fiberwise regular values}  \label{section51}

Let us explain more concretely how we will use the Isotopy Extension Theorem 
(or, more precisely, Corollary \ref{isotopy2}) in the proof of Proposition \ref{makereg}. First, fix a closed PL manifold $M$, an element $W$ of $\psi_d(N,1)(M)$, and a fiber $W_{\lambda}$ of the projection $\pi: W \rightarrow M$ over some point $\lambda \in M $. Recall that we may regard $W_{\lambda}$ as a PL submanifold of $\mathbb{R}\times (-1,1)^{N-1}$ which is closed as a topological subspace of $\mathbb{R}^N$.  
Since the projection $x_1: W_{\lambda} \rightarrow \mathbb{R}$ onto the first component of $\mathbb{R}\times (-1,1)^{N-1}$ is a proper map, we can find simplicial complexes $K$ and $L$ which triangulate $W_{\lambda}$ and $\mathbb{R}$ respectively and which make the map $x_1: W_{\lambda} \rightarrow \mathbb{R}$ simplicial. Thus, by Proposition \ref{williamson}, any point $a$ in the interior of a 1-simplex of 
$L$ is a regular value for the map 
$x_1: W_{\lambda} \rightarrow \mathbb{R}$. The idea is to then use the Isotopy Extension Theorem to deform $W$ near the fiber $W_{\lambda}$ (and near the height $a$) so that $a$ becomes a fiberwise regular value 
of $x_1: W \rightarrow \mathbb{R}$ over an open neighborhood of $\lambda$. 
We will make this procedure more explicit in the next lemma. 

\theoremstyle{plain} \newtheorem{makereg1}[longman]{Lemma}

\begin{makereg1}  \label{makereg1}

Fix a closed PL manifold $M$ of dimension $m$, an element $W$ of $\psi_d(N,1)(M)$, a point 
$\lambda_{0} \in M $, and a value $\epsilon>0$. 
If $a_0 \in \mathbb{R}$ is a regular value of the projection $x_1: W_{\lambda_0} \rightarrow \mathbb{R}$, then there exists an open neighborhood $V$ of $\lambda_0$ in $M$ and  a PL automorphism $F$ of $[0,1]\times V \times \mathbb{R} \times (-1,1)^{N-1}$ satisfying the following properties:

\begin{itemize}

\item[(i)] $F$ is an ambient isotopy of $\mathbb{R} \times (-1,1)^{N-1}$ over  $[0,1]\times V$
(see Remark \ref{isotopy.over}). Thus, we can view $F$
as a family of PL automorphisms $F_{t,\lambda}$ of \hspace{0.03cm} $\mathbb{R} \times (-1,1)^{N-1}$ parameterized by 
$(t,\lambda) \in [0,1]\times V$, or as a path $F_t$ of PL automorphisms of $V \times \mathbb{R} \times (-1,1)^{N-1}$ which commute with the projection onto $V$.

\item[(ii)]  $F_0$ is the identity map on $V\times \mathbb{R}\times (-1,1)^{N-1}$. 

\item[(iii)] $F_{t,\lambda}$ is supported on $(a_0 - \epsilon, a_0 + \epsilon)\times (-1,1)^{N-1}$ for all $(t,\lambda) \in [0,1]\times V$. 

\item[(iv)]  $a_0$ is a fiberwise regular value of the projection $x_1: F_1(W_V) \rightarrow \mathbb{R}$, where $W_V$ is the restriction of $W$ over the open set $V$.   

\end{itemize}

\end{makereg1}

\begin{proof}

Before we begin with the details of this proof, let us fix the following notation: For any point $\lambda \in M$ and any interval $I$ in $\mathbb{R}$, we shall denote the intersection of the fiber $W_{\lambda}$ with 
$I \times (-1,1)^{N-1}$ by $W_{\lambda, I}$. Also, given any $\lambda \in M$ and any value $a \in \mathbb{R}$, we will denote the pre-image of $a$ under the projection $x_1: W_{\lambda} \rightarrow \mathbb{R}$ by $W_{\lambda,a}$. 

Since $a_0$ is a regular value of the projection $x_1: W_{\lambda_0} \rightarrow \mathbb{R}$, there exists a
$\delta > 0$ and a PL homeomorphism 
$h: (a_0 - \delta, a_0 + \delta) \times W_{\lambda_0, a_0}  \rightarrow W_{\lambda_0, (a_0 - \delta, a_0 + \delta)}$ such that the composition  $x_1 \circ h$ agrees with the standard projection 
$\mathrm{pr}_1: (a_0 - \delta, a_0 + \delta) \times W_{\lambda_0, a_0}   \rightarrow (a_0 - \delta, a_0 + \delta)$. By shrinking the value $\delta$ if necessary, we may assume that $0< \delta < \epsilon$.
As we have typically done throughout this paper, we will denote the projection from $W$ onto $M$ by $\pi$. 
Recall that this projection is a PL submersion of codimension $d$. Also, note that $W_{\lambda_0, [a_0 - \delta, a_0 + \delta]}$ is a compact subspace of the fiber $W_{\lambda_0}$. Then, by the Union Lemma for PL submersions (see page 150 in \cite{Si}), we can find an open neighborhood $V$ of $\lambda_0$ and a PL embedding 
$f: \overline{V} \times W_{\lambda_0, [a_0 - \delta, a_0 + \delta]} \rightarrow \pi^{-1}(\overline{V})$ satisfying the following properties:

\begin{itemize}

\item[$\cdot$] $\overline{V}$ is PL homeomorphic to $[-1,1]^m$ and $\pi \circ f$ is equal to the standard projection onto 
$\overline{V}$. Thus, we can view $f$ as a family of PL embeddings $f_{\lambda}$ 
parametrized by $\lambda \in \overline{V}$.

\item[$\cdot$]  $f_{\lambda_0}$ maps $W_{\lambda_0, [a_0 - \delta, a_0 + \delta]}$ identically to itself. 

\item[$\cdot$]  $f_{\lambda}\big( W_{\lambda_0, [a_0 - \delta, a_0 + \delta]} \big) \subseteq 
(a_0 - \epsilon, a_0 + \epsilon)\times (-1,1)^{N-1}$ for all $\lambda \in \overline{V}$. 

\end{itemize}  

Moreover, let us fix a value $\epsilon' > 0$ which is less than $\delta$. By shrinking the neighborhood $V$ if necessary, we can also assume that the map
$f:\overline{V}\times W_{\lambda_0, [a_0 - \delta, a_0 + \delta]} \rightarrow \pi^{-1}(\overline{V})$ 
satisfies the following condition: 

\begin{itemize}

\item[(*)] For any $\lambda \in \overline{V}$, we have that $W_{\lambda, [a_0 - \epsilon', a_0 + \epsilon']} \subseteq f_{\lambda}\big(W_{\lambda_0, (a_0 - \delta, a_0 + \delta)}\big)$.

\end{itemize}

We will use this last condition in the final stage of this proof. Now, let us denote the product 
$[a_0 - \epsilon, a_0 + \epsilon]\times (-1,1)^{N-1}$ by $Q$. Once we identify $\overline{V}$ with $[-1,1]^m$,
the first and third conditions given above imply that the map $f$  is an $m$-isotopy of $W_{\lambda_0, [a_0 - \delta, a_0 + \delta]}$ (which is a compact $d$-dimensional PL manifold) in $Q$ (which is an $N$-dimensional PL manifold). In fact, $f$ is an allowable $m$-isotopy because, for any $\lambda \in \overline{V}$, the pre-image of 
$\partial Q$ under the embedding 
$f_{\lambda}$ is empty. Thus, since $N -d \geq 3$, the Isotopy Extension Theorem guarantees that there is an ambient $m$-isotopy $H: \overline{V} \times Q \rightarrow \overline{V} \times Q$ which extends 
$f$, fixes   $\partial Q$ pointwise, and has the property that 
$H_{\lambda_0} = \mathrm{Id}_{Q}$.  Moreover, by 
Corollary \ref{isotopy2}, we can choose $H$ so that there is an ambient isotopy 
\[
\widetilde{H}: [0,1]\times \overline{V} \times Q \rightarrow  
 [0,1]\times \overline{V} \times Q
\]
of $Q$ over $[0,1]\times \overline{V}$ such that $\widetilde{H}_{t,\lambda}$ fixes $\partial Q$ pointwise
for all $(s,\lambda) \in [0,1]\times \overline{V}$,  
$\widetilde{H}_0 = H$, $\widetilde{H}_1 = \mathrm{Id}_{\overline{V} \times Q}$, and
$\widetilde{H}_{(s, \lambda_0)} = \mathrm{Id}_Q$ for all $s \in [0,1]$.

Now, let $F:  [0,1]\times V \times Q \rightarrow  [0,1]\times V \times Q$ be the ambient isotopy of $Q$ over $[0,1]\times V$ obtained by first pre-composing $\widetilde{H}$ with the  product isotopy 
$\mathrm{Id}_{[0,1]}\times H^{-1}$, and then restricting this composition on $[0,1]\times V \times Q$.  Since both $\widetilde{H}$ and $\mathrm{Id}_{[0,1]}\times H^{-1}$ are fixed on $\partial Q$, then so is $F$. Therefore, we can extend $F$ to an ambient isotopy 
of $\mathbb{R} \times (-1,1)^{N-1}$ over $[0,1]\times V$ which is fixed on the closure of the complement of $Q$. To avoid introducing more notation, we will also denote this extension by $F$.
Since $\widetilde{H}_0 = H$, it follows that $F_0$ is the identity map on 
$V\times \mathbb{R} \times (-1,1)^{N-1}$.   
Thus, if we apply the ambient isotopy $F$ to the product $[0,1]\times W_V \in \psi_d(N,1)([0,1]\times V)$, we will obtain a concordance 
from  $W_V$ to the element 
$F_1(W_V) \in  \psi_d(N,1)(V)$. From now on, we will denote the images 
$F([0,1]\times W_V)$ and $F_1(W_V)$ by $\widehat{W}$ and $W'$ respectively. Note that, since $F$ is fixed on $\mathrm{cl}(\mathbb{R}\times (-1,1)^{N-1} - Q)$, the element $\widehat{W}$ agrees with the constant concordance $[0,1] \times W_V$ in the complement of  $[0,1]\times V \times Q$.  

Our next step is to examine the behavior of the concordance $\widehat{W}$ inside the product $[0,1] \times V \times Q$. More specifically, we wish to show (after possibly shrinking $V$) that $a_0$ is a fiberwise regular value of the projection  
$x_1: W' \rightarrow \mathbb{R}$. To this end, let us consider again the PL embedding $f: \overline{V} \times W_{\lambda_0, [a_0 - \delta, a_0 + \delta]} \rightarrow W_{\overline{V}}$ that we obtained via the Union Lemma at the beginning of this proof.  
Since $H^{-1}_{\lambda}\circ f_\lambda = f_{\lambda_0}$ for all $\lambda \in V$, and since $f_{\lambda_0}$ is the standard inclusion $W_{\lambda_0, [a_0 - \delta, a_0 + \delta]} \hookrightarrow Q$, the image of the composition $F_1 \circ f|_{V\times W_{\lambda_0, (a_0 - \delta, a_0 + \delta)}}$ is equal to the product 
$V \times W_{\lambda_0, (a_0 - \delta, a_0 + \delta)}$.  Therefore, if 
$h:(a_0 - \delta, a_0 + \delta) \times W_{\lambda_0, a_0}   \rightarrow W_{\lambda_0, (a_0 - \delta, a_0 + \delta)}$ is the PL homeomorphism we obtained by using the fact that $a_0$ is a regular value of 
$x_1: W_{\lambda_0} \rightarrow \mathbb{R}$, then the map 
$\widetilde{h}: V \times (a_0 - \delta, a_0 + \delta) \times W_{\lambda_0, a_0} \rightarrow W'$
defined by $\widetilde{h}(\lambda, t, x) = (\lambda, h(t,x))$ is a PL embedding whose image is equal to 
$V \times  W_{\lambda_0, (a_0 - \delta, a_0 + \delta)}$. Moreover, $\widetilde{h}$ commutes with the projections onto $V$ and $(a_0 - \delta, a_0 + \delta)$. Thus, to conclude that $a_0$ is indeed a fiberwise regular value of the map $x_1:  W' \rightarrow \mathbb{R}$, it suffices to find some value $0 < \delta' \leq \delta$ such that 
the image $\widetilde{h}( V \times (a_0 - \delta', a_0 + \delta') \times W_{\lambda_0, a_0})$, which is equal to 
$V \times W_{\lambda_0, (a_0 - \delta', a_0 + \delta')}$, agrees with the pre-image $x^{-1}_1\big((a_0 - \delta', a_0 + \delta')\big) \subseteq  W'$. To find such a value $\delta'$, we will use the condition (*) that we imposed on $f$. So let $\epsilon'>0$ be the value that appears in that condition, and choose any number $\delta'$ such that $0 < \delta' < \epsilon'$. 
Note that $F_{1,\lambda_0} = \mathrm{Id}_{\mathbb{R}\times (-1,1)^{N-1}}$, which implies that $F_{1,\lambda_0}(W_{\lambda_0}) = W_{\lambda_0}$. Then, by making the neighborhood $V$ smaller if necessary, we can also guarantee that 
\begin{equation} \label{cond.makereg1}
W'_{\lambda, (a_0 - \delta', a_0 + \delta')} \subseteq 
F_{1,\lambda}\big( W_{\lambda, (a_0 - \epsilon', a_0 + \epsilon' )} \big)
\end{equation}
for all $\lambda \in V$. By combining (\ref{cond.makereg1}) and condition (*), we obtain 
\begin{equation} \label{cond.makereg2}
W'_{\lambda, (a_0 - \delta', a_0 + \delta')} \subseteq  
F_{1,\lambda}\big(f_{\lambda}(W_{\lambda_0, (a_0 - \delta, a_0 + \delta)})\big)
\end{equation}
for all points $\lambda$ in $V$. Since $F_{1,\lambda}\big(f_{\lambda}(W_{\lambda_0, (a_0 - \delta, a_0 + \delta)})\big)$ is equal to $W_{\lambda_0, (a_0 - \delta, a_0 + \delta)}$, we can rewrite (\ref{cond.makereg2})  as $W'_{\lambda, (a_0 - \delta', a_0 + \delta')} \subseteq W_{\lambda_0, (a_0 - \delta, a_0 + \delta)}$, which clearly implies that 
\begin{equation} \label{cond.makereg3}
W'_{\lambda, (a_0 - \delta', a_0 + \delta')} \subseteq W_{\lambda_0, (a_0 - \delta', a_0 + \delta')}.
\end{equation}
On the other hand, since the embedding $\widetilde{h}: V \times (a_0 - \delta, a_0 + \delta) \times W_{\lambda_0, a_0} \rightarrow W'$ commutes with the projection onto 
$(a_0 - \delta, a_0 + \delta)$ and its image is equal to $V\times W_{\lambda_0, (a_0 - \delta, a_0 + \delta)}$, we also have 
\begin{equation} \label{cond.makereg4}
W_{\lambda_0, (a_0 - \delta', a_0 + \delta')} \subseteq W'_{\lambda, (a_0 - \delta', a_0 + \delta')}
\end{equation} 
for all $\lambda \in V$. Since (\ref{cond.makereg3}) also holds for any point $\lambda \in V$, it follows that the pre-image  $x^{-1}_1\big((a_0 - \delta', a_0 + \delta')\big)$ is equal to $V \times W_{\lambda_0, (a_0 - \delta', a_0 + \delta')}$. Thus, we can conclude that $a_0$ is a fiberwise regular value of the map
$x_1: W' \rightarrow \mathbb{R}$. Moreover, it is clear that the ambient isotopy $F$ satisfies all the other properties listed in this lemma. 
\end{proof}

\subsection{Proof of Proposition \ref{makereg}}  \label{section52}

At this point, it is helpful to explain our strategy for proving Proposition \ref{makereg}. Let then $M$ be a closed PL manifold and $W$ an element of $\psi_d(N,1)(M)$.
Using Lemma \ref{makereg1}, we can find a finite open cover $ U_1, \ldots, U_p$ of $M$ such that, for each $j=1,\ldots,p$, there is an ambient isotopy $F^j$ of $\mathbb{R}\times (-1,1)^{N-1}$ over $[0,1]\times M$ which produces a fiberwise regular value over $U_j$. Naively, 
to prove Proposition \ref{makereg},
one could try applying each isotopy $F^j$ successively with the hopes of producing a fiberwise regular value over each $U_j$ by the end of the process. However, the obvious problem with this approach is that if one has two open sets $U_j$ and $U_i$ that overlap, then applying either $F^j$ or $F^i$ first might affect the output of the other isotopy. To avoid this issue, we will prove that it is possible to define an open cover $ U_1, \ldots, U_p$ and isotopies $F^1, \ldots, F^p$ such that $F^j$ produces a fiberwise regular value over $U_j$  and, if $U_i$ and $U_j$ overlap, $F^j$ and $F^i$ will have disjoint supports in  $\mathbb{R}\times (-1,1)^{N-1}$. Thus, it will not matter if we apply either $F^j$ or $F^i$ first. The following lemma is our first step to construct this data.  

\theoremstyle{plain} \newtheorem{makereg2}[longman]{Lemma}
 
\begin{makereg2}  \label{makereg2}

Let $M$ be a closed $m$-dimensional PL manifold, $W$ an element of $\psi_d(N,1)(M)$, and $\lambda_0$ a point in $M$. Moreover, fix the following: 

\begin{itemize}

\item[$\cdot$] A value $\beta \in \mathbb{R}$. 

\item[$\cdot$] $m + 1$ regular values $a_0, \ldots, a_m$ of the projection $x_1: W_{\lambda_0} \rightarrow \mathbb{R}$ with the property that $a_j \in (\beta + j + \frac{1}{4} , \beta + j + \frac{3}{4})$ for each 
$j=0, \ldots, m$. 

\item[$\cdot$] A value $\epsilon > 0$ such that 
$(a_j - \epsilon, a_j + \epsilon) \subseteq (\beta + j + \frac{1}{4} , \beta + j + \frac{3}{4})$ for all 
$j=0, \ldots, m$.

\end{itemize}

Then, if $N - d \geq 3$, we can find an open neighborhood $V$ of $\lambda_0$ and $m+1$ ambient isotopies $F^{0}, \ldots, F^{m}$ of $\hspace{0.05cm}\mathbb{R}\times (-1,1)^{N-1}$ over $[0,1]\times V$ such that, for each 
$j=0, \ldots, m$, the 4-tuple $(V, \epsilon, a_j, F^{j})$ satisfies all the properties listed
in Lemma \ref{makereg1}. 

\end{makereg2}

\begin{proof}
First of all, recall that the set of regular values of $x_1: W_{\lambda_0} \rightarrow \mathbb{R}$ is dense in $\mathbb{R}$ (see Remark \ref{remark.williamson2}). Therefore, it is indeed possible to pick $m+1$ regular values $a_0, \ldots, a_m$ such that 
$a_j \in (\beta + j + \frac{1}{4} , \beta + j + \frac{3}{4})$ for each $j=0, \ldots, m$. Now, for any $j$ in 
$\{0, \ldots, m\}$, we can find an open neighborhood $V_j$ of $\lambda_0$ and an ambient isotopy $G^j$ 
of $\mathbb{R}\times (-1,1)^{N-1}$ over $[0,1]\times V_j$ such that  $(V_j, \epsilon, a_j, G^{j})$ satisfies all the properties from Lemma \ref{makereg1}. Then, if $V$ is any open neighborhood of $\lambda_0$ contained in 
$V_0\cap \ldots \cap V_m$ and $F^j = G^j|_{[0,1]\times V \times \mathbb{R}\times (-1,1)^{N-1}}$ for $j=0,\ldots,m$, it is evident that each 4-tuple of the form $(V, \epsilon, a_j, F^{j})$ will also satisfy the properties given in the statement of Lemma \ref{makereg1}.
 \end{proof}

In Lemma \ref{coverk} below, we will construct the type of open covers that we wish to use in the proof of Proposition \ref{makereg}. For the statement of this lemma, and for its proof, we shall need the following technical definition. 

\theoremstyle{definition} \newtheorem{concentric}[longman]{Definition}

\begin{concentric} \label{concentric}
 
Let $K$ be a simplicial complex in some
Euclidean space $\mathbb{R}^m$,  and let $\sigma$ be a simplex of $K$
with barycentric point $b(\sigma)$. 
A subset 
$\widetilde{\sigma}$ of $\sigma$ is called a
\textit{concentric subsimplex of $\sigma$} if:

\begin{enumerate}

\item $\widetilde{\sigma}=\sigma$ if $\mathrm{dim}\hspace{0.05cm}\sigma = 0$.

\item If $\mathrm{dim}\hspace{0.05cm}\sigma > 0$, then there
is a value $0< t_0 < 1$ such that
\[
\widetilde{\sigma} = \{ (1-t)\cdot b(\sigma) + t\cdot x \hspace{0.1cm} | \hspace{0.1cm} x\in \partial\sigma, 0\leq t \leq t_0\}.
\]
\end{enumerate}

\end{concentric}

In the next lemma, we shall use the following definition: A \textit{PL ball of dimension $m$} is a PL space that is PL homeomorphic to the cube $[-1,1]^m$.  

\theoremstyle{plain}  \newtheorem{coverk}[longman]{Lemma}

\begin{coverk} \label{coverk}
Fix a closed $m$-dimensional PL manifold $M$.
Without loss of generality, we will suppose that $M$ is embedded in some $\mathbb{R}^n$
of sufficiently high dimension.  
Moreover, let $\mathcal{U} = \{V_{\alpha}\}_{\alpha\in\Lambda}$
be an open cover of $M$, and $K$ a finite simplicial complex 
which triangulates $M$ and
is subordinate to $\mathcal{U}$ (i.e., for each simplex $\sigma \in K$, there is a $V_{\alpha} \in \mathcal{U}$ such that $\sigma \subseteq V_{\alpha}$). Then, 
there exists a family of concentric subsimplices $\{ \widetilde{\sigma} \}_{\sigma \in K}$ such that   
$\widetilde{\sigma} \subseteq \sigma$ for each $\sigma \in K$, and
two collections $\{ U_{\sigma}\}_{\sigma \in K}$, $\{ U'_{\sigma}\}_{\sigma \in K}$
of open sets in $M$ 
which satisfy the following properties:

\begin{itemize}
\item[$(i)$]  $\widetilde{\sigma} \subseteq U_{\sigma}$ and  $\overline{U_{\sigma}} \subseteq U'_{\sigma}$ 
for all $\sigma \in K$. 

\item[$(ii)$] For each $\sigma \in K$, the closures $\overline{U_{\sigma}}$ and $\overline{U'_{\sigma}}$ are PL balls of dimension $m$. 

\item[$(iii)$] For each $\sigma \in K$, the closure
$\overline{U'_{\sigma}}$ is contained in some open set $V_{\alpha}$ of the collection $\mathcal{U}$.

\item[$(iv)$] $\overline{U'_{\sigma_1}} \cap \overline{U'_{\sigma_2}} = \varnothing$
if $\sigma_1$ and $\sigma_2$ are distinct simplices of $K$ of the same dimension.

\item[$(v)$] The collection $\{ U_{\sigma} \}_{\sigma\in K}$ is 
an open cover of $M$. 

\end{itemize}

\end{coverk} 

\begin{proof}
We are going to define the concentric simplices
$\widetilde{\sigma}$ and the open neighborhoods
$U_{\sigma}$,  $U'_{\sigma}$ by induction
on the dimension of the simplices of $K$. 
Let then $\alpha_1, \ldots, \alpha_p$ be the vertices of
$K$ and, for each $i=1,\ldots, p,$ let $V_i$ be an open set of the cover $\mathcal{U}$
which contains $\alpha_i$. Since $M$ is Hausdorff, we can find a collection of $p$ pairwise disjoint open sets 
$U'_{\alpha_1}, \ldots, U'_{\alpha_p} \subseteq M$ such that $\alpha_i \in U'_{\alpha_i}$ and 
$U'_{\alpha_i} \subseteq V_i$ for each $i = 1, \ldots, p$. 
Moreover, for each $i = 1, \ldots, p$, we may assume that $\overline{U'_{\alpha_i}}$ is a PL ball of dimension $m$.
By shrinking each ball $\overline{U'_{\alpha_i}}$ if necessary, we can also assume that 
the closures $\overline{U'_{\alpha_1}}, \ldots, \overline{U'_{\alpha_p}}$ are pairwise disjoint.
To complete the first step of this induction argument, we take for each vertex $\alpha_i \in K$ an open set $U_{\alpha_i} \subseteq M$ such that $\alpha_i \in U_{\alpha_i}$ and
$\overline{U_{\alpha_i}} \subseteq U'_{\alpha_i}$. Again, we can assume that each closure 
$\overline{U_{\alpha_i}}$ is a PL ball of dimension $m$. 

Now, let $q$ be a non-negative integer less than $m$ (recall that $m$ is the dimension of $M$), and let $K^q$ denote the $q$-skeleton of $K$. Suppose that, for each simplex $\sigma$ of $K^q$, we have defined a concentric subsimplex 
$\widetilde{\sigma}$ and a pair of open sets 
$U_{\sigma}, U'_{\sigma}$ in $M$ which satisfy the following: 

\begin{itemize}

\item[$\cdot$] The collection of open sets $\{U_{\sigma}\}_{\sigma \in K^q}$ covers $|K^q|$.

\item[$\cdot$] The collection of open sets $\{  U_{\sigma}, U'_{\sigma} \}_{\sigma \in K^q}$ satisfy properties 
(i), (ii), (iii), and (iv) given in the statement of this lemma. 

\end{itemize} 

For each $(q+1)$-simplex $\beta$ of $K$,
pick a concentric subsimplex 
$\widetilde{\beta}$ such that 
\[ 
\mathrm{cl}(\beta - \widetilde{\beta}) \subseteq \bigcup_{\sigma \in K^q}U_{\sigma}
\]
and an open set $V_{\beta} \in \mathcal{U}$ which contains $\beta.$ 
If $\beta_1$ and $\beta_2$ are distinct $(q+1)$-simplices of $K$, then note that 
$\widetilde{\beta}_1 \cap \widetilde{\beta}_2 = \varnothing$ since 
the interiors of $\beta_1$ and $\beta_2$ do not overlap. 
Since $M$ is a normal space, we can choose for each $(q+1)$-simplex $\beta$ an open neighborhood  $U'_{\beta} \subseteq M$ of $\widetilde{\beta}$ so that $U'_{\beta_1}\cap U'_{\beta_2} = \varnothing$ for any two distinct $(q+1)$-simplices $\beta_1$ and $\beta_2$.
In fact, for each $(q+1)$-simplex $\beta$ of $K$, we can choose $U'_{\beta}$ so that 
$\overline{U'_{\beta}}$ is a PL ball of dimension $m$ contained in $V_{\beta}$. For example, we can take 
$\overline{U'_{\beta}}$ to be a regular neighborhood of $\widetilde{\beta}$ in $V_{\beta}$. Since $\widetilde{\beta}$ is contractible and the dimension of $M$ is $m$, then standard facts from regular neighborhood theory 
ensure that $\overline{U'_{\beta}}$ is a PL ball of dimension $m$. 
Moreover, by shrinking each $\overline{U'_{\beta}}$ if necessary, we can also guarantee that 
$\overline{U'_{\beta_1}} \cap \overline{U'_{\beta_2}} = \varnothing$ for any pair of distinct 
$(q+1)$-simplices $\beta_1$ and $\beta_2$. 
Finally, to conclude the induction step, we just choose for  each $(q+1)$-simplex $\beta$  an open subset $U_{\beta}\subseteq M$ with the property that $\widetilde{\beta} \subseteq U_{\beta}$ and $\overline{U_{\beta}} \subseteq U'_{\beta}$. Again, we may assume that $\overline{U_{\beta}}$ is a regular neighborhood of $\beta$ 
contained in $U'_{\beta}$, which would imply that $\overline{U_{\beta}}$ is a PL ball of dimension $m$. 
By the way we chose the open sets $U_{\beta}$ and $U'_{\beta}$, we have that 
$\{U_{\sigma}\}_{\sigma \in K^{q+1}}$
is an open cover of  $|K^{q+1}|$ and that the collection
$\{U_{\sigma}, U'_{\sigma}\}_{\sigma \in K^{q+1}}$ 
satisfies conditions (i), (ii), (iii), and (iv) listed
in this lemma. 
\end{proof}

For the next remark, it is helpful to recall the following 
notation, which we used already in the statement of Lemma
\ref{makereg1}: If $f: [0,1]\times V \times M \rightarrow [0,1]\times V \times Q$ is an isotopy of $M$ in $Q$
over $[0,1]\times V$, then we shall denote the value of $f$ at $(t,\lambda) \in [0,1]\times V$ by $f_{t,\lambda}$. 

\theoremstyle{definition}  \newtheorem{dampening.isotopies}[longman]{Remark}

\begin{dampening.isotopies} \label{dampening.isotopies}
For the proof of Proposition \ref{makereg}, we also need to discuss how to \textit{dampen isotopies} over a 
base-space of the form $[0,1]\times V$. The space $V$  in the product $[0,1]\times V$ shall typically 
be an open set in a PL manifold.  
To explain this dampening procedure, let us fix an open set $V$ in an $m$-dimensional PL manifold
and a PL isotopy 
$f: [0,1] \times V \times M \rightarrow [0,1]\times V \times Q$ of $M$ in $Q$ over $[0,1]\times V$. 
Additionally, fix two nested open sets $U \subseteq U'$ contained in $V$ satisfying the following properties:
\begin{itemize}
\item[$\cdot$] $\overline{U} \subseteq U'$ and $\overline{U'}\subseteq V$. 

\item[$\cdot$] $\overline{U}$ and $\overline{U'}$ are PL balls of dimension $m$.
 \end{itemize}
Finally, fix a PL bump function $\alpha: V \rightarrow [0,1]$ such that 
$\alpha(\overline{U}) = \{1\}$ and $\alpha(\lambda) = 0$ for all $\lambda$ in $V - U'$.  
We wish to use the 
bump function $\alpha$ to produce a new isotopy $g$ over $[0,1]\times V$ which agrees with $f$ over 
$[0,1] \times \overline{U}$ and has the property that $g_{t,\lambda} = f_{0,\lambda}$ for all $(t,\lambda)$ in 
$[0,1] \times (V - U')$. One might be tempted to define 
this new isotopy $g$ by setting $g_{t,\lambda} = f_{\alpha(\lambda)t,\lambda}$
for all $(t,\lambda)$ in $[0,1] \times V$,
where $\alpha(\lambda)t$ is the product of the real values $\alpha(\lambda)$ and $t$.  
The problem with this definition is that the ordinary product of real 
numbers is not a PL function, so this way of defining $g$ might not give a PL isotopy. 
To fix this problem, consider the function 
$p: [0,1]\times V \rightarrow [0,1]$ given by $p(t, \lambda) = \alpha(\lambda)t$.
Note that for any point $(t,\lambda)$ in 
$[0,1]\times \overline{U}$ we have $p(t, \lambda) = t$. On the other hand, $p$ maps all points in 
$[0,1] \times (V - U')$ to $0$. In particular,
 the restrictions $p|_{[0,1]\times \overline{U}}$  and
$p|_{[0,1] \times (V - U')}$ are piecewise linear. Then, since both 
$[0,1]\times \overline{U}$ and $[0,1] \times (V - U')$ are closed PL subspaces of $[0,1]\times V$, 
the relative version of the Simplicial Approximation Theorem (proven in \cite{relative.iso}) guarantees that $p$ is homotopic to a PL map
$q: [0,1]\times V \rightarrow [0,1]$ which agrees with $p$ on 
$[0,1]\times \overline{U}$ and $[0,1] \times (V - U')$.  
Thus, this new PL map $q$ also  
satisfies $q(t, \lambda) = t$ for all $(t,\lambda) \in [0,1] \times \overline{U}$ and
$q(t, \lambda) = 0$ for all $(t,\lambda) \in [0,1] \times (V - U')$. 
It follows that the PL isotopy 
$g: [0,1] \times V \times M \rightarrow [0,1]\times V \times Q$ defined by 
setting $g_{t,\lambda} = f_{q(t,\lambda), \lambda}$ 
agrees with $f$ over $[0,1]\times \overline{U}$ and satisfies 
$g_{t,\lambda} = f_{0,\lambda}$ for all $(t,\lambda)$ in $[0,1] \times (V - U')$.
In future arguments (e.g., the proof of Proposition \ref{makereg}, which we will give immediately after this remark), if $g$ is an isotopy obtained from another isotopy $f$ 
via the procedure described above, then we shall say that $g$ 
is obtained  by \textit{dampening} $f$.  
\end{dampening.isotopies}

\begin{proof}[\textbf{\textit{Proof of Proposition \ref{makereg}}}] 
As indicated in Remark \ref{makereg.remark}, it is enough to prove this proposition in the case when the 
base-space is a closed PL manifold. Let us fix then a closed PL manifold $M$
of dimension $m$, and let $W$ be an element of $\psi_d(N,1)(M)$. 
For each point $\lambda$ in $M$, 
we will fix $m+1$ regular values $a_{\lambda,0}, \ldots, a_{\lambda,m}$ of the projection $x_1: W_{\lambda} \rightarrow \mathbb{R}$ with the property that $a_{\lambda, \hspace{0.03cm}j} \in (\beta + j + \frac{1}{4}, \beta + j + \frac{3}{4} )$ for $j=0, \ldots, m$. 
The value  $\beta$ that appears in each interval $(\beta + j + \frac{1}{4}, \beta + j + \frac{3}{4} )$ is the one that we fixed in the statement of this proposition. 
Also, for each $\lambda \in M$, we will fix a value $\epsilon_{\lambda} > 0$ such that
$(a_{\lambda,  \hspace{0.03cm} j} - \epsilon_{\lambda}, a_{\lambda,  \hspace{0.03cm} j} + \epsilon_{\lambda}) \subseteq (\beta + j + \frac{1}{4} , \beta + j + \frac{3}{4})$
for $j=0, \ldots, m$. According to Lemma 
\ref{makereg2}, for each $\lambda \in M$ we can find an open neighborhood $V_{\lambda}$ of $\lambda$
and  $m+1$ ambient isotopies $F^{\lambda,0}, \ldots, F^{\lambda,m}$ of $\mathbb{R}\times (-1,1)^{N-1}$ over 
$[0,1]\times V_{\lambda}$ with the following properties: 

\begin{itemize}

\item[$\cdot$] For $j=0,\ldots, m$, the isotopy $F^{\lambda, \hspace{0.03cm} j}$ is supported in $(a_{\lambda,  \hspace{0.03cm}j} - \epsilon_{\lambda}, a_{\lambda,  \hspace{0.03cm}j} + \epsilon_{\lambda}) \times (-1,1)^{N-1}$  and $F^{\lambda,  \hspace{0.03cm}j}_0$ is equal to the identity isotopy over $V_{\lambda}$. 

\item[$\cdot$] If $W_{V_{\lambda}}$ is the restriction of $W$ over $V_{\lambda}$, then  $a_{\lambda,  \hspace{0.03cm}j}$ is a fiberwise regular value of the projection 
$x_1: F^{\lambda,  \hspace{0.03cm}j}_1(W_{V_{\lambda}}) \rightarrow \mathbb{R}$. 

\end{itemize}

By the compactness of $M$, we can choose finitely many points $\lambda_1, \ldots, \lambda_p \in M$ such that the corresponding open sets $V_{\lambda_1}, \ldots, V_{\lambda_p}$ cover $M$. From now on, we will denote the open sets  $V_{\lambda_1}, \ldots, V_{\lambda_p}$ by 
 $V_1, \ldots, V_p$ and, for each $i = 1, \ldots, p$, we will relabel the fiberwise regular values
 $a_{\lambda_i,0}, \ldots, a_{\lambda_i,m}$ and the ambient isotopies 
 $F^{\lambda_i,0}, \ldots, F^{\lambda_i,m}$ corresponding to $\lambda_i$ as 
 $a_{i,0}, \ldots, a_{i,m}$  and $F^{i,0}, \ldots, F^{i,m}$ respectively. Moreover, we relabel the values 
 $\epsilon_{\lambda_1}, \ldots, \epsilon_{\lambda_p}$ as $\epsilon_1, \ldots, \epsilon_p$. 
  
Next, take a finite simplicial complex $K$ which triangulates $M$ and is subordinate to the open cover 
 $\{V_1, \ldots, V_p\}$. By Lemma \ref{coverk}, we can construct two open covers
 $\mathcal{U} = \{U_{\sigma}\}_{\sigma \in K}$ and  $\mathcal{U}' = \{ U'_{\sigma}\}_{\sigma \in K}$ of $M$ which satisfy the following conditions:  
 
 \begin{itemize}
 
 \item[$\cdot$]  For each $\sigma \in K$, $\overline{U_{\sigma}} \subseteq U'_{\sigma}$ and 
 $\overline{U'_{\sigma}}$ is contained in some open set $V_j$ of the collection  $\{V_1, \ldots, V_p\}$. 
 
  \item[$\cdot$]  For each $\sigma \in K$, $\overline{U_{\sigma}}$ and 
  $\overline{U'_{\sigma}}$ are PL balls of dimension $m$.
 
  \item[$\cdot$]  $\overline{U'_{\sigma_1}} \cap \overline{U'_{\sigma_2}} = \varnothing$ for any pair of distinct simplices $\sigma_1, \sigma_2$ of the same dimension. 
 
 \end{itemize}
 
For each $\sigma \in K$, let us fix an index $i_{\sigma} \in \{1, \ldots, p\}$ with the property that 
$\overline{U'_{\sigma}} \subseteq V_{i_{\sigma}}$.  Additionally, for each simplex $\sigma$ in $K$, let $|\sigma|$ denote the dimension of $\sigma$. For each $\sigma \in K$, we can dampen the isotopy $F^{\hspace{0.03cm} i_{\sigma}, |\sigma|}$ (using the procedure described in Remark \ref{dampening.isotopies}) 
to produce a new ambient isotopy $G^{\hspace{0.03cm}i_{\sigma}, |\sigma|}$ of 
$\mathbb{R} \times (-1,1)^{N-1}$ over $[0,1] \times V_{i_{\sigma}}$ which agrees with 
$F^{\hspace{0.03cm} i_{\sigma}, |\sigma|}$ over $[0,1] \times \overline{U_{\sigma}}$, and
agrees with the identity isotopy over 
the union $\big( \{0\}\times V_{i_{\sigma}} \big)\cup \big( [0,1]\times (V_{i_{\sigma}}- U'_{\sigma}) \big)$.
By this last property, we may assume that $G^{\hspace{0.03cm}i_{\sigma}, |\sigma|}$ is defined over
$[0,1]\times M$ and that $G^{\hspace{0.03cm}i_{\sigma}, |\sigma|}$ agrees with the
identity isotopy over $\big( \{0\}\times M  \big)\cup \big( [0,1]\times (M - U'_{\sigma}) \big)$.
Since $G^{\hspace{0.03cm}i_{\sigma}, |\sigma|}$ was obtained by dampening 
$F^{\hspace{0.03cm} i_{\sigma}, |\sigma|}$, the isotopy $G^{\hspace{0.03cm}i_{\sigma}, |\sigma|}$ will also be supported in $(a_{i_{\sigma}, |\sigma|} - \epsilon_{i_{\sigma}}, a_{i_{\sigma}, |\sigma|} + \epsilon_{i_{\sigma}}) \times (-1,1)^{N-1}$, where  $a_{i_{\sigma}, |\sigma|}$ is the fiberwise regular value  over $V_{i_{\sigma}}$ produced by the isotopy $F^{\hspace{0.03cm} i_{\sigma}, |\sigma|}$. 

Note that any two distinct ambient isotopies in the collection 
$\{G^{\hspace{0.03cm}i_{\sigma}, |\sigma|}\}_{\sigma \in K}$ will commute with each other. 
Indeed, if $\sigma_1, \sigma_2 \in K$ have different dimensions, then the intervals 
$[a_{i_{\sigma_1}, |\sigma_1|} - \epsilon_{i_{\sigma_1}}, a_{i_{\sigma_1}, |\sigma_1|} + \epsilon_{i_{\sigma_1}}]$ and 
$[a_{i_{\sigma_2}, |\sigma_2|} - \epsilon_{i_{\sigma_2}}, a_{i_{\sigma_2}, |\sigma_2|} + \epsilon_{i_{\sigma_2}}]$ do not intersect. Consequently, $G^{\hspace{0.03cm}i_{\sigma_1}, |\sigma_1|}$ and 
$G^{\hspace{0.03cm}i_{\sigma_2}, |\sigma_2|}$ have disjoint supports in $\mathbb{R}\times (-1,1)^{N-1}$. 
On the other hand, if $\sigma_1$ and $\sigma_2$ have the same dimension, then the corresponding PL balls 
$\overline{U'_{\sigma_1}}$ and $\overline{U'_{\sigma_2}}$ do not overlap, which then implies that 
$G^{\hspace{0.03cm}i_{\sigma_1}, |\sigma_1|}$ and 
$G^{\hspace{0.03cm}i_{\sigma_2}, |\sigma_2|}$ have disjoint supports in the base-space $[0,1]\times M$. 
Therefore, if we plug in the constant concordance $[0,1]\times W$ in the ambient isotopy $G$ obtained by composing all the isotopies $G^{\hspace{0.03cm}i_{\sigma}, |\sigma|}$ in any given order, we will obtain an element $\widetilde{W} \in \psi_d(N,1)([0,1]\times M)$   which is a concordance from $W$ to an element 
$W' \in \psi_d(N,1)(M)$ with the property that, for any simplex $\sigma \in K$, the restriction 
$W'_{U_{\sigma}}$ over $U_{\sigma}$
agrees with $F^{i_{\sigma}, |\sigma|}_1(W_{U_{\sigma}})$. In particular, for each 
$\sigma \in K$, we have that
$a_{i_{\sigma}, |\sigma|}$ is a fiberwise regular value of the projection 
$x_1: W'_{U_{\sigma}}\rightarrow \mathbb{R}$. 
Thus, the concordance $\widetilde{W}$ and the open cover $\{ U_{\sigma}\}_{\sigma \in K}$
satisfy properties (i) and (iii) of Proposition \ref{makereg}.
To see that $\widetilde{W}$ also satisfies property (ii), just note that the support of each isotopy $G^{\hspace{0.03cm}i_{\sigma}, |\sigma|}$ in $\mathbb{R}\times (-1,1)^{N-1}$ lies above the height $\beta$. Therefore, $\widetilde{W}$ agrees with the constant concordance $[0,1]\times W$ when we restrict the background space to $(-\infty,\beta)\times (-1,1)^{N-1}$. 
\end{proof}

In Proposition \ref{longman.k} below, 
we shall give a modified version of Proposition \ref{makereg} 
which we will use in the next section. 
For the statement of this proposition and its proof, we will continue to use the following notation: If $W \subseteq P \times \mathbb{R}^k \times (-1,1)^{N-k}$ is an element of the set $\psi_d(N,k)(P)$, then $x_k : W \rightarrow \mathbb{R}^k$ shall denote the standard projection from $W$ to the second factor of $P \times \mathbb{R}^k \times (-1,1)^{N-k}$. Also, we will denote by 
 $x_1 : W \rightarrow \mathbb{R}$ the projection from $W$ onto the first coordinate of the factor $\mathbb{R}^k$. 

\theoremstyle{plain}  \newtheorem{longman.k}[longman]{Proposition}

\begin{longman.k} \label{longman.k}
Suppose that $N - d \geq 3$ and $1 < k \leq N$.
Fix a closed PL manifold $M$ and an element $W \in \psi_d(N,k)(M)$.  
Moreover, let $i_0, i_1: M \hookrightarrow [0,1]\times M$ be the inclusions defined by 
$\lambda \mapsto (0,\lambda)$
and $\lambda \mapsto (1,\lambda)$ respectively. Then,
given any value $\beta \in \mathbb{R}$, we can find a concordance 
$\widetilde{W} \in \psi_d(N,k)([0,1] \times M)$ with the following properties: 
\begin{itemize}

\item[(i)] $i^*_0\widetilde{W} = W$.

\item[(ii)] $\widetilde{W}$ agrees with the constant concordance $[0,1] \times W$ when we restrict the background space to $(-\infty, \beta) \times \mathbb{R}^{N-1}$. 

\item[(iii)]  For the element $W' = i_1^*\widetilde{W}$, there exists a finite open cover $\mathcal{U}$ of $M$ such that, for each open set $U \in \mathcal{U}$, the map $x_k: W'_{U} \rightarrow \mathbb{R}^k$ has a fiberwise regular value $a = (a_1,\ldots, a_k)\in \mathbb{R}^k$ such that $a_1 > \beta$. 
\end{itemize}

\end{longman.k}

\begin{proof}[Proof sketch]
To prove this result, we just need to adjust the proof of Proposition \ref{makereg} at the right places.
 Let us denote the dimension of $M$ by $m$. 
Next, consider an arbitrary point $\lambda \in M$ and fix a value $i \in \{0, \ldots, m\}$. 
If we perform the proof of Lemma \ref{makereg1} with the projection   $x_k : W \rightarrow \mathbb{R}^k$ instead of $x_1 : W \rightarrow \mathbb{R}$, 
we can find an open neighborhood  $V_i\subseteq M$ of $\lambda$ and an ambient isotopy 
\[
F^{\lambda,i}: [0,1]\times V_i \times \mathbb{R}^k\times (-1,1)^{N-k} \longrightarrow [0,1]\times V_i \times \mathbb{R}^k\times (-1,1)^{N-k} 
\]
of $\mathbb{R}^k\times (-1,1)^{N-k}$ over $[0,1]\times V_i$ with the following properties:

\begin{itemize}

\item[$\cdot$] $F^{\lambda,i}_0 =\mathrm{Id}_{ V_i \times \mathbb{R}^k\times (-1,1)^{N-k}}$.

\item[$\cdot$] The standard projection $x_k: F^{\lambda,i}_1(W_{V_i}) \rightarrow \mathbb{R}^k$ has a fiberwise regular value 
$a^i = (a^i_1, \ldots, a^i_k)$ with the property that $a^i_1 \in  (\beta + i + \frac{1}{4} , \beta + i + \frac{3}{4})$. 

\item[$\cdot$] $F^{\lambda,i}$ is supported on $B(a^i,\epsilon_i)\times (-1,1)^{N-k}$, where $\epsilon_i > 0$ is a value such that 
$(a^i_1 - \epsilon_i, a^i_1 + \epsilon_i) \subset (\beta + i + \frac{1}{4} , \beta + i + \frac{3}{4})$.

\end{itemize} 

We can obtain such an isotopy $F^{i,\lambda}$ for each $i \in \{0, \ldots, m\}$. From now on, we will assume that all the ambient isotopies $F^{0,\lambda}, \ldots, F^{m,\lambda}$ are defined over the same product 
$[0,1]\times V_{\lambda}$. For example, $V_{\lambda}$ could be the intersection of all the neighborhoods $V_i$.

By the compactness of $M$, we can find finitely many points $\lambda_1, \ldots, \lambda_p \in M$ such that 
the corresponding open neighborhoods $V_{\lambda_1}, \ldots, V_{\lambda_p}$ cover $M$. Next, fix a finite simplicial complex $K$ which triangulates $M$ and is subordinate to the open cover $\{V_{\lambda_1}, \ldots, V_{\lambda_p}\}$.  By Lemma
\ref{coverk}, we can find two collections of open sets $\{U_{\sigma}\}_{\sigma \in K}$ and 
$\{U'_{\sigma}\}_{\sigma \in K}$ in $M$ indexed by $K$ which are subordinate to 
$\{V_{\lambda_1}, \ldots, V_{\lambda_p}\}$ and  satisfy the following properties: 

\begin{itemize}

\item[(i)] $\overline{U_{\sigma}} \subseteq U'_{\sigma}$ for all $\sigma \in K$.  

\item[(ii)] For all $\sigma \in K$, the closures $\overline{U_{\sigma}}$ and $\overline{U'_{\sigma}}$
are PL balls of dimension $m$.

\item[(iii)] $\overline{U'_{\sigma_1}} \cap \overline{U'_{\sigma_2}} = \varnothing$
if $\sigma_1$ and $\sigma_2$ are distinct simplices of $K$ of the same dimension.

\item[(iv)] The collection $\{ U_{\sigma} \}_{\sigma\in K}$ covers $M$, and the collection of PL balls 
$\{\overline{U'_{\sigma}}\}_{\sigma \in K}$ is subordinate to the cover
$\{V_{\lambda_1}, \ldots, V_{\lambda_p}\}$.
 
\end{itemize}

For each $\sigma \in K$, let us fix an open set of the collection $\{V_{\lambda_1}, \ldots, V_{\lambda_p}\}$ that contains $\overline{U'_{\sigma}}$. From now on, we will denote this open set containing $\overline{U'_{\sigma}}$ by $V_{\sigma}$. Also, the ambient isotopies over $[0,1]\times V_{\sigma}$ that we obtained at the beginning of this proof will be relabeled as $F^{\sigma,0}, \ldots, F^{\sigma,m}$, and the dimension of a simplex $\sigma \in K$  will be denoted again by $|\sigma|$.  
As we did in the last step of the proof of Proposition \ref{makereg}, we can dampen the ambient isotopy $F^{\sigma, |\sigma|}$ to produce a new ambient isotopy $G^{\sigma,|\sigma|}$ 
of $\mathbb{R}^k\times (-1,1)^{N-k}$
over $[0,1]\times M$ which agrees with $F^{\sigma, |\sigma|}$ over $[0,1]\times \overline{U_{\sigma}}$ and with the identity isotopy over 
$[0,1]\times (M - U'_{\sigma})$. Doing an argument identical to the one we did at the very end of the proof of Proposition \ref{makereg}, we can show that any two isotopies in the collection $\{G^{\sigma, |\sigma|}\}_{\sigma \in K}$ will have disjoint supports. Thus, we can compose all the maps $G^{\sigma, |\sigma|}$ in any order to produce an ambient isotopy $G$ 
of $\mathbb{R}^k\times (-1,1)^{N-k}$ over $[0,1]\times M$ which is the identity at time $t=0$, and agrees with $F^{\sigma,|\sigma|}$ over $[0,1]\times U_{\sigma}$. 
Therefore, the element $\widetilde{W}$ of $\psi_d(N,k)([0,1]\times M)$
obtained by applying $G$ to $[0,1]\times W$ will be a concordance from $W$ to an element $W'$ such that, for each $\sigma \in K$, 
the projection $x_k: W'_{U_{\sigma}}\rightarrow \mathbb{R}^k$ has a fiberwise regular value $a \in \mathbb{R}^k$ whose first coordinate is greater than $\beta$. Finally, since the support of each isotopy 
$F^{\sigma, |\sigma|}$ is contained in 
$(\beta,\infty)\times \mathbb{R}^{N-1}$, we also have that $\widetilde{W}$ agrees with the constant concordance 
$[0,1]\times W$ when we restrict the background space to $(-\infty,\beta)\times \mathbb{R}^{N-1}$. 
\end{proof}

\section{The equivalence $|\psi_d(N,1)_{\bullet}| \stackrel{\simeq}{\longrightarrow} \Omega^{N-1}|\Psi_d(\mathbb{R}^N)_{\bullet}|$}  \label{section6}  

\subsection{The scanning map}     \label{section60}

To complete the proof of the main theorem, it remains to show that 
 $|\psi_d(N,1)_{\bullet}|$ and $\Omega^{N-1}|\Psi_d(\mathbb{R}^N)_{\bullet}|$ have the same weak homotopy type. Recall that the canonical base-point of $|\Psi_d(\mathbb{R}^N)_{\bullet}|$ is the geometric realization of the subsimplicial set $\varnothing_{\bullet}$ consisting of all empty simplices in $\Psi_d(\mathbb{R}^N)_{\bullet}$  (see Remark \ref{spacerem}). This will also be the canonical base-point for any space of the form $|\psi_d(N,k)_{\bullet}|$. For the space 
 $\Omega^{N-1}|\Psi_d(\mathbb{R}^N)_{\bullet}|$, the preferred base-point will be the point corresponding to the constant map 
 $S^{N-1} \rightarrow |\Psi_d(\mathbb{R}^N)_{\bullet}|$ which sends all points in $S^{N-1}$ to $|\varnothing_{\bullet}|$. 
 
The first thing we will do in this section is define the map 
 $|\psi_d(N,1)_{\bullet}| \rightarrow \Omega^{N-1}|\Psi_d(\mathbb{R}^N)_{\bullet}|$ that we will use to compare the spaces  $|\psi_d(N,1)_{\bullet}|$ and $\Omega^{N-1}|\Psi_d(\mathbb{R}^N)_{\bullet}|$. First, for any integer $1 \leq k \leq N -1$, we will define a map of the form
 \begin{equation} \label{scanning.map}
 \mathcal{S}_k: |\psi_d(N,k)_{\bullet}| \longrightarrow \Omega|\psi_d(N,k+1)_{\bullet}|.
 \end{equation}
 This map 
$\mathcal{S}_k$, which we will refer to as \textit{the scanning map},
will be the adjoint of a map 
$\mathcal{E}: S^1 \wedge |\psi_d(N,k)_{\bullet}| \rightarrow |\psi_d(N,k+1)_{\bullet}|$
whose construction is almost identical to that of the structure maps $\mathcal{E}_N$
of the spectrum $\Psi^{\mathrm{PL}}_d$.  However, instead of pushing manifolds in an entirely new direction (as we did for the 
spectrum $\Psi^{\mathrm{PL}}_d$), we will define the map $\mathcal{E}: S^1 \wedge |\psi_d(N,k)_{\bullet}| \rightarrow |\psi_d(N,k+1)_{\bullet}|$ by pushing manifolds in 
the new open direction we have available in $\psi_d(N,k+1)_{\bullet}$, i.e., 
the one corresponding to the coordinate $x_{k+1}$. 

As was the case with the structure maps $\mathcal{E}_N$ of $\Psi^{\mathrm{PL}}_d$,  the construction of the map 
$\mathcal{E}: S^1 \wedge |\psi_d(N,k)_{\bullet}| \rightarrow |\psi_d(N,k+1)_{\bullet}|$
requires fixing an increasing PL homeomorphism $f: [0,1) \rightarrow [0,\infty)$. Next, for each $p \geq 0$, 
we define a PL embedding
\[
F_p: [0,1)\times\Delta^p\times \mathbb{R}^N \rightarrow [0,1]\times\Delta^p\times \mathbb{R}^{N}  
\]
by setting
\[ 
F_p(t,\lambda, x_1, \ldots,x_k,x_{k+1}, \ldots, x_{N}) = (t,\lambda,x_1, \ldots, x_k, x_{k+1} + f(t),\ldots, x_{N}). 
\]
We can use these PL embeddings $F_p$ to push  a $p$-simplex $W$ of $\psi_d(N,k)_{\bullet}$ to $\infty$ along the new open direction $x_{k+1}$. Then, following the same procedure we used in the  definition of the structure maps of $\Psi_d^{\mathrm{PL}}$ in \S \ref{section2.5}, we can use the embeddings $F^p$ 
to construct a map $\mathcal{T}: [-1,1]\times |\psi_d(N,k)_{\bullet}| \rightarrow  |\psi_d(N,k+1)_{\bullet}|$ which sends $[-1,1]\times |\varnothing_{\bullet}|$
and $\{-1,1\}\times |\psi_d(N,k)_{\bullet}|$ to the base-point $|\varnothing_{\bullet}|$
of $|\psi_d(N,k+1)_{\bullet}|$. Consequently, $\mathcal{T}$ induces  a map of the form 
$\mathcal{E}: S^1 \wedge |\psi_d(N,k)_{\bullet}| \rightarrow |\psi_d(N,k+1)_{\bullet}|$.
As mentioned earlier, we define the scanning map as follows. 

\theoremstyle{definition} \newtheorem{scan}{Definition}[section] 

\begin{scan} \label{scan}
The \textit{scanning map} 
\[
\mathcal{S}_k: |\psi_d(N,k)_{\bullet}| \longrightarrow \Omega|\psi_d(N,k+1)_{\bullet}|
\]
is the adjoint of the map
$\mathcal{E}: S^1 \wedge |\psi_d(N,k)_{\bullet}| \rightarrow |\psi_d(N,k+1)_{\bullet}|$  defined above.

\end{scan}

The map 
$\widetilde{\mathcal{S}}: |\psi_d(N,1)_{\bullet}| \rightarrow \Omega^{N-1}|\Psi_d(\mathbb{R}^N)_{\bullet}|$
that we will use to compare the spaces $|\psi_d(N,1)_{\bullet}|$ and $\Omega^{N-1}|\Psi_d(\mathbb{R}^N)_{\bullet}|$ is defined as the following composition:
\begin{equation}  \label{prescaneq}
\widetilde{\mathcal{S}} = \Omega^{N-2}\mathcal{S}_{N-1} \circ \ldots \circ \mathcal{S}_1.
\end{equation}

To finish the proof of the main theorem, we must show that $\widetilde{S}$ is a weak homotopy equivalence. This will be a consequence of the following result. 

\theoremstyle{plain}  \newtheorem{scanweak}[scan]{Theorem}

\begin{scanweak} \label{scanweak}
The scanning map 
\[
\mathcal{S}_k: |\psi_d(N,k)_{\bullet}| \longrightarrow \Omega |\psi_d(N,k+1)_{\bullet}|
\]
is a weak homotopy equivalence if $N-d\geq 3$ and $1 \leq k \leq N-1$. 
\end{scanweak}     

The proof of Theorem \ref{scanweak} will occupy most of the rest of this section. Again, we include the condition $N-d\geq 3$  because we will use the Isotopy Extension Theorem at some point in our proof.

\theoremstyle{definition}  \newtheorem*{notadd5}{Note}

\begin{notadd5}

For the proof of Theorem \ref{scanweak}, we shall assume
that the underlying PL space $W$ of a $p$-simplex in $\psi_d(N,k)_{\bullet}$
is contained in $\Delta^p\times \mathbb{R}^k\times (0,1)^{N-k}$.
\end{notadd5}

\subsection{Decomposition of the scanning map}     \label{section61}  

To prove Theorem \ref{scanweak}, we are going to express 
the scanning map $\mathcal{S}_k$ as the composition of three maps (up to homotopy), and then show that each of these is a weak equivalence. To describe this decomposition of the scanning map, we shall need the following definitions.

\theoremstyle{definition} \newtheorem{emptylevel}[scan]{Definition}

\begin{emptylevel}  \label{emptylevel}

Let $1 \leq k \leq N-1$.

\begin{enumerate} 
\item $\psi_d^{\varnothing}(N,k)_{\bullet}$
will denote the path component of the vertex $\varnothing$
in $\psi_d(N,k)_{\bullet}$.

\item $\psi_d^0(N,k)_{\bullet}$ will denote the subsimplicial set
of $\psi_d(N,k)_{\bullet}$ whose set of $p$-simplices
consists of those elements $W \in \psi_d(N,k)_{p}$
for which there is a piecewise linear function
$f: \Delta^p \rightarrow \mathbb{R}$ such that
\begin{equation} \label{empty.function}
W_{\lambda} \cap \big( \mathbb{R}^{k-1}\times \{f(\lambda)\} \times (0,1)^{N-k}\big) = \varnothing
\end{equation}
for all $\lambda \in \Delta^p$.
\end{enumerate}

\end{emptylevel}     

It is clear how one can extend the simplicial set $\psi_d^0(N,k)_{\bullet}$ 
to a functor of the form 
$\psi_d^0(N,k): \mathbf{PL}^{op} \rightarrow \mathbf{Sets}$. 
For the simplicial set $\psi_d^{\varnothing}(N,k)_{\bullet}$, 
we will explain how to define the corresponding extension $\psi_d^{\varnothing}(N,k):\mathbf{PL}^{op} \rightarrow \mathbf{Sets}$ in Definition \ref{pl.psi.empty}. 

\theoremstyle{definition} \newtheorem{emptylevel.remark}[scan]{Remark}

\begin{emptylevel.remark}  \label{emptylevel.remark}
In this remark, we will show that any 0-simplex
of $\psi_d^0(N,k)_{\bullet}$ is concordant to the empty $0$-simplex $\varnothing$ of $\psi_d(N,k)_{\bullet}$. Indeed, pick a $W$ in $\psi_d^0(N,k)_0$ and fix a value $c \in \mathbb{R}$ such that the intersection of $W$ and $\mathbb{R}^{k-1}\times \{c\} \times (0,1)^{N-k}$ is empty. By translating $W$ if necessary, we can assume that $c = 0$. Since $W$ is closed as a topological subspace of 
$\mathbb{R}^N$, we can find a value $\epsilon > 0$ so that  $W$ is disjoint from $(- \epsilon, \epsilon)^k \times (0,1)^{N-k}$. Thus, by stretching each interval $(-\epsilon, \epsilon)$ to $(-\infty, \infty)$, we will produce a concordance from $W$ to the empty manifold
 $\varnothing$. In particular, this shows that 
$\psi_d^0(N,k)_{\bullet}$ is a subsimplicial set of $\psi_d^{\varnothing}(N,k)_{\bullet}$. 

\end{emptylevel.remark}

\theoremstyle{definition} \newtheorem{emptylevel.remark2}[scan]{Remark}

\begin{emptylevel.remark2}  \label{emptylevel.remark2}
It is not hard to verify that the simplicial set $\psi_d^{\varnothing}(N,k)_{\bullet}$  is Kan. 
On the other hand, 
we claim that $\psi_d^0(N,k)_{\bullet}$ is not Kan. To show this, we will produce a map 
$g: \Lambda^2_{0\bullet} \rightarrow \psi_d^0(N,k)_{\bullet}$ of simplicial sets which does not admit an extension 
of the form $\Delta^2_{\bullet} \rightarrow \psi_d^0(N,k)_{\bullet}$. Before we start constructing the map 
$g: \Lambda^2_{0\bullet} \rightarrow \psi_d^0(N,k)_{\bullet}$, 
let us recall some notational conventions that we set in \S \ref{section2}:
\begin{itemize}
\item[$\cdot$] For any integer $p\geq 0$, 
we shall continue to denote the elements of the standard basis of $\mathbb{R}^{p + 1}$ by
$e_0, e_1, \ldots, e_p$, and we declare the standard $p$-simplex $\Delta^p$ to be the convex hull of
the set  $\{ e_0, e_1, \ldots, e_p \}$. 

\item[$\cdot$] For any set $\{v_0, \ldots,v_m\} \subset \mathbb{R}^{p+1}$, we will denote its convex hull
by  $\langle v_0, \ldots, v_m  \rangle$.

\end{itemize} 
In this discussion, we will only use these conventions in the case when $p =2$.
Now, fix a non-empty element $W \in \psi_d(N,k-1)_0$ (recall that the definition of
the simplicial set
$\psi_d^0(N,k)_{\bullet}$ requires $k \geq 1$), and an increasing  
PL homeomorphism $f:[0,1) \rightarrow [0,\infty)$. 
As we did in the construction of the scanning map
in  \S \ref{section60}, we can use the map $f:[0,1) \rightarrow [0,\infty)$ to produce a concordance 
$\widetilde{W}^+ \in \psi_d(N,k)([0,1])$ from $W$ to $\varnothing$. More concretely, 
$\widetilde{W}^+$ is the trace that we obtain by pushing $W$ to $+\infty$ along the direction $x_k$.  
In a similar fashion, we can obtain a concordance 
$\widetilde{W}^- \in \psi_d(N,k)([0,1])$ from $W$ to $\varnothing$ by pushing
$W$ to $-\infty$ along the direction $x_k$.  
Note that both $\widetilde{W}^+$ and $\widetilde{W}^-$ are elements of 
$\psi_d^{0}(N,k)([0,1])$. Indeed, if $f_0, \hspace{0.1cm} f_1: [0,1] \rightarrow \mathbb{R}$ are the constant functions defined by $f_0(\lambda) = 0$ and $f_1(\lambda) = 1$ for all $\lambda \in [0,1]$, then 
we evidently have that 
$\widetilde{W}^+_{\lambda} \cap \big( \mathbb{R}^{k-1}\times \{f_0(\lambda)\} \times (0,1)^{N-k}\big) = 
\varnothing$ and
$\widetilde{W}^-_{\lambda} \cap \big( \mathbb{R}^{k-1}\times \{f_1(\lambda)\} \times (0,1)^{N-k}\big) = \varnothing$ for any point $\lambda \in [0,1]$.    
Moreover, if $\{e_0, e_1, e_2\}$ is the standard basis of $\mathbb{R}^3$, we can
assume that $\widetilde{W}^+ \in \psi_d^{0}(N,k)(\langle e_0, e_1 \rangle)$ and 
$\widetilde{W}^- \in \psi_d^{0}(N,k)(\langle e_0, e_2 \rangle)$ by pulling back
$\widetilde{W}^+$ along the linear map $h_+: \langle e_0, e_1 \rangle \rightarrow [0,1]$
determined by $e_0 \mapsto 0$, $e_1 \mapsto 1$, and by pulling back
$\widetilde{W}^-$ along the linear map $h_-: \langle e_0, e_2 \rangle \rightarrow [0,1]$
given by $e_0 \mapsto 0$, $e_2 \mapsto 1$. 
By the way we defined the maps $h_+$ and $h_-$, we have
that the fibers $\widetilde{W}^+_{e_0}$ and $\widetilde{W}^-_{e_0}$ 
are both equal to $W$. Thus, since $\psi_d(N,k)$ is a quasi-PL space (see Remark \ref{moreclasssub}), we can glue 
$\widetilde{W}^+$ and $\widetilde{W}^-$ along $W$ to produce an element 
$\widetilde{W} \in \psi_d(N,k)(\Lambda^2_0)$, where $\Lambda^2_0$ is the 0th horn
of the standard simplex $\Delta^2$. By Theorem \ref{classsub}, this element 
$\widetilde{W}$ induces a simplicial set map 
$g: \Lambda^2_{0\bullet} \rightarrow \psi_d(N,k)_{\bullet}$. In fact, since 
$\widetilde{W}^+$ and $\widetilde{W}^-$ are elements of 
$\psi_d^{0}(N,k)(\langle e_0, e_1 \rangle)$ and $\psi_d^{0}(N,k)(\langle e_0, e_2 \rangle)$ respectively,
$g$ maps $\Lambda^2_{0\bullet}$ to $\psi_d^{0}(N,k)_{\bullet}$. Thus, we can view $g$ as a map of the form 
$g: \Lambda^2_{0\bullet} \rightarrow \psi_d^0(N,k)_{\bullet}$. We now claim that this map 
$g: \Lambda^2_{0\bullet} \rightarrow \psi_d^0(N,k)_{\bullet}$ does not admit a lift to  $\Delta^2_{\bullet}$. 
If such a lift for $g: \Lambda^2_{0\bullet} \rightarrow \psi_d^0(N,k)_{\bullet}$ did exist, then
it would be possible to find an element $\widehat{W} \in \psi_d(N,k)(\Delta^2)$  
such that $\widehat{W}_{\Lambda^2_0} = \widetilde{W}$ and for which there exists a PL function
$\tilde{f}: \Delta^2 \rightarrow \mathbb{R}$ with the property that
\begin{equation} \label{non.kan.cond}
\widehat{W}_{\lambda} \cap \big( \mathbb{R}^{k-1}\times \{\tilde{f}(\lambda)\} \times (0,1)^{N-k}\big) = 
\varnothing
\end{equation}
for all points $\lambda \in \Delta^2$. In particular, the pair $(\widetilde{W}, \tilde{f}|_{\Lambda^2_0})$ would satisfy
(\ref{non.kan.cond}) for all $\lambda \in \Lambda^2_0$. However, by the way we constructed $\widetilde{W}$, 
it is impossible to find a PL function 
$f: \Lambda^2_0 \rightarrow \mathbb{R}$ satisfying the condition  
$\widetilde{W}_{\lambda} \cap \big( \mathbb{R}^{k-1}\times \{f(\lambda)\} \times (0,1)^{N-k}\big) = 
\varnothing$ globally over all $\Lambda^2_0$. 
Thus, our map $g: \Lambda^2_{0\bullet} \rightarrow \psi_d^0(N,k)_{\bullet}$
does not have a lift to $\Delta^2_{\bullet}$, and we can therefore conclude that 
$\psi_d^0(N,k)_{\bullet}$ is not Kan.
\end{emptylevel.remark2}

We will also need the following bisimplicial set for the proof of Theorem \ref{scanweak}.

\theoremstyle{definition} \newtheorem{monoid}[scan]{Definition}

\begin{monoid} \label{monoid}
$N\psi_d(N,k)_{\bullet,\bullet}$ is the bisimplicial set 
whose set of $(p,q)$-simplices consists of 
all $(q+2)$-tuples $(W, f_0,\ldots, f_q)$, where $W$ is
a $p$-simplex of $\psi_d(N,k+1)_{\bullet}$ and
the $f_i$'s  are piecewise linear functions
$\Delta^p \rightarrow \mathbb{R}$ satisfying the following properties:

\begin{itemize}

\item[(i)] For any two consecutive functions $f_j$ and $f_{j+1}$ in $(f_0, \ldots, f_q)$, we have either 
$f_j(\lambda) < f_{j+1}(\lambda)$ for all $\lambda \in \Delta^p$, or 
$f_j(\lambda) = f_{j+1}(\lambda)$ for all $\lambda \in \Delta^p$.

\item[(ii)] For each function $f_j$ in $(f_0, \ldots, f_q)$, we have that 
\begin{equation} \label{monprop2}
W_{\lambda} \cap \big( \mathbb{R}^{k}\times \{f_j(\lambda)\} \times (0,1)^{N-k-1}\big)  = \varnothing
\end{equation} 
for all fibers $W_{\lambda}$ of the projection $\pi: W \rightarrow \Delta^p$.

\end{itemize}

The structure map $\big(\eta\times\delta\big)^*$ induced by a morphism 
$\eta\times\delta:[p']\times[q'] \rightarrow [p]\times [q] $
in the category $\Delta\times\Delta$ is defined by
\[ 
\big(\eta\times\delta\big)^*\big(W, \hspace{0.05cm} f_0, \ldots, f_q\big) = \big( \eta^*W, \hspace{0.05cm} f_{\delta(0)}\circ \tilde{\eta}, \ldots,f_{\delta(q')}\circ \tilde{\eta} \big),
\]
where $\tilde{\eta}:\Delta^{p'} \rightarrow \Delta^p$ is the linear map induced by $\eta$.

\end{monoid}

\theoremstyle{definition} \newtheorem{monoid.remark0}[scan]{Remark}

\begin{monoid.remark0} \label{monoid.remark0}

The notation $N\psi_d(N,k)_{\bullet,\bullet}$  suggests that we can regard this bisimplicial set as the nerve of a monoid. Even though we have not defined a monoid structure on $\psi_d(N,k)_{\bullet}$, 
we can think of a $(p,q)$-simplex $(W, f_0, \ldots, f_q)$ of $N\psi_d(N,k)_{\bullet,\bullet}$
as the result of multiplying $q$ elements of $\psi_d(N,k)_p$  plus some manifolds that we can push to infinity 
(namely, the parts of $W$ below $f_0$ and above $f_q$). Note that, despite the notation 
$N\psi_d(N,k)_{\bullet,\bullet}$, the first component of a $(p,q)$-simplex $(W, f_0, \ldots, f_q)$ is a $p$-simplex of
$\psi_d(N,k+1)_{\bullet}$, not $\psi_d(N,k)_{\bullet}$.

\end{monoid.remark0}

\theoremstyle{definition} \newtheorem{monoid.remark}[scan]{Remark}

\begin{monoid.remark} \label{monoid.remark}

Observe also that there is a forgetful map 
\begin{equation} \label{forgetful.map.scan}
F: \left\|N\psi_d(N,k)_{\bullet,\bullet}\right\| \rightarrow |\psi^{0}_d(N,k+1)_{\bullet}|
\end{equation}
obtained by forgetting all the functions $f_j$. Now, for each $W \in \psi^{0}_d(N,k+1)_p$, define $\mathcal{F}_W$ to be the poset of all functions $f:\Delta^p \rightarrow \mathbb{R}$ that satisfy
\[
W_{\lambda} \cap \big( \mathbb{R}^{k}\times \{f(\lambda)\} \times (0,1)^{N-k -1}\big) = \varnothing
\]
for all $\lambda$ in $\Delta^p$. For any $p$-simplex $W$ of $\psi_d^0(N,k)_{\bullet}$, the poset $\mathcal{F}_W$ has a contractible classifying space. Thus, by doing an argument almost identical to the one we did
in the proof of
Proposition \ref{DtoPsi}, we can show that the forgetful map $F$ is a weak equivalence. 

\end{monoid.remark}

We claim that the scanning map $\mathcal{S}_k$ is homotopic to a composition of the following form:
\begin{equation} \label{descan}
|\psi_d(N,k)_{\bullet}| \rightarrow \Omega\left\|N\psi_d(N,k)_{\bullet,\bullet}\right\| \rightarrow \Omega |\psi_d^{0}(N,k+1)_{\bullet}|
\rightarrow \Omega|\psi_d^{\varnothing}(N,k+1)_{\bullet}|.
\end{equation}
The second and third maps are induced respectively by the forgetful map (\ref{forgetful.map.scan}) and the inclusion $\psi_d^{0}(N,k+1)_{\bullet} \hookrightarrow 
\psi_d^{\varnothing}(N,k+1)_{\bullet}$. By Remark \ref{monoid.remark}, the second map in (\ref{descan}) is a weak equivalence. We will define the leftmost map of (\ref{descan}) in \S \ref{section62}. The rest of this section will be structured as follows: In \S \ref{section62}, we will construct the map $|\psi_d(N,k)_{\bullet}| \rightarrow \Omega\left\|N\psi_d(N,k)_{\bullet,\bullet}\right\|$ and prove that it is a weak equivalence. Also, in \S \ref{section62}, we will show that the composition 
(\ref{descan}) is homotopic to the scanning map $\mathcal{S}_k$.
Next, in 
\S  \ref{section63}, we prove that the inclusion $\psi_d^{0}(N,k+1)_{\bullet} \hookrightarrow 
\psi_d^{\varnothing}(N,k+1)_{\bullet}$ is also a weak homotopy equivalence, which would then conclude the proof of Theorem \ref{scanweak}. Finally, in \S \ref{section64}, we will give the final details of the proof of the article's main theorem. 

\subsection{The group completion argument}    \label{section62} 

In Proposition \ref{propmonoid} below, we collect several
properties of the bisimplicial set $N\psi_d(N,k)_{\bullet,\bullet}$
that we will need for the proof of Theorem \ref{scanweak}.
We shall use the
following notation
in the statement of Proposition \ref{propmonoid} and its proof:

\begin{itemize}

\item[$\cdot$] If $W$ is a $p$-simplex of 
$\psi_d(N,k+1)_{\bullet}$ and $f: \Delta^p \rightarrow \mathbb{R}$
is a piecewise linear function, we will denote by 
$W + f$
the image of $W$ under the piecewise  linear
automorphism of  
$\Delta^p \times \mathbb{R}^N$ defined by
\[
(\lambda,x_1,\ldots,x_{k+1},\ldots,x_N) \mapsto (\lambda,x_1, \ldots,x_{k+1} + f(\lambda), \ldots, x_N).
\]
It is clear that $W+f$ is also a $p$-simplex 
of $\psi_d(N,k+1)_{\bullet}$.

\item[$\cdot$]  For any value $a \in \mathbb{R}$, we will denote by $c_{a}$ the constant map $\Delta^p \rightarrow \mathbb{R}$
which sends all points $\lambda \in \Delta^p$ to $a$. 

\item[$\cdot$] $d_1:  N\psi_d(N,k)_{\bullet,2} \rightarrow N\psi_d(N,k)_{\bullet,1}$ will be the map of simplicial sets which sends a $p$-simplex
$(W, f_0, f_1, f_2)$ of $N\psi_d(N,k)_{\bullet,2}$ to $(W, f_0, f_2).$ We will use this map in part (iv)
of Proposition \ref{propmonoid}. 
\end{itemize}

\theoremstyle{plain} \newtheorem{propmonoid}[scan]{Proposition}

\begin{propmonoid} \label{propmonoid}

\begin{itemize}
\item[(i)] For any integer $q>0$, the inclusion of simplicial sets
\[
\eta^q:\underbrace{\psi_d(N,k)_{\bullet}\times \ldots \times \psi_d(N,k)_{\bullet}}_{q} \longrightarrow N\psi_d(N,k)_{\bullet,q}
\]
which sends a $q$-tuple $(W_1, \ldots, W_q)$ 
to $\big( \coprod_{j=1}^{q}( W_j + c_{j-1}) , c_0, \ldots, c_{q} \big)$
is a weak homotopy equivalence. 

\item[(ii)] If $\varnothing_{\bullet}$ is the subsimplicial set
of $N\psi_d(N,k)_{\bullet,0}$ consisting of all degeneracies of the 0-simplex $(\varnothing, c_0)$,
then the inclusion 
$\varnothing_{\bullet} \hookrightarrow N\psi_d(N,k)_{\bullet,0}$
is a weak homotopy equivalence. In particular,
$|N\psi_d(N,k)_{\bullet,0}|$ is contractible.

\item[(iii)] If $\beta_{j}: [1] \rightarrow [q]$ 
is the morphism in $\Delta$ defined by 
$\beta_j(0)= j-1$ and $\beta_j(1)=j$, 
then the morphism
\[
\mathcal{B}^q :=(\beta_1^* , \ldots , \beta_q ^*): N\psi_d(N,k)_{\bullet,q} \rightarrow \underbrace{N\psi_d(N,k)_{\bullet,1}\times \ldots \times N\psi_d(N,k)_{\bullet,1}}_{q}
\]
is a weak homotopy equivalence.  

\item[(iv)]  
If $\bullet: \pi_0 (N\psi_d(N,k)_{\bullet,1}) \times  \pi_0 (N\psi_d(N,k)_{\bullet,1}) \rightarrow  \pi_0 (N\psi_d(N,k)_{\bullet,1})$ is the product
obtained by composing the bijection
\[
\pi_0( N\psi_d(N,k)_{\bullet,1}) \times \pi_0( N\psi_d(N,k)_{\bullet,1}  ) \stackrel{\cong}{\longrightarrow} 
\pi_0(N\psi_d(N,k)_{\bullet,2})
\] 
induced by $\mathcal{B}^2$ (see part (iii) of this proposition) 
and
the function between 
path components induced by the map $d_1:  N\psi_d(N,k)_{\bullet,2} \rightarrow N\psi_d(N,k)_{\bullet,1}$,
then the tuple $\big(\pi_0 ( N\psi_d(N,k)_{\bullet,1}), \bullet\big)$
is a group.

\end{itemize}

\end{propmonoid}

\begin{proof}

In order to make this proof easier to 
follow, we are going to switch the roles of the coordinates 
$x_1$ and $x_{k+1}$. Thus, if $(W, f_0, \ldots, f_q)$ is a $(p,q)$-simplex of $N\psi_d(N,k)_{\bullet,\bullet}$, then
$W_{\lambda}\cap\big( \{f_j(\lambda)\} \times \mathbb{R}^k \times (0,1)^{N-k-1}\big) = \varnothing$ 
for all $j$ in $\{ 0,\ldots, q\}$ and  all $\lambda$ in $\Delta^p$. 

To prove (i), we start by noting that the simplicial sets $N\psi_d(N,k)_{\bullet,q}$ and 
$\mathrm{Im}\hspace{0.06cm}\eta^q$ are both Kan.
Thus, to show that the inclusion $\eta^q$ is a weak homotopy equivalence, 
it suffices to show that the geometric realization of any morphism of pairs 
$g: (\Delta^p_{\bullet}, \partial\Delta^p_{\bullet}) \rightarrow (N\psi_d(N,k)_{\bullet,q}, \hspace{0.06cm} \mathrm{Im}\hspace{0.06cm}\eta^q)$ represents the trivial class in $\pi_p\big( |N\psi_d(N,k)_{\bullet,q}|, |\mathrm{Im}\hspace{0.06cm}\eta^q| \big)$. 
Fix then such a morphism $g$, and let $(W, f_0, \ldots, f_q)$ be the $p$-simplex of $N\psi_d(N,k)_{\bullet,q}$
classified by $g$.
Since $g(\partial\Delta^p_{\bullet}) \subseteq \mathrm{Im}\hspace{0.06cm}\eta^q$, 
for each $j$ in 
$\{0,\ldots,q\}$ we must have  that $f_j$ agrees with the constant map $c_j$ on $\partial \Delta^p$. In particular, by the way we defined the bisimplicial set $N\psi_d(N,k)_{\bullet,\bullet}$, 
we have strict inequalities $f_0(\lambda) < f_1(\lambda) < \ldots < f_q(\lambda)$ for all $\lambda$ in $\Delta^p$. 
Our first step is to deform the maps $f_0, \ldots, f_q$ into $c_0, \ldots, c_q$ respectively. 
We will do this inductively.  In other words, we first deform (relative to $\partial \Delta^p$) 
$f_0$ into $c_0$, then $f_1$ into $c_1$, and so on. To transform $f_0$ into $c_0$, 
the reader might be tempted to use the linear homotopy
$L_t(\lambda) = (1 - t)\cdot f_0(\lambda) + t\cdot c_0(\lambda)$. 
The problem with this construction is that (despite the name) linear homotopies are in general not piecewise linear. 
However, we can fix this issue by applying the Simplicial Approximation Theorem. 
This result allows us to deform (relative to $\partial \Delta^p$) the linear homotopy $L_t$ into a piecewise linear one. Therefore, by inductively applying the Simplicial Approximation Theorem, 
we can construct $q+1$ PL maps 
$H^0,\ldots, H^q$ from $[0,1]\times \Delta^p$ to $\mathbb{R}$ satisfying the following properties: 

\begin{itemize}

\item[$\cdot$] For each $j$ in $\{ 0, \ldots, q\}$, the PL map $H^j: [0,1]\times \Delta^p \rightarrow \mathbb{R}$ is a homotopy from 
$f_j$ to the constant map $c_j$. 

\item[$\cdot$] For any $j \in \{ 0, \ldots, q \}$ and any $t \in [0,1]$, the restriction  $H^j_t|_{\partial \Delta^p}$ is equal to $c_j|_{\partial\Delta^p}$. 

\item[$\cdot$] Furthermore, for any $j$ in $\{0, \ldots, q -1\}$, we can guarantee that $H^{j}_t(\lambda) < H^{j+1}_t(\lambda)$ for all $t \in [0,1]$ and $\lambda \in \Delta^p$. 

\end{itemize}

Consider now the map 
$H:[0,1]\times\Delta^p \times \{ 0, \ldots, q \} \rightarrow [0,1]\times\Delta^p \times \mathbb{R}$ 
defined by $H(t, \lambda,j) =(t, \lambda, H^j_t(\lambda))$. 
In the domain $[0,1]\times\Delta^p \times \{ 0, \ldots, q \}$ of $H$, the reader should regard 
$\{ 0, \ldots, q \}$ as a 0-dimensional submanifold of $\mathbb{R}$. 
The last property listed above ensures
that $H$ is a PL isotopy of embeddings of $\{0,\ldots,q\}$ in $\mathbb{R}$ over $[0,1] \times\Delta^p$
(which we can view as a $(p+1)$-isotopy once we identify $[0,1]\times\Delta^p$ with $I^{p+1}$). In fact,
it is not hard to prove that $H$ is actually a locally trivial isotopy
(see the discussion given before Theorem \ref{loc.trivial.iso}). Therefore, by the version 
of the Isotopy Extension Theorem given in Theorem \ref{loc.trivial.iso},  
we can find an ambient isotopy $\widehat{H}: [0,1]\times\Delta^p \times  \mathbb{R} \rightarrow [0,1]\times\Delta^p\times \mathbb{R}$ of $\mathbb{R}$ over $[0,1]\times \Delta^p$ which extends $H$. That is, for any point $(t,\lambda) \in [0,1]\times \Delta^p$ and any $j \in \{0,\ldots,q\}$, we have that 
$\widehat{H}_{(t,\lambda)}(j) = H^j_t(\lambda)$.  
Moreover, for any point  $(t,\lambda)$ in the union 
of  $[0,1]\times \partial \Delta^p$ and $\{1\}\times \Delta^p$,
the map $H_{(t, \lambda)}$ is just the natural inclusion $\{0,\ldots,q\} \hookrightarrow \mathbb{R}$ that maps each point in $\{0,\ldots,q\}$ to itself. In particular, 
the identity ambient isotopy of $\mathbb{R}$ over 
$([0,1]\times \partial \Delta^p)\cup(\{1\}\times \Delta^p)$ extends the isotopy $H$
over the base-space $([0,1]\times \partial \Delta^p)\cup(\{1\}\times \Delta^p)$.
Therefore,  
since there is a PL retraction from $[0,1]\times \Delta^p$ onto the union 
$([0,1]\times \partial \Delta^p)\cup(\{1\}\times \Delta^p)$, we can use an argument identical
to the one given in Remark \ref{isotopy.addendum} to ensure that
$\widehat{H}$ agrees with the identity ambient isotopy over 
$[0,1]\times \partial \Delta^p$ and $\{1\}\times \Delta^p$. 
Finally, consider the PL automorphism 
$\widehat{H}_0: \Delta^p\times \mathbb{R} \rightarrow \Delta^p\times \mathbb{R}$ obtained by restricting 
$\widehat{H}$ on $\{0\}\times \Delta^p \times \mathbb{R}$, and let 
$\widetilde{H}: [0,1]\times \Delta^p \times \mathbb{R} \rightarrow [0,1]\times \Delta^p \times \mathbb{R}$
be the ambient isotopy defined by setting
$\widetilde{H} := \widehat{H} \circ \big( \mathrm{Id}_{[0,1]}\times \widehat{H}_0^{-1}\big)$.
 Note that this new map $\widetilde{H}$ satisfies the following properties:

\begin{itemize}
\item[$\cdot$] $\widetilde{H}_0$ is equal to the identity map $\mathrm{Id}_{\Delta^p\times \mathbb{R}}$. 
\item[$\cdot$] For any point $(t,\lambda)$ in
$[0,1]\times \Delta^p$, we have 
$\widetilde{H}_{(t,\lambda)}(f_j(\lambda)) = H_t^j(\lambda)$.   
\end{itemize}

Now, let $F$ be the PL automorphism from $[0,1]\times \Delta^p \times \mathbb{R}^N$ to itself defined as $F = \widetilde{H}\times \mathrm{Id}_{\mathbb{R}^{N-1}}$, and let $\widetilde{W}$ denote the image of the constant concordance $[0,1]\times W$ under $F$.  
Note that $\widetilde{W}$ is an element of $\psi_d(N,k+1)([0,1]\times \Delta^p)$ with the property that, for each $(t,\lambda) \in [0,1]\times \Delta^p$ and each $j \in \{0, \ldots, q\}$, the fiber $\widetilde{W}_{(t,\lambda)}$ does not intersect the hyperplane $\{H^j_t(\lambda)\}\times \mathbb{R}^{N-1}$. 
Thus, even though we never defined an extension 
$\mathbf{PL}^{op} \rightarrow \mathbf{Sets}$ for the simplicial set $N\psi_d(N,k)_{\bullet, q}$, 
we can view the $(q+2)$-tuple 
$(\widetilde{W}, H^0, \ldots, H^q)$ as a concordance between two elements of $N\psi_d(N,k)_{p,q}$. 
More precisely, 
since $H^j$ is a PL homotopy from $f_j$ to $c_j$, 
$(\widetilde{W}, H^0, \ldots, H^q)$ will act as a concordance from the element $(W, f_0, \ldots, f_q)$ to another element of the form $(W', c_0, \ldots, c_q)$. 
Thus, by applying the construction we carried out in the proof of Proposition \ref{concord.homotopy} to the concordance $(\widetilde{W}, H^0, \ldots, H^q)$, we can produce a homotopy $h: [0,1]\times \Delta^p \rightarrow |N\psi_d(N,k)_{\bullet, q}|$ from $|g|$ to the geometric realization of the map
$g': \Delta^p_{\bullet}\rightarrow N\psi_d(N,k)_{\bullet, q}$ which classifies the $p$-simplex 
$(W', c_0, \ldots, c_q)$.  Moreover, since each of the maps $H^0, \ldots, H^q$ satisfies 
$H^j_t|_{\partial\Delta^p} = c_j|_{\partial\Delta^p}$ for all times $t \in [0,1]$, we can guarantee that 
$h_t(\partial\Delta^p)\subseteq |\mathrm{Im}\hspace{0.06cm}\eta^q|$ for each $t \in [0,1]$. 
Finally, since the fibers of $W'$ do not intersect either 
$\{0\}\times\mathbb{R}^{N-1}$ or $\{q\}\times\mathbb{R}^{N-1}$, we can push to infinity
those parts of $W'$ below $\{0\}\times\mathbb{R}^{N-1}$
and above $\{q\}\times\mathbb{R}^{N-1}$ to produce 
a concordance of the form $(\widetilde{W}', c_0, \ldots, c_q)$ between $(W',c_0, \ldots, c_q)$ and
a $p$-simplex $(W'', c_0, \ldots, c_q)$ which
lies in the image of $\eta^q$. This new concordance will also induce a homotopy $h': [0,1]\times \Delta^p \rightarrow |N\psi_d(N,k)_{\bullet, q}|$ such that $h'_t(\partial\Delta^p) \subseteq |\mathrm{Im}\hspace{0.06cm}\eta^q|$ for all $t \in [0,1]$. Therefore, by concatenating the homotopies $h$ and $h'$, we obtain a homotopy of maps of pairs $(\Delta^p, \partial\Delta^p) \rightarrow (|N\psi_d(N,k)_{\bullet, q}|,  |\mathrm{Im}\hspace{0.06cm}\eta^q|)$ from $|g|$ to a map that represents the trivial class in 
$\pi_p\big( |N\psi_d(N,k)_{\bullet,q}|, |\mathrm{Im}\hspace{0.06cm}\eta^q| \big)$. 
This concludes the proof of part (i). 
Also, observe that the arguments we used in the previous
proof can be used to show that
any map of pairs of the form 
$(\Delta^p_{\bullet}, \partial\Delta^p_{\bullet}) \rightarrow (N\psi_d(N,k)_{\bullet,0}, \varnothing_{\bullet})$ 
represents the trivial class in
$\pi_p\big(|N\psi_d(N,k)_{\bullet,0}|, |\varnothing_{\bullet}| \big)$.
This proves claim (ii). 

Part (iii) is basically a consequence of 
part (i) and of the techniques we used to prove that statement. Indeed, consider the canonical homeomorphisms
\[
P: |\prod_{i=1}^q N\psi_d(N,k)_{\bullet,1}| \longrightarrow   \prod_{i=1}^q |N\psi_d(N,k)_{\bullet,1}|, \quad Q: |\prod_{i=1}^q\psi_d(N,k)_{\bullet}| \longrightarrow \prod_{i=1}^q |\psi_d(N,k)_{\bullet}|.
\]
Applying the techniques we used to prove part (i), we can show that 
$P\circ|\mathcal{B}^q|\circ |\eta^q|$
is homotopic to
$\big(|\eta^1|\times \ldots \times|\eta^1|\big)\circ Q$. Since $\big(|\eta^1|\times \ldots \times|\eta^1|\big)\circ Q$
is a weak homotopy equivalence, then so is 
$P\circ|\mathcal{B}^q|\circ |\eta^q|$. In particular, we must have that
$\mathcal{B}^q$ is also a weak homotopy equivalence.

To prove (iv), we are going to use the following notation: Given any open interval $J \subseteq \mathbb{R}$, we will denote by $\psi^{\hspace{0.03cm}J}_d(N,k+1)_{\bullet}$ the subsimplicial set of $\psi_d(N,k+1)_{\bullet}$  
whose set of $p$-simplices is the subset
of $\psi_d(N,k+1)_p$  consisting of all $W$ that are contained in 
$\Delta^p \times J \times \mathbb{R}^k \times (0,1)^{N-k -1}$. Also, for any $0$-simplex 
$W \in \psi_d(N,k+1)_0$ and any open interval $J \subseteq \mathbb{R}$, we will denote by $W_J$ 
the intersection of $W$ and $J \times \mathbb{R}^k \times (0,1)^{N-k-1}$.
 Note that any $0$-simplex of 
$N\psi_d(N,k)_{\bullet,1}$ is path connected
to a vertex $(W',a,b)$ with  
$W' \in \psi_{d}^{(a,b)}(N,k)_{0}$. Indeed, if $(W, a, b)$ is a $0$-simplex of $N\psi_d(N,k)_{\bullet,1}$, then we can obtain a concordance from 
$(W, a, b)$ to $(W_{(a,b)}, a, b)$ by pushing $W_{(b, \infty)}$ and $W_{(-\infty, a)}$ to $+\infty$ and $- \infty$
respectively along the $x_1$-direction. 
Thus, throughout the remainder of this proof, whenever we take a $0$-simplex $(W,a,b)$ of 
$N\psi_d(N,k)_{\bullet,1}$, we may assume that $W \in \psi_d^{(a,b)}(N, k+1)_{\bullet}$.   

Taken together, statements (ii) and (iii) say that the bisimplicial set $N\psi_d(N,k)_{\bullet,\bullet}$ is a Segal space. In particular, the product 
$\bullet$ on $\pi_0 (N\psi_d(N,k)_{\bullet,1})$
that we defined in part (iv) of this proposition is both unital and associative. The unit is the path component containing all 
0-simplices of the form $(\varnothing, a, b)$.  
Thus, to conclude the proof of (iv), we just need to show that any element in  $\pi_0 (N\psi_d(N,k)_{\bullet,1})$ 
has an inverse with respect to the product  $\bullet$. 
To do this, we again reverse the roles 
of $x_1$ and $x_{k+1}$. Thus,  
a triple $(W,a,b)$ is a 0-simplex of $N\psi_d(N,k)_{\bullet,1}$
if it satisfies 
$W\cap \big( \mathbb{R}^k \times \{ a\} \times (0,1)^{N-k-1}\big) = \varnothing$ and
$W\cap \big( \mathbb{R}^k \times \{ b\} \times (0,1)^{N-k-1}\big) = \varnothing$.
Pick then a 0-simplex
$(W,a, b)$. Without loss of
generality, we may assume that $a=0$, $b=1$,
and that $W \in \psi_d^{(0,1)}(N,k+1)_{0}$. In other words, we assume 
that $W \in \psi_d(N,k)_{0}$. 
Also, let us pick a regular value $a \in \mathbb{R}^k$ of the standard projection $x_k: W \rightarrow \mathbb{R}^k$, and let $M$ denote the pre-image $x_k^{-1}(a)$. Note that such a value $a \in \mathbb{R}^k$ exists because the set of regular values 
of $x_k: W \rightarrow \mathbb{R}^k$ is dense in $\mathbb{R}^k$. By a technique that we will discuss in Remark \ref{inter.remark},   
we can guarantee that there is a concordance $\widetilde{W} \in \psi_d(N,k)([0,1])$ from $W$ to the product 
$\mathbb{R}^k \times M$. Note that $\mathbb{R}^k \times M$ is also a 0-simplex  
of $\psi_d(N,k)_{\bullet}$. 
It follows that $(W,0,1)$ and $(\mathbb{R}^k\times M, 0, 1)$ lie in the same path component of $N\psi_d(N,k)_{\bullet,1}$.  Thus, without loss of generality, we can assume from now on that 
$W = \mathbb{R}^k \times M$. Next, pick a PL embedding 
$g: \mathbb{R} \times (0,1) \rightarrow \mathbb{R} \times (0,3)$
whose image is equal to the one illustrated  
in the following figure:

\begin{center}
\setlength{\unitlength}{1.2mm} 
\begin{picture}(77,62)
\put(20,20){\line(0,1){23}}
\put(58,20){\line(0,1){23}}
\put(26,20){\line(0,1){17}}
\put(52,20){\line(0,1){17}}
\put(20,43){\line(1,0){38}}
\put(26,37){\line(1,0){26}}

\put(20,20){\line(0,-1){8}}
\put(58,20){\line(0,-1){8}}
\put(26,20){\line(0,-1){8}}
\put(52,20){\line(0,-1){8}}
\thicklines
\put(15,20){\vector(0,-1){7}}
\put(15,20){\vector(0,1){38}}
\put(15,20){\vector(1,0){56}}
\put(15,20){\vector(-1,0){7}}

\put(11,58){$x_1$}
\put(71,17){$x_2$}
\put(65,16){\footnotesize{$3$}}
\put(65,19){\line(0,1){2}}
\put(30,19){\line(0,1){2}}
\put(48,19){\line(0,1){2}}
\put(14,45){\line(1,0){2}}
\put(12,44){\footnotesize{$2$}}
\put(30,16){\footnotesize{$1$}}

\put(48,16){\footnotesize{$2$}}

\put(26,6){\footnotesize{Figure 1. Image of $g$.}}
\end{picture}
\end{center}
Moreover, let $G: \mathbb{R}^{k-1}\times \mathbb{R}\times (0,1) \times \mathbb{R}^{N-k-1}  \rightarrow \mathbb{R}^N$ 
be the embedding defined by 
$G=\mathrm{Id}_{\mathbb{R}^{k-1}}\times g \times \mathrm{Id}_{\mathbb{R}^{N-k-1}}$. 
The image $G(W)$
is a 0-simplex of $\psi_d^{(0,3)}(N,k+1)_{\bullet}$
which is empty at all heights $x_{k} > 2$. Therefore, we can produce a concordance
between $G(W)$ and $\varnothing$ by 
pushing $G(W)$ towards $- \infty$ along the $x_{k}$-direction. On the other hand,
recall that $W$ is of the form $\mathbb{R}^k \times M$. Thus,
if we push  $G(W)$ towards $x_{k} = \infty$,
we obtain a 
concordance from  $G(W)$ to a 0-simplex of the
form
$W \sqcup W'$ with $W' \in \psi_d^{(2,3)}(N,k+1)_{0}$.
This implies that the product
$[(W,0,1)] \bullet [ (W',2,3)]$
is equal to
$[ (\varnothing,0,3)]$, 
and we can therefore conclude that the element $[(W,0,1)] $
has an inverse with respect to $\bullet$.
\end{proof}

The following proposition,
stated and proven in \cite{Seg},  will be our main tool to prove that the
first map in (\ref{descan}) is a weak homotopy
equivalence (compare with Lemma 3.14
of \cite{GRW}). For any positive integer $p$ and any $j$ in $\{1, \ldots, p\}$, 
we will denote again by 
$\beta_j$ the morphism $[1] \rightarrow [p]$
in $\Delta$ defined by $\beta_j(0)=j-1$ and $\beta_j(1)=j$. 

\theoremstyle{plain} \newtheorem{segal}[scan]{Proposition}

\begin{segal} \label{segal}
Let $X_{\bullet}$ be a simplicial space such that,
for each non-negative integer $p$, we have that
\[ 
(\beta_1^*,\ldots,\beta_p^*):X_{p}\rightarrow \underbrace{X_1 \times \ldots \times X_1}_{p}
\]
is a homotopy equivalence
(when $p=0$, this means that $X_0$ is contractible).
Then, the adjoint 
$X_1 \rightarrow \Omega_{X_0} |X_{\bullet}|$
of the natural map $[-1,1]\times X_1 \rightarrow |X_{\bullet}|$
is a homotopy equivalence if and only if 
$X_{\bullet}$ is group-like, i.e., $\pi_0(X_1)$ 
is a group with respect to the product induced by
the face map 
$d_1: X_2 \rightarrow X_1$.
\end{segal} 

In this statement, $\Omega_{X_0} |X_{\bullet}|$
denotes the space of paths in $|X_{\bullet}|$ 
which start and end in $X_0$, which we view as a
subspace of $|X_{\bullet}|$. If we
apply the geometric realization functor to 
the first simplicial direction of $N\psi_d(N,k)_{\bullet,\bullet}$,
we obtain a simplicial space whose space
of $q$-simplices is equal to $|N\psi_d(N,k)_{\bullet,q}|$. The following is an immediate consequence of
Proposition \ref{propmonoid}. 

\theoremstyle{plain} \newtheorem{trueseg}[scan]{Lemma}

\begin{trueseg} \label{trueseg}
The simplicial space $[q] \mapsto |N\psi_d(N,k)_{\bullet,q}|$ satisfies
the assumptions stated in Proposition \ref{segal}.
\end{trueseg} 

We will now define the first map of (\ref{descan}). Recall that the preferred base-point for any space of the form 
$|\psi_d(N,k)_{\bullet}|$ is the geometric realization of the subsimplicial set
$\varnothing_{\bullet}$ consisting of all empty simplices $\varnothing \in \psi_d(N,k)_p$. On the other hand, our preferred base-point 
for the space $\left\|N\psi_d(N,k)_{\bullet,\bullet}\right\|$ will be the vertex corresponding to the (0,0)-simplex $(\varnothing, 0)$. We will denote this base-point by $x_0$. 

\theoremstyle{definition} \newtheorem{leftmost.map}[scan]{Note}

\begin{leftmost.map} \label{leftmost.map} 

\textbf{(Construction of the first map in (\ref{descan}))}  Recall that the scanning map $\mathcal{S}_k$ was induced by a map 
$\mathcal{T}: [-1,1]\times |\psi_d(N,k)_{\bullet}| \rightarrow  |\psi_d(N,k+1)_{\bullet}|$ which sends
$[-1,1]\times |\varnothing_{\bullet}|$
and $\{-1,1\}\times |\psi_d(N,k)_{\bullet}|$ to the base-point $|\varnothing_{\bullet}|$
of $|\psi_d(N,k+1)_{\bullet}|$. See the discussion given before Definition \ref{scan}. 
In this note, we will call this map $\mathcal{T}$   \textit{the adjoint of the scanning map}. Also, let us denote the restriction of $\mathcal{T}$ on 
$[-1,0]\times  |\psi_d(N,k)_{\bullet}|$ and $[0,1]\times  |\psi_d(N,k)_{\bullet}|$ by $\mathcal{T}^-$ and 
$\mathcal{T}^+$ respectively.  To construct the
map $\widetilde{\mathcal{T}}: [-1,1]\times |\psi_d(N,k)_{\bullet}| \rightarrow \left\|N\psi_d(N,k)_{\bullet,\bullet}\right\|$ that will induce the 
leftmost map in (\ref{descan}), we will first modify
$\mathcal{T}^-$ and $\mathcal{T}^+$ so that their targets become  $\left\|N\psi_d(N,k)_{\bullet,\bullet}\right\|$. However, once we do this modification, the two maps will not agree on $\{0\}\times |\psi_d(N,k)_{\bullet}|$. Thus, we will also define a map of the form
$[-1,1]\times |\psi_d(N,k)_{\bullet}| \rightarrow \left\|N\psi_d(N,k)_{\bullet,\bullet}\right\|$ that will allow us to transition between the new versions of 
$\mathcal{T}^-$ and $\mathcal{T}^+$. The rigorous details of the construction of the map 
$\widetilde{\mathcal{T}}: [-1,1]\times |\psi_d(N,k)_{\bullet}| \rightarrow \left\|N\psi_d(N,k)_{\bullet,\bullet}\right\| $ are given in the following steps: 

\textit{Step 1.} Let $\psi_d^{< 1}(N,k + 1)_{\bullet}$ be the subsimplicial set of $\psi_d(N,k+1)_{\bullet}$ whose $p$-simplices
are all the $W \in \psi_d(N,k+1)_p$ such that $W \subseteq \Delta^p \times \mathbb{R}^k\times (-\infty,1)\times (0,1)^{N-k-1}$.  
We also define $\psi_d^{> 0}(N,k + 1)_{\bullet}$ in a similar fashion. That is, $W$ is in $\psi_d^{> 0}(N,k + 1)_p$ if
$W \subseteq \Delta^p \times \mathbb{R}^k\times (0, \infty)\times (0,1)^{N-k-1}$. It is easy to check that 
$\mathrm{Im}\hspace{0.06cm}\mathcal{T}^- \subseteq |\psi_d^{< 1}(N,k + 1)_{\bullet}|$ and
$\mathrm{Im}\hspace{0.06cm}\mathcal{T}^+ \subseteq |\psi_d^{> 0}(N,k + 1)_{\bullet}|$. 
Thus, we can view $\mathcal{T}^-$ and $\mathcal{T}^+$ as maps into 
$|\psi_d^{< 1}(N,k + 1)_{\bullet}|$ and $|\psi_d^{> 0}(N,k + 1)_{\bullet}|$ respectively. Also, note that we can
define inclusions $j^-: \psi_d^{< 1}(N,k + 1)_{\bullet} \hookrightarrow N\psi_d(N,k)_{\bullet,0}$ and
$j^+: \psi_d^{> 0}(N,k + 1)_{\bullet} \hookrightarrow N\psi_d(N,k)_{\bullet,0}$ by setting 
$j^-(W) = (W,c_1)$ and $j^+(W) = (W,c_0)$. 
 We will modify
$\mathcal{T}^-$ and $\mathcal{T}^+$ by taking the following compositions: 
\vspace{-0.05cm}
\begin{equation*} 
[0,1]\times |\psi_d(N,k)_{\bullet}| \stackrel{\mathcal{T}^{\hspace{0.03cm}-}}{\longrightarrow} 
|\psi_d^{<1}(N,k+1)_{\bullet}| \stackrel{\hspace{-0.05cm}|j^{-}|}{\longrightarrow}  |N\psi_d(N,k)_{\bullet,0}| \hookrightarrow 
\left\|N\psi_d(N,k)_{\bullet,\bullet}\right\| 
\end{equation*}
\vspace{-0.2cm}
\begin{equation*}
[-1,0]\times |\psi_d(N,k)_{\bullet}| \stackrel{\mathcal{T}^{\hspace{0.03cm}+}}{\longrightarrow} 
|\psi_d^{>0}(N,k+1)_{\bullet}| \stackrel{\hspace{-0.05cm}|j^{+}|}{\longrightarrow} |N\psi_d(N,k)_{\bullet,0}| \hookrightarrow 
\left\|N\psi_d(N,k)_{\bullet,\bullet}\right\|
\end{equation*}
In both of these compositions, the last map is the canonical inclusion of  $|N\psi_d(N,k)_{\bullet,0}|$
into $\left\|N\psi_d(N,k)_{\bullet,\bullet}\right\|$. In the remainder of this construction, we will label these compositions as
$\widetilde{\mathcal{T}}^-$ and $\widetilde{\mathcal{T}}^+$ respectively. 

\textit{Step 2.} Next, we will define the map that we will use to interpolate between $\widetilde{\mathcal{T}}^-$ and 
$\widetilde{\mathcal{T}}^+$. To do this, consider the natural map $q: \Delta^1 \times |N\psi_d(N,k)_{\bullet,1}| \rightarrow \left\|N\psi_d(N,k)_{\bullet,\bullet}\right\|$. From now on, in the domain of $q$, we will replace $\Delta^1$ with $[-1,1]$. The map 
$g:[-1,1]\times |\psi_d(N,k)_{\bullet}| \rightarrow \left\|N\psi_d(N,k)_{\bullet,\bullet}\right\|$ 
that we will use to interpolate from $\widetilde{\mathcal{T}}^-$ to $\widetilde{\mathcal{T}}^+$ is defined 
as the composition
\[
[-1,1]\times |\psi_d(N,k)_{\bullet}| \longrightarrow  [-1,1] \times |N\psi_d(N,k)_{\bullet,1}|  \stackrel{q}{\longrightarrow} 
 \left\|N\psi_d(N,k)_{\bullet,\bullet}\right\|,
\] 
where the first map is the product of $\mathrm{Id}_{[-1,1]}$ and the geometric realization of the morphism 
$\eta^1: \psi_d(N,k)_{\bullet} \rightarrow N\psi_d(N,k)_{\bullet,1}$ that we defined in part (i) of Proposition 
\ref{propmonoid}.  

\textit{Step 3.} In this step, we concatenate the maps $\widetilde{\mathcal{T}}^-$, $g$, and $\widetilde{\mathcal{T}}^+$. More concretely, we define
a map
$\widetilde{\mathcal{T}}': [-1,1]\times |\psi_d(N,k)_{\bullet}| \rightarrow \left\|N\psi_d(N,k)_{\bullet,\bullet}\right\|$ as follows:
\begin{equation} \label{adjointfirst} 
\widetilde{\mathcal{T}}'(t,x) = \left\{
\begin{array}{rl}
\widetilde{\mathcal{T}}^{-}(2t+1,x) & \text{if } t \in [-1, - \frac{1}{2}],\\
g(2t,x) & \text{if } t \in [-\frac{1}{2}, \frac{1}{2}],\\
\widetilde{\mathcal{T}}^{+}(2t-1,x) & \text{if } t\in [\frac{1}{2},1].
\end{array} \right.
\end{equation}

\textit{Step 4.} Note that the map $\widetilde{\mathcal{T}}'$ that we defined in Step 3 will not induce a map of the form 
$|\psi_d(N,k)_{\bullet}| \rightarrow \Omega \left\|N\psi_d(N,k)_{\bullet,\bullet}\right\|$ since $\widetilde{\mathcal{T}}'$ maps 
$\{-1\} \times |\psi_d(N,k)_{\bullet}|$ and $\{1\} \times |\psi_d(N,k)_{\bullet}|$ to different points in 
$ \left\|N\psi_d(N,k)_{\bullet,\bullet}\right\|$. More precisely, the map $\widetilde{\mathcal{T}}'$ sends
$\{-1\} \times |\psi_d(N,k)_{\bullet}|$ to the vertex corresponding to the
$0$-simplex $(\varnothing, 1)$ and  $\{1\} \times |\psi_d(N,k)_{\bullet}|$ to 
the vertex corresponding to the $0$-simplex $(\varnothing, 0)$ (which, recall, we are denoting by $x_0$). 
To fix this, and to give the definition of our desired map $\widetilde{\mathcal{T}}$, consider the characteristic map
\[
c: \Delta^{1} \cong [0,1] \rightarrow |N\psi_d(N,k)_{\bullet,1}|
\]
of the $(0,1)$-simplex $(\varnothing, 0, 1)$.
Note that, by reversing the orientation of this edge, we obtain a path from $x_0$
to the vertex corresponding to $(\varnothing,1)$. Then, we define the map 
$\widetilde{\mathcal{T}}: [-1,1]\times |\psi_d(N,k)_{\bullet}| \rightarrow \left\|N\psi_d(N,k)_{\bullet,\bullet}\right\|$ by setting

\begin{equation} \label{adjointfirst2} 
\widetilde{\mathcal{T}}(t,x) = \left\{
\begin{array}{rl}
c(-t) & \text{if } t \in [-1,0]\\
\widetilde{\mathcal{T}}'(-1 + 2t, x)   &  \text{if } t \in [0,1].
\end{array} \right.
\end{equation}

By construction, this map $\widetilde{\mathcal{T}}$ sends $\{-1,1\}\times  |\psi_d(N,k)_{\bullet}|$ to the base-point $x_0$ of
$\left\|N\psi_d(N,k)_{\bullet,\bullet}\right\|$. Therefore, the adjoint of $\widetilde{\mathcal{T}}$ will be a map of the form
\begin{equation} \label{leftmost.map.h}
\mathcal{H}:  |\psi_d(N,k)_{\bullet}| \rightarrow \Omega \left\|N\psi_d(N,k)_{\bullet,\bullet}\right\|.
\end{equation}
We take this $\mathcal{H}$ to be the first map in the composition (\ref{descan}). Note that 
$\mathcal{H}(|\varnothing_{\bullet}|)$ is not equal to the constant loop $[-1,1]\rightarrow x_0$, so $\mathcal{H}$ is not a pointed map. By following the construction given above, one can see that 
$\mathcal{H}(|\varnothing_{\bullet}|)$ is homotopic to the loop which first traverses the image of $c$ from
$c(1)$ to $c(0)$, and then traverses the image of $c$ again from $c(0)$ back to $c(1)$. In particular, 
$\mathcal{H}(|\varnothing_{\bullet}|)$ is a nullhomotopic loop.  

\end{leftmost.map} 

Now that we have introduced all the maps involved in the composition (\ref{descan}), we can prove the following fact. 

\theoremstyle{plain}  \newtheorem{descan.homotopic}[scan]{Lemma}

\begin{descan.homotopic} \label{descan.homotopic}

The composition given in (\ref{descan}) is homotopic to the scanning map $\mathcal{S}_k$.

\end{descan.homotopic}

\begin{proof}
To prove this lemma, we will show  that the adjoint $\mathcal{T}:[-1,1]\times |\psi_d(N,k)_{\bullet}| \rightarrow |\psi_d(N,k+1)_{\bullet}|$
of the scanning map is homotopic to the adjoint 
of the map (\ref{descan}).
The adjoint
\[
g: [-1,1] \times |\psi_d(N,k)_{\bullet}| \longrightarrow |\psi_d^{\varnothing}(N,k+1)_{\bullet}|
\]
of (\ref{descan}) is equal 
to the composite
\[
[-1,1] \times |\psi_d(N,k)_{\bullet}| \stackrel{\widetilde{\mathcal{T}}}{\rightarrow} \left\|N\psi_d(N,k)_{\bullet,\bullet}\right\| \stackrel{F}{\rightarrow}
|\psi_d^{0}(N,k+1)_{\bullet}| \stackrel{i}{\hookrightarrow} |\psi_d^{\varnothing}(N,k+1)_{\bullet}|,
\]
where $\widetilde{\mathcal{T}}$ is the map we defined in Note \ref{leftmost.map},
$F$ is the forgetful map introduced in Remark \ref{monoid.remark}, and 
$i$ is the inclusion from $|\psi_d^{0}(N,k+1)_{\bullet}|$
into $|\psi_d^{\varnothing}(N,k+1)_{\bullet}|$.  
Using the definition of the map $\widetilde{\mathcal{T}}$
given in (\ref{adjointfirst2}), it
is easy to verify 
that for each $(t,x)\in [-1,1] \times |\psi_d(N,k)_{\bullet}|$ we have
\[
g(t,x) = \mathcal{T}\big( \phi(t),x \big),
\] 
where $\phi: [-1,1] \rightarrow [-1,1]$
is the continuous function defined by
\begin{equation*}
\phi(t) = \left\{
\begin{array}{rl}
-1  & \text{if } t \in [-1,0],\\
4t -1  & \text{if } t \in [0, 1/4],\\
0 & \text{if } t\in [1/4,3/4],\\
4t -3 & \text{if } t \in [ 3/4, 1].
\end{array} \right.
\end{equation*}

If $h:[0,1] \times [-1,1] \rightarrow [-1,1]$
is any homotopy between $\phi$
and $\mathrm{Id}_{[-1,1]}$ relative to the endpoints of
$[-1,1]$, then the map
\[
H: [0,1] \times [-1,1] \times |\psi_d(N,k)_{\bullet}| \rightarrow |\psi_d(N,k+1)_{\bullet}|
\]
defined by $(s,t, x) \mapsto \mathcal{T}\big( h(s,t), x \big)$
is a homotopy between $g$ and 
$\mathcal{T}$ which is fixed 
on $\{-1,1 \} \times |\psi_d(N,k)_{\bullet}|$.
Consequently, the scanning map 
$\mathcal{S}_k$ is homotopic to the composition
given in (\ref{descan}).
\end{proof}

As we indicated earlier, the leftmost map in (\ref{descan}) has the following property. 

\theoremstyle{plain} \newtheorem{firstmapweak}[scan]{Proposition}

\begin{firstmapweak} \label{firstmapweak}
The map $\mathcal{H}: |\psi_d(N,k)_{\bullet}| \rightarrow \Omega \left\|N\psi_d(N,k)_{\bullet,\bullet}\right\|$
defined in Note \ref{leftmost.map} is a weak homotopy equivalence.
\end{firstmapweak} 

\begin{proof}
Denote the geometric realization of $N\psi_d(N,k)_{\bullet,0}$ by $N_0$, and let 
$\Omega_{N_0} \left\|N\psi_d(N,k)_{\bullet,\bullet}\right\|$ be the space of paths 
$[-1,1]\rightarrow \left\|N\psi_d(N,k)_{\bullet,\bullet}\right\|$ which start and end at $N_0$. Since $N_0$ is contractible
(part (ii) of Proposition \ref{propmonoid}), the natural inclusion 
\[
\iota: \Omega \left\|N\psi_d(N,k)_{\bullet,\bullet}\right\| \hookrightarrow \Omega_{N_0} \left\|N\psi_d(N,k)_{\bullet,\bullet}\right\|
\]
is a weak homotopy equivalence. Thus, to prove this proposition, it suffices to show that the composition
\begin{equation} \label{firstmap.eq1}
\gamma_1: |\psi_d(N,k)_{\bullet}| \stackrel{\mathcal{H}}{\longrightarrow}  \Omega \left\|N\psi_d(N,k)_{\bullet,\bullet}\right\| 
 \stackrel{\iota}{\longrightarrow} \Omega_{N_0} \left\|N\psi_d(N,k)_{\bullet,\bullet}\right\|
\end{equation}
is a weak equivalence.  To this end, let us also consider the composition 
\begin{equation} \label{firstmap.eq2}
\gamma_2: |\psi_d(N,k)_{\bullet}| \stackrel{|\eta^1|}{\longrightarrow} |N\psi_d(N,k)_{\bullet,1}| \longrightarrow 
\Omega_{N_0} \left\|N\psi_d(N,k)_{\bullet,\bullet}\right\|,
\end{equation}
where the first map is the geometric realization of the morphism of simplicial sets $\eta^1$
defined in part (i) of Proposition \ref{propmonoid}, and the
second map is the adjoint of the natural map 
$q: [-1,1] \times |N\psi_d(N,k)_{\bullet,1}| \rightarrow \left\|N\psi_d(N,k)_{\bullet,\bullet}\right\|$. 
Recall that we also used this map $q$ in Step 2 of Note \ref{leftmost.map}. By Proposition \ref{propmonoid}, Proposition \ref{segal}, and Lemma \ref{trueseg}, the map $\gamma_2$ given in (\ref{firstmap.eq2}) 
is a weak homotopy equivalence.  
Thus, to conclude this proof, it is enough to show that $\gamma_1$ and $\gamma_2$ are homotopic to each other. For this last step, we will use the following notation:
For any $x \in  |\psi_d(N,k)_{\bullet}|$ and $t \in [-1,1]$, we will denote the image of $t$ in the path $\gamma_1(x)$ (resp. $\gamma_2(x)$) by $\gamma_1(x,t)$ (resp. $\gamma_2(x,t)$). A careful inspection of the definition of $\gamma_1$ and $\gamma_2$ yields the following observations: 

\begin{itemize}

\item[$\cdot$] For any $x \in  |\psi_d(N,k)_{\bullet}|$, we have that $\gamma_1(x,t) = \gamma_2(x, 4t - 2)$ if $t$ is in 
$[\frac{1}{4}, \frac{3}{4}]$. 

\item[$\cdot$] For any $x \in  |\psi_d(N,k)_{\bullet}|$ and any $t \in [-1,\frac{1}{4}] \cup [\frac{3}{4}, 1]$, the image $\gamma_1(x,t)$ is in $N_0$. 

\end{itemize}
Therefore, we can define a homotopy from $\gamma_1$ to $\gamma_2$ by gradually shrinking the restrictions 
$\gamma_1(x)|_{[-1,\frac{1}{4}]}$ and  
$\gamma_1(x)|_{[\frac{3}{4}, 1]}$ for all $x$ in $|\psi_d(N,k)_{\bullet}|$. Consequently, since $\gamma_2$ is a weak homotopy equivalence, then so is the map $\gamma_1$. 
\end{proof}

\subsection{The inclusion $\psi_d^0(N,k)_{\bullet} \hookrightarrow \psi_d^{\varnothing}(N,k)_{\bullet}$} \label{section63}

It remains to show that the last map in the  
composite (\ref{descan}) is a weak homotopy equivalence.
This will be a consequence of the following proposition. 

\theoremstyle{plain}  \newtheorem{incempty}[scan]{Proposition}

\begin{incempty} \label{incempty}
If $N-d \geq 3$, then the inclusion 
$\psi_d^0(N,k)_{\bullet} \hookrightarrow \psi_d^{\varnothing}(N,k)_{\bullet}$
is a weak homotopy equivalence when $k>1$.  
\end{incempty}

For the proof of this proposition, we need to extend the simplicial set $\psi_d^{\varnothing}(N,k)_{\bullet}$ to a functor of the form 
$\mathbf{PL}^{op} \rightarrow \mathbf{Sets}$. 

\theoremstyle{definition}  \newtheorem{pl.psi.empty}[scan]{Definition} 

\begin{pl.psi.empty} \label{pl.psi.empty}

$\psi_d^{\varnothing}(N,k): \mathbf{PL}^{op} \rightarrow \mathbf{Sets}$ shall be the functor that sends a PL space $P$ to the subset $\psi_d^{\varnothing}(N,k)(P) \subseteq \psi_d(N,k)(P)$ consisting of all 
$W \in  \psi_d(N,k)(P)$ with the property that each fiber $W_{\lambda}$ is concordant (when viewed as an element of 
$\psi_d(N,k)_0$) to $\varnothing$. Given any PL map $f: Q \rightarrow P$, the functor $\psi_d^{\varnothing}(N,k)$ will send $f$ to the function $f^*: \psi_d(N,k)(P) \rightarrow \psi_d(N,k)(Q)$ defined by taking pull-backs. 
\end{pl.psi.empty}

Evidently, we have that $\psi_d^{\varnothing}(N,k)(\Delta^p) = \psi_d^{\varnothing}(N,k)_p$, so the functor defined above is indeed an extension of $\psi_d^{\varnothing}(N,k)_{\bullet}$. 

\theoremstyle{definition} \newtheorem{Extra}[scan]{Note}

\begin{Extra} \label{Extra}
As we did for Theorem \ref{longman}, we will prove Proposition  
\ref{incempty} by showing that
the corresponding inclusion of semi-simplicial sets 
$\widetilde{\psi}_d^0(N,k)_{\bullet} \hookrightarrow \widetilde{\psi}_d^{\varnothing}(N,k)_{\bullet}$
is a weak homotopy equivalence. Throughout the rest of this section, unless 
noted otherwise, we shall denote the semi-simplicial sets 
$\widetilde{\psi}_d^0(N,k)_{\bullet}$ and
$\widetilde{\psi}_d^{\varnothing}(N,k)_{\bullet}$ by
$\psi_d^0(N,k)_{\bullet}$ and
$\psi_d^{\varnothing}(N,k)_{\bullet}$ respectively. Also, for the proof of Proposition \ref{incempty}, we shall interchange once again the roles of $x_k$ and $x_1$. Thus,  a $p$-simplex
$W$ of $\psi_d(N,k)_{\bullet}$ is in $\psi_d^0(N,k)_p$
if there is a piecewise linear function $f:\Delta^p \rightarrow \mathbb{R}$
such that, for each $\lambda$ in $\Delta^p$, we have that 
$W_{\lambda} \cap \big( \{ f(\lambda)\} \times \mathbb{R}^{k-1} \times (0,1)^{N-k} \big) = \varnothing.$
\end{Extra}

Before we begin proving Proposition \ref{incempty}, we point out that, 
by adjusting the argument given in Remark \ref{emptylevel.remark2} to the semi-simplicial setting, 
we can show that the semi-simplicial set 
$\psi_d^0(N,k)_{\bullet}$ is also not Kan. Thus, we run into the same issue we discussed in Remark \ref{remark.psireg}. That is, if 
$f: (\Delta^p, \partial \Delta^p) \rightarrow \big(|\psi_d^{\varnothing}(N,k)_{\bullet}|,|\psi_d^{0}(N,k)_{\bullet}|\big)$
represents an element in  $\pi_p(|\psi_d^{\varnothing}(N,k)_{\bullet}|, |\psi_d^{0}(N,k)_{\bullet}|)$, then we cannot guarantee that $f$
 is homotopic to the geometric realization of a morphism of the form
$g: (\Delta^p_{\bullet}, \partial \Delta^p_{\bullet}) \rightarrow (\psi_d^{\varnothing}(N,k)_{\bullet}, \psi_d^{0}(N,k)_{\bullet})$. 
To fix this, as we did in the proof of Theorem \ref{longman}, we will use again Proposition \ref{subsimphop}.
The proof of Proposition 
\ref{incempty}
is practically identical to that of 
Theorem \ref{longman} once we 
have the following result. 

\theoremstyle{plain}  \newtheorem{nestenempty}[scan]{Proposition}

\begin{nestenempty}  \label{nestenempty}

Fix a compact PL space $P$ and let $i_0, i_1: P \hookrightarrow [0,1]\times P$ be the inclusions defined by 
$i_j(x) = (j, x)$ for $j=0,1$. Given any $W$ in $\psi_d^{\varnothing}(N,k)(P)$ and any real value $\beta$, we can find a concordance $\widetilde{W} \in \psi_d^{\varnothing}(N,k)([0,1]\times P)$ with the following properties: 

\begin{itemize}

\item[(i)] $i^*_0\widetilde{W} = W$.

\item[(ii)] $\widetilde{W}$ agrees with the constant concordance $[0,1]\times W$ when we restrict the background space to $(-\infty, \beta) \times \mathbb{R}^{N-1}$. 

\item[(iii)] For the element $W'= i_1^*\widetilde{W}$, there exists a finite open cover $U_1, \ldots, U_q$ of $P$ and real values  $a_1, \ldots, a_q \in (\beta, \infty)$
such that, for each $j \in \{1, \ldots, q\}$, we have that
$W'_{\lambda} \cap \big( \{a_j\} \times \mathbb{R}^{N-1} \big) = \varnothing$
for all $\lambda \in U_j$.  
\end{itemize}
\end{nestenempty}

\begin{proof}[\textbf{Proof of Proposition \ref{incempty}}] 
Let us give the proof of Proposition 
\ref{incempty} assuming Proposition 
\ref{nestenempty}.
 Consider then an arbitrary map of pairs
\begin{equation} \label{diaginc}
f:(\Delta^p, \partial\Delta^p) \rightarrow \big(|\psi_d^{\varnothing}(N,k)_{\bullet}|,|\psi_d^{0}(N,k)_{\bullet}|\big).
\end{equation}
We need to prove that $f$ represents the trivial class in 
$\pi_p(|\psi_d^{\varnothing}(N,k)_{\bullet}|, |\psi_d^{0}(N,k)_{\bullet}|)$. First, note that 
both $|\psi_d^{\varnothing}(N,k)_{\bullet}|$ and $|\psi_d^{0}(N,k)_{\bullet}|$ are invariant under the subdivision map $\rho$ of $|\Psi_d(\mathbb{R}^N)_{\bullet}|$ and the homotopy between $\rho$ and 
$\mathrm{Id}_{|\Psi_d(\mathbb{R}^N)_{\bullet}|}$ that we defined in \S \ref{secsubmap}.  Thus, by Proposition \ref{subsimphop}, $f$ is homotopic (as a map of pairs)
to a composite of the form
\begin{equation} \label{class.varnothing}
(\Delta^p,\partial\Delta^p) \stackrel{f'}{\rightarrow} (\left|K_{\bullet}\right|, \left|K'_{\bullet}\right|) \stackrel{\left|h\right|}{\rightarrow} (\left|\psi_d^{\varnothing}(N,k)_{\bullet}\right|,\left|\psi_d^{0}(N,k)_{\bullet}\right|),
\end{equation}
where $(K_{\bullet},K'_{\bullet})$ is a pair
of finite semi-simplicial sets induced by a pair
of finite ordered simplicial complexes $(K,K')$. Now, consider the bijective correspondence 
$\Psi_d(\mathbb{R}^N)(|K|) \rightarrow \mathbf{Semi}$-$\mathbf{Ssets}(K_{\bullet}, \Psi_d(\mathbb{R}^N)_{\bullet})$ given by the version of Theorem \ref{classsub} for semi-simplicial sets (see Remark \ref{semi.class.version}). Clearly, if we restrict this function on $\psi_d^{\varnothing}(N,k)(|K|)$,  we will obtain a bijection 
$\psi_d^{\varnothing}(N,k)(|K|) \stackrel{\cong}{\rightarrow}  \mathbf{Semi}$-$\mathbf{Ssets}(K_{\bullet}, \psi_d^{\varnothing}(N,k)_{\bullet})$. Thus, the map $h: K_{\bullet} \rightarrow \psi_d^{\varnothing}(N,k)_{\bullet}$ that appears in (\ref{class.varnothing}) classifies (relative to the triangulation $(K, \mathrm{Id}_{|K|})$) a unique element 
$W \in \psi_d^{\varnothing}(N,k)(|K|)$. 
Moreover, since $h(K'_{\bullet}) \subseteq \psi_d^{0}(N,k)_{\bullet}$, for each simplex $\sigma \in K'$ we can find a PL function $f_{\sigma}: \sigma \rightarrow \mathbb{R}$ such that
\[
W_{\sigma} \cap (\Gamma(f_{\sigma}) \times \mathbb{R}^{k-1} \times (0,1)^{N-k}) = \varnothing,
\]
where $\Gamma(f_{\sigma})$ is the graph of the function $f_{\sigma}$. 
Now, fix a real constant $\beta$ such that $\mathrm{Im}\hspace{0.05cm} f_{\sigma} \subset (-\infty, \beta)$ for all simplices $\sigma$ of $K'$.  
Then, by applying Proposition \ref{nestenempty} to $W$ and $\beta$, we can find a concordance $\widetilde{W} \in \psi_d^{\varnothing}(N,k)\big( [0,1]\times |K| \big)$ from $W$ to an element 
$W' \in \psi_d^{\varnothing}(N,k)\big( |K| \big)$ for which there exists a finite open cover $U_1, \ldots, U_q$ of $|K|$ and real values  $a_1, \ldots, a_q > \beta$ such that
\begin{equation} \label{nestenempty.eq}
W'_{U_j}\cap \big( \Gamma(c_{a_j})\times \mathbb{R}^{k-1} \times (0,1)^{N-k}\big) = \varnothing
\end{equation}
for each $j = 1, \ldots, q.$ In (\ref{nestenempty.eq}), $\Gamma(c_{a_j})$ is the graph of the constant function $c_{a_j}: U_j \rightarrow \mathbb{R}$ which maps every point in $U_j$ to $a_j$. Also, by Proposition \ref{nestenempty}, we can guarantee that $\widetilde{W}$ agrees with the constant concordance $[0,1]\times W$ when we restrict the background space to $(-\infty, \beta) \times \mathbb{R}^{k-1} \times (0,1)^{N-k}$. At this point, we have a scenario similar to the one we had in the proof of Theorem \ref{longman}. In that proof, we obtained an open cover $U_1, \ldots, U_q$ such that the corresponding element $W'$ admitted a fiberwise regular value over each $U_j$. On the other hand, in our current proof, the local condition that holds over each $U_j$ is the one expressed in (\ref{nestenempty.eq}). That is,
for each $U_j$ in our open cover, we have that $W'$ does not intersect the graph of the constant function $c_{a_j}$ over $U_j$.
Therefore, by doing arguments almost identical to the ones we did to conclude the proof of Theorem \ref{longman}, we can prove that the map $f$ in (\ref{diaginc}) represents the trivial class in  $\pi_p(|\psi_d^{\varnothing}(N,k)_{\bullet}|, |\psi_d^{0}(N,k)_{\bullet}|)$. 
\end{proof}

The proof of Proposition \ref{nestenempty} is rather intricate and it requires proving several preliminary lemmas. Before jumping into any technicalities, we will spend a few moments motivating our strategy for this proof. An essential idea that will underlie our argument is that 
\textit{null-bordisms} can be pushed to infinity. In the next definition, we clarify what we mean by null-bordism. 

\theoremstyle{definition}  \newtheorem{null.bordant}[scan]{Definition}

\begin{null.bordant}  \label{null.bordant}
Let $k,d,N$ be non-negative integers such that $k < d < N$. 
We say that a $0$-simplex $M$ of $\psi_{d-k}(N-k,0)_{\bullet}$ is \textit{null-bordant} if there exists a compact PL submanifold $C$ of $[0,1]\times (0,1)^{N-k}$ of dimension $d-k +1$ which satisfies the following: 

\begin{itemize}

\item[(i)] $\partial C = M$ and $\partial C \subseteq \{ 0\}\times (0,1)^{N-k}$.

\item[(ii)] $C \cap \big( \{ 1\}\times (0,1)^{N-k} \big) = \varnothing$.

\end{itemize}

A PL submanifold $C$ of $[0,1]\times (0,1)^{N-k}$  satisfying the above conditions will be called a \textit{null-bordism for $M$}. 

\end{null.bordant}
 
 The reason why we formulate Definition \ref{null.bordant} for $\psi_{d-k}(N-k,0)_{\bullet}$ instead of only $\psi_d(N,0)_{\bullet}$ is because we will typically use this definition while working with manifolds and ambient spaces whose dimensions are smaller than our usual $d$ and $N$. Before we continue motivating the proof of Proposition \ref{nestenempty}, we need to recall the following notation: 
 
 \begin{itemize}
 
 \item[$\cdot$] For any PL space $P$ and open set $U$ in $\mathbb{R}^N$, $\Psi_d(U)(P)$ is the set of all closed PL subspaces $W$ of 
 $P\times U$ with the property that the standard projection $\pi: W \rightarrow P$ is a PL submersion of codimension $d$. Also, 
 recall that $\Psi_d(U)_{\bullet}$ is the simplicial set whose set of $p$-simplices is equal to $\Psi_d(U)(\Delta^p)$. 
 
\item[$\cdot$] For any element $W$ in $\Psi_d(\mathbb{R}^N)(P)$ (in particular, in $\psi_d^{\varnothing}(N,k)(P)$) and any point $\lambda \in P$, $x_1: W_{\lambda} \rightarrow \mathbb{R}$ shall continue to denote the standard projection from the fiber $W_{\lambda}$ onto the first coordinate of $\mathbb{R}^N$.  
 
 \end{itemize}

Now, let $P$ be a compact PL space and $W$ an element of $\psi_d^{\varnothing}(N,k)(P)$. To prove Proposition \ref{nestenempty}, we need to show that it is possible to deform $W$ so that each fiber $W_{\lambda}$ has an empty level set of the form $x^{-1}_1(a) = \varnothing$. Moreover, throughout the deformation, the fibers of $W$ need to remain unchanged below the height $\beta$ that we fixed in the statement of Proposition \ref{nestenempty}. In Lemma \ref{maneuver} below, we will introduce the basic maneuver that we will use to achieve this type of deformation. As the reader shall see, the fact that we can push null-bordisms to infinity will play a crucial role in the proof of this lemma. We will use the following terminology in the statement of Lemma \ref{maneuver}: Fix an open set $U$ in $\mathbb{R}^k$, a 0-simplex $M$ of $\psi_{d-k}(N-k, 0)_{\bullet}$, a PL space $P$, and an element $W$ of $\psi_d(N,k)(P)$.   
We shall say that a fiber $W_{\lambda}$ of the projection $\pi:W \rightarrow P$ is  \textit{$M$-standard in $U$}  if
$W_{\lambda}\cap\big( U \times (0,1)^{N-k}\big) = U \times M$.  

\theoremstyle{plain} \newtheorem{maneuver}[scan]{Lemma}

\begin{maneuver} \label{maneuver}
Fix integers $1 < k < d < N$ and let $W$ be a 0-simplex of $\psi_d(N,k)_{\bullet}$, which we can view as the unique fiber of a map $W \rightarrow *$ onto a point. Suppose there is a $0$-simplex $M$ of  $\psi_{d-k}(N-k,0)_{\bullet}$, a real number $a$, and a value $\delta>0$ such that $W$ is $M$-standard in $(a-\delta, a+\delta)\times \mathbb{R}^{k-1}$. Then, if $M$ is null-bordant, there exists a concordance $\widetilde{W}\in \psi_d(N,k)([0,1])$ with the following properties: 

\begin{itemize}

\item[(i)] $\widetilde{W}$ is a concordance from $W$ to an element $W'\in \psi_d(N,k)_0$ with the property that $W'\cap \big(\{a\}\times \mathbb{R}^{N-1}\big) = \varnothing$.

\item[(ii)]  $\widetilde{W}$ agrees with the constant concordance 
$[0,1]\times W$ when we restrict the background space to the complement of $[a - \frac{\delta}{2}, a + \frac{\delta}{2}]\times \mathbb{R}^{N-1}$. 

\end{itemize}

\end{maneuver}

\begin{proof}
Let $W(\delta)$ denote the intersection of $W$ with $(a -\delta, a + \delta)\times \mathbb{R}^{N-1}$. This $W(\delta)$ is a
 0-simplex of $\Psi_d\big((a - \delta, a+ \delta)\times \mathbb{R}^{N-1}\big)_{\bullet}$. Moreover, since $W$ is $M$-standard in 
 $(a - \delta, a + \delta)\times \mathbb{R}^{k-1}$, we have that $W(\delta) = (a - \delta, a + \delta)\times \mathbb{R}^{k-1} \times M$. 
 Our strategy for constructing the concordance $\widetilde{W}$ is roughly the following. First, we shall construct an element
$\widetilde{M}_{\delta}$ of $\Psi_d\big((a - \delta, a+ \delta)\times \mathbb{R}^{N-1}\big)([0,1])$ which will be a concordance between 
$W(\delta)$ and another $0$-simplex of 
$\Psi_d\big((a - \delta, a+ \delta)\times \mathbb{R}^{N-1}\big)_{\bullet}$ which does not intersect the hyperplane $\{a\}\times \mathbb{R}^{N-1}$. 
Moreover, all the fibers of $\widetilde{M}_{\delta}$ will be $M$-standard in both 
$(a - \delta, a - \frac{\delta}{2} )\times \mathbb{R}^{k-1}$ and $(a + \frac{\delta}{2}, a + \delta) \times \mathbb{R}^{k-1}$.   Note that the fibers of the constant concordance $[0,1]\times W$ are also $M$-standard in both of these open sets of $\mathbb{R}^k$. In fact, the fibers of $[0,1]\times W$ are $M$-standard in all of 
$(a-\delta, a + \delta)\times \mathbb{R}^{k-1}$. Then, we will obtain the concordance $\widetilde{W}$ by cutting out the subspace $[0,1] \times (a-\delta, a + \delta)\times \mathbb{R}^{k-1}\times M$ from $[0,1]\times W$ and then gluing in the concordance $\widetilde{M}_{\delta}$. We will break down this proof into fours steps. In Steps 1, 2, and 3, we will construct the concordance $\widetilde{M}_{\delta}$. Finally, in Step 4, we carry out the gluing process that we described above. 

\textit{Step 1.} We start by constructing a concordance between the product $\mathbb{R}\times M$ (which is a $0$-simplex of 
 $\psi_{d-k+1}(N-k+1, 1)_{\bullet}$) and the empty manifold $\varnothing$. First, let us fix a null-bordism $C \subseteq [0,1]\times (0,1)^{N-k}$ of $M$ (in the sense of Definition \ref{null.bordant}) and let 
$D$ be the 0-simplex of $\psi_{d-k+1}(N-k+1, 1)_{\bullet}$ defined as the following union: 
\begin{equation} \label{infinite.cob}
D = \big( (-\infty, 0] \times M \big) \cup C.
\end{equation}
That is, $D$ is the `extended' null-bordism obtained by attaching the half-closed cylinder $(-\infty, 0] \times M$ to the boundary of $C$. Now, take the constant concordance $[0,1]\times D \in \psi_{d-k+1}(N-k+1, 1)([0,1])$ and let   $e: [0,1]\times \mathbb{R} \rightarrow [0,1]\times \mathbb{R}$ be an open PL embedding which commutes with the projection onto $[0,1]$ and has the following properties: 

\begin{itemize}

\item[$\cdot$] The image of the map $e_0$ is the open interval $(-1,0)$.

\item[$\cdot$] The map $e_{\frac{1}{2}}$ is equal to the identity map $\mathrm{Id}_{\mathbb{R}}$.

\item[$\cdot$] The image of the map $e_1$ is the open interval $(1,2)$.

\end{itemize}

If $E: [0,1] \times \mathbb{R} \times (0,1)^{N-k} \rightarrow  [0,1] \times \mathbb{R} \times (0,1)^{N-k}$
is the open PL embedding defined by $E = e\times \mathrm{Id}_{(0,1)^{N-k}}$, then Proposition \ref{pullemb} guarantees that the pre-image 
$E^{-1}([0,1]\times D)$ is an element of $\psi_{d-k+1}(N-k+1, 1)([0,1])$. From now on, we will denote this new element by  
$\widetilde{C}$. By the way we chose the embedding $e$, we have that $\widetilde{C}$ satisfies the following three conditions: 
\[
\widetilde{C}_0 = \mathbb{R}\times M \qquad \widetilde{C}_{\frac{1}{2}} = D  \qquad \widetilde{C}_1 = \varnothing. 
\]
In particular, $\widetilde{C}$ is a concordance from $\mathbb{R}\times M$ to $\varnothing$. Intuitively, the process that is being described by the concordance $\widetilde{C}$ is the following. As we go from $t = \frac{1}{2}$ to $t = 1$ along the interval $[0,1]$, we are pushing the extended null-bordism $D$ towards $-\infty$. At time $t=1$, we reach the empty manifold $\varnothing$. On the other hand, if we go from $t = \frac{1}{2}$ to $t =0$, we are pulling $D$ towards $+\infty$. At time $t = 0$, we reach the infinite cylinder $\mathbb{R}\times M$.     

\textit{Step 2.} Consider again the 0-simplex $W(\delta)$ of 
$\Psi_d\big((a - \delta, a+ \delta)\times \mathbb{R}^{N-1}\big)_{\bullet}$ obtained by intersecting $W$ with 
$(a -\delta, a + \delta)\times \mathbb{R}^{N-1}$.
Since $W(\delta) = (a - \delta, a + \delta)\times \mathbb{R}^{k-1} \times M$, it follows that the obvious projection $x_1: W(\delta) \rightarrow (a- \delta, a+ \delta)$ is a PL submersion of codimension $d-1$. 
Thus, we can also regard $W(\delta)$ as an element of the set $\Psi_{d-1}(\mathbb{R}^{N-1})\big( (a-\delta, a + \delta)\big)$, where the projection from $W(\delta)$ onto the base-space $(a - \delta, a + \delta)$ is given by 
$x_1: W(\delta) \rightarrow (a-\delta, a + \delta)$.
 In this step, we will show that $W(\delta)$, when viewed as an element of 
$\Psi_{d-1}(\mathbb{R}^{N-1})\big( (a-\delta, a + \delta)\big)$, is concordant to the empty manifold $\varnothing$. To do this, 
consider the 
 PL homeomorphism defined as the composition 
\begin{equation} \label{mapsPS}
\xymatrix{
(a - \delta, a + \delta) \times \mathbb{R}^{k-2} \times [0,1] \times \mathbb{R} \times (0,1)^{N-k} 
\ar[d]_P\\
[0,1] \times (a - \delta, a + \delta) \times \mathbb{R}^{k-2} \times \mathbb{R} \times (0,1)^{N-k} 
\ar[d]_S\\
[0,1] \times (a - \delta, a + \delta) \times \mathbb{R}^{k-1}  \times (0,1)^{N-k},}
\end{equation}
where $P$ is the obvious permutation map and $S$ is the map defined via the identification 
$\mathbb{R}^{k-2} \times \mathbb{R}\cong \mathbb{R}^{k-1}$. 
Now, if $\widetilde{C} \in \psi_{d-k+1}(N-k+1, 1)([0,1])$ is the concordance that we constructed in Step 1,
then the image of the product $(a-\delta, a + \delta)\times \mathbb{R}^{k-2}\times \widetilde{C}$ under the PL homeomorphism 
$S\circ P$ will be an element of $\Psi_{d-1}(\mathbb{R}^{N-1})\big( [0,1]\times(a-\delta, a + \delta)\big)$. It is not hard to verify that this element, which we will denote by $\widetilde{W}(\delta)$, is a concordance from 
$W(\delta) \in \Psi_{d-1}(\mathbb{R}^{N-1})\big( (a-\delta, a + \delta)\big)$ to $\varnothing$. We point out that it is exactly in this step where we need the condition $k>1$. Besides the $x_1$-axis, where we have the interval $(a-\delta, a + \delta)$, we need at least one more direction to push the fibers of $x_1: W(\delta) \rightarrow (a-\delta, a + \delta)$ to infinity.

\textit{Step 3.} To complete the construction of the concordance $\widetilde{M}_{\delta}$, we need to fix a PL map $r: [0,1]\times (a - \delta, a+ \delta) \rightarrow [0,1]\times (a- \delta, a + \delta)$ with the following properties:
 
 \begin{itemize}
 
\item[(i)] $r(t,x) = (0,x)$ if $x \in (a - \delta, a - \frac{\delta}{2}) \cup (a+ \frac{\delta}{2}, a + \delta)$ or if $t =0$.

 \item[(ii)] $r(t, x) = (t,x)$ if $x \in [a - \frac{\delta}{4}, a + \frac{\delta}{4}]$.  
 
 \end{itemize}

We define $\widetilde{M}_{\delta}$ as the pull-back $r^*\widetilde{W}(\delta)$ of $\widetilde{W}(\delta)$ along the PL map $r$. Since $\widetilde{W}(\delta)$ is an element of $\Psi_{d-1}(\mathbb{R}^{N-1})\big( [0,1]\times(a-\delta, a + \delta)\big)$, then so is $\widetilde{M}_{\delta}$. 
In particular, the standard projection $\pi: \widetilde{M}_{\delta} \rightarrow [0,1]\times(a-\delta, a + \delta)$ is a PL submersion of codimension $d-1$. However, the obvious projection $\mathrm{pr}_1: [0,1]\times(a-\delta, a + \delta) \rightarrow [0,1]$ is clearly a PL submersion of codimension 1. Therefore, the composition 
$\mathrm{pr}_1\circ \pi: \widetilde{M}_{\delta} \rightarrow [0,1]$, which we shall denote by $\widetilde{\pi}$, is a PL submersion of codimension $d$. Thus, if we take this projection $\widetilde{\pi}$ instead of $\pi$, we can regard $\widetilde{M}_{\delta}$ as an element of
$\Psi_d\big((a - \delta, a + \delta\big)\times \mathbb{R}^{N-1})([0,1])$. As an element of this set, $\widetilde{M}_{\delta}$ will satisfy the following properties:

 \begin{itemize}
 
\item[(i$^*$)] Each fiber of $\widetilde{\pi}: \widetilde{M}_{\delta} \rightarrow [0,1]$ is $M$-standard in  $(a - \delta, a - \frac{\delta}{2} )\times \mathbb{R}^{k-1}$ and $(a + \frac{\delta}{2}, a + \delta) \times \mathbb{R}^{k-1}$.

 \item[(ii$^*$)] The fiber of $\widetilde{\pi}: \widetilde{M}_{\delta} \rightarrow [0,1]$ over $t =0$ is equal to $W(\delta)$.
 
  \item[(iii$^*$)] If $W'(\delta)$ is the fiber of $\widetilde{\pi}: \widetilde{M}_{\delta} \rightarrow [0,1]$ over $t =1$, then $W'(\delta)$ is a $0$-simplex of
  $\Psi_d\big((a - \delta, a + \delta)\times \mathbb{R}^{N-1}\big)_{\bullet}$ which does not intersect the hyperplane $\{a\}\times \mathbb{R}^{N-1}$.

 \end{itemize} 

Properties (i$^*$) and (ii$^*$) follow easily from property (i) of the PL map $r$ that we fixed at the beginning of this step. To verify that property (iii$^*$) also holds, we first note that property (ii) of the PL map $r$ implies that $\widetilde{M}_{\delta}$ is equal to the concordance $\widetilde{W}(\delta)$ when we restrict the background space to $(a - \frac{\delta}{4}, a + \frac{\delta}{4})\times \mathbb{R}^{N-1}$. In particular, since the concordance $\widetilde{W}(\delta)$  is equal to $\varnothing$ at $t=1$, the fiber $W'(\delta)$ becomes empty when we intersect it with the open set  $(a - \frac{\delta}{4}, a + \frac{\delta}{4})\times \mathbb{R}^{N-1}$. In particular, $W'(\delta)$ and
$\{a\} \times \mathbb{R}^{N-1}$ are disjoint.  

\textit{Step 4.} Finally, we will give the details of the gluing process that we described at the beginning of this proof. As we observed earlier, the fibers of $\widetilde{M}_{\delta}$ are $M$-standard in $(a - \delta, a - \frac{\delta}{2} )\times \mathbb{R}^{k-1}$ and $(a + \frac{\delta}{2}, a + \delta) \times \mathbb{R}^{k-1}$. Also, since the fibers of the constant concordance $[0,1]\times W$ are $M$-standard in $(a-\delta, a + \delta)\times \mathbb{R}^{k-1}$, the intersection 
\begin{equation} \label{cutting.out}
\big( [0,1]\times W\big) \cap \big( [0,1]\times(\mathbb{R} - [a - \frac{\delta}{2}, a + \frac{\delta}{2}])\times \mathbb{R}^{N-1} \big)
\end{equation}
will be an element of $\Psi_d\big( (\mathbb{R} - [a - \frac{\delta}{2}, a + \frac{\delta}{2}])\times \mathbb{R}^{N-1} \big)([0,1])$ whose fibers are also 
$M$-standard in  $(a - \delta, a - \frac{\delta}{2} )\times \mathbb{R}^{k-1}$ and $(a + \frac{\delta}{2}, a + \delta) \times \mathbb{R}^{k-1}$. 
Then, by the gluing property of the sheaf $\Psi_d: \mathcal{O}(\mathbb{R}^N) \rightarrow \mathbf{Ssets}$
(see Remark \ref{spacerem2}), we can glue together 
$\widetilde{M}_{\delta}$ and the manifold given in (\ref{cutting.out})
 to produce an element of $\Psi_d(\mathbb{R}^N)([0,1])$ which agrees with $\widetilde{M}_{\delta}$ in 
 $(a-\delta, a + \delta)\times \mathbb{R}^{N-1}$ and with $[0,1]\times W$ in 
 $(\mathbb{R} - [a - \frac{\delta}{2}, a + \frac{\delta}{2}])\times \mathbb{R}^{N-1}$. 
 As we indicated at the beginning of this proof, we define $\widetilde{W}$ to be the concordance obtained via this gluing argument. Evidently, $\widetilde{W}$ satisfies conditions (i) and (ii) listed in the statement of this lemma. To finish this proof, we need to explain why $\widetilde{W}$ is an element of $\psi_d(N,k)([0,1])$. To see this, note that any fiber of the projection $\widetilde{\pi}: \widetilde{M}_{\delta} \rightarrow [0,1]$ is contained in 
 $(a-\delta, a + \delta)\times \mathbb{R}^{k-1}\times (0,1)^{N-k}$. In particular, any fiber of this projection will be contained in 
 $\mathbb{R}^{k}\times (0,1)^{N-k}$. Since the same is true for the standard projection $[0,1]\times W \rightarrow [0,1]$, it follows that any fiber of the projection $\pi: \widetilde{W}\rightarrow [0,1]$ is contained in  $\mathbb{R}^{k}\times (0,1)^{N-k}$. In other words, 
 $\widetilde{W}$ is an element of  $\psi_d(N,k)([0,1])$.
\end{proof}

Note that the concordance $\widetilde{M}_{\delta}$ 
that we constructed in Steps 1, 2, and 3 of the previous proof is 
independent of the behavior of the manifold $W$ outside of 
 $(a - \delta, a + \delta)\times \mathbb{R}^{N-1}$. 
 Any $0$-simplex of $\psi_d(N,k)_{\bullet}$ which is $M$-standard in 
 $(a - \delta, a + \delta)\times \mathbb{R}^{k-1}$ would give the same $\widetilde{M}_{\delta}$. 
 From now on, any concordance of the form  $\widetilde{M}_{\delta}$ 
 will be called a \textit{standard $M$-concordance.} 

\theoremstyle{definition} \newtheorem{example.non.fib}[scan]{Remark}

\begin{example.non.fib} \label{example.non.fib} 
We can use the construction given in Step 1 of the previous proof to produce an element 
$W$ in $\psi_d(N,1)_{\bullet}$ which does not have a fiberwise regular value. 
More concretely, we will give an example of an element in the set $\psi_d(N,1)([0,1])$ (which we can identify with $\psi_d(N,1)_1$) with no fiberwise regular values. 
First, let us fix a non-empty null-bordant element $M \in \psi_d(N-1,0)_0$.  By replicating the argument given in Step 1 of the previous proof, we can obtain a concordance $W \in \psi_d(N,1)([0,1])$ with the property that 
$W_0 = \mathbb{R}\times M$ and $W_1 = \varnothing$. 
As we have done several times already, we will denote the standard projections from 
$W \subseteq [0,1] \times \mathbb{R} \times (0,1)^{N-1}$ onto
$[0,1]$ and $\mathbb{R}$ by $\pi$ and $x_1$ respectively. 
We claim that the projection $x_1: W \rightarrow \mathbb{R}$
does not have a fiberwise regular value.  
If a fiberwise regular value $a \in \mathbb{R}$ did exist, then we would have either that 
$(\pi, x_1)^{-1}\big((\lambda, a)\big)$  is PL homeomorphic to $M$ for all
$\lambda \in [0,1]$, or that $(\pi, x_1)^{-1}\big((\lambda, a)\big) = \varnothing$ for 
all $\lambda \in [0,1]$.
But neither of these two options hold for the concordance $W$ that we constructed. Therefore, 
$x_1: W \rightarrow \mathbb{R}$
cannot have any fiberwise regular values. 
\end{example.non.fib}

\theoremstyle{definition} \newtheorem{outline.nestenempty}[scan]{Note}

\begin{outline.nestenempty} \label{outline.nestenempty} 
\textbf{(Outline of the proof of Proposition \ref{nestenempty})}  First, as was the case with Proposition \ref{makereg}, it is enough to prove Proposition \ref{nestenempty} in the case when the 
base-space $P$ is a closed PL manifold. 
Indeed, assuming that Proposition \ref{nestenempty} is true in this special case, one can 
then show that the result is also true for any compact PL space $P$
by doing an argument identical to the one given in Remark \ref{makereg.remark}. 
Let us fix then a closed PL manifold $P$, an element $W \in \psi_d^{\varnothing}(N,k)(P)$, and a value $\beta \in \mathbb{R}$. The overall structure for the proof of Proposition \ref{nestenempty} will be similar to that of Proposition \ref{makereg}. The proof of that proposition had two essential parts: (1) We showed how to produce fiberwise regular values locally, and (2) we then showed how to obtain an open cover $\mathcal{U}$ for the base-space $P$ such that, over each open set in $\mathcal{U}$, it was possible to produce a fiberwise regular value. We will divide the proof of Proposition \ref{nestenempty} into two similar parts: 

\begin{enumerate}
\item For a fixed point $\lambda_0 \in P$, we will show that it is possible to deform $W$ over a neighborhood $V$ of $\lambda_0$  so that all the fibers over $V$ become disjoint from a hyperplane of the form $\{a\}\times \mathbb{R}^{N-1}$.
\item We will then obtain an open cover $\{U_1, \ldots, U_q\}$ of $P$, real values $a_1, \ldots, a_q$, and a concordance $\widetilde{W}$ from $W$ to an element $W'$ with the property that each fiber of $W'$ over $U_j$ does not intersect the hyperplane $\{ a_j\}\times \mathbb{R}^{N-1}$. 
\end{enumerate}

Moreover, we will never change the fibers of $W$ below the height $x_1 = \beta$. The local deformation described in (1) will be obtained via Lemmas \ref{stretching} and \ref{preincempty}. More specifically, 
in these two lemmas, we will achieve the following: 
Fix a point $\lambda_0 \in P$ and two nested PL balls
$V \subset V' \subset P$ of dimension $m$ (where $m = \mathrm{dim}\hspace{0.05cm}P$) such that 
$\lambda_0 \in \mathrm{Int}\hspace{0.05cm}V$ and $V \subset \mathrm{Int}\hspace{0.05cm}V'$. 
In Lemma \ref{stretching}, we will show that  the fibers of $W$ over $V'$ can be made $M$-standard in $(a - \delta, a + \delta)\times \mathbb{R}^{k-1}$ for a suitable 
$M \in \psi_{d-k}(N-k,0)_0$ and suitable values $a \in \mathbb{R}$ and $\delta>0$.
Then, in Lemma \ref{preincempty}, we will use the maneuver from Lemma \ref{maneuver} 
to make each fiber over the neighborhood $V$ disjoint from 
$\{a\}\times \mathbb{R}^{N-1}$. In order to use the maneuver from Lemma \ref{maneuver}, we need to show that the manifold $M$ is null-bordant. We will prove this in Lemma \ref{nullbord}. The data described in (2) will be 
produced using Proposition \ref{longman.k}.  
\end{outline.nestenempty}  

\theoremstyle{definition} \newtheorem{reg.notation}[scan]{Notation}

\begin{reg.notation}  \label{reg.notation}
In order to avoid any notational ambiguities, we shall adopt the following 
conventions throughout the rest of this section: 

\begin{itemize}

\item[$\cdot$] We will use the letter $a$ (or symbols of the form $a_i$ and $a'$) to denote real numbers, 
and we will use the letter $v$ (or alternatives of the form $v_j$ and $v'$) to denote points in $\mathbb{R}^k$.

\item[$\cdot$] Consider an element $W \subseteq P\times \mathbb{R}^k \times (0,1)^{N - k}$ of
$\psi_d(N,k)(P)$ and let $\pi: W \rightarrow P$ be the natural projection onto the base-space $P$. 
From now on, the standard projection $W \rightarrow \mathbb{R}^k$ onto the factor
$\mathbb{R}^k$ will be denoted by $\widetilde{x}_k$. Similarly, the projection 
$W \rightarrow \mathbb{R}$ onto the first coordinate of $\mathbb{R}^k$ will be denoted by $\widetilde{x}_1$. 
On the other hand, for a single fiber $W_{\lambda}$ of $\pi: W \rightarrow P$, we will denote the 
corresponding projections $W_{\lambda} \rightarrow \mathbb{R}^k$ 
and $W_{\lambda} \rightarrow \mathbb{R}$ by $x_k$ and $x_1$ respectively. The purpose of these notational conventions is to make it clear to the reader whether we are projecting from the total space
$W \in \psi_d(N,k)(P)$ or from a single fiber $W_{\lambda}$. 

\end{itemize}

\end{reg.notation}

\theoremstyle{plain} \newtheorem{nullbord}[scan]{Lemma}

\begin{nullbord}  \label{nullbord}
Let $W\subseteq \mathbb{R}^k \times (0,1)^{N-k}$ be a $0$-simplex of 
$\psi_d(N,k)_{\bullet}$. Also, suppose that
$v_0\in \mathbb{R}^k$ is a regular value 
of the projection $x_k:W \rightarrow \mathbb{R}^k$
(in the sense of Definition \ref{regular}) and let
$M_0 := x_k^{-1}(v_0)$. If $W \in \psi_d^{\varnothing}(N,k)_{0}$, then 
$M_0$ must be null-bordant. 
\end{nullbord}

\begin{proof}
The goal of this lemma is to construct a null-bordism $C$ for $M_0$. To do so, let us fix a concordance 
$\widetilde{W}\in \psi_d(N,k)([0,1])$ from $W$ to $\varnothing$. Such a concordance exists because we are assuming that $W$ is a 0-simplex of $\psi_d^{\varnothing}(N,k)_{\bullet}$. Note that $\widetilde{W}$ is a PL manifold with boundary such that $\partial \widetilde{W} = W$. Also, since $v_0$ is a regular value of $x_k: W \rightarrow \mathbb{R}^k$, there exists 
an open PL ball $B(v_0,\delta)\subset \mathbb{R}^k$ (with respect to the norm $\left\| x \right\| = \mathrm{max}\{|x_1|, \ldots, |x_k| \}$) and  a PL homeomorphism $h: M_0 \times B(v_0, \delta) \rightarrow x_k^{-1}\big(B(v_0,\delta)\big)$ such that $x_k\circ h$ is equal to the canonical projection 
$\mathrm{pr}_2: M_0 \times B(v_0, \delta) \rightarrow B(v_0, \delta)$. It follows that the restriction of $x_k: W \rightarrow \mathbb{R}^k$ 
on $x_k^{-1}\big(B(v_0,\delta)\big)$ is a PL submersion of codimension $d-k$. Consequently, we can regard the pre-image $x_k^{-1}\big(B(v_0,\delta)\big)$ as an element of $\psi_{d-k}(N-k,0)\big( B(v_0,\delta) \big)$. 

Now, consider the standard projection $\widetilde{x}_k: \widetilde{W} \rightarrow \mathbb{R}^k$ 
from $\widetilde{W}\subseteq [0,1]\times \mathbb{R}^k \times (0,1)^{N-k}$ to $\mathbb{R}^k$. Since the set of regular values of $\widetilde{x}_k$ is dense in $\mathbb{R}^k$ (see Remark \ref{remark.williamson2}), we can find a regular value $v_1$ for 
$\widetilde{x}_k$ inside the open ball $B(v_0,\delta)$. Moreover, by Remark \ref{remark.williamson}, the pre-image 
$\widetilde{x}_k^{-1}(v_1)$ (which from now on we will denote by $C'$) will be a compact and proper PL submanifold with boundary of $\widetilde{W}$ of dimension $d - k + 1$.  
Consequently, after we identify 
$[0,1]\times \{ v_1\}\times (0,1)^{N-k}$ with $[0,1]\times (0,1)^{N-k}$,  $C'$ will be a compact PL submanifold of $[0,1]\times (0,1)^{N-k}$ such that $\partial C' \subseteq \{0\}\times (0,1)^{N-k}$ and 
$C\cap\big( \{1\}\times (0,1)^{N-k}\big) = \varnothing$ (this last property follows from the fact that $\widetilde{W}_1 = \varnothing$). Therefore, $C'$ is a null-bordism for the pre-image $x_k^{-1}(v_1)$. In the remainder of this proof, we will denote $x_k^{-1}(v_1)$ by $M_1$.
Finally, we will obtain a null-bordism $C$ for $M_0$ via the following steps: 

\begin{enumerate}

\item First, we embed $C'$ in $[1, 2]\times (0,1)^{N-k}$ via the PL homeomorphism $t\times \mathrm{Id}_{(0,1)^{N-k}}: [0,1] \times (0,1)^{N-k} \rightarrow [1,2] \times (0,1)^{N-k}$, where $t$ is the unique increasing linear homeomorphism $[0,1]\rightarrow [1,2]$. Let us denote the image of $C'$ under this embedding by $C''$. 

\item Let $s: [0, 1] \rightarrow B(v_0, \delta)$ be any PL embedding such that $s(0) = v_0$ and $s(1) = v_1$. Recall that we can regard $x_k^{-1}\big(B(v_0,\delta)\big)$ as an element of the set $\psi_{d-k}(N-k,0)\big( B(v_0,\delta) \big)$. By pulling back $x_k^{-1}\big(B(v_0,\delta)\big)$ along the PL embedding $s: [0, 1] \rightarrow B(v_0, \delta)$, we obtain an element of $\psi_{d-k}(N-k,0)([0, 1])$, which we shall denote by $D$. Note that  $D$ is a concordance from $M_0$ to $M_1$ which is PL homeomorphic to $[0,1]\times M_0$. 

\item Finally, we take the union $D \cup C'' \subseteq [0,2]\times (0,1)^{N-k}$ and define the null-bordism $C$ for $M_0$ as the image of 
$D \cup C''$ under the PL homeomorphism $j\times \mathrm{Id}_{(0,1)^{N-k}}: [0,2] \times (0,1)^{N-k} \rightarrow [0,1] \times (0,1)^{N-k}$, where $j$ is the unique increasing linear homeomorphism $[0,2]\rightarrow [0,1]$. 
\end{enumerate}
\end{proof}

Let $P$ be a closed PL manifold and $W$ an element in $\psi_d^{\varnothing}(N,k)(P)$. 
Moreover, fix a point 
$\lambda_0 \in P$ and  two nested PL balls
$V \subset V' \subset P$ of dimension $m = \mathrm{dim}\hspace{0.05cm}P$ satisfying 
$\lambda_0 \in \mathrm{Int}\hspace{0.05cm}V$ and $V \subset \mathrm{Int}\hspace{0.05cm}V'$. 
As we explained in Note \ref{outline.nestenempty}, 
we will show in Lemma \ref{stretching} that it is possible to deform $W$ 
so that its fibers over $V'$ become $M$-standard in a set of the form 
$(a- \delta, a+\delta) \times \mathbb{R}^{k-1}$. Then, in Lemma \ref{preincempty}, we will make the fibers over
$V$ disjoint from the hyperplane $\{a\}\times \mathbb{R}^{N-1}$. Since both Lemma \ref{stretching} and Lemma \ref{preincempty} describe deformations that take place over 
a PL ball of dimension $m = \mathrm{dim}\hspace{0.05cm}P$, 
it is enough to prove these two lemmas in the case when $P = [-1,1]^m$. 
This assumption will make it easier to carry out the local constructions described above. As we have already done in this section, we shall continue to denote by $B(v,\delta)$ the open ball centered at $v\in \mathbb{R}^k$ of radius $\delta$ with respect to the norm $\left\| x \right\| = \mathrm{max}\{|x_1|, \ldots, |x_k| \}$. 
Also, we will denote the point $(0,\ldots,0)\in [-1,1]^m$ by $\mathbf{0}$. 

\theoremstyle{plain} \newtheorem{stretching}[scan]{Lemma}

\begin{stretching} \label{stretching}
Fix an element $W$  of $\psi_d(N,k)([-1,1]^m)$ and  suppose that 
$v = (a_1, \ldots, a_k) \in \mathbb{R}^k$ is a fiberwise regular value for the standard projection $\widetilde{x}_k: W \rightarrow \mathbb{R}^k$. Additionally, let $x_k: W_{\mathbf{0}} \rightarrow \mathbb{R}^k$ be the standard projection from $W_{\mathbf{0}}$ to $\mathbb{R}^k$, and define $M :=x_k^{-1}(v)$. Then, for any values $0 < \epsilon< 1$ and $0 <\widetilde{\delta}$, there exists an element $\widetilde{W} \in \psi_d(N,k)([0,1]\times [-1,1]^m)$ with the following properties: 

\begin{itemize}

\item[(i)] $\widetilde{W}$ is a concordance from $W$ to an element $W' \in \psi_d(N,k)([-1,1]^m)$ with the property that, for each 
$\lambda \in [-\epsilon, \epsilon]^m$, the fiber $W'_{\lambda}$ is $M$-standard in an open set  $(a_1-\delta, a_1 + \delta)\times \mathbb{R}^{k-1}$ for some value $\delta> 0$ satisfying $0< 2\delta < \widetilde{\delta}$. 

\item[(ii)] $\widetilde{W}$ agrees with the constant concordance $[0,1]\times W$ over an open neighborhood of 
$[0,1] \times \partial([-1,1]^m)$. Also, $\widetilde{W}$ will agree with $[0,1]\times W$ when we restrict the background space to 
$\big( (\mathbb{R} - [a_1 - 2\delta, a_1 + 2\delta]\big) \times \mathbb{R}^{N-1}$.   

\end{itemize}

\end{stretching}

\begin{proof}
We will also split up this proof into several steps.

\textit{Step 1.} Let us first review some facts that we discussed in \S \ref{section4}. Let $\pi: W \rightarrow [-1,1]^m$ be the projection from $W$ to $[-1,1]^m$. Recall that $\pi$ must be a PL submersion of codimension $d$. As we explained in Remark \ref{remark.fib.reg}, since $v = (a_1, \ldots, a_k)$ is a fiberwise regular value of $\widetilde{x}_k: W \rightarrow \mathbb{R}^k$, we can find an open ball $B(v,\delta)$ centered at $v$ so that the restriction of the projection 
$(\pi, \widetilde{x}_k): W \rightarrow [-1,1]^m\times \mathbb{R}^k$ on the pre-image 
$(\pi,\widetilde{x}_k)^{-1}\big( [-1,1]^m \times B(v,\delta) \big)$ is a PL bundle whose fibers are PL homeomorphic to $M = x_k^{-1}(v)$. 
In particular, we can regard $(\pi,\widetilde{x}_k)^{-1}\big( [-1,1]^m \times B(v,\delta) \big)$
as an element of $\psi_{d-k}(N-k,0)\big( [-1,1]^m \times B(v,\delta) \big)$.  From now on, we will denote this element by 
$\widetilde{M}$. Also, we will assume that the radius $\delta$ of $B(v,\delta)$ is small enough 
so that $2\delta < \widetilde{\delta}$. 

\textit{Step 2.} In this step, we will transform $W$ so that every fiber over $[-\epsilon, \epsilon]^m$ becomes $M$-standard in $B(v,\delta)$, possibly after shrinking the value $\delta$. In fact, we will accomplish this for a slightly bigger radius than $\epsilon$. Let us then fix values $\epsilon'$, $\epsilon'', \delta', \delta''$ such that $0 < \epsilon < \epsilon' < \epsilon'' < 1$ and $0 < \delta' < \delta'' < \delta$.  
Also, consider the constant concordance $[0,1]\times W$, which is a closed PL subspace of $[0,1]\times [-1,1]^m\times \mathbb{R}^k\times (0,1)^{N-k}$. To avoid introducing more notation, we will also denote the standard projections $[0,1]\times W \rightarrow [0,1]\times [-1,1]^m$ and $[0,1]\times W \rightarrow \mathbb{R}^k$ by $\pi$ and $\widetilde{x}_k$ respectively. Note that the pre-image of $[0,1]\times [-1,1]^m\times B(v,\delta)$ under the projection $(\pi,\widetilde{x}_k):[0,1]\times W \rightarrow [0,1]\times [-1,1]^m\times \mathbb{R}^k$ is equal to $[0,1]\times \widetilde{M}$, where $\widetilde{M}$ is the element of $\psi_{d-k}(N-k,0)\big( [-1,1]^m \times B(v,\delta) \big)$ that we constructed in Step 1. In particular, 
$[0,1]\times \widetilde{M}$ is an element of the set $\psi_{d-k}(N-k,0)\big([0,1]\times [-1,1]^m \times B(v,\delta) \big)$.  To make the fibers of $W$ over $[-\epsilon, \epsilon]^m$ $M$-standard, we are going to pull back $[0,1]\times \widetilde{M}$ along a PL map 
\[
f: [0,1]\times [-1,1]^m \times B(v,\delta) \rightarrow  [0,1]\times [-1,1]^m \times B(v,\delta) 
\]
which satisfies the following:

\begin{itemize}

\item[(i)] $f|_{\{0\}\times [-1,1]^m \times B(v,\delta)} = \mathrm{Id}_{\{0\}\times [-1,1]^m \times B(v,\delta)}$.  

\item[(ii)] $f$ maps every point in $\{1\} \times [-\epsilon',\epsilon']^m\times B(v,\delta')$ to $(1,\mathbf{0}, v)$. 

\item[(iii)]$f$ fixes every point outside of $[0,1]\times [-1,1]^m \times B(v,\delta'')$.

\item[(iv)] $f$ fixes every point outside of $[0,1]\times [-\epsilon'',\epsilon'']^m \times B(v,\delta)$. 

\end{itemize}

Fix such a PL map $f$ and take the pull-back of $[0,1]\times \widetilde{M}$ along $f$. 
We will denote this pull-back by 
$\widetilde{M}_f$. Evidently, 
$\widetilde{M}_f$ is an element of $\psi_{d-k}(N-k,0)\big([0,1]\times [-1,1]^m \times B(v,\delta) \big)$. Moreover, the standard projection 
$\widetilde{M}_f \rightarrow [0,1]\times [-1,1]^m$ is a PL submersion of codimension $d$, so we can also regard 
$\widetilde{M}_f$ as an element of the set  
\[
\Psi_d(U_1)([0,1]\times [-1,1]^m),
\]
where $U_1 = B(v,\delta)\times \mathbb{R}^{N-k}$. 
Now, let $\overline{B(v,\delta'')}$ be the closure of the open ball $B(v,\delta'')$ and define $U_2 = \big( \mathbb{R}^k - \overline{B(v,\delta'')}\big)\times \mathbb{R}^{N-k}$. We will use the open sets $U_1$ and $U_2$ to 
define the following elements: 

\begin{itemize}

\item[$\cdot$] $W_{U_1}$ will be the element of $\Psi_d(U_1)([-1,1]^m)$ obtained by intersecting $W$ with 
$[-1,1]^m\times U_1$.

\item[$\cdot$]  $([0,1]\times W)_{U_2}$ will be the element of the set $\Psi_d(U_2)([0,1] \times [-1,1]^m)$ obtained by intersecting $[0,1]\times W$
with $[0,1]\times [-1,1]^m\times U_2$. We define $([0,1]\times W)_{U_1}$ in a similar fashion. 
\end{itemize}

As an element of $\Psi_d(U_1)([0,1]\times [-1,1]^m)$, $\widetilde{M}_f$ is a concordance from $W_{U_1}$ to an element of 
$\Psi_d(U_1)([-1,1]^m)$ whose fibers over $[-\epsilon', \epsilon']^m$ are $M$-standard in $B(v,\delta')$. This is a consequence of properties (i) and (ii) of the PL map $f$. 
Also, if we define $U_3$ as $U_3 := U_1\cap U_2$, then property (iii) of the PL map $f$ implies that 
$\widetilde{M}_f$ and  
$([0,1]\times W)_{U_2}$ become equal when we intersect both concordances with 
$[0,1]\times [-1,1]^m\times U_3$. 
Thus, using the gluing property of the sheaf $\Psi_d: \mathcal{O}(\mathbb{R}^N) \rightarrow \mathbf{PL}\text{-}\mathbf{Sets}$ introduced in Remark \ref{spacerem2}, we can glue together $\widetilde{M}_{f}$ and $([0,1]\times W)_{U_2}$ to produce a concordance 
$\widehat{W} \in \psi_d(N,k)([0,1]\times [-1,1]^m)$ from $W$ to an element $W''$ whose fibers over $[-\epsilon', \epsilon']^m$ are $M$-standard in $B(v,\delta')$. 
Moreover, by the way we performed this construction, $\widehat{W}$ agrees with $[0,1]\times W$ when we restrict the background space to $U_2$. Also, by property (iv) of the PL map $f$, we have
that $\widetilde{M}_f$ agrees with $([0,1]\times W)_{U_1}$ over the complement of $[0,1] \times [-\epsilon'',\epsilon'']^m$. Therefore, 
$\widehat{W}$ and $[0,1]\times W$ will also be equal outside of  $[0,1] \times [-\epsilon'',\epsilon'']^m$. For the next step, we will relabel the radius $\delta'$ as $\delta$. 
Note that we have maintained the property $2\delta < \widetilde{\delta}$. 

\textit{Step 3.} Consider the element $W'' \in \psi_d(N,k)([-1,1]^m)$ concordant to $W$ 
that we obtained in Step 2. Recall that the fibers of $W''$ over $[-\epsilon', \epsilon']^m$ 
are $M$-standard in $B(v,\delta)$. 
In this step, we will deform $W''$ (via a concordance) so that its fibers over $[-\epsilon, \epsilon]^m$ become $M$-standard in the strip $(a_1-\delta, a_1 + \delta)\times \mathbb{R}^{k-1}$, where $a_1$ is the first component of the point $v = (a_1, \ldots, a_k)$. 
To do this, let us fix an isotopy of open PL embeddings 
$e: [0,1]\times \mathbb{R}^{k} \rightarrow [0,1]\times \mathbb{R}^{k}$ such that 
$e_0 = \mathrm{Id}_{\mathbb{R}^{k}}$ and 
\[
e_1(\mathbb{R}^{k}) = \mathbb{R} \times \Big(\prod_{j=2}^{k} (a_j - \delta, a_j + \delta) \Big).
\]
In other words, for each $j=2, \ldots, k$, the isotopy $e$ will compress the $x_j$-axis into $(a_j - \delta, a_j + \delta)$ while keeping the first coordinate fixed. Next, let $E$ be the PL map from $[0,1]\times [-1,1]^m\times \mathbb{R}^k \times (0,1)^{N-k}$ to itself which satisfies the following: 

\begin{itemize}

\item[(i)] $E$ is an isotopy of open PL embeddings $\mathbb{R}^k\times (0,1)^{N-k} \hookrightarrow \mathbb{R}^k\times (0,1)^{N-k}$ parameterized by $[0,1]\times [-1,1]^m$.

\item[(ii)] For any $\lambda \in [-1,1]^m$, we have that $E_{(t,\lambda)} = e_t \times \mathrm{Id}_{(0,1)^{N-k}}$.

\end{itemize}

After fixing a suitable PL bump function $\phi:[-1,1]^m \rightarrow [0,1]$,
we can apply the technique described in Remark \ref{dampening.isotopies}
to dampen the isotopy $E$ in order to obtain a new isotopy $F$
of open PL embeddings 
$\mathbb{R}^k \times (0,1)^{N-k} \hookrightarrow \mathbb{R}^k \times (0,1)^{N-k}$  parameterized by  
$[0,1]\times [-1,1]^m$ which agrees with $E$ over $[0,1]\times [-\epsilon,\epsilon]^m$, and agrees with the identity isotopy over the 
union
\begin{equation} \label{union.stretch}
\Big( \{0\} \times [-1,1]^m \Big) \cup \Big( [0,1] \times \big([-1,1]^m - [-\epsilon', \epsilon']^m\big)\Big).
\end{equation}

By Proposition \ref{pullemb}, the pre-image $F^{-1}(\widehat{W})$ 
(which from now on we will denote by $\widehat{W}'$) 
will be an element of $\psi_d(N,k)([0,1]\times[-1,1]^m)$. Since $F$ agrees with the identity isotopy over 
$ \{0\} \times [-1,1]^m $ and with $E$ over $[0,1]\times [-\epsilon,\epsilon]^m$, the element $\widehat{W}'$ will be a concordance
from $W''$ to an element $W'$ of $\psi_d(N,k)([-1,1]^m)$ such that, for each $\lambda \in [-\epsilon, \epsilon]^m$, the fiber 
$W'_{\lambda}$ is $M$-standard in $(a_1-\delta, a_1 + \delta)\times \mathbb{R}^{k-1}$. Moreover, since $F$ also agrees with the identity isotopy over $[0,1] \times \big([-1,1]^m - [-\epsilon', \epsilon']^m\big)$, the concordance $\widehat{W}'$ will agree with 
$[0,1]\times W$ over a neighborhood of the product $[0,1]\times \partial([-1,1]^m)$. Therefore, by concatenating the concordances 
$\widehat{W}$ from Step 2 and $\widehat{W}'$, we obtain a concordance $\widetilde{W}$ satisfying conditions (i) and (ii) given in the statement of this lemma. 
\end{proof}

\theoremstyle{definition}  \newtheorem{inter.remark}[scan]{Remark}

\begin{inter.remark} \label{inter.remark}
Now that we have proven Lemma \ref{stretching}, we can address one of the details of the proof of Proposition \ref{propmonoid}. Namely, suppose that $W$ is a $0$-simplex of $\psi_d(N,k)_{\bullet}$, let $v = (a_1, \ldots, a_k) \in \mathbb{R}^k$ be a regular value for the standard projection $x_k: W \rightarrow \mathbb{R}^k$,  and define $M := x_k^{-1}(v)$. In the last part of the proof of Proposition 
\ref{propmonoid}, we claimed that $W$ is concordant to the element 
$\mathbb{R}^k\times M \in \psi_d(N,k)_0$. 
We can prove this as follows. First, we claim that $W$ is concordant to an element of 
$\psi_d(N,k)_0$ which is $M$-standard 
in an open ball $B(v,\delta)$ for some suitable value $\delta>0$. 
This is similar to what we wanted to prove in  
Step 2 of Lemma \ref{stretching}.  In that step, we had a family of manifolds 
$W_{\lambda} \in \psi_d(N,k)_0$ parameterized by $[-1,1]^m$ and a 
fiberwise regular value $v$ for the projection onto $\mathbb{R}^k$. 
Then, using the fiberwise regularity of $v$, we made all the fibers
$W_{\lambda}$ over $[-\epsilon', \epsilon']^m$ $M$-standard in an open ball $B(v,\delta)$.
On the other hand, in this remark, we are working with a single manifold $W$ (in other words, we can view $W$ as a family parameterized by a point $*$). 
Thus, by adjusting the argument given in Step 2 of Lemma \ref{stretching} to the case when we have a single manifold, we can obtain a concordance $\widetilde{W} \in \psi_d(N,k)([0,1])$ from $W$ to an element 
$W' \in \psi_d(N,k)_0$ which is $M$-standard in some $B(v,\delta)$.  Now, pick an isotopy 
of open PL embeddings $e:[0,1]\times \mathbb{R}^k \rightarrow [0,1]\times \mathbb{R}^k$  such that
$e_0= \mathrm{Id}_{\mathbb{R}^k}$ and $e_1(\mathbb{R}^k) =  B(v,\delta)$,  and define a new isotopy of embeddings
\[
E: [0,1]\times \mathbb{R}^k \times (0,1)^{N-k} \longrightarrow [0,1]\times \mathbb{R}^k \times (0,1)^{N-k}
\]
by setting $E = e \times \mathrm{Id}_{(0,1)^{N-k}}$. 
By Proposition \ref{pullemb}, the pre-image $E^{-1}([0,1]\times W')$ is an element of 
the set $\psi_d(N,k)([0,1])$. In fact, by the way we defined the isotopy $E$, we have that $E^{-1}([0,1]\times W')$ is a concordance from 
$W'$ to the element $\mathbb{R}^k \times M$. Therefore, by concatenating the concordances $\widetilde{W}$ and 
$E^{-1}([0,1]\times W')$, we obtain a concordance from $W$ to $\mathbb{R}^k \times M$. 
\end{inter.remark}

\theoremstyle{plain}  \newtheorem{preincempty}[scan]{Lemma}

\begin{preincempty}  \label{preincempty}
Fix an element $W \in \psi_d^{\varnothing}(N,k)([-1,1]^m)$ and let $v = (a_1, \ldots, a_k)$ be as in Lemma \ref{stretching}. Given any values $0 < \epsilon < 1$ and $0< \widetilde{\delta}$, there exists an element 
$\widetilde{W} \in \psi_d^{\varnothing}(N,k)([0,1]\times [-1,1]^m)$ with the following properties: 

\begin{itemize}

\item[(i)] $\widetilde{W}$ is a concordance from $W$ to an element 
$W' \in \psi_d^{\varnothing}(N,k)([-1,1]^m)$ with the property that, for each 
$\lambda \in [-\epsilon, \epsilon]^m$, the fiber $W'_{\lambda}$ does not intersect the hyperplane $\{a_1\}\times \mathbb{R}^{N-1}$. 

\item[(ii)] $\widetilde{W}$ satisfies property (ii) from Lemma \ref{stretching}. In particular, $\widetilde{W}$ agrees with $[0,1]\times W$ when we restrict the background space to
$\big( (\mathbb{R} - [a_1 - 2\delta, a_1 + 2\delta]\big) \times \mathbb{R}^{N-1}$ for some value $\delta$ such that $0 < 2\delta < \widetilde{\delta}$. 

\end{itemize}

\end{preincempty} 

\begin{proof}
As we did in Lemma \ref{stretching}, we will denote the point $(0,\ldots,0)\in [-1,1]^m$ by $\mathbf{0}$.
Also, let $x_k:W_{\mathbf{0}}\rightarrow \mathbb{R}^k$ be the standard projection onto $\mathbb{R}^k$, 
and denote the pre-image $x_k^{-1}(v)$ by $M$. 
Since we can regard the fiber $W_{\mathbf{0}}$ as a $0$-simplex of  $\psi_d^{\varnothing}(N,k)_{\bullet}$, we have by Lemma \ref{nullbord} that the manifold $M$ is null-bordant, in the sense of Definition \ref{null.bordant}.  Now, let $\widetilde{\delta}$ be the value given in this lemma and fix a value 
$\epsilon'$ such that $\epsilon < \epsilon' < 1$. 
By virtue of Lemma \ref{stretching}, we may assume that each fiber of $W$ over 
$[-\epsilon', \epsilon']^m$ is $M$-standard in $(a_1 - \delta, a_1 + \delta)\times \mathbb{R}^{k-1}$, 
where $\delta>0$ is a value such that $0 < 2\delta < \widetilde{\delta}$. 
The idea of this proof is to perform the maneuver from Lemma \ref{maneuver} to make each fiber over 
$[-\epsilon, \epsilon]^m$ disjoint from the hyperplane $\{a_1\}\times \mathbb{R}^{N-1}$. More rigorously, we will do the following. 
 
 \textit{Step 1.} First, to make our notation less cumbersome, we will denote the open sets 
\[
(a_1 - \delta, a_1 + \delta)\times \mathbb{R}^{N-1} \quad \text{and} \quad \big( (-\infty, a_1 - \frac{\delta}{2}) \cup (a_1 + \frac{\delta}{2}, \infty) \big)\times \mathbb{R}^{N-1}
\]
in $\mathbb{R}^N$ by $U_1$ and $U_2$ respectively. 
Also, let us denote the $m$-cube $[-\epsilon', \epsilon']^m$ by $V'$.  
Now, consider the standard $M$-concordance $\widetilde{M}_{\delta}$ that we constructed  in the proof of Lemma \ref{maneuver}. Recall that $\widetilde{M}_{\delta}$ is an element of $\Psi_d(U_1)([0,1])$. We will use $\widetilde{M}_{\delta}$, $V'$, $U_1$, and $U_2$ to define the following two elements $\widetilde{M}(\epsilon', U_1)$ and $\widetilde{W}(\epsilon', U_2)$: 

\begin{enumerate}

\item $\widetilde{M}(\epsilon', U_1)$ will be the element of $\Psi_d(U_1)([0,1]\times V')$ obtained by taking the image of 
the product $\widetilde{M}_{\delta}\times V'$ under the obvious permutation homeomorphism 
$[0,1]\times U_1 \times V' \stackrel{\cong}{\longrightarrow} [0,1] \times V' \times U_1$. 

\item Next, let $W_{V'}$ be the restriction of $W$ over the $m$-cube $V'$. We define $\widetilde{W}(\epsilon', U_2)$ to be the element of 
$\Psi_d(U_2)([0,1]\times V')$ obtained by intersecting the constant concordance 
$[0,1] \times W_{V'}$ with $[0,1]\times V' \times U_2$. 

\end{enumerate}

Recall that all the fibers of the concordance $\widetilde{M}_{\delta}$ are $M$-standard in the open set
\begin{equation} \label{M.standard.set}
\big((a_1 - \delta, a_1 - \frac{\delta}{2}) \cup (a_1 + \frac{\delta}{2}, a_1 + \delta)\big)\times \mathbb{R}^{k-1}.
\end{equation}

However, we are also assuming that the fibers of $W$ are $M$-standard in the open set $(a_1 - \delta, a_1 + \delta)\times \mathbb{R}^{k-1}$. It follows that 
the fibers of the elements $\widetilde{M}(\epsilon', U_1)$ and $\widetilde{W}(\epsilon', U_2)$ are also $M$-standard in the set given
in (\ref{M.standard.set}), which implies that $\widetilde{M}(\epsilon', U_1)$ and $\widetilde{W}(\epsilon', U_2)$ become equal when we intersect the fibers of both elements with $U_1\cap U_2$.
Therefore, using the gluing property of the sheaf
$\Psi_d: \mathcal{O}(\mathbb{R}^N) \rightarrow \mathbf{PL}\text{-}\mathbf{Sets}$ defined 
in Remark \ref{spacerem2}, we can glue 
$\widetilde{M}(\epsilon', U_1)$ and $\widetilde{W}(\epsilon', U_2)$ 
together to produce an element $\widehat{W}$ 
of $\psi_d^{\varnothing}(N,k)([0,1]\times V')$ which agrees with  $\widetilde{M}(\epsilon', U_1)$ when 
we restrict the background space to
$U_1$ and with $\widetilde{W}(\epsilon', U_2)$ when 
we restrict the background space to $U_2$. Note that all the fibers of 
$\widetilde{M}(\epsilon', U_1)$ over $\{1\}\times V'$ are disjoint from $\{a_1\}\times \mathbb{R}^{N-1}$. Consequently, the same property holds for all the fibers of $\widehat{W}$ over $\{1\}\times V'$. 

\textit{Step 2.} Now, let $V$ denote the $m$-cube $[-\epsilon, \epsilon]^m$. In this step, we will obtain our desired concordance $\widetilde{W} \in \psi_d^{\varnothing}(N,k)([0,1]\times [-1,1]^m)$ by pulling back the concordance $\widehat{W}$ along a PL map $f: [0,1]\times V' \rightarrow [0,1]\times V'$ with the following properties:

\begin{itemize}

\item[(i)] $f(t,x) = (0,x)$ for all $(t,x)$ in $[0,1]\times V''$,
where $V''$ is some compact neighborhood of $\partial V'$ such that $V''\cap V = \varnothing$. For example, $V''$ could be a collar of $\partial V'$ disjoint from $V$. 

\item[(ii)] $f(t,x) = (t,x)$ for all $(t,x)$ in $( \{0\} \times V' ) \cup ([0,1]\times V)$. 
 
\end{itemize}

As indicated earlier, we define $\widetilde{W}$ as the pull-back $f^*\widehat{W}$. Since $\widehat{W}$ is an element of 
$\psi_d^{\varnothing}(N,k)([0,1]\times V')$, then so is $\widetilde{W}$. 
However, property (i) of the map $f$ implies that 
$\widetilde{W}$ agrees with the constant concordance $[0,1]\times W$ over $[0,1]\times V''$. Therefore, we can make $\widetilde{W}$ into an element of $\psi_d^{\varnothing}(N,k)([0,1]\times [-1,1]^m)$ by setting $\widetilde{W} = [0,1]\times W$ over the set 
$[0,1] \times ([-1,1]^m - V')$. 
Also, by property (ii) of the map $f$, $\widetilde{W}$ agrees with $\widehat{W}$ over $[0,1]\times V$. In particular, the fibers of $\widetilde{W}$ over $\{1\} \times V$ are disjoint from the hyperplane $\{a_1\}\times \mathbb{R}^{N-1}$. Finally,
in Step 1, we obtained $\widehat{W}$ from $[0,1]\times W_{V'}$ 
without changing the fibers of $[0,1]\times W_{V'}$ in the open set $U_2$. 
It follows that the concordance $\widetilde{W}$ agrees with $[0,1]\times W$ 
when we restrict the background space to $U_2$. Thus, $\widetilde{W}$ 
satisfies all the conditions listed in the statement of this lemma. 
\end{proof}

\theoremstyle{definition}  \newtheorem{perturb.reg}[scan]{Remark}

\begin{perturb.reg}  \label{perturb.reg}
Let $W \subseteq P \times \mathbb{R}^k \times (0,1)^{N-k}$ be some element of 
$\psi_d(N,k)(P)$, and let $\pi: W \rightarrow P $ be the natural
projection onto $P$. As we explained at the end of Remark \ref{remark.fib.reg}, 
if the base-space $P$ is compact, we have the following alternative way of characterizing fiberwise regular values of the projection $\widetilde{x}_k: W \rightarrow \mathbb{R}^k$: \textit{a point $v_0$ in $\mathbb{R}^k$ is a fiberwise regular value of $\widetilde{x}_k$ if and only if there exists a $\delta>0$ such that the pre-image 
$\widetilde{x}_k^{-1}\big( B(v_0,\delta)\big)$ is an element of $\psi_{d-k}(N-k,0)\big(P\times B(v_0,\delta)\big)$}. In this statement, we are assuming that the projection from $\widetilde{x}_k^{-1}\big( B(v_0,\delta)\big)$ onto 
$P\times B(v_0,\delta)$ is the restriction of $(\pi, \widetilde{x}_k)$ on $\widetilde{x}_k^{-1}\big( B(v_0,\delta)\big)$.
Let us assume then that $P$ is compact and suppose that $v_0$ is a fiberwise regular value of 
$\widetilde{x}_k: W \rightarrow \mathbb{R}^k$. Also, let $\delta>0$ be a value whose existence is 
ensured by the equivalence stated above, and let us denote the pre-image
$\widetilde{x}_k^{-1}\big( B(v_0,\delta)\big)$ by $W_0$. 
The purpose of this remark is to show that any other point in
$B(v_0, \delta)$ is also a fiberwise regular value of $\widetilde{x}_k: W \rightarrow \mathbb{R}^k$. 
Indeed, pick any point $v_1 \in B(v_0, \delta)$ and 
let $\delta'>0$ be any value such that $B(v_1, \delta')\subset B(v_0, \delta)$. 
The restriction of  the projection 
$(\pi, \widetilde{x}_k)|_{W_0}: W_0 \rightarrow P\times B(v_0,\delta)$ 
on $W_1 := \widetilde{x}_k^{-1}\big( B(v_1,\delta')\big)$ is equal to the projection 
$(\pi, \widetilde{x}_k)|_{W_1}: W_1 \rightarrow P\times B(v_1,\delta')$. Since 
$(\pi, \widetilde{x}_k)|_{W_0}$ is a PL submersion of codimension $d-k$, then so is 
$(\pi, \widetilde{x}_k)|_{W_1}$. Therefore, $W_1$ is an element of 
$\psi_{d-k}(N-k,0)\big(P\times B(v_1,\delta')\big)$, which 
(according to the equivalence stated above) implies that 
$v_1$ is also a fiberwise regular value of $\widetilde{x}_k: W \rightarrow \mathbb{R}^k$.  
It follows that, for any compact PL space $P$ and any $W \in \psi_d(N,k)(P)$,
the set of fiberwise regular values of $\widetilde{x}_k: W \rightarrow \mathbb{R}^k$ is open in 
$\mathbb{R}^k$. Thus, if $v_0$ is a fiberwise regular value of $\widetilde{x}_k: W \rightarrow \mathbb{R}^k$, any
sufficiently small perturbation of $v_0$ will also be a fiberwise regular value of $\widetilde{x}_k$, 
as long as the base-space $P$ is compact. We will use these 
observations in the following argument.  
\end{perturb.reg}

\begin{proof}[\textbf{Proof of Proposition \ref{nestenempty}}]
Fix a closed PL manifold $P$, an element $W \in \psi_d^{\varnothing}(N,k)(P)$, and a value 
$\beta\in \mathbb{R}$. By applying Lemma \ref{longman.k}, we can find a concordance 
$\widehat{W} \in \psi_d^{\varnothing}(N,k)([0,1]\times P)$ and a finite collection 
$\mathcal{U} = \{U_1', \ldots, U_q'\}$ of open sets in $P$ which satisfy the following conditions: 

\begin{itemize}

\item[(i)] $\mathcal{U} = \{U_1', \ldots, U_q'\}$ covers $P$. 

\item[(ii)] $\widehat{W}$ is a concordance from $W$ to an element 
$W''\in \psi_d^{\varnothing}(N,k)(P)$ with the property that, for each $U_i' \in \mathcal{U}$, 
the standard projection $\widetilde{x}_k: W''_{U_i'} \rightarrow \mathbb{R}^k$ has a fiberwise regular value 
$v_i = (a^i_1, \ldots, a^i_k) \in \mathbb{R}^k$ such that $a^i_1 > \beta$. 

\item[(iii)] $\widehat{W}$ agrees with the constant concordance $[0,1]\times W$
when we restrict the background space to $(-\infty,\beta)\times \mathbb{R}^{N-1}$. 
\end{itemize}

Next, by the compactness of $P$, we can find two finite collections 
$\{V_{\alpha}\}_{\alpha \in \Lambda}$,  $\{V'_{\alpha}\}_{\alpha \in \Lambda}$ of subsets of $P$
(indexed by the same set $\Lambda$) subordinate to $\mathcal{U}$ with the following properties:

\begin{itemize}

\item[(i$^*$)]  For each $\alpha \in \Lambda$, both $V_{\alpha}$ and $V'_{\alpha}$ are PL balls of dimension $m$, 
where $m$ is the dimension of the manifold $P$. 

\item[(ii$^*$)] $V_{\alpha} \subset \mathrm{Int}\hspace{0.03cm}V'_{\alpha}$ for each $\alpha \in \Lambda$.  

\item[(iii$^*$)] The collection $\{\mathrm{Int}\hspace{0.03cm}V_{\alpha}\}_{\alpha \in \Lambda}$ is an open cover for $P$.

\end{itemize}

From property (iii) of $\widehat{W}$ and $\mathcal{U}$ it follows that, 
for each $\alpha \in \Lambda$, the projection $\widetilde{x}_k: W''_{V'_{\alpha}} \rightarrow \mathbb{R}^k$ has a fiberwise regular value.
Pick such a fiberwise regular value for $\widetilde{x}_k: W''_{V'_{\alpha}} \rightarrow \mathbb{R}^k$  and 
denote it by $v_{\alpha} = (a^{\alpha}_1, \ldots, a^{\alpha}_k)$. Recall that the first coordinate of $v_{\alpha}$  satisfies  $a^{\alpha}_1 > \beta$. 
By the discussion given in Remark \ref{perturb.reg}, any sufficiently small perturbation 
of $v_{\alpha}$ will also be a fiberwise regular value of
$\widetilde{x}_k: W''_{V'_{\alpha}} \rightarrow \mathbb{R}^k$. Thus, if necessary,  
we can perturb the fiberwise regular values $v_{\alpha}$ to ensure
that there are no repetitions in the set of real numbers $\{a_1^{\alpha}\}_{\alpha \in \Lambda}$. We can perform this perturbation without altering the condition $a^{\alpha}_1 > \beta$. 
Next, pick a small enough value 
$\widetilde{\delta}> 0$ so that all the intervals
in the collection $\{[a^{\alpha}_1 - \widetilde{\delta}, a^{\alpha}_1 + \widetilde{\delta}]\}_{\alpha \in \Lambda}$ are pairwise disjoint and so that $\beta < a^{\alpha}_1 - \widetilde{\delta}$ for all $\alpha \in \Lambda$. 
Since all the intervals  
$[a^{\alpha}_1 - \widetilde{\delta}, a^{\alpha}_1 + \widetilde{\delta}]$ 
are pairwise disjoint, we can apply
Lemma \ref{preincempty} to define
inductively an element 
$\widehat{W}' \in \psi_d^{\varnothing}(N,k)([0,1]\times P)$
satisfying the following:

\begin{itemize}

\item[$\cdot$] $\widehat{W}'$ is a concordance between $W''$ and an element $W'$ in 
$\psi_d^{\varnothing}(N,k)(P)$ such that, for each $\alpha \in \Lambda$, we have
\[
W'_{\lambda} \cap \big( \{ a_1^{\alpha} \} \times \mathbb{R}^{N-1}  \big) = \varnothing
\]
for all points $\lambda \in V_{\alpha}$.

\item[$\cdot$] $\widehat{W}'$ agrees with the constant concordance $[0,1]\times W''$ when 
we restrict the background space to $(-\infty,\beta)\times \mathbb{R}^{N-1}$. 

\end{itemize}
 
The desired concordance $\widetilde{W}$ is obtained by concatenating 
$\widehat{W}$ and  $\widehat{W}'$.  Evidently, $\widetilde{W}$ satisfies 
claims (i) and (ii) of Proposition \ref{nestenempty}. Moreover, let
$\{U_{\alpha}\}_{\alpha \in \Lambda}$ be the open cover of $P$ defined by $U_{\alpha}:= \mathrm{Int}\hspace{0.03cm}V_{\alpha}$. Given any $\alpha \in \Lambda$, 
it is straightforward to verify that
$W'_{\lambda} \cap \big( \{ a_1^{\alpha} \} \times \mathbb{R}^{N-1}  \big) = \varnothing$
for all $\lambda \in U_{\alpha}$. Thus, $\widetilde{W}$ also satisfies 
claim (iii). 
\end{proof}

\subsection{Proof of the main theorem} \label{section64}

 Recall that in Lemma \ref{descan.homotopic} we showed that the composition in (\ref{descan})
is homotopic to the scanning map $\mathcal{S}_{k}: |\psi_d(N,k)_{\bullet}| \rightarrow \Omega  |\psi_d(N,k+1)_{\bullet}|$. 
By Remark \ref{monoid.remark} and Proposition \ref{firstmapweak}, the first two maps in (\ref{descan}) are weak homotopy 
equivalences. On the other hand, 
from Proposition \ref{incempty} it follows that the last map in (\ref{descan}) is also a weak equivalence. 
Therefore, we have proven that the scanning map $\mathcal{S}_k$ is a weak homotopy equivalence, as long as we  have $N - d \geq 3$ and $k>1$. This was the claim we made back in Theorem \ref{scanweak}. Using this theorem, we can prove the following by induction. 

\theoremstyle{plain}  \newtheorem{premaintheorem}[scan]{Theorem}

\begin{premaintheorem} \label{premaintheorem}
If $N-d \geq 3$, then the map 
$\widetilde{\mathcal{S}}:|\psi_d(N,1)_{\bullet}| \rightarrow \Omega^{N-1}|\Psi_d(\mathbb{R}^N)_{\bullet}|$
defined in (\ref{prescaneq}) is a weak homotopy equivalence. 
\end{premaintheorem}
 
We will now indicate how to obtain the main theorem of this article from Theorems   
\ref{DesCob} and \ref{premaintheorem}. 
 Let $i_N$ denote the obvious inclusion  
$|\psi_d(N,1)_{\bullet}| \hookrightarrow |\psi_d(N+1,1)_{\bullet}|$. 
By Theorem \ref{premaintheorem}, there is a weak homotopy equivalence
\begin{equation} \label{finalequiv}
\hocolim_{N\rightarrow\infty}|\psi_d(N,1)_{\bullet}| \stackrel{\simeq}{\longrightarrow} \Omega^{\infty-1}\Psi^{\mathrm{PL}}_d,
\end{equation}  
where the homotopy colimit on the left-hand side is
the mapping telescope of the diagram
\begin{equation} \label{sluttdiagram}
\dots \longrightarrow |\psi_d(N,1)_{\bullet}| \stackrel{i_N}{\longrightarrow}  |\psi_d(N+1,1)_{\bullet}|\stackrel{i_{N+1}}{\longrightarrow} |\psi_d(N+2,1)_{\bullet}| \longrightarrow \dots.
\end{equation}
Since each map in this diagram is an inclusion of CW-complexes, we have that the natural
map $F:\hocolim_{N\rightarrow\infty}|\psi_d(N,1)_{\bullet}| \rightarrow |\psi_d(\infty,1)_{\bullet}|$ 
is a weak homotopy equivalence. Moreover, since both  
$|\psi_d(\infty,1)_{\bullet}|$ and the mapping telescope  
of (\ref{sluttdiagram}) are CW-complexes, the map $F$ has a 
homotopy inverse
\begin{equation} \label{finalequivtwo}
|\psi_d(\infty,1)_{\bullet}| \stackrel{\simeq}{\longrightarrow} \hocolim_{N\rightarrow\infty}|\psi_d(N,1)_{\bullet}|.  
\end{equation}
Finally, we obtain the main theorem of this article by combining (\ref{finalequiv}), (\ref{finalequivtwo}) and the  
homotopy equivalence from Theorem \ref{DesCob}.

\theoremstyle{plain} \newtheorem{maintheorem}[scan]{Theorem}

\begin{maintheorem} \label{maintheorem}
There is a weak homotopy equivalence
\[
B\mathsf{Cob}_d^{\mathrm{PL}} \stackrel{\simeq}{\longrightarrow} \Omega^{\infty-1}\Psi^{\mathrm{PL}}_d.
\]
\end{maintheorem}


\begin{thebibliography}{99}               

\bibitem[Br]{JLB} J.L. Bryant, 
\emph{Piecewise linear topology},
Handbook of Geometric Topology, North-Holland, 2002. 
 
\bibitem[Ga]{Ga} S. Galatius, 
\emph{Stable homology of automorphism groups of free groups},
Annals of Mathematics, Volume \textbf{173} (2011), 705-768. 
 
\bibitem[GMTW]{GMTW} S. Galatius, I. Madsen, U. Tillmann, M. Weiss, 
\emph{The homotopy type of the cobordism category},
Acta Mathematica, \textbf{202} (2009), 195-239.

\bibitem[GRW]{GRW} S. Galatius, O. Randal-Williams, 
\emph{Monoids of moduli spaces of manifolds},
Geometry and Topology, \textbf{14} (2010), 1243-1302.


\bibitem[GL]{Gomez2} M. Gomez Lopez,
\emph{The homotopy type of the PL cobordism category. II}.
To appear in Algebraic and Geometric Topology.
arXiv:2308.04612, 2023.

\bibitem[GLK]{GomezKupers} M. Gomez Lopez, A. Kupers,
\emph{The homotopy type of the topological cobordism category},
Documenta Mathematica \textbf{27} (2022), 2107-2182. 
         
\bibitem[Hu]{plhud} J.F.P. Hudson,
\emph{Piecewise linear topology},
W.A. Benjamin, New York, 1969.         

\bibitem[Hu2]{ext} J.F.P. Hudson,
\emph{Extending piecewise linear isotopies},
Proceedings of the London Mathematical Society \textbf{s3-16} (1966), 651-668.

\bibitem[Jo]{Jo} P.T. Johnstone,
\emph{Sketches of an elephant: A topos theory compendium.} Volume 2,
 Clarendon Press, Oxford, 2002. 

\bibitem[KS]{KS} R.C. Kirby, L.C. Siebenmann,
\emph{Foundational essays on topological manifolds, smoothings, and triangulations}, 
Princeton University Press, Princeton, 1977.


\bibitem[KL]{KL} N.H. Kuiper, R.K. Lashof, 
\emph{Microbundles and bundles II},
Inventiones Mathematicae \textbf{1} (1966), 243-259.


\bibitem[RS1]{RS} C.P. Rourke, B.J. Sanderson, 
\emph{Introduction to piecewise linear topology},
Springer-Verlag, New York, 1972.

\bibitem[RS2]{RSD} C.P. Rourke, B.J. Sanderson, \emph{$\Delta$-sets I: Homotopy theory}, Quart. Journal of Mathematics Oxford, Volume 2-22 (1971), 321-338. 
      
\bibitem[Se]{Seg} G. Segal,
\emph{Categories and cohomology theories},
Topology \textbf{13} (1974), 293-312.


\bibitem[Si]{Si} L.C. Siebenmann,
\emph{Deformation of homeomorphisms on stratified sets}, 
 Commentarii Mathematici Helvetici \textbf{47} (1972), 123-163.

\bibitem[Wi]{cobordism} R.E. Williamson Jr.,
\emph{Cobordisms of combinatorial manifolds},
Annals of Mathematics, Volume \textbf{83} (1966), 1-33.

\bibitem[Ze]{relative.iso} E.C. Zeeman, 
\emph{Relative simplicial approximation}, 
Mathematical Proceedings of the Cambridge Philosophical Society \textbf{60} (1964), 39-43.

\end{thebibliography}
\end{document}